\documentclass[12pt]{amsart}

\usepackage{latexsym,enumitem}
\usepackage{amssymb}
\usepackage{epsfig}
\usepackage{psfrag}
\usepackage{amsthm}
\usepackage{amscd}

\usepackage{amsmath}
\usepackage{amsfonts}
\usepackage{graphics,caption}
\usepackage[all]{xy}
\usepackage{etoolbox}
\patchcmd{\quote}{\rightmargin}{\leftmargin 2em \rightmargin}{}{}
\newtheorem{sat}{Theorem}[section]		\newtheorem{lem}[sat]{Lemma}
\newtheorem{kor}[sat]{Corollary}			\newtheorem{prop}[sat]{Proposition}
\newtheorem{bei}{Example}				
		\newtheorem*{bei*}{Example}
\newtheorem*{sat*}{Theorem}				\newtheorem*{kor*}{Corollary}
\newtheorem*{rmk*}{Remark}				\newtheorem{quest}{Question}

\newtheorem{claim}[sat]{Claim}	
\newtheorem*{claim*}{Claim}	
\newtheorem{fact}[sat]{Fact}	
\newtheorem*{fact*}{Fact}	

\let\ssection=\section
\renewcommand{\section}{\setcounter{equation}{0}\ssection}

\newtheorem*{namedtheorem}{\theoremname}
\newcommand{\theoremname}{testing}
\newenvironment{named}[1]{\renewcommand{\theoremname}{#1}\begin{namedtheorem}}{\end{namedtheorem}}

\theoremstyle{remark}
\newtheorem*{bem}{Remark}
\newtheorem{example}[sat]{Example}	

\theoremstyle{definition}
		\newtheorem{defi}[sat]{Definition}
		\newtheorem*{defi*}{Definition}	

\newcommand{\BC}{\mathbb C}			\newcommand{\BH}{\mathbb H}
\newcommand{\BR}{\mathbb R}			\newcommand{\BD}{\mathbb D}
\newcommand{\BN}{\mathbb N}			
\newcommand{\BS}{\mathbb S}			
				\newcommand{\BT}{\mathbb T}

\newcommand{\CE}{\mathcal E}		
		\newcommand{\CH}{\mathcal H}
\newcommand{\CI}{\mathcal I}		
		\newcommand{\CL}{\mathcal L}
		\newcommand{\CN}{\mathcal N}
		
		\newcommand{\CR}{\mathcal R}
		\newcommand{\CT}{\mathcal T}
\newcommand{\CU}{\mathcal U}		\newcommand{\CV}{\mathcal V}

\newcommand{\D}{\partial}

\newcommand{\into}{\hookrightarrow}

\DeclareMathOperator{\PSL}{PSL}		%	Spezielle lineare Gruppe
\DeclareMathOperator{\area}{area}

\DeclareMathOperator{\inj}{inj}
\DeclareMathOperator{\diam}{diam}
\DeclareMathOperator{\rank}{rank}

\DeclareMathOperator{\width}{width}

\DeclareMathOperator{\length}{length}

\begin{document}

\title[]{Thick hyperbolic $3$-manifolds with bounded rank}
\author{Ian Biringer \& Juan Souto}
\begin{abstract}
 We construct a geometric decomposition  of the convex cores of $\epsilon$-thick hyperbolic $3$-manifolds $M$ with bounded rank.  Corollaries include upper bounds in terms of rank and injectivity radius on the Heegaard genus of $M$ and on the radius of  any embedded ball in the convex core of $M$.
\end{abstract}
\maketitle

\tableofcontents
\clearpage

\section{Introduction}

Thurston's Geometrization Conjecture gives topological criteria that determine exactly when a closed $3$-manifold admits a hyperbolic metric.  With its resolution by Perelman \cite {Perelmanentropy,Perelmanfinite,Perelmanricci}, new questions now arise. For instance, Mostow's Rigidity Theorem implies that a hyperbolic metric on a closed $3$-manifold $M$ is unique, if it exists, so it is natural to try to extract concrete geometric information about the metric from the topology of $M$. This program is usually referred to as \emph{effective} geometrization, and has been studied by Brock, Canary, Minsky, Namazi, Souto, e.g.\ \cite{Brockclassification,Brockbounded,Brockvolume,Namaziheegaardbounded}, among others.  

The $\emph{rank}$  of a group is the minimum size of a generating set, and if $M$ is a $3$-manifold we define the \emph{rank} of $M$ to be the rank of its fundamental group. Our main focus here is on the  constraints imposed on the geometry of a hyperbolic $3$-manifold $M$ by its rank, and most of our results will be phrased in terms of understanding the geometry of collections of hyperbolic $3$-manifolds with bounded rank.

It is not hard to construct infinite collections of hyperbolic $3$-manifolds with bounded rank.  For instance, if $\Sigma $  is a closed surface of genus $g$ and $\phi : \Sigma \longrightarrow\Sigma$  is a homeomorphism, then the mapping torus $M_\phi$ has rank  at most $2g+1$, and Thurston~\cite{Thurstonhyperbolic2}  showed that $M_\Sigma$ is hyperbolic as long as $\phi$ is pseudo-Anosov.  Similarly, if $H$ is a genus $g$  handlebody and $\phi : \partial H \longrightarrow \partial H$  is a homeomorphism, then any  manifold $$M=H \sqcup_\phi H$$ obtained by gluing $H$ to itself via $\phi$ (i.e.\ a manifold obtained as a   \emph{Heegaard splitting} from $\phi$) has rank  at most $g$, and here Hempel \cite{Hempel3-manifolds} has given a hyperbolicity condition for $M$  in terms of how $\phi$  interacts with the set of curves on $\partial H$  that bound disks in $H$. See \S \ref{corollariesintro} below for more on the relationship of  rank with Heegaard splittings.

 More generally, suppose that one is given a  finite collection $\mathcal M $ of compact $3$-manifolds with boundary.  One can create closed $3$-manifolds $M $ by gluing  boundedly many copies of these `building blocks' together along their boundary components, and infinitely many closed $3$-manifolds can usually be produced by varying the gluing maps.  However, there is a universal upper bound on the rank of any $3$-manifold $M$ produced from such a gluing. 
And again one can often ensure that $M$ is hyperbolic by choosing the gluing maps to be complicated enough, see Brock--Minsky--Namazi--Souto \cite{Brockbounded} in this general case.

 Our main result is a `converse'  of this general construction for $\epsilon$-thick manifolds. (Recall that  the \emph {injectivity radius} $inj(M)$ of a  hyperbolic $3$-manifold $M$ is  half the length of the shortest closed geodesic, and that $M$ is \emph {$\epsilon$-thick} if $\inj(M)\geq \epsilon$, see also \S \ref{tt}.)   It says that the convex core $CC(M)$ of any $\epsilon$-thick hyperbolic $3$-manifold $M$ with bounded rank decomposes  as a gluing  of manifolds from some fixed finite collection,  and that this decomposition  is geometric, not just topological. Here, recall that  $CC(M)$ is  the smallest convex submanifold of $M$   whose inclusion is a homotopy equivalence,  see \S \ref{convexcores}. 

\begin{named}{Theorem \ref{general case}, informally\footnote{Since many of our theorem statement are somewhat complicated, we will only state informal versions of many of them in the introduction.  The reader who wants to know the details can look for the theorem with the same label later in the paper.
}}
 Fix $k,\epsilon>0$ and let $M$ be an orientable hyperbolic $3$-manifold $M$ such that $CC(M)$ is $3$-dimensional and 
$$\rank(\pi_1(M))\le k, \ \ \inj(M)\ge\epsilon.$$
Then $M$ can be written as a union of  `building blocks' glued together along `thin $1$-handles' and `product regions'. The number of terms in this decomposition is at most some $n(k)$, and each building block has  diameter at most some $B(k,\epsilon)$ and is homeomorphic to some manifold on a finite list  that depends only on $k,\epsilon$.
\end{named}

\begin{figure}
	\centering
\includegraphics{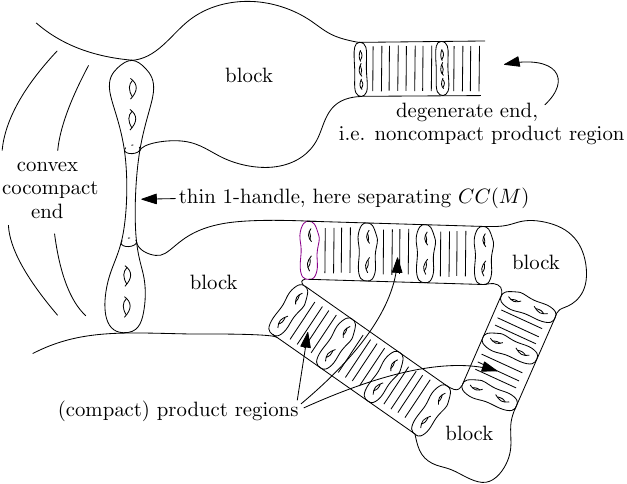}
\caption{An illustration of Theorem \ref{general case}.}\label{mainpic}
\end{figure}
See Figure \ref{mainpic}  for an illustration of the theorem.  The assumption that $CC(M)$ is $3$-dimensional is just for  simplicity in the statement, and if desired, a similar geometric decomposition theorem is trivially true for manifolds with lower dimensional convex cores. Informally,  thin $1$-handles and product regions  should be thought of as follows. 
\begin{enumerate}
\item A \emph {product region}  is a  properly embedded $3$-submanifold $U \subset M $  that is homeomorphic to $\Sigma \times I$,  for some closed surface $\Sigma $ and some closed, open, or half open interval $I$,  in such a way that the fibers $\Sigma \times \{t\}$  have bounded diameter. 
 \item A \emph{thin $1$-handle} is a embedded copy of $D^2 \times [0,1]$ in $ CC(M)$  such that each disk $D^2 \times \{t\}$ has very small diameter, and where $(D^2 \times [0,1]) \cap \partial CC( M)=(\partial D^2) \times [0,1]$, which is a component of the $\delta$-thin part of $\partial CC(M)$ for some small $\delta$.%  
\end{enumerate}
See  Definitions \ref{prdef} and \ref{epsilon handle}, respectively, for the real definitions of the above. In (1), any sufficiently large fixed  diameter bound will give basically the same definition of product region, up to a boundary error, see essentially Proposition \ref{Bsurfaces}. We should mention that  (1)  looks quite a bit different from Definition \ref{prdef}, but one can translate between the two using Lemma \ref{limit-double} and work of Nina White \cite{white2013spectral}, if desired. The terminology in (2) is inspired by Bowditch, who uses such $1$-handles in similar manner in  \cite{bowditch2010upper}. Note that when $M$ is closed, $CC(M)=M$ and has no boundary, so the theorem says that $M$  decomposes as a union of building blocks and product regions.

 The hypothesis that $M$ is $\epsilon$-thick is essential in our proof.  The main reason for this is that the building blocks are constructed from compact cores of  geometric limits of \emph{sequences} of such $M$, see \S \ref{limitsintro} below, and without the  thickness assumption these limits can be quite wild. However, much of the work we do on product regions, and even some parts of the proof  of Theorem \ref{general case},  can be done with minimal modification in the thin case, by working relative to the thin part $M_{< \epsilon}$.  So, with an eye toward future applications, we have written almost all of Sections \S \ref{topologysec}--\S \ref{splitting along barriers}  without any assumption on injectivity radius.

 This paper  is the culmination of our earlier papers \cite{Biringergeometry,Biringerfiniteness,Biringerranks,Soutorank}, which all illustrate  in some sense how rank  interacts with product regions, and \cite{Biringeralgebraic}, in which we prove some technical results concerning limits of Kleinian groups that  are necessary for  Theorem \ref{general case}.

 We should say that at least some parts of the rank $2$ case of our theorem were known to Ian Agol, who proved (but did not make publicly available) the  related fact that there are only finitely many closed $\epsilon$-thick hyperbolic $3$-manifolds $M$ with $\rank \pi_1 M =2$ but Heegaard genus bigger than $2$. See \S \ref{corollariesintro} below for more about Heegaard genus.

\medskip

The rest of this introduction will be as follows. In \S \ref{corollariesintro}, we will state three corollaries of Theorem \ref{general case}. In \S \ref{carrierintro}, we  introduce \emph {carrier graphs}  an important technical tool  we use to understand rank geometrically, and in \S \ref{wideprintro} and \S \ref{limitsintro} we present some results on product regions, and on limits of Kleinian groups, that are used in multiple places in the paper. And finally in \S \ref{msrs}, we will  briefly  mention \emph{minimizing simplicial ruled surfaces}, which appear everywhere in the paper, and for which one can prove elegant versions of the well-known interpolation and  shrinkwrapping theorems, as in Canary \cite{Canarycovering} and Soma \cite{Somaexistence}.

\subsection {Conventions}  For simplicity, we'll assume that \emph{all $3$-manifolds in this paper are orientable}, although some of the results are true without this assumption. Also, all Riemannian manifolds are complete.

\subsection{Some corollaries}
\label{corollariesintro}  We present in this section three corollaries of Theorem \ref{general case}. Their brief proofs can be found in \S \ref{corollaries}.

First,  note that both product regions and thin $1$-handles fiber over an interval, with bounded diameter fibers. As building blocks have entirely bounded diameter, this  shows that $CC(M)$ has the coarse structure of a metric graph, where the product regions and thin $1$-handles are edges, and where the building blocks are vertices.  More formally,

\begin{named}{Corollary \ref{qicor}}
 There is some $C=C(k,\epsilon)$  such that the convex core of any  hyperbolic $3$-manifold $M$ with $$\rank(\pi_1(M))\le k, \ \  \inj(M)\ge\epsilon$$
is $(1,C)$-quasi-isometric to a metric graph  with at most $n=n(k)$ edges, some of which some of which may be half-infinite or biinfinite.
\end{named}

As a corollary of the corollary, one can then prove the following:

\begin{named}{Corollary \ref{embedded balls cor}}
	 Suppose that $M$  is a hyperbolic $3$-manifold  such that $\rank(\pi_1(M))\le k$  and $ \inj(M)\ge\epsilon$. Then  the radius $r_M$ of the largest embedded ball in $CC(M)$ is at most some  function of $k,\epsilon$.
\end{named}

This result  was one of our initial motivations here. It resolves the following 1989 question of McMullen, see \cite{conformalproblemlist}, in the thick case.

\begin{quest}[McMullen]
If $M$ is a hyperbolic $3$-manifold with rank at most $k$, is the radius $r_M$ of the largest embedded ball in $CC(M)$  bounded above by some function of $k$?
\end{quest}

  We should mention that in \cite{bowditch2010upper}, Bowditch gave a  uniform bound for $r_M$ that depends only on the topology of $M$.  However, he does not relate it to any specific  algebraic invariants of $\pi_1 M$.

\medskip

Recall that a \emph {Heegaard splitting} of a  compact $3$-manifold $M$ is a surface $S \subset M $ that divides $M $ into two  compression bodies, $M=C_1 \cup_S C_2$, where $S$ is the exterior boundary of each $C_i$. (See \S \ref{compression bodies} for definitions; when $M$ is  closed, $C_1,C_2$ are handlebodies.) Any compact, orientable $M$ admits a Heegaard splitting, and the minimal genus of such a splitting is the \emph {Heegaard genus} $g(M) $.  Note that $$\rank (M ) \leq g(M) ,$$ since the fundamental group of each  compression body in a Heegaard splitting surjects onto $\pi_1 M$.

In the 1960s, Waldhausen conjectured that the rank and genus of a closed orientable $3$-manifold are always equal.  This became known as the \emph {Generalized Poincar\'e Conjecture}, since the rank zero case is the Poincar\'e conjecture.  In 1984, Boileau-Zieschang \cite{Boileauheegaard} constructed Seifert-fibered $3$-manifolds with rank two and Heegaard genus three, disproving the conjecture, and later Schultens-Weidmann \cite {Schultensgeometric} showed that the difference between rank and genus can be arbitrarily large.  The question of whether rank and genus were equal for \emph {hyperbolic} $3$-manifolds remained open (and active) until Li \cite{Lirank} constructed hyperbolic examples where they differ in 2013.  

On the other hand, in all known examples the Heegaard genus is at most twice the rank.  Considerable effort has been made to either prove or disprove the existence of a linear bound \cite{Abertrank,Lackenbyasymptotic,Lackenbyexpanders,Lackenbyheegaard}, but currently  it is not even known whether the Heegaard genus  of a $3$-manifold is bounded above by any function of  rank. However, using our geometric decomposition  one can prove the following.

\begin {named}{Corollary \ref {genus cor}}
The Heegaard genus $g(M)$ of a  hyperbolic $3$-manifold $M $ is bounded above by some function of $\rank(\pi_1 M)$ and $\inj(M)$.
\end {named}

Here, the Heegaard genus of a noncompact $M$ with finite rank is that of any manifold compactification, which must exist by the Tameness  Theorem \cite{Agoltameness,Calegarishrinkwrapping}.  Note that when $M$  is closed, Corollary \ref{genus cor} actually implies Corollary \ref{embedded balls cor},  using work of  Bachman--Cooper--White \cite{Bachmanlarge}.

\subsection {The role of carrier graphs}\label {carrierintro}  It is crucial in our work to have a tool that  allows one to understand a generating set for $\pi_1 M$  in a way that interfaces with the hyperbolic geometry of $M$. 

A \emph{carrier graph} is a $\pi_1$-surjective map $X \longrightarrow M$, where $X$  is a graph. Any generating set for $\pi_1 M$ can be represented by a carrier graph in $M$, for instance as a wedge of circles. Two carrier graphs $f : X \longrightarrow M $ and $g : Y \longrightarrow M$ 
are \it equivalent \rm if there is a homotopy equivalence $h: X \longrightarrow Y $ such that $g \circ h $ is homotopic to $f $, and a carrier graph has \it minimal length \rm if it minimizes edge length sum over all equivalent carrier graphs.   

In \cite {Whiteinjectivity}, White noticed that minimal length  carrier graphs\footnote{In White's work, he assumes that his carrier graphs have minimal rank, i.e.\ $\rank \pi_1 X = \rank \pi_1 M$. However, the same arguments prove the statements we make in this paragraph. Alternatively, see the appendix of \cite{Biringergeometry}.} have restricted geometry: for instance, they are trivalent with $3\rank (\pi_1 X) - 3 $ geodesic edges that connect at $\frac {2\pi} 3 $-angles.  He used this to show that if $\pi_1 M$ is not free, then $X$ contains a cycle whose image in $M $ has length bounded above by a function of $\rank (\pi_1 X) $.  We extend White's observation as follows, generalizing our earlier work in \cite{Soutorank,Biringergeometry,Biringerranks}.

\begin{named}{Theorem \ref{chainsprop}}
Given $\epsilon>0 $, there is  some  constant $D =D(\epsilon)$ as follows. Let $M$ be a  complete $\epsilon$-thick hyperbolic $3$-manifold  and assume that  the surface $\partial CC(M)$ has  injectivity radius at  least $\epsilon$. If $$f: X \longrightarrow M$$ is a  minimal length carrier graph and $Y \subset X$  is a proper subgraph of $X$, there is an edge $e \subset X \setminus Y$ that spends a length of at most $D$ outside of the `hyperbolic convex hulls of the adjacent components of $f(Y)$'. \end{named}

 See \S \ref{carrier}  for a precise statement; the convex hull indicated above should be taken in the universal cover rather than in $M$. Recall that the surface $\partial CC(M)$ is intrinsically hyperbolic, and that it may be thin even if $M$ is thick, if it is compressible in $M$.

\medskip 

 The basic idea behind  the existence of the thin $1$-handles and   product regions in Theorem \ref{general case}, then, is as follows. Each thin part of $\partial CC(M)$  determines a thin $1$-handle. If we split $M$  along all these $1$-handles,  we can reduce to the case where $\partial CC(M)$ is thick, so that  Theorem \ref{chainsprop} applies.  Take a  minimal length, minimal rank carrier graph $$X \longrightarrow M.$$ Since $\rank(M)$ is bounded and $X$  is trivalent, $X$ has  a bounded number of edges. So, we can divide $X$ into a `short' subgraph $Y \subset X$ and its `long' complement $X \setminus Y$,  both possibly disconnected,  such that \begin {itemize}
 \item 	all edges of $Y$ have  length at most some $L=L(k,\epsilon)$,
\item  all edges of $X \setminus Y$ have length \emph{much} bigger than $L$. 
 \end {itemize}
Theorem \ref{chainsprop}  then says that there is an edge $e \subset X \setminus Y$  that only spends a  bounded  amount of its length outside the convex hulls of  the adjacent components $Y_0 \subset Y$.  Since $e$ is extremely long,  some  such $Y_0$ has a huge convex hull. And since all edges of $Y_0$  have bounded length, this convex hull is much bigger than $Y_0$ itself, which means the cover of $M$ corresponding to  the image of $\pi_1(Y_0) \longrightarrow \pi_1 M$ has a huge diameter convex core, which implies that it  contains a product region. (This is a finite version of the statement that any $\epsilon$-thick hyperbolic $3$-manifold with finitely generated fundamental group and \emph{noncompact} convex core has  a degenerate end. Note that degenerate ends have neighborhoods that are noncompact product regions, see Proposition \ref{constructing prs}.)

  Except for the part about thin $1$-handles, the argument above was essentially  present in our earlier paper \cite{Biringerfiniteness}, where we showed that all but finitely many thick, closed hyperbolic $3$-manifolds with bounded rank contain a wide product region.  In some sense, the kernel of this paper is Lemma \ref{kernel}, which is a similar, but much more involved,  manipulation of minimal length carrier graphs. See also \ref{pfsketch} below, where we outline the proof of Theorem \ref{general case} (and of Lemma \ref{kernel}) in detail.

\subsection{Wide product regions are unknotted} \label{wideprintro}While the kernel of this paper may  involve carrier graph manipulations, much of the surrounding work involves understanding how product regions that are wide relative to their genera can sit inside of $M$.  In particular, it is crucial for us to show that such wide product regions cannot be `knotted' inside $M$. In hopes of later applications, we develop almost all of this theory without the $\epsilon$-thick assumption, using a notion of `width'  relative to the $\epsilon$-thin part of $M$.

Let $M$ be a hyperbolic $3$-manifold, let $\epsilon>0$, and let $U \subset M$  be a product region $U \subset M$, say with $U\cong \Sigma_g \times [0,1]$. The \emph{genus} of $U$ is defined to be the genus $g$ of the surface $\Sigma_g$, and the \emph {$\epsilon$-width} of $U$ is 
$$\width_\epsilon(U) := \inf_{\gamma} \length (\gamma \cap M_{\geq \epsilon}),$$
where the  infimum is over paths $\gamma$ in $U$  that join the two boundary components of $U$, and $M_{\geq\epsilon}$  is the $\epsilon $-thick part of $M$.

A special case of  one of our results is the following:

\begin{named}{Special case of Corollary \ref{unknotted}}
Given $g\geq 2,\epsilon>0$,  there is some $L$ as follows.  Suppose $M$ is the interior of a compression body $\bar M$, and $U \subset M$ is a genus $g$  product region with $\epsilon$-width at least $L$. Then  each component of $\bar M \setminus int(U)$ is a compression body. 
\end{named}

 So for instance, if $M$ is the interior of a handlebody, the surface in Figure \ref{knottedcore} cannot be a level surface in a wide product region.% cannot be a neighborhood of  the boundary of some knotted core in a handlebody, as in Figure \ref{knottedcore}.

\begin{figure}
	\centering
\includegraphics{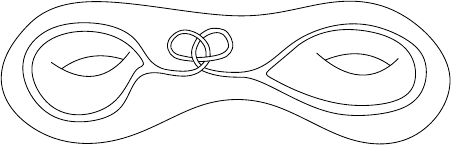}
\caption{A `knotted' genus $2$ surface in a handlebody.}\label {knottedcore}
\end{figure}

 Here is another related result. Given a  product region $U \subset M$, say with $U \cong \Sigma_g \times [0,1]$, we can form the \emph{double compression body} $$\mathcal {DC}(M,U)\subset M$$ from $U$ by  attaching a  maximal collection of disjoint $2$-handles in $M\setminus int(U)$ to the two components of $\partial U$, and  filling in any $2$-sphere boundary components  of the result  with the balls they bound in $M$.

\begin {named}{Theorem \ref{double compression body theorem}}
Given	$g\geq 2,\epsilon>0$, there is $L$ as follows. Suppose $M$ has no cusps, and $U$ has genus $g$ and $\epsilon$-width at least $L$.
\begin{enumerate}
	\item If $\Sigma $ is a  boundary component of $\mathcal {DC}(M,U)$  that is not isotopic to a level surface of $U$, then $\Sigma$ is incompressible in $M$.
\item Either the inclusion $\mathcal{DC}(M,U) \hookrightarrow M$ is $\pi_1$-injective, or  $U$ compresses to one side in $M$ and bounds a twisted interval bundle over a nonorientable surface to the other side.
\end{enumerate}
\end {named}

Part (2)  is the more difficult part of this result.  Now, if $\partial U = S_1 \cup S_2$, one way the double compression body $\mathcal {DC}(M,U)$  could fail to be $\pi_1$-injective is that there could be curves $\alpha_1,\beta_1 \subset S_1$, where $\alpha_1$ compresses in $M \setminus int(U)$  and $\beta_1 $  does not, but where these curves are homotopic in $U$ to $\alpha_2,\beta_2 \subset S_2$, which bound an immersed\footnote{In the exceptional case described in part (2), the annulus goes through the twisted interval bundle and is not embedded --- its two boundary components intersect. Indeed, one can use part (1) to show that when $U$ has large width, no curve that is a boundary component of an embedded annulus in the twisted interval bundle can bound a disk on the other side of $U$.} annulus in $M\setminus int(U)$.  So, $\beta_1$ is nullhomotopic in $M$, but might not be in $\mathcal {DC}(M,U)$. Part (2) says that this can only happen if $S_2$ bounds a twisted interval bundle.

In contrast, suppose that $S$  is the boundary of a regular neighborhood $\CN(K)$ of a nontrivial knot $K \subset S^3$. Then  the  double compression body of (a regular neighborhood of) $S$  can be  defined in the same way as above, and will be isotopic to $\CN(K)$.  However,  the inclusion $\CN(K) \hookrightarrow S^3$  is certainly not $\pi_1$-injective.

\subsection{The structure of strong limits}
\label{limitsintro}
The building  blocks in Theorem \ref{general case} are constructed as compact cores of certain limits of \emph{sequences} of manifolds with bounded rank and injectivity radius. Hence, it is important for us to be able to control the  topology of such limits, and relate their geometry to  that of the approximating manifolds.   

If $\rho_i : \Gamma \longrightarrow \PSL_2\BC$  is a sequence of discrete representations,  one says that $(\rho_i)$ converges {\em algebraically} to a representation $\rho$ if for every $\gamma\in\Gamma$ we have $\rho_i(\gamma)\to\rho(\gamma)$ in $\PSL_2\BC$. And if the images $\rho_i(\Gamma)$ converge in the Chabauty topology  to some subgroup $G<\PSL_2 \BC$, one says that $\rho_i \to G$  \emph {geometrically.} If $(\rho_i)$  converges both algebraically to $\rho$  and geometrically to $G=\rho(\Gamma)$, one says that $\rho_i \to \rho$ \emph {strongly}.  The reader can refer to \S \ref{convergence section} for more details.

The main reason why we consider $\epsilon$-thick manifolds in this paper is that otherwise, cusps may appear in the limit, drastically complicating its structure. For instance, by work of Thurston and Bonahon--Otal~\cite{Bonahonvarieties}, there is a  sequence of  discrete, faithful representations $$\rho_i : \pi_1 S \longrightarrow \PSL_2\BC,$$ where $S$ is a  closed surface of genus $2$, that converges both algebraically and geometrically, but where the geometric limit is a  subgroup of $\PSL_2\BC$  that is not finitely generated.   In contrast, if $\rho_i\to \rho$ is an algebraically convergent  sequence of discrete, faithful surface group representations, and $\rho$ has no parabolics (e.g., if each $\BH^3 / \rho_i(\pi_1 S)$ is $\epsilon$-thick for some uniform $\epsilon>0$) then it is known that $\rho_i\to \rho$ converges strongly. See e.g.\ \cite[Proposition 7.39 and Theorem 7.41]{Matsuzakihyperbolic}.

The structure of algebraic, geometric and strong limits has been heavily studied for sequences of $(\rho_i)$ of \emph{faithful} representations, see for instance \cite{Andersoncores,Andersoncores2,Evansweakly,Ohshikacontinuity}.  However,  in this paper our representations all come from (usually unfaithful) markings of  the  fundamental groups of hyperbolic $3$-manifolds $M$ by free groups that are given by  carrier graphs $f : X \longrightarrow M$. In \cite{Biringeralgebraic}, we studied the algebraic and geometric limits of discrete, unfaithful representations. Combining our work there with our work above on wide product regions gives the following, which is used several times in the proof of Theorem \ref{general case}.

\begin{named}{Theorem \ref{strong limits}, informally}
	Suppose $\Gamma$  is a finitely generated group, $\rho_i : \Gamma \longrightarrow \PSL_2 \BC$  is a sequence of discrete representations, $\rho_i \to \rho_\infty$   strongly, and none of these representations have parabolics.  Set $$M_i = \BH^3 / \rho_i(\Gamma), \ \ i = 1,2,\ldots,\infty.$$
Then  there is a standard compact core $C_\infty \subset M_\infty$ that pulls back to subsets $C_i \subset M_i$ as follows.  Suppose that a component $S \subset \partial C_\infty$ pulls back to components $S_i \subset\partial C_i$. Then for large $i$,
\begin{enumerate}
\item if $S$ faces a degenerate end of $M_\infty$,  the surface $S_i$ is adjacent to a wide product region in $M_i \setminus int(C_i)$,  on the other side of which is a compression body, and 
\item if $S$  faces a convex cocompact end of $M_\infty$,  the surface $S_i $ faces a  convex cocompact end of $M_i$.
\end{enumerate}
\end{named}

See Figure \ref{stronglimpic} for an illustration. Above, a compact core is \emph{standard} if its complementary components are topological products, see \S \ref{cores}.

\begin{figure}
\centering
\includegraphics{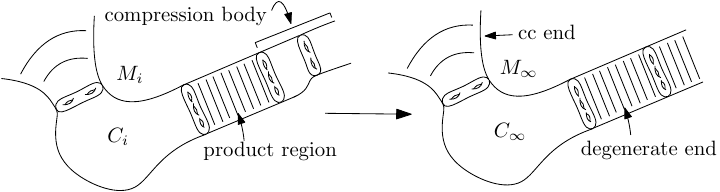}
\caption{An  illustration of Theorem \ref{strong limits}.}
\label{stronglimpic}
\end{figure}

\subsection{Minimizing simplicial ruled surfaces} \label{msrs} In the 1970s, Thurston understood that the geometry of a hyperbolic $3$-manifold is in some sense controlled by the negatively curved surfaces it contains.  Most of his work involved \emph{pleated surfaces}, intrinsically hyperbolic surfaces $S \looparrowright M$ that are totally geodesic except along some lamination in $S$. Recently, many authors have instead turned to the more robust \emph{simplicial  hyperbolic surfaces}, or \emph{simplicial ruled surfaces (SRSs)} when  the underlying manifold has variable negative curvature. These surfaces are no longer intrinsically hyperbolic, but have intrinsic singular metrics of negative curvature. See \S \ref{sec:shs}  for details.

Negatively curved surfaces in $3$-manifolds are useful because their geometry is constrained, and because they are ubiquitous.  Early results of Thurston and Bonahon \cite{Bonahonbouts} construct pleated and simplicial hyperbolic surfaces in hyperbolic $3$-manifolds with great flexibility, and there is a family of results that allows one to  interpolate between negatively curved surfaces in $3$-manifolds with surfaces of the same type.   Perhaps the most well-known of these is due to Canary \cite{Canarycovering}, who proved that  in a hyperbolic $3$-manifold without cusps, any two homotopic, $\pi_1$-injective  simplicial hyperbolic surfaces that are `useful' (see \S \ref{homotopies}) are homotopic through simplicial hyperbolic surfaces. While Canary's argument does not directly apply to variable curvature, Agol  \cite{Agoltameness} gave a slightly weaker interpolation result for SRSs in manifolds with pinched negative curvature  as part of his proof of the Tameness Theorem. Soma \cite{Somaexistence} showed that in the absence of $\pi_1$-injectivity, one can also work with negatively curved surfaces that are `$2$-incompressible' relative to an appropriate geodesic link $\Lambda \subset M $,  extending and simplifying earlier technology of Calegari-Gabai \cite{Calegarishrinkwrapping}. 

For our work in this paper, we found it necessary to develop a more general theory of \emph{minimizing simplicial ruled surfaces} in manifolds with pinched negative curvature. While the theorems mentioned above all require some sort of incompressibility, and are awkward if there are cusps in the manifold, in our setting we prove an Interpolation Theorem (see \S \ref{homotopies}) that works in complete generality, requiring only that the relevant surfaces $S \longrightarrow M$ have nonelementary $\pi_1$-image in $M$. When  $\pi_1$-injectivity is desired for applications, we can find a geodesic link $\Lambda\subset M$ such that $S$ is $\pi_1$-injective in the complement $M\setminus \Lambda$, endow $M\setminus \Lambda$ with  a metric of pinched negative curvature, and work simply with simplicial ruled surfaces in $M\setminus \Lambda$  homotopic to $S$.  We find this new theory to be elegant, and more powerful than its predecessors. For instance, using our theory one can give a clean proof of a strong version of Canary's Filling Theorem in variable negative curvature, as well as an associated convex core filling theorem, see \S \ref{filling thm sec}.

 Briefly,  the idea behind minimizing simplicial ruled surfaces is  that we start with a triangulated surface $S$ in which the edges of the triangulation have been assigned positive weights and a homotopy class of maps $S \longrightarrow M$, find a map  in the right homotopy class from the $1$-skeleton of $S$ to $M$ that minimizes  the \emph{weighted}  sum of the edge lengths, and then extend the map to all of $S$ via ruled triangles.  The result  is unique, at least on the $1$-skeleton, and will be a simplicial ruled surface. To  interpolate  between two such maps, we just take  a common refinement of the two triangulations of $S$ and take any interpolation from the induced system of weights  supported on the edges of the first, to that on the edges of the second. See \S \ref{homotopies}.
 
There is one technical point we should mention. In previous references, a simplicial ruled surface $S \longrightarrow M$ is required to map every triangle in $S$ to a nondegenerate geodesic triangle in $M$. Our theory requires that we allow triangles to degenerate. So, a given triangle in $S$ can map to a point or a geodesic segment in $M$. When $S \longrightarrow M$ is not $\pi_1$-injective, this allows entire essential subsurfaces of $S$ to map to a point or a segment, so the image of $S$ actually looks like a graph of surfaces and edges. Still, these more general degenerate surfaces have many of the geometric bounds that one is used to having in the nondegenerate case. And while the possibility of degeneration complicates some of the notation and requires some extra care to treat, after the basic facts are established the theory is quite easy to use.

\subsection {Plan of the paper}

 In the next section, \S \ref{pfsketch}, we give a detailed outline of the paper  that includes informal statements of many of the main results.  Sections \ref{topologysec} and \ref{negcurved} furnish some necessary background on $3$-manifold topology and  the geometry of manifolds with pinched negative curvature, respectively. In \S \ref{sec:shs}, we discuss some familiar background on simplicial ruled surfaces, and  show how to prove variants of  some of the standard related theorems using minimizing simplicial ruled surfaces.  Section \ref{sec: WPR} sets up the machinery necessary to manipulate wide product regions in $M$, and  we understand how products regions can compress in \S \ref{Double compression bodies}. Sections \ref{carrier} and \ref{convergence section} flesh out the material on carrier graphs and strong limits presented above. The  actual proof of Theorem \ref{general case} is contained in Sections \ref{splitting along barriers}--\ref{section general}; the reader should refer to the outline in \S \ref{pfsketch}  to understand how the individual sections fit into the proof.

\subsection{Acknowledgements}

We would like to start by thanking the referee, who sacrificed an enormous amount of their time to help this paper get published. It is hard to express how deeply we appreciate their efforts, which have greatly improved the paper.  We would also like to thank Ian Agol, Jason DeBlois and Mujie Wang for helpful conversations.  The first author was partially supported by NSF grant DMS-1611851 and CAREER Award DMS-1654114.

\section{An outline of the proof}  \label{pfsketch}

 In this section, we sketch how to construct  the geometric decomposition in Theorem \ref{general case}. Fix $k,\epsilon>0$ and assume that $M$ is an $\epsilon$-thick hyperbolic $3$-manifold  with $\rank(\pi_1 M) \leq k$.  We want to construct a collection of at most $n=n(k)$ pairwise disjoint thin $1$-handles and product regions in  the convex core $CC(M)$  such that the complementary components have bounded diameter, and only boundedly many  topological types. (Everywhere below, `bounded' means bounded above by some constant depending only on $k,\epsilon$.)

As mentioned in \S \ref{carrierintro},  the first step is to  construct the thin $1$-handles, one for each component of the thin part of $\partial CC(M)$. By passing to the cover of $M$  corresponding to some  component of their complement in $CC(M)$, we  can reduce Theorem \ref{general case}  to the case that $\partial CC(M)$ is thick. This  reduction is done in \S \ref{section general}. So, assuming from now on that $\partial CC(M)$ is thick, we want to show that there is a collection of at most some $n=n(k)$  product regions in $CC(M)$ such that the complementary components have bounded diameter and   boundedly many topology types. (This is stated in \S \ref{the proof} as  Theorem \ref{main}.)

 Ideally, one would now just take any  maximal collection of product regions in $M$, and show that all the complementary components have bounded diameter and topology. This kind of statement is true after the fact, but it is hard to make this into a proof of Theorem \ref{general case}, primarily because there are no a priori rank bounds for the complementary components.  So while we will at some point below say `take a maximal collection of product regions', we need to first decompose $M$ enough so that we can understand the  topology of the result.

\smallskip

\noindent \it Step 1, splitting along barriers, see \S \ref{splitting along barriers} and \S \ref{the proof}. \rm
 A \emph{barrier}  of a product region $U\subset M$  is a  boundary component $$S \subset \partial \mathcal {DC}(M,U)$$ of  the double compression body of $U$ that is  incompressible in $ M$.  Note that by Theorem~\ref{double compression body theorem} (1),  stated above in \S \ref{wideprintro}, if $U$ is sufficiently wide then the only way that a component $S \subset \partial\mathcal {DC}(M,U)$ can be  compressible in $M$ is if it is isotopic to a boundary component of $U$, and the other boundary component of $U$ is compressible in  $M\setminus int(U)$.

 The reason we consider barriers is that while the decomposition of $M$ given by an  arbitrary collection of product regions is hard to control, the splitting of $\pi_1 M$ given by a union $\Sigma$ of  disjoint, pairwise non-isotopic barriers  of  sufficiently wide product regions in $M$ is  almost always \emph {$7$-acylindrical}, i.e.\ there is no annulus in $M$  that is a concatenation of $7$ essential annuli in $M \setminus \Sigma$.  This is Lemma \ref{acylindrical}---the $7$ is not optimal, see \S \ref{splitting along barriers}. Work of Weidmann~\cite{Weidmannrank} on acylindrical splittings then allows us to construct and control a splitting of $M$ along a `maximal' collection of barriers.

\begin{named}{Theorem \ref{maxsplitting}, informally}
	For some $k'=k'(k)$, there is a union $\Sigma$ of at most $15k$  pairwise disjoint and non-isotopic barriers of wide genus-at-most-$k'$ product regions,  such that $\Sigma$ is `maximal',  in the sense that even for some huge $K>>k'$, any barrier of a sufficiently wide genus-at-most-$K$ product region is isotopic to  a component of $\Sigma.$  Furthermore, every component of $M\setminus \Sigma$ has rank bounded above by a function of $k$.
\end{named}

\begin {figure}
\centering
\includegraphics{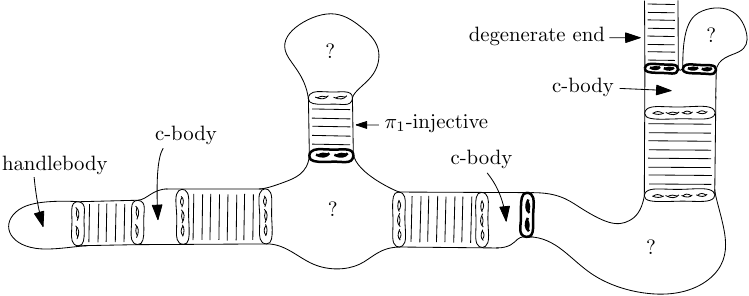}
\caption{The  $\Sigma$ in  Theorem \ref{maxsplitting}  is in heavy ink,  and each component of $\Sigma$ is  a barrier of one of the pictured product regions.}
\label {splittingpic}	
\end {figure}

See Figure \ref{splittingpic}  for an illustration. The interdependencies of the constants in the actual statement of Theorem \ref{maxsplitting} are a bit complicated. For instance, we start by fixing a function $K : \BN \longrightarrow \BN$ and then we choose $k'$ depending on $K,k$, and then the $K$ referenced in the informal statement above is really $K(k')$. 

\smallskip

The proof is inductive. We start with $\Sigma=\emptyset$, and look for any barrier that comes from a product region with genus at most some huge bound, and width at least some even larger constant---it turns out that any such collection of barriers can be realized disjointly in $M$. We  then increase the genus upper bound and the width lower bound,  and continue.  The process terminates because of the  universal upper bound (depending on $k$,  and discussed in the previous paragraph) for the number of  disjoint,  non-isotopic barriers one can place in $M$.

\medskip

\noindent \it Step 2, reducing to the case where barriers are peripheral, see \S \ref{the proof}. \rm
 Let $\Sigma$ be as in  Theorem \ref{maxsplitting}. For each component $N \subset M \setminus \Sigma$, let $$M^N\longrightarrow M$$  be the cover corresponding to the subgroup $\pi_1 N \subset \pi_1 M$. We show:

\begin{named}{Theorem \ref{main}, informally}
	One can  construct a geometric decomposition for $CC(M)$ from geometric decompositions for each  $CC(M^N)$. 
\end{named}  

 In other words, if for \emph{every} $M^N$,  the convex core $CC(M^N)$  admits a  collection of disjoint product regions with bounded complementary components, these  collections will project down and union over $N$ to a similar geometric decomposition for $M$.

Now any sufficiently wide product region in $M^N$ projects  down to give a wide product region in $N$, and any barrier of a genus-at-most-$K$ product region in $N$ is isotopic  to  one of the  adjacent components of $\Sigma$, by the maximality condition in Theorem \ref{maxsplitting}.  So, any barrier of a genus-at-most-$K$  product region in $M^N$ 
 is \emph{peripheral},  meaning that it bounds  a product neighborhood of an end of $M^N$.  
In other words, after replacing $M$ with $M^N$, we can assume from now on that

\begin{enumerate}
	\item[($\star$)] for some huge $K$, all  barriers of  sufficiently wide, genus-at-most-$K$ product regions in $M$ are peripheral in $M$.
\end{enumerate}

\medskip

\noindent \it Step 3, finding short carrier graphs, see \S \ref{shortgraphssec}. \rm  The  reason that  $(\star)$ makes  the problem of finding geometric decompositions tractable is the following, which is in some sense the core result of the paper.

\begin{named}{Proposition \ref{no barriers implies}, informally}
When $M$ satisfies ($\star$),  the manifold $M$   admits a carrier graph $f : X \longrightarrow M$ whose rank is at most some function of $k$, and whose total length is bounded.  
\end{named}

Remember that when unqualified, `bounded' means at most some constant depending only on $k,\epsilon$.  The proof of Proposition \ref{no barriers implies} is  what we called in \S \ref{carrierintro}  the kernel of the entire paper.  The idea is to  start with an  arbitrary minimal rank,  minimal length  carrier graph $$X\longrightarrow M,$$ and to perform a sequence of surgeries on $X$ to eliminate all  extremely long edges. Each surgery can increase the rank of $X$ by some constant factor, but will strictly decrease the total number of `long' edges, so that the total number of steps is at most the number of long edges in the original $X$. Therefore, the final carrier graph will still have rank at most some constant depending on $k$.

 Formally, the proof of Proposition \ref{no barriers implies} is a massive contradiction argument.  We assume that there is no length bound that works, take a sequence of counterexamples $(M_i)$ to larger and larger bounds, and do surgeries simultaneously on all $M_i$ while passing to an extraordinary number of subsequences.   Instead of adopting that perspective here, we instead explain how to perform a single surgery on a minimal length carrier graph $X \longrightarrow M$ that reduces the number of long edges. 

Let $X^{sh} \subset X$ be  the subgraph consisting of all `short' edges, i.e.\ those with bounded length. Let's assume $X^{sh}\neq X$, since otherwise $X$ has bounded length and we are done. In \S \ref{carrierintro} above, we mentioned Theorem \ref{chainsprop}, says that there is  some long edge $e \subset X \setminus X^{sh}$  that has only a bounded amount of length outside of the hyperbolic convex hulls $CH(Y_i)$ of the adjacent components $Y_0,Y_1 \subset X^{sh}$.   For simplicity, let's assume that $e$ spends all but a bounded amount of its (large) length inside one of these, say $CH(Y_0)$.   Replace $e$ by a  geodesic segment with the same endpoints that is distance minimizing (in the cover corresponding to $\pi_1 Y_0$). If $e$ is still long, it must be that $CH(Y_0)$ is huge.  One can then show that the image of $\pi_1 Y_0 \longrightarrow \pi_1 M$ contains the fundamental group of some  wide product region $$U \subset M$$ with genus at most $\rank(\pi_1 Y_0)$ that the edge $e$ traverses soon after it leaves $Y_0$. %and that $e$  expends all but a bounded amount of its length  inside of $U$.
By $(\star)$, all barriers of $U$ are peripheral in $M$. Together with Theorem \ref{double compression body theorem}, this implies that some component $$C \subset M \setminus int(U)$$ is a compression body  with missing interior boundary. (So, the exterior boundary is  the adjacent component of $\partial U$.)
 
Now, $C\cup U$  is also a compression body with missing interior boundary. Create a new carrier graph $X'$ from $X$ by  chopping off the part of $X$ that lies in $C \cup U$ and  replacing it with a bounded length, minimal rank carrier graph for the surface $\partial (C \cup U)$, as in Figure \ref{surgerypic}. Since 
$$\partial (C \cup U) \hookrightarrow C \cup U$$
is $\pi_1$-surjective, this $X'$  is still a carrier graph for $M$. Since $$genus(U) \leq \rank \pi_1 Y_0,$$ the number of edges in $X'$ is at most a  constant factor times the number in $X$. One can show that Theorem \ref{chainsprop} still applies, so it  remains only to show that the number of long edges strictly decreases.

 For simplicity, let's just assume that $Y_0 \cap U = \emptyset$. There are now two cases, depending on whether $Y_0 \subset C$ or not.  In the former case, one can show (see Corollary \ref{extending product regions}) that the product region $U$ extends all the way out to the boundary of $CH(Y_0)$, so that $e$ expends all but a bounded amount of its length inside of $C\cup U$. Hence, $e$ is no longer long in $X'$. If $Y_0 \not \subset C$, then  the fact that $e$ is distance minimizing implies that $e$ does not exit $C \cup U$ after it enters $U$, so again at most a bounded subsegment of $e$  is left in $X'$.

\begin{figure}
\centering
\includegraphics{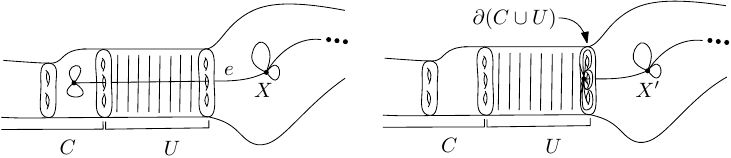}
\caption{Before and after shots of the surgery.}
\label{surgerypic}
\end{figure}
\medskip

\noindent \it Step 4, constructing the   geometric decomposition assuming the existence of a short carrier graph, see \S \ref{when all barriers}. \rm In light of  Proposition \ref{no barriers implies}, we can now assume that  there is a carrier graph $$X \longrightarrow M$$ with rank at most some constant  depending on $k$, and with total length bounded  in terms of $k,\epsilon.$  The rest of the proof is a(nother) big contradiction argument. We assume that we have a sequence $M_i$ of  such manifolds, each containing a bounded rank, bounded length carrier graph $X_i$, and that for each possible diameter bound $B$,  and each  finite collection $\mathcal M$ of  compact $3$-manifolds, there is some $M_i$ that does not admit a geometric decomposition  in which the building blocks have diameter at most $B$ and are  homeomorphic to elements of $\mathcal M$.  Hoping for a contradiction, we want to show that after passing to a subsequence, the $M_i$ do all admit  geometric decompositions  in which the building blocks  have bounded diameter and assume  only finitely many topological types,  where `bounded' and `finite'  are over all $i$.

\medskip

\begin{minipage}{4.5in}
	\center \emph {For simplicity, we assume  here that $CC(M_i)=M_i$ for all $i$. So, we are looking for geometric decompositions of the $M_i$ themselves into product regions and building blocks,  instead of decompositions of their convex cores. The proof is essentially the same in the general case.}
\end{minipage}

\medskip

\medskip

We first construct in each $M_i$ a `central' building block $$N^c_i\subset M_i$$  that lies near $X_i$.  Since the $X_i$  all have bounded rank,  we can pass to a subsequence in which every $\pi_1 X_i $  can be identified with some fixed free group $F$. We then use the carrier graphs $X_i$ to mark the  fundamental groups of the $M_i$,  producing associated  discrete representations 
$$\rho_i :  F \longrightarrow \PSL_2 \BC, \ \ \BH^3 / \rho_i( F ) = M_i.$$
 Using  what is essentially a theorem from our earlier paper \cite{Biringeralgebraic}, see  Proposition \ref{limits-fg}  in this paper,  after possibly doubling the rank of the free group $F$  and passing to a subsequence, we can assume that $(\rho_i)$  converges strongly to some $\rho_\infty$.  Theorem \ref{strong limits} then implies that there is some compact core $C_\infty$ for the limit $$M_\infty := \BH^3 / \rho_\infty(F)$$ that  pulls back to a compact set $N_i^c \subset M_i$ that is bounded by product regions and  neighborhoods of convex cocompact ends of $M_i$. This $N_i^c $ will have diameter bounded above independently of $i$,  since  for large $i$ it is $2$-bilipschitz image of the fixed compact set $C_\infty $. See Figure \ref{centralcomp}  for an illustration.

\begin{figure}
	\centering
\includegraphics{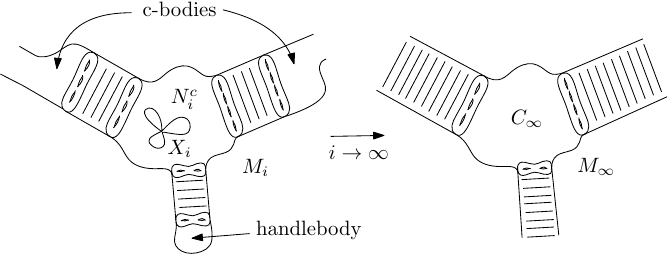}
\caption{We construct the `central'  building block $N_i^c$ around the carrier graph $X_i$  by pulling back the core of a geometric limit.}
\label {centralcomp}
\end{figure}

 So, we have now  constructed the building block $N_i^c$  and its  collection of adjacent product regions, which we call $\mathcal U_i^c$, say.  The idea now is that by Theorem \ref{strong limits},  the parts of $M_i$  that lie on the other sides of  the product regions $U \in \mathcal U_i^c$ are all compression bodies.  To further divide up all of these compression bodies into  bounded diameter building blocks and product regions, we just extend $\mathcal U_i^c$  to any \emph {maximal}  collection of  pairwise disjoint, non-parallel wide product regions $\mathcal U_i$ in $M_i \setminus N_i^c$. By Theorem \ref{unknotted}, any sufficiently wide product region in a compression body divides it into two smaller compression bodies. So,  the components
$$ N_i \subset M_i \setminus \cup_{U \in \mathcal U_i} U$$  consist of $N_i^c$ and  a number of  bounded rank compression bodies,  as  illustrated in  Figure \ref{nonhakenpic}. 
These components will be our building blocks. Note that the topological type of $N_i^c$ is fixed, independent of $i$, so there are a finite number of topological types  in total.  It remains to show that all these components  have bounded diameter.

 So,  fixing a sequence $N_i$ of  such components,  we want to show that $$\sup_i \diam N _i < \infty.$$ 
 To do this,  we use another  limiting argument. We know the $N_i^c$ have bounded diameter, so we can assume $N_i\neq N_i^c$ for all $i$, in which case $N_i$ is a compression body that is bounded in $M$ by wide product regions. We can assume that the widths of  these product regions go to infinity with $i$. (If not, just throw out any product region in $\mathcal U_i$  that has bounded length, and look at complementary components of the result.)   By arguments similar to those we used to produce the $N_i^c$, one can now show that there is a geometric limit of the $M_i$, with base points chosen within $N_i$, that is a tame hyperbolic $3$-manifold with  only degenerate ends. Pieces of the degenerate ends pull back to product regions in $M_i$, which must be  the product regions adjacent to $N_i$, since the collection $\mathcal U_i$ was maximal. Hence, a  compact core in the limit pulls back to a submanifold of $M_i$  that contains $N_i$,  showing that $\sup_i \diam N_i < \infty.$

\begin{figure}
	\centering
\includegraphics{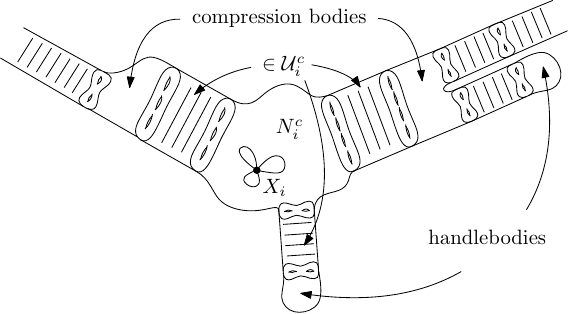}
\caption{The collection $\mathcal U_i^c$  of product regions adjacent to $N_i^c$  extends to a maximal collection $\mathcal U_i$.}
\label {nonhakenpic}
\end{figure}

\section{Some topology}\label {topologysec}
A compact (orientable) $3$-dimensional manifold $M$ with possible non-empty boundary $\D M$ is said to be {\em irreducible} if every embedded sphere $\BS^2$ bounds a ball, and does not contain embedded real projective planes $\BR P^2$. A compact, connected, orientable, properly immersed surface $S\neq\BS^2,\BD^2,\BR P^2$ in $M$ is {\em $\pi_1$-injective} if the homomorphism $\pi_1(S)\longrightarrow\pi_1(M)$ is injective. A $\pi_1$-injective embedded surface is said to be {\em incompressible}.  At the other extreme, a properly embedded surface $S\subset M$ is {\em geometrically compressible} if there is an embedded disk $D\subset M$ such that $\D D=D\cap S$ is a homotopically essential curve in $S$; a surface which is not {\em geometrically compressible} is said to be {\em geometrically incompressible}. By the loop theorem, a two-sided embedded surface is incompressible if and only if it is geometrically incompressible. On the other hand, a one-sided geometrically incompressible surface does not need to be $\pi_1$-injective. Observe that the orientability of $M$ implies that a one-sided surface has to be non-orientable. An irreducible and atoroidal 3-manifold $M$ has {\em incompressible boundary} if every properly embedded disk $(\BD^2,\D\BD^2)$ in $(M,\D M)$ is properly homotopic into $\D M$. By the discussion above this is equivalent to the assumption that every component of $\D M$ is $\pi_1$-injective.
%An irreducible 3-manifold $M$ is {\em atoroidal} if every embedded incompressible torus is homotopic to the boundary. We will be mostly interested in 3-manifolds which don't contain any incompressible tori at all. 

\begin{lem}[Freedman-Hass-Scott, Bonahon]\label {getembedded}
Assume $f: S\longrightarrow M $ is a $\pi_1 $-injective map of a closed orientable surface into a compact, orientable irreducible $3$-manifold $M$, and that $f$ is homotopic to a embedding.  Then given any neighborhood $U$ of $f(S) $, one can find an embedding $g :S \longrightarrow  U $ that is homotopic to $f $ within $M $.
\end{lem}

A word is in order about the attributions.  Freedman-Hass-Scott \cite[Theorem 5.3]{FreedmanIleast} proved that any least area surface in a compact Riemannian $3$-manifold is either embedded or double covers an embedding, and then in our case can be perturbed to be an embedding.  Bonahon \cite{Bonahonbouts} was the first to observe that one can apply their results to a metric on $M $ that is blown up outside of a small neighborhood of $f (S)$, and that the least area surface that is produced lies nearby. See also \cite[Theorem 2.5]{Canarylimits} and \cite{Kapovichhyperbolic}, where it is  explained that the technicalities that arise in this argument can be simplified by working with a combinatorial version of `least area' instead of the usual Riemannian one.

\subsection{Interval bundles}
\label{intervalbundle}
An {\em interval bundle} over a closed surface $S$ is a compact, orientable and irreducible manifold $M$ homeomorphic to the total space of a $[0,1]$-bundle over $S$. Observe that the bundle has to be trivial if the surface $S$ is orientable. If not, the bundle has to be twisted and $\D M$ is connected. 

\begin{named}{Waldhausen's cobordism theorem}[see \cite{Waldhausenirreducible}]
Let $M$ be an irreducible and orientable 3-manifold. 
\begin{itemize}
\item Assume that $S,S'\subset M$ are closed, disjoint, embedded and incompressible surfaces in $M$. If $S$ and $S'$ are homotopic then they bound a trivial interval bundle in $M$.
\item Assume that $S\subset M$ is a closed, embedded and incompressible surface in $M$ and 
$$f:(S\times[0,1],S\times\{0,1\})\longrightarrow(M,S)$$
is a map whose restriction to each component of $S\times\{0,1\}$ is a homeomorphism onto $S$. Either $f$ is homotopic as a map of pairs to some map $f'$ with $f'(S\times[0,1])\subset S$, or one of the connected components of $M\setminus S$ is homeomorphic to a possibly twisted interval bundle.
\end{itemize}
\end{named}

Waldhausen's cobordism theorem is not explicitly stated in Waldhausen's paper \cite{Waldhausenirreducible} but it follows easily from the results therein and is known to all experts in the field.

\subsection{Fibers}
An orientable 3-manifold $M$ {\em fibers over the circle with fiber of genus $g$} if it is homeomorphic to the total space of a $\Sigma_g$-bundle over $\BS^1$. Some of these bundles admit orientation preserving, fixed-point free involutions mapping fibers to fibers in such a way that the induced homeomorphism of $\BS^1$ is orientation reversing. The quotient of the bundle under such an involution is a manifold which {\em fibers over the orbifold $\BS^1/\langle z\mapsto\bar z\rangle$}. The {\em  regular fibers} have genus $g$ and there are two {\em singular fibers} homeomorphic to the non-orientable surface with euler-characteristic $1-g$. Observe that while a connected 3-manifold $M$ without boundary fibers over the circle if and only if it is obtained from two trivial interval bundles by identifying boundary components, $M$ fibers over $\BS^1/\langle z\mapsto\bar z\rangle$ if and only if it is constructed by identifying the boundaries of two twisted interval bundles. %We will say that $M$ {\em fibers} if it does either over $\BS^1$ or over $\BS^1/\langle z\to\bar z\rangle$. 

%An immersed surface $S$ in $M$ is a {\em virtual fiber} if there are a finite cover $M'\to M$ to which $S$ lifts as an embedded surface, and a fibration of $M'$ on a 1-dimensional orbifold such that $S$ is a fiber of that fibration; notice that every virtual fiber is $\pi_1$-injective. Abusing terminology we will often say that the subgroup group $\pi_1(S)$ of $\pi_1(M)$ is also a virtual fiber. Finally, we will refer to the cover of $M$ associated to $\pi_1(S)$ as a virtual fiber as well.

\subsection{Compact cores and ends}\label{cores}
Let $M$ be an orientable, irreducible noncompact 3-manifold with $\pi_1 M$  finitely generated.  

\begin{defi}\label{deficompactcore}
A \emph{compact core} for $M$ is a compact submanifold $C\subset M$ such that the inclusion is a homotopy equivalence. 
\end{defi} 

This definition is common in papers on hyperbolic geometry, but is atypical from a topologist's perspective. Namely, in $3$-manifold topology, a compact core is usually defined to be a compact submanifold $C \subset M$ whose inclusion is a $\pi_1$-isomorphism. This is strictly weaker than requiring the inclusion to be a homotopy equivalence, since deleting a ball from the interior of a compact core as in Definition \ref{deficompactcore} gives a submanifold whose inclusion is still a $\pi_1$-isomorphism, but where $\pi_2$ is nontrivial. However, we have the following well known fact:

\begin{fact}\label{corefact}
	If $C \subset M$ is irreducible and the inclusion is a $\pi_1$-isomorphism, then $C \subset M$ is a compact core.
	\end{fact}
\begin{proof}
$M$ is a $K(\pi,1)$. Indeed, the universal cover $\tilde M$ is also irreducible (c.f.\ Hatcher \cite[Theorem 3.15]{hatcher2007notes}), so $\pi_2(\tilde M)=0$ by the Sphere Theorem, and then since $H_k(\tilde M)=0$ for $k\geq 3$, the Hurewicz theorem and Whitehead's Theorem imply that $\tilde M$ is contractible. If $C$ is irreducible, it is also a $K(\pi,1)$ by the same argument. Hence, if the inclusion $C \hookrightarrow M$ is a $\pi_1$-isomorphism, it is a homotopy equivalence.
\end{proof}

Scott \cite{Scottcompact} showed that any $M$ as above admits a compact submanifold $C\subset M$ such that the inclusion $C\hookrightarrow M$ is an isomorphism on $\pi_1$. Such a $C$ may not be irreducible, but every $S^2$ component of $\partial C$ bounds a ball in $M\setminus C$, and if we add those balls to $C$, the result is irreducible, and hence a compact core by the fact above. %e.g.\ deleting a ball from any given $C$ doesn't change that the inclusion is a $\pi_1$-isomorphism. However, by irreducibility of $M$, any $S^2$ component of $\partial C$ bounds a ball in $M\setminus C$, which we can add to $C$, so that now $C$ has no $S^2$-boundary components. Any such $C $ is irreducible\footnote{If not, there is a sphere $S \hookrightarrow C$ that doesn't bound a ball. Since $M$ is irreducible, $S$ bounds a ball in $M$, and this ball must contain some boundary component of $C$ that maps $\pi_1$-trivially into $M$, and hence also into $C$. However, the inclusion of a boundary component of $C$ into $C$ can only be $\pi_1$-trivial if the boundary component is a sphere. For instance, if $X \subset \partial C$ is a $\pi_1$-trivial boundary component, use the loop theorem to find an embedded disk in $C$ with boundary a nonseparating curve on $X$. Any loop on $X$ that intersects this disk once should be essential in $C$.}, so is a compact core by Fact~\ref{corefact}.
 
\medskip

Recall that when $X$ is a topological space, an \emph{end} of $X$ is defined to be an element of the inverse limit $$e(X):=\varprojlim \pi_0(X\setminus K),$$ where $K$ ranges over compact subsets of $X$.

\begin{fact}[Ends $\leftrightarrow$ core boundary components]\label{oneend}
With $M$ as above, let $C \subset M$ be a compact core, and let $S $ be a component of $\partial C$. Let $E_S \subset M\setminus int(C)$ be the component that has $S$ as a boundary component. Then 
$S = \partial E_S$, and $E_S$ has one end, $\CE_S \in e(E_S) \subset e(M)$. Hence, 
$$\pi_0(\partial C) \longrightarrow e(M), \ \ 
 S \longmapsto \CE_S$$
is a bijection.
\end{fact}

Using the notation of the fact, we'll say that a component $S \subset \partial C$ \emph{faces} the corresponding end $\CE_S \in e(M)$.

Strangely, we're not sure where to find a reference for Fact \ref{oneend}, although it is assumed everywhere in the literature. Briefly, the idea of the proof is as follows. First, $S=\partial E_S$, since if $E_S$ shares more than one boundary component with $C$, then one can construct a loop in $M$ that intersects $S$ once, which cannot be homotoped into $C$. For the other assertion, let $K\subset C$ be the compact submanifold obtained by removing all characteristic compression bodies of components of $\partial C$, as defined further down in \S \ref {compression bodies}. Then $K \subset M$ has incompressible boundary, and the inclusion map induces a bijection between components of $M\setminus int(C)$ and components of $M\setminus int(K)$ that preserves the number of ends a component has. One can show that for any component $E \subset M\setminus int(K)$, we have $H_2(E,\partial E)=0$, c.f.\ the first paragraph of the proof of Lemma 5.6 in \cite{Calegarishrinkwrapping}. Poincar\'e-Lefschetz duality then implies that if we look at cohomology with compact supports, we have $H_c^1(E)=0$. And then it follows from Proposition 5.1 of Scott-Wall \cite{scott1979topological} and the long exact sequence immediately above its statement that $E$ has one end.

\medskip

We'll need the following uniqueness theorem for convex cores.

\begin{fact}[McCullough-Miller-Swarup uniqueness]\label{coreuniqueness}
If $C,C'$ are compact cores for $M$, there is a homeomorphism $f: C\longrightarrow C'$ that is homotopic to the inclusion $C \hookrightarrow M$, and where if $S \subset \partial C$, we have that $S$ and $f(S)$ face the same end of $M$. 
\end{fact}
\begin{proof}
This follows from the proof of Theorem 2 of \cite{mccullough1985uniqueness}. Indeed, they show that $C,C'$ are both obtained up to isotopy by attaching collars and $1$-handles to a fixed compact submanifold $K\subset M$, and they then build their homeomorphism so that it fixes $K$ and maps $1$-handles to $1$-handles. This map satisfies the desired condition on ends, and they are able to show that the induced map on $\pi_1$ is the same as that of $\iota_{C'}^{-1} \circ \iota_C$, where the $\iota$'s are the obvious inclusion maps. Since $C,C',M$ are all $K(\pi,1)$ spaces, $f$ is homotopic to the inclusion.
\end{proof}

In general, the ends of $M$ may be topologically wild. However, the noncompact $M$ that appear in this paper will all be `tame' in the following sense.

\begin{defi}
A $3$-manifold $M$ is \emph{tame} if it is homeomorphic to the interior of a compact $3$-manifold $\bar M$. 
\end{defi}

A compact core $C \subset M$ is {\em standard} if $M\setminus int(C)$ is homeomorphic to $\D C\times [0,\infty)$. Note that $M$ admits a standard compact core if and only if it is tame, and in that case $C\subset M$ is standard if and only if the induced embedding $C\longrightarrow\bar M$ is isotopic to a homeomorphism. For a nonstandard example, one can take a knotted embedding of  a solid torus into itself that induces an isomorphism on $\pi_1$.

The following is well known, c.f.\ the proof of Theorem 1 in \cite{mccullough1985uniqueness}.

\begin{lem}\label{unique-compact}
If $M=int(\bar M)$ and $\bar M$ has incompressible boundary, then every compact core of $M$ is standard.
\end{lem}

Here is an application of compact cores that we'll use below. 

\begin{lem}[Pushing homotopies into submanifolds]\label{pushinghomotopieslem}
	Suppose that $M$ is an orientable, irreducible $3$-manifold, $N\subset M$ is a compact, connected irreducible $3$-submanifold with boundary, and the inclusion $ N \hookrightarrow M$ is $\pi_1$-injective. If $S$ is a locally path connected space and $f_0,f_1 : S \longrightarrow N$ are homotopic in $M$, then $f_0,f_1$ are also homotopic in $N$.
\end{lem}

It is necessary to assume that $ N \hookrightarrow M$ is $\pi_1$-injective here. For instance, if $T \subset \BR^3$ is a solid torus, then loop in $T$ generating $\pi_1 T$ is homotopic in $\BR^3$ to a loop in $T$ that is homotopically trivial in $T$.

\begin{proof}
If $\hat M\longrightarrow M$ is the cover corresponding to $\pi_1 N$, then $N$ lifts homeomorphically to a submanifold $\hat N \subset \hat M$ such that the inclusion is an isomorphism on $\pi_1$. Since $N$ is irreducible, $\hat N$ is a compact core for $\hat M$, by Fact \ref{corefact}. Take a homotopy $(f_t)$ between $f_0,f_1$, and lift it to a homotopy $(\hat f_t)$ in $\hat M$, where $\hat f_0,\hat f_1 $ map into $\hat N$. (We assume $S$ is locally path connected so that we can apply the lifting property of covering maps.) Pick a homotopy inverse $h : \hat M \longrightarrow \hat N$ to the inclusion. Then $(h \circ \hat f_t)$ is a homotopy in $\hat N$ from $h\circ \hat f_0 \sim f_0$ to $h\circ \hat f_1 \sim f_1$, where $\sim$ indicates homotopy in $\hat N$. In other words, $\hat f_0,\hat f_1$ are homotopic in $\hat N$, so projecting down, we get that $f_0,f_1$ are homotopic in $N$.
\end{proof}

\subsection{Compression bodies}\label {compression bodies}
A {\em compression body} is a compact, orientable and irreducible 3-manifold $C$ which has a boundary component $\partial_{e}C$ such that the homomorphism 
$$\pi_1(\partial_{e}C)\longrightarrow\pi_1(C)$$ 
is surjective. Here, $\partial_{e}C$ is called the {\em exterior boundary} of $C$. The remaining boundary components form the {\em interior boundary} $\partial_{int} C$. A {\em handlebody} $C$ is a compression body with empty interior boundary. A compression body is {\em trivial} if $\partial_{e}C$ is incompressible; equivalently,  $C$ is homeomorphic to the trivial interval bundle $\partial_{e}C\times[0,1]$. Observe that twisted interval bundles are not compression bodies.

\begin{lem}\label {euler decrease} If $C$  is a nontrivial compression body  that is not a solid torus, we have $\chi(\partial_{int} C) > \chi(\partial_{e} C) $.
\end{lem}
\begin {proof}
Regard $C$  as being obtained from $S=\partial_{e} C$  by compressing some collection of  disjoint,  essential and pairwise non-homotopic simple  closed curves $a_1,\ldots,a_n$ on $S$, see e.g.\ \cite{Biringerautomorphisms}. This means that we attach $2$-handles to $S\times\{0\} \subset S \times [0,1]$ along annular neighborhoods of the curves $a_1,\ldots,a_n$, and then fill in any $2$-sphere boundary components that we make with balls.

To construct the interior boundary $\partial_{int} C$  from such a description, cut  $S$ open along all the curves $a_1,\ldots,a_n$, cap off all resulting boundary components with discs, and throw away spheres.  So, $$\chi(\partial_{int}C) = \chi(\partial_{e} C) +2n - 2s,$$  where $s$  is the number of spheres. Each sphere corresponds to a component of $S \setminus a_1 \cup\cdots\cup a_n$ that is a planar surface.  Since the $a_i$ are essential, no complementary component can be a disk. Since no two $a_i\neq a_j$ are homotopic, the only way  a complementary component can be an annulus is if the two boundary components of the annulus are actually the same curve $a_i$,  in which case $S$  is a torus and $C$ is a solid torus. Hence,  we can assume that each planar surface that is a component of $S \setminus a_1\cup\cdots\cup a_n$ has at least three boundary components. 

Now, a given $a_i$ is adjacent to  at most two planar surfaces. So, if $X$  is the number of pairs $(P,a_i)$, where $P$ is one of  these planar surfaces and $a_i$ is a component $\partial P$,  it follows that
\[3s \leq X \leq 2n, \ \ \implies n>s.\qedhere\]
\end {proof}

The  following fact will also be useful.

\begin{fact}[see, e.g.\  Corollary 5.22 \cite{johnson2006notes}]\label{incincomp}
	Let $C$  be a compression body and let $S \subset C$ an embedded, closed, connected, incompressible surface. Then $S$  is isotopic to a component of $\partial_{int} C$.
\end{fact}

Note that when $C$ is a trivial compression-body, i.e.\ a trivial interval bundle, Fact \ref{incincomp} says that any incompressible surface in $C$ is a level surface. This is an earlier result of Waldhausen \cite[Prop 3.1]{Waldhausenirreducible}.

Let $F$ be a boundary component of a compact, orientable and irreducible 3-manifold $M$. The {\em characteristic compression body} of $F$ is a submanifold $C \subset M$ with $F\subset\D C$ satisfying the following conditions:
\begin{itemize}
	\item $C$ is a compression body with exterior boundary $F$, and
	\item each interior boundary component of $C$ is incompressible in $M$.
\end{itemize}

\begin{prop}[Bonahon \cite{Bonahonbouts}]\label{prop:rel-compbody}
Let $M$ be a compact orientable irreducible $3$-manifold and $S$ be a boundary component of $M$. Then $S$ has a characteristic compression body, and any two are isotopic in $M$ relative to $S$. Moreover, the characteristic compression bodies of different boundary components of $M$ can be isotoped to be disjoint. \end{prop}

  We will now denote  the characteristic compression body of $S \subset \partial M$ by $C(M,S)$. Here is a way to recognize $\partial_{int} C(M,S)$.

\begin{fact}\label {compressions isotopies}
 Suppose that $M$  is a compact, orientable, irreducible $3$-manifold and $S$ is a boundary component of $M$. Then up to isotopy, the components of $\partial_{int} C(M,S)$  are exactly the incompressible surfaces in $M$ that are not spheres and are obtained from $S$ by a sequence of isotopies and compressions.  Finally, if $C \subset M$  is any compression body with $\partial_e C = S$, then  every component of $\partial_{int} C(M,S)$  can be obtained by a sequence of isotopies and compressions from a component of $\partial_{int} C$.
\end{fact}

Here, a  surface $T'$ is  obtained from $T$ by a \emph{compression} if $T'$ is isotopic to  a boundary component of a regular neighborhood of $T \cup D$, where $D$ is an embedded disc in $M$  such that $\partial D \subset T$ and $int(D) \cap T =\emptyset.$

\begin{proof}
 It is a standard fact that the components of $\partial_{int} C(M,S)$  can  be obtained from $S$ by compressions and isotopies, see e.g.\ \cite[\S 2]{Biringerautomorphisms}.  Similarly, if $C \subset M$  is a compression body with $\partial_e C = S$, then $C$ is isotopic to a subcompression body of $C(M,S)$, and  the same logic implies that every component of $\partial_{int} C(M,S)$  can be obtained by a sequence of isotopies and compressions from a component of $\partial_{int} C$. 

 For the converse, let $M_S$ be  the cover of $M$  corresponding to the subgroup $\pi_1 S$. Then $C(M,S)$ lifts to a  submanifold of $M_S$, whose  inclusion is a homotopy equivalence. Since $\partial_{int} C(M,S)$  is incompressible in $C(M,S)$, it  follows from Lemma \ref{unique-compact} that $M_S \setminus C(M,S)$ is homotopic to  $\partial_{int} C(M,S) \times\BR$. In  particular, $M_S$  admits a compactification that is naturally homeomorphic to $C(M,S)$. 

Any  finite sequence of isotopies and compressions starting with $S$ lifts  homeomorphically to $M_S$. If  the resulting surface is incompressible in $M$, its lift is incompressible in $M_S$, so if it  not a sphere it is isotopic to a component of $\partial_{int} C(M,S)$ by  Fact \ref{incincomp}.
\end{proof}

%Using the characteristic compression body, we can prove a variant of Waldhausen's Cobordism Theorem that applies to \emph{compressible} surfaces in the boundary of a manifold. This must exist somewhere in the literature, but we are not sure where to find it.
%
%\begin{kor}\label{compressiblewaldhausen}
%	Let $M $ be an irreducible and orientable $3$-manifold, and that $S,S'$ are boundary components of $M$. If $S,S'$ are homotopic in $M$, then either $S=S'$, or $M \cong S \times [0,1]$.
%\end{kor}
%\begin{proof}
%If $S,S'$ are both incompressible, we are done by Waldhausen's Cobordism Theorem. So, suppose $S$ is compressible, and let $C \subset M$ be its characteristic compression body. Pick a simple closed curve $\gamma$ on $S$ that is not homotopic into any component of the interior boundary of $C$. Then $\gamma$ is homotopic to a curve $\gamma'$ on $S'$. It follows from the Annulus Theorem, c.f. \cite[Theorem 1]{cannon1976essential}, that $\gamma,\gamma'$ bound an embedded annulus $A \subset M$. Any closed curve $\alpha \subset A \cap \partial_{int} C$ that is inessential in $A$ is also inessential in $\partial_{int} C$, since the latter is incompressible, and therefore we can perform innermost disk surgeries to get rid of all such intersections. So, every curve of $A \cap \partial_{int} C$ generates $\pi_1 A$, implying that $\gamma$ is homotopic into $\partial_{int}C$, a contradiction.
%\end{proof}

\section{Negatively curved manifolds}\label {negcurved}
The main theorems in this paper conccern hyperbolic 3-manifolds. However, manifolds of variable negative curvature will be used as tools in the proofs. In these section we remind the reader of a few facts about such manifolds. A few of these facts can be found in the literature only for hyperbolic manifolds; the proofs remain exactly the same and we will only add a few remarks needed to clarify some minor differences.
\medskip

Throughout this section assume that $M$ is a complete orientable Riemannian $3$-manifold with pinched negative sectional curvature 
$$\kappa^-\le\kappa_M\le\kappa^+<0.$$
Below, we will refer to $\kappa^\pm$ as the pinching constants. Recall that $M$ is {\em hyperbolic} if its sectional curvature is $-1$ everywhere; equivalently, $M$ is locally isometric to hyperbolic space $\BH^3$.

By the Cartan-Hadamard theorem, the universal cover of $M$ is diffeomorphic to $\BR^3$. This implies that $M$ is irreducible and that its fundamental group has no torsion. In particular, $\pi_1(M)$ is infinite unless $M$ is simply connected. Those manifolds $M$ as above with abelian fundamental group are well-understood and we will assume from now on and without further mention that $\pi_1(M)$ is not abelian.

\subsection{Thin-thick decomposition}\label{tt}
The \emph{injectivity radius} $\inj_x(M)$ of $M$ at the point $x$ is half of the length of the shortest homotopically essential loop in $M$ through $x$, and the \emph{injectivity radius of $M$} is 
$$\inj(M)=\inf\{\inj_x(M) \ \vert \ x\in M\}.$$
Given a small $\epsilon>0$, we will denote by $M_{<\epsilon}$ the set of points in $M$ with injectivity radius less than $\epsilon$ and refer to it as the {\em $\epsilon$-thin part} of $M$. The  {\em $\epsilon$-thick part} $M_{\geq\epsilon}=M\setminus M_{<\epsilon}$ is its complement.  The manifold $M$ is called \emph{$\epsilon$-thick} if $\inj(M)\geq \epsilon$, or equivalently if $M_{\geq\epsilon}=M$. 

The Margulis Lemma \cite{Benedettilectures} asserts that there is some $\mu>0$, called {\em Margulis constant}, depending only on the pinching constants $\kappa^\pm$, with the property that for every positive $\epsilon<\mu$, every connected component of the $\epsilon$-thin part $M_{<\epsilon}$ has abelian fundamental group. For hyperbolic $M$, i.e.\ when $\kappa^\pm=-1$, Meyerhoff \cite{meyerhoff1987lower} has estimated that $\mu\ge0.1$.

Suppose that $\epsilon>0$ is smaller than the Margulis constant. The non-compact components of $M_{<\epsilon}$ are called {\em $\epsilon$-cusp neighborhoods} or just {\em cusp neighborhoods} when $\epsilon$ is understood from the context. Since $M$ is orientable and $3$-dimensional, a cusp $U$ is homeomorphic to either $S^1 \times \BR \times (0,\infty)$ or $S^1 \times S^1 \times (0,\infty)$. In the first case we say that $U$ has \emph{rank 1}, while in the second case we say that $U$ has \emph{rank 2}.

 The manifold $M$ is said to have {\em hyperbolic cusps} if there is some $\epsilon>0$ such that the curvature is constant $-1$ on all the $\epsilon$-cusp neighborhoods. Throughout the paper, we will only encounter this kind of manifold. 

\subsection{Distances rel the thin part}\label {relative distances}  Fix some $\epsilon >0$ and define the \it  $\epsilon $-length \rm of a path $\gamma$ in $M$ as $$\length_\epsilon(\gamma)= \length ( \gamma \cap M_{\geq\epsilon} ). $$
If $x,y\in M$, the \emph{$\epsilon $-distance} $d_\epsilon(x,y)$ is the infimum of $\epsilon$-lengths of paths from $x$ to $y$. The \emph{$\epsilon$-distance} between two sets $A,B \subset M$ is the infimum of $d_{\epsilon}(a,b)$, where $a\in A$ and $b\in B$, and the \emph{$\epsilon$-diameter} of a set $A \subset M$ is the supremum of $d_\epsilon(x,y),$ where $x,y\in A$. See also our earlier paper \cite{Biringerranks},  in which we called $\epsilon$-distances \emph{`$\epsilon$-electric distances'}.

The following lemma shows that changing $\epsilon $ results in coarsely equivalent  $\epsilon $-distances.  Namazi proved this lemma for hyperbolic $3$-manifolds, but the argument works in pinched negative curvature under the obvious modifications.

\begin{lem}[Namazi, Lemma 5.2 \cite{Namaziquasiconvexity}] 
\label {differentepsilons}
Given $\epsilon_2 > \epsilon_1 > 0$ less than the Margulis constant, there exists a
constant $C = C (\epsilon_1,\epsilon_2, \kappa^ -,\kappa^ +) \geq 0 $ such that for every Riemannian $3$-manifold $M $ with pinched negative curvatures $\kappa^ - \leq K_M \leq \kappa^ + < 0 $ and $x,y \in M $ we have
$$d_{\epsilon_2 } (x, y) \leq  d_{\epsilon_1 } (x, y) \leq (C + 1) d_{\epsilon_2 } (x, y) + C.$$
\end{lem}

Here is a well known lemma about distances modulo the thin part that will be useful later. In particular, it implies that a bounded length loop homotopic into a cusp or Margulis tube in $M$ lies at bounded distance from that cusp or tube.

\begin{lem}[{compare with \cite [Lemma 2.1]{Namaziquasiconvexity}}]\label {distancetogeodesic}
Let $\epsilon > 0$, and $\gamma$ be a homotopically essential loop in $M $.  Then the following estimates hold.
\begin{enumerate}
	\item If $\gamma$ is homotopic to a geodesic $\gamma^* $, then $$ \cosh \Big (\sqrt { | \kappa^ +| } d(\gamma, T)\Big ) \leq \frac {\length (\gamma)} {\epsilon }, $$
where $T=\gamma^*$ if $\length \gamma^* \geq \epsilon$, and otherwise $T$ is the component of $M_{<\epsilon}$ that is the Margulis tube containing $\gamma^*$.
\item If $\gamma$ is parabolic, then
 $$ \cosh \Big (\sqrt { | \kappa^ +| } d(\gamma,U)\Big ) \leq \frac {\length (\gamma)} {\epsilon },$$
where $U$ is the non-compact component of the $\epsilon$-thin part of $M$ (i.e.\ the cusp neighborhood) into which $\gamma$ is homotopic.\end{enumerate} \end{lem}

Note that the two inequalities above can be interpreted as giving bounds on the $\epsilon$-distance from $\gamma$ to either the homotopic geodesic $\gamma^*$ or the corresponding cusp of $M$, depending on whether $\gamma$ has hyperbolic or parabolic type.

The key to Lemma \ref{distancetogeodesic} is the following well known estimate.

\begin{fact}[Exponential decrease of path length under projection]
\label{expdecrease}Suppose that $X$ is a simply connected $n$-manifold with sectional curvatures at most $\kappa^+<0$ and $C \subset X$ is convex. Let $\pi : X \longrightarrow C$ be the closest point projection, and let $\gamma $ be a path in $X \setminus \CN_r(C)$, where $r>0$. Then $$\length(\gamma) \geq \cosh(\sqrt{|\kappa^+|} \cdot r) \length(\pi\circ \gamma).$$
\end{fact}
\begin{proof}
Take two points $y,y' \in X \setminus \CN_r(C)$ and let $x,x'$ be their projections onto $\partial C$. We first claim that
\begin{equation}
\sinh \big (|\sqrt{\kappa^+}| d(y,y')/2\big ) \geq \cosh(\sqrt{|\kappa^+|} \cdot r) \sinh \big (|\sqrt{\kappa^+}| d(x,x')/2\big ).\label{saccheri}
\end{equation}
To prove this, consider the geodesic quadrilateral with vertices $x,x',y,y'$. The angles at $x,x'$ are both at least $\pi/2$, since otherwise there would be closer points to $y,y'$ on the geodesic from $x$ to $x'$, which lies inside $C$. Take a diagonal from $x$ to $y'$, say, subdividing the quadrilateral into two triangles, and construct  comparison triangles with the same side lengths in the $2$-dimensional model space $\BH^2_{\kappa^+}$ of curvature $\kappa^+$, arranged side-by-side to form a quadrilateral with vertices $\bar x,\bar x',\bar y,\bar y'$.

 Since $X$ is CAT($\kappa^+$), by \cite[pg 161, Prop 1.7 (3)]{Bridsonmetric} the angles at $\bar x,\bar x'$ are both at least $\pi/2$. Hence, the distance between $\bar y$ and $\bar y'$ is at least what it would be if both those angles were $\pi/2$, and if the distances from $\bar x$ to $\bar y$ and from $\bar x'$ to $\bar y'$ were both $r$. Such a quadrilateral is called a \emph{Saccheri quadrilateral}, and scaling the usual side length formula in $\BH^2$ then gives \eqref{saccheri}.

To prove the fact, write $\gamma : [0,1] \longrightarrow X$, let $0=t_0<\cdots <t_n=1$ be a fine partition of the unit interval, set $x_i = \pi\circ \gamma(t_i)$ and $y_i = \gamma(t_i)$, and apply \eqref{saccheri} above. As the partition becomes arbitrarily fine, the polygonal approximations $$\sum_i d(x_i,x_{i+1}) \to \length(\pi \circ \gamma), \ \ \ \sum_i d(y_i,y_{i+1}) \to \length(\gamma).$$ Using \eqref{saccheri} and the fact that $\sinh'(0)= 1$, for any $\lambda<1$, when the partition is sufficiently fine we have
$$ d(y_i,y_i') \geq \lambda \cosh(\sqrt{|\kappa^+|} \cdot r) d(x_i,x_i'),$$
so the Fact follows.\end{proof}

We now prove Lemma \ref{distancetogeodesic}.

\begin{proof}
Let's assume we're in in case (1), for simplicity, as (2) is similar. Suppose also that the length of our geodesic is less than $\epsilon$ for concreteness, since the other case is less complicated. After lifting everything to the universal cover and replacing the original curves by appropriate lifts, we assume we have the following data:
\begin{enumerate}
	\item an isometry $g: \tilde M \longrightarrow \tilde M$ stabilizing a geodesic $\gamma^*$,
\item a path $\gamma : [0,1] \longrightarrow \tilde M$ such that $g(\gamma(0))=\gamma(1)$,
\item the metric neighborhood $T \supset \gamma^*$ consisting of all $x\in \tilde M$ such  that $d(x,g(x))\leq \epsilon$.
\end{enumerate} 
We want to show that $\length(\gamma) \geq \epsilon \cdot \cosh( \sqrt{|\kappa^+|} d(\gamma,T) )$. Note that this $T$ is not quite the lift of the Margulis tube in the statement of the lemma, because Margulis tubes also take into account translation distances of powers of $g$. However, the $T$ in (3) is \emph{contained} in the lift of the Margulis tube, so it suffices to use it in the desired inequality. 

As $T \subset X$ is convex, we have
$$\length(\gamma) \geq  \cosh( \sqrt{|\kappa^+|} d(\gamma,T) ) \length(\pi \circ \gamma)$$ by Fact \ref{expdecrease}. But $$\length(\pi \circ \gamma) \geq d(\pi(\gamma(0)),\pi(\gamma(1)))$$
and the latter is the distance between a point on $\partial T$ and its $g$-translate, which is at least $\epsilon$, so we're done.
\end{proof}

 Here is a useful consequence of the above.

\begin{kor}\label {distancetogeodesic2}
	 Suppose that $G$ is a finite graph, $p \in G$ is a vertex, and $g: G \hookrightarrow M$ is a continuous map such that $g_*(\pi_1 M)$ is not elementary parabolic.  For every $L>0$, let $K$ be
 the set of points $x\in M$ for which there is a map $f: G \longrightarrow M$ that is homotopic to $g$, such that $f(p)=x$ and $\length f(G) \leq L.$ Then $K$ is compact.
\end{kor}

For instance, set $G$ to be a loop that maps to a hyperbolic type element $\alpha \in \pi_1 M$.  The corollary says that the set of points $x\in M$ through which there is a loop of length at most $L$ freely homotopic to $\alpha$ is compact. One can also apply the corollary when $G$ is a wedge of two circles both of which map to parabolics in $\pi_1 M$, as long as the parabolics have different fixed points at infinity. (Note that `the fixed point' on $\partial \tilde M$ of an element of $\pi_1 M$ is not well defined unless we fix an identification of $\pi_1 M$ with an appropriate deck group acting on $\tilde M$, but two group elements `having the same fixed point' is well defined.)

\begin {proof}
Since $g_*(\pi_1 M)$ is not elementary parabolic, there is some loop $\alpha$ on $G$ such that $g_*(\alpha) \in \pi_1 M$ has hyperbolic type. If $f : G \longrightarrow M$ is as in the statement of the corollary, then the length of $f(\alpha)$ is bounded by some constant depending on $\alpha$ and $L$, but not on $f$.  So, part 1 of Lemma \ref{distancetogeodesic} implies that $f(\alpha)$ lies within a bounded neighborhood of the homotopic geodesic, or its containing Margulis tube. Since $\length(G)\leq L$, the point $f(p)$ also lies in a slightly bigger bounded neighborhood of this geodesic or tube. Since tubes are compact, we are done.\end {proof}

\subsection{Ends of negatively curved 3-manifolds}\label {sec:ends}
In this section we describe the geometry of the ends of $3$-manifolds with pinched negative curvature. We start with the following theorem.  

\begin{named}{Tameness Theorem}[Agol, Calegari-Gabai]
Let $M$ be a complete Riemannian manifold with pinched negative curvature and hyperbolic cusps. If $M$ has finitely generated fundamental group, then $M$ is tame, i.e.\ it is homemorphic to the interior of a compact $3$-manifold.
\end{named}

Agol \cite{Agoltameness} and Calegari--Gabai \cite{Calegarishrinkwrapping} gave simultaneous proofs of the tameness theorem for hyperbolic manifolds. Agol's argument is written directly in the context of non-positively curved manifolds with hyperbolic cusps, and one can also apply the outline of Calegari-Gabai's argument in the variable curvature case.

\medskip

Let $C\subset M$ be a compact core, see \S \ref{cores}. As described in Fact \ref{oneend}, the ends of $M$ are in bijective correspondence with components of $\partial C$. An end $\CE$ of $M$ {\em has no cusps} if  it has a neighborhood whose intersection with every component of $M_{<\epsilon}$ is bounded. Here, $\epsilon$ is positive and smaller than the Margulis constant. We now describe how to classify ends of $M$ with no cusps geometrically. 

The {\em convex-core} of $M$, written $CC(M)$, is the smallest closed convex\footnote{Here, `convex' means that any geodesic segment with endpoints in $CC(M)$ is contained in $CC(M)$. However, given that the inclusion into $M$ is a homotopy equivalence, this is equivalent to saying that any arc in $CC(M)$ is homotopic rel endpoints to a geodesic in $CC(M)$, which is the other definition of convex you might consider for non-simply connected spaces.}  subset of $M$ whose inclusion into $M$ is a homotopy equivalence. This is well-defined unless $M$ is simply connected, in which case we set $CC(M)=\emptyset$. An end $\CE$ of $M$ is called \emph {convex-cocompact} if it has a neighborhood disjoint from the convex core $CC(M)$. An end $\CE$ of $M$ facing $S \subset \partial C$ is {\em degenerate} if there is a sequence $f_i : S \longrightarrow M \setminus C$ of simplicial ruled surfaces (see \S \ref{sec:shs}) that are homotopic to $S$ within $M\setminus C$ and where the images $f_i(S)$ exit the end $\CE$. See Bowditch \cite{bowditch1995geometrical} for a dive into some of the subtleties of these definitions in negative curvature.

The following is a corollary of the Tameness Theorem and earlier work of Hou \cite{hou2003critical}, following Canary \cite{Canaryends} and Bonahon \cite{Bonahonbouts}.

\begin{named}{Geometric Tameness Theorem}
Assume that $M$ is a complete Riemannian $3$-manifold with pinched negative curvature, finitely generated fundamental group and hyperbolic cusps. Every end $\CE$ of $M$ that has no cusps is either convex-cocompact or degenerate.
\end{named}

Specifically, a lemma of Bonahon \cite{Bonahonbouts} shows that any end $\CE$ that has no cusps and is not convex-cocompact as defined above admits an exiting sequence of closed geodesics, i.e.\ a sequence of closed geodesics $(\gamma_i)$ where every neighborhood of $\CE$ contains all by finitely many of the $\gamma_i$. In the terminology of Hou \cite{hou2003critical}, this means $\CE$ is geometrically infinite, and then the main result of that paper implies that $\CE$ is degenerate as defined above. We note that while technically, Hou assumes throughout the paper that $M$ has no cusps, this is just for simplicity of the statements, and all that is needed in the argument is that $\CE$ has no cusps. Hou's result is a modified version of an earlier argument of Canary \cite[Theorem 4.1]{Canaryends}; the main trick in Canary's paper is to show that a certain branched cover of a hyperbolic manifold has a metric of pinched negative curvature, and Hou shows that one can also do this if the original manifold has variable negative curvature.

For later convenience, we note this property of ends without cusps.

\begin{fact}[$\epsilon$-distance is proper on ends without cusps]\label{fact: dist proper}
	Suppose that $M$ is a complete Riemannian $3$-manifold with pinched negative curvature $\kappa^- \leq \kappa \leq \kappa^+ < 0$ and hyperbolic cusps, $\CE$ is an end of $M$ with no cusps, and $\epsilon$ is less than half the $\kappa^+$-Margulis constant. Then if $(p_i) $ is a sequence that exits $\CE$, and $p\in M$, we have $d_\epsilon(p_i,p) \to \infty$. 
	
	In particular, if we fix $R>0$, let $(p_i)$ be a sequence of points that exits $\CE$, and let $B_i \subset M$ be subsets such that $d_\epsilon(x,p_i)\leq R$ for all $x\in B_i$, then $(B_i)$ also exits $\CE$. 
\end{fact}

Here, a sequence of subsets $(X_i)$ in $M$ \emph{exits} $\CE$ if any neighborhood of $\CE$ contains all but finitely many $X_i$. 

\begin{proof}
Since $\CE$ has no cusps, there is a neighborhood $U$ of $\CE$ such that every component of $M_{<\epsilon}$ that intersects $U$ is compact. Taking $U$ to be a component of the complement of a compact set, we can assume that the frontier $\partial U \subset M$ is compact. For the first part of the Fact, it suffices to show that for each $L>0$, the set of all $x\in U$ such that $d_\epsilon(x,\partial U) \leq L$ is bounded in $M$. 
	
It's helpful to formulate this inductively. We claim that for each $n$, the set of all $x\in U$ such that there is a path $\gamma$ from $\partial U$ to $x$ with $\epsilon$-length at most $L$ that only intersects $n$ components of $M_{<\epsilon}$ is at most some $B_n<\infty$. The proof is easy. The $n=0$ case is trivial, and assuming the $n$ case, suppose $\gamma$ is a path as above in the $n+1$ case. By the induction hypothesis, the point at which $\gamma$ enters the last component $T \subset M_{<\epsilon}$ it intersects lies at most distance $B_n$ from $\partial U$. The set of all $T$ at distance $B_n$ from $\partial U$ is finite, and each is compact, so there is an upper bound $D$ for all their diameters. It follows that the entire path $\gamma $ lies at most $B_{n+1} := B_n + D + L$ from $\partial U$. 

As long as $\epsilon$ is less than half the Margulis constant, the distance between any two components of the $\epsilon$-thin part is at least $\epsilon$: indeed, the injectivity radius function is $1$-lipschitz, so the distance between $M_{<\epsilon}$ and $\partial M_{<2\epsilon}$ is at least $\epsilon$, and disjoint components of $M_{<\epsilon}$ are contained in disjoint components of $M_{<2\epsilon}$. So, a path of $\epsilon$-length $L$ in $U$ can intersect at most $L/\epsilon$ distinct components of $M_{<\epsilon}$. It follows from the inductive result in the previous paragraph that the set of all $x\in U$ such that $d_\epsilon(x,\partial U) \leq L$ is bounded in $M$. 

For the last statement of the Fact, it suffices to show that for any $U$ as above, the sets $B_i$ are eventually contained in $U$. For this, just note that for large $i $, the points $p_i$ are contained in $U$ and are at large $\epsilon$-distance from $\partial U$. So, the sets $B_i$ are also in $U$.
 \end{proof}

Also, we note the following well-known fact about cusps of PNC $3$-manifolds, just to reassure the reader that it is still true in variable negative curvature.

\begin {fact}[Cusps generated by simple loops] \label{simplecusp} Suppose $M$ is a complete Riemannian $3$-manifold with pinched negative curvature and hyperbolic cusps, and $\epsilon$ is less than the Margulis constant. Suppose that $C \subset M$ is a compact core and $T$ is an $\epsilon$-cusp of $M$ that goes out the end of $M$ facing a component $S \subset \partial C$. Then $\pi_1 T$ is generated by a loop that is homotopic in $N$ to a {simple} closed curve on $S$. 

\end{fact}
\begin{proof}
Let $N_0 =N \setminus T$. Then $\partial N_0$ is an open annulus. McCullough's relative core theorem \cite{mccullough1986compact} says that there is a compact core $K \subset N_0$ that intersects $\partial N_0$ in a compact annulus $A$. Then $\pi_1 T$ is generated by a loop homotopic to the core curve of $A$, and hence to a simple closed curve $\gamma$ on some component $\Sigma \subset \partial K$.
Since the inclusion $N_0 \hookrightarrow N$ is a homotopy equivalence, $K$ is a compact core of $N$ as well. Fact \ref{coreuniqueness} says that there is a homeomorphism $h : C \longrightarrow K$ homotopic to the inclusion $C \hookrightarrow N$ such that $h(S)$ faces the same end as $S$. It follows that $\Sigma=h(S)$, so $\gamma$ is homotopic to a simple closed curve on $S$.\end{proof} 

\subsection{Geometry of the convex core  boundary}\label {convexcores}
Suppose now that $M$ is a complete Riemannian $3$-manifold with finitely generated fundamental group, pinched negative curvature $$\kappa^-\le\kappa\le\kappa^+<0,$$ and where $M$ has hyperbolic cusps, all of which have rank $2$.

 %Also, for simplicity in exposition we will assume throughout the section that $CC(M)$ is $3$-dimensional. If not, $CC(M)$ is either empty, a closed geodesic or a totally geodesic surface of negative curvature, and the appropriate analogues of all the statements we make are trivial.

Agol has shown that the area of a component of $\D CC(M)$ is bounded  in terms of the topological type and the pinching constants.

\begin{sat}[{Agol, \cite[Appendix]{Agoltameness}}] \label{area-boundary}
	For all $\kappa^-,\kappa^+<0$, there is some $A$ such that for every component $S\subset \D CC(M)$, we have $$\area(S)\le A\cdot \vert\chi(S)\vert.$$
\end{sat}

Agol's theorem statement technically requires no cusps, but his proof works just fine in the setting above. More directly, though, if $M$ has hyperbolic cusps, all of which have rank $2$, then we can fill those cusps using the $2\pi$-theorem (see the proof of Theorem 9 in \cite{Bleilerspherical}) to create a new manifold $M'$ with pinched negative curvature and no cusps, where $\partial CC(M')$ is isometric to $\partial CC(M)$, and then apply Agol's statement.

In his proof, instead of talking directly about the `area' of $S$, which is not a differentiable surface, Agol really just talks about the Riemannian areas on approximating \emph{equidistant surfaces} $S_t$, defined as follows. Let $\pi : M \longrightarrow CC(M)$ be the nearest point retraction, fix $t>0$, and let $S_t$ be the preimage of $S \times \{t\}$ under the map $$M \longrightarrow CC(M) \times \BR_{\geq 0}, \ \ p \longmapsto (\pi(p),d(p,\pi(p))).$$ 
Then for each $t>0$, we have that $S_t$ is a $C^{1,1}$-surface\footnote{The point is to prove that if $C$ is a closed, convex set in a simply connected manifold $X$ of negative curvature and $C_t$ is its $t$-neighborhood, then $\partial C_t$ is $C^{1,1}$. It suffices to show that the distance function $p \mapsto d(C,p)$ is $C^{1,1}$ along $C_t$; indeed, $t$ is a regular value of this function, so the preimage $C_t$ is a $C^{1,1}$-submanifold. The $C^{1,1}$-regularity of the distance function follows from Proposition 2.7 in \cite{ghomi2022total}. In their notation, $\Gamma=\partial C$, and  the negative curvature and the convexity of $C$ imply that $\mathrm{cut}(\Gamma) \subset C$. Alternatively, one can cite Theorem 2 and Corollary 1 in \cite{walter1974metric}.} that is homotopic to $S$. Endowing each $S_t$ with the Lipschitz-varying Riemannian metric induced from $M$, what Agol proves above is that for some $A>0$, \begin{equation}
	\limsup_{t\to 0} \area(S_t) \leq A \cdot |\chi(S_t)|.\label{areaform}
\end{equation}

Note that since $S_t$ admits a lipschitz map onto $S$ and lipchitz maps do not increase Hausdorff measure, Theorem \ref{area-boundary} makes sense as stated and follows from \eqref{areaform} if you take the area of $S$ to mean $2$-dim Hausdorff measure, with the definition normalized so that Hausdorff measure on $\BR^2$ is Lebesgue measure. The theorem is even true when $CC(M)$ has dimension less than $3$, although probably a few more words should be said about how the $S_t$ collapse onto $\partial CC(M)$ in that case. Below, we'll always cite \eqref{areaform} instead of Theorem \ref{area-boundary}, though. 

\begin{kor}\label {ccshortloop}
With $S_t$ as above, for small $t$ there is a homotopically essential loop $\gamma $ in $S_t$ with length  at most some $C=C(\chi(S_t),\kappa^+)$.
\end{kor}
\begin{proof}
	Applying \cite[4.5 3/4]{Gromovmetric}, the area bounds from \eqref{areaform} give a  bounded length curve $\gamma$ on $S_t$ for all small $t$.
\end{proof}

As another application of the area bound, we prove a bounded diameter lemma for the equidistant surfaces $S_t$. %Since $S_t$ does not have negative curvature, for complete clarity we begin by defining what we mean by lengths relative to the thin part. If $X \subset S_t$ and $\gamma \subset S_t$ is a path, define the \emph{length of $\gamma$ rel $X$} to be the length of $\gamma \cap (S_t \setminus X)$. The \emph{distance between two points $p,q \in S_t$ rel $X$} is the infimal length rel $X$ of a path from $p$ to $q$, and the \emph{diameter of $S_t$ rel $X$} is the supremal distance rel $X$ between two points in $S_t$. Let $(S_t)_{< \epsilon}$ be the set of points in $S_t$ through which there is a loop of length less than $2\epsilon$ that is essential in $S_t$, and let $(S_t)_{\geq \epsilon} := S_t \setminus (S_t)_{<\epsilon}$.

\begin{sat}[Relative BDL for equidistant surfaces]\label{bdlcc}
	Suppose as above that $M$ has sectional curvature in $[\kappa^-,\kappa^+]$ and all of its cusps are hyperbolic of rank $2$. Suppose that $\epsilon,t$ are small and positive, $S_t$ is the equidistant surface above, and that $S_t$ is incompressible in $M$. %Then  $S_t$ has diameter at most $C \cdot |\chi(S)|/\epsilon$ rel $(S_t)_{<\epsilon}$, for some $C=C(\kappa^{\pm})$. 

Then there is a collection of disjoint, pairwise non-isotopic, embedded, essential annuli $A_i \subset S_t$ such that 
\begin{enumerate}
\item each boundary component of each $A_i$ has length at most $3\epsilon$,
\item if $p\in A_i$, there is a loop $\gamma$ in\footnote{Perhaps one can take $\gamma \subset S_t$, but our proof doesn't produce such a loop.} $M$ with length at most $3\epsilon + 2t$ that passes through $p$, and that is homotopic in $M$ to an essential simple closed curve in $A_i$,
\item each component of the complement $S_t \setminus \cup_i A_i$ has (intrinsic) diameter at most some constant depending only on $\chi(S_t), \epsilon, \kappa^\pm$.
\end{enumerate}\end{sat}

This theorem may look unsurprising to hyperbolic geometers, and it's trivial when $M$ is hyperbolic since then the convex core boundary is intrinsically hyperbolic. But when $M$ has variable curvature it is hard to adequately control the curvature of the convex core boundary, or that of the associated equidistant surfaces, so we have to give an argument that somehow avoids any discussion of curvature.

\medskip

	Before starting the proof of Theorem \ref{bdlcc}, we need to understand the geometry of the universal cover of $S_t$. Write $M=\Gamma \backslash X,$ where $X$ is a simply connected manifold with pinched negative curvature. Let $CH(\Gamma)$ be the convex hull in $X$ of the limit set of $\Gamma$ and let $CH_t(\Gamma)$ be the closed $t$-neighborhood of $CH(\Gamma)$. Then $S_t$ is covered by some component $T \subset \partial CH_t(\Gamma)$. Since $S_t$ is incompressible, $T$ is homeomorphic to an open disk, and therefore $T$ is the universal cover of $S_t$. 
	
	Our first result says that a subset of $T$ doesn't stray far from the convex hull of its boundary, as long as the subset has the same limit points in $\partial X$ as its boundary.
	
	\begin{lem}\label{close to ch}
		Suppose that $A \subset T$ is a closed subset with frontier $\partial A \subset T$, and that $\overline{\partial A} \cap \partial X =  \overline A \cap \partial X$, where the closures are taken in $X \cup \partial X$. Then $A $ is contained in the $t$-neighborhood of the convex hull $CH(\partial A) \subset X.$
	\end{lem}
	
		In particular, Lemma \ref{close to ch} applies when $A$ is a closed disk in $T$, or when $A\cong \BR \times [0,1]$ is a strip that is a lift of an essential annulus in $S_t$.  These are the two cases in which we will apply the lemma.

	\begin{proof}
	The argument is an extension of that of Lemma 11.1 in Agol's paper \cite{Agoltameness}. Briefly, Agol's argument is as follows. Suppose for a minute that $T$ is a component of $\partial CH(\Gamma)$ instead of $\partial CH_t(\Gamma)$, and that $A \subset T$ is a closed disc. The convex hull of $\partial A$ in $X$ is a ball contained in $CH(\Gamma)$, and the loop $\partial A$ separates the boundary of this ball into two discs $D_1,D_2$, where $D_1$ faces $A$. In fact, $D_1=A$: if not, we could cut $CH(\Gamma)$ along $D_1$ and get a smaller convex subset of $X$ limiting onto $\Gamma$, a contradiction to the definition of $CH(\Gamma)$. So, $A \subset CH(\partial A)$. 
	
	Here are the complexities in the real argument below. Of course, we must work with the $t$-neighborhood $CH_t(\Gamma)$ instead of the convex hull $CH(\Gamma)$, and we work with more general $A$. But also, the part about \emph{`cutting along the disk $D_1 \subset \partial CH(\partial A)$ that faces $A$'} is a little vague, both in the paragraph above and in Agol's paper. Most of the length of the argument below is related to the second point.
	
	\medskip
	
Let's return to the statement of the lemma, where $T\subset CH_t(\Gamma)$ and $A\subset T$ is a closed subset such that $\overline{\partial A} \cap \partial X =  \overline A \cap \partial X$. Set $$K:=CH(\partial A) \subset X.$$ If $x\in A$, let 
	$d_x := d(x,K)$ be the distance to $K$, let 
	$$\CI :=\big \{(x,t) \ \big | \ x\in  A,\, t\in [0,d_x]\big \}$$ and let
	$\gamma : \CI \longrightarrow X$
	be the map such that for every $x\in A$, the path $t\mapsto \gamma(x,t)$ is the shortest path from $K$ to $x$, parameterized so that $d(\gamma(x,t),K)=t$. We consider the domain of $\gamma$ with its induced topology as a subset of $ A \times \BR$. 
	 Then $\gamma$ is continuous, $\gamma(x,\cdot)=x$ if $x\in \partial A$, and the image of $\gamma$ is contained in $CH_t(\Gamma)$. Also, since $\overline{\partial A} \cap \partial X =  \overline A \cap \partial X$ the map $\gamma$ is proper.  Set
	$$\mathring \CI= \{(x,t) \in \CI \ | \ x\in   A, \ t>0\}.$$

	\begin{fact*} $\gamma(\mathring \CI)$ is open in $CH_t(\Gamma)$, and the frontier of $\gamma(\mathring \CI)$ in $CH_t(\Gamma)$ is contained in $\partial K$.
	\end{fact*}
	
	With respect to the naive proof sketch above, the set $\gamma(\mathring \CI)$ corresponds to the set of points that lie between $D_1$ and $A$.
	
	\begin{proof}
		Fixing $r>0$, we'll actually show that the image of the set 
		$$\mathring \CI_r := \{ (x,t)\in \CI \ | \  x\in A, t>r \}$$ is open in $CH_t(\Gamma)$. Openness of $\gamma(\mathring \CI)$ will follow by taking the union over all $r>0$. For the second part of the lemma, properness of $\gamma$ implies that the frontier of $\gamma(\mathring \CI)$ in $CH_t(\Gamma)$ is contained in $\gamma(\CI)$, and therefore in $\gamma(\CI \setminus \mathring \CI)$ by the first part of the lemma. But if $(x,t) \in \CI \setminus \mathring \CI$, then $t=0$, so $d(\gamma(x,t),K)=t=0$ by construction.

		We now show $\gamma(\mathring \CI_r)$ is open. Let $A'\subset A$ be the subset on which $d(\cdot,K) > r$ and consider the closest point projection $$\rho : A' \longrightarrow \partial K_r,$$ where $K_r$ is the $r$-neighborhood of $K$. 
		
		We claim that $\rho$ is injective. If $y,z\in A'$ have the same $\rho$-image, they both lie on a geodesic ray $\alpha$ emanating out from $K_r$, say hitting $y$ before $z$. Negative curvature implies that the distance function $d(\cdot, CH(\Gamma))$ is strictly convex along $\alpha$ wherever it is positive, and hence $\alpha$ must exit $CH_t(\Gamma)$ after passing through $y\in T\subset \partial CH_t(\Gamma)$. The part of $\alpha$ that lies between $y,z$ then violates convexity of $CH_t(\Gamma)$. 
		
		Since $\rho$ is a continuous and injective map between surfaces, it is open by invariance of domain, and hence a homeomorphism onto its image $V:=\rho(A')$, which is an open subset of $\partial K_r$. Consider the radial parametrization $$R: V \times (0,\infty) \longrightarrow X$$ that takes a point $(v,s)$ to the point $s$ away from $K_r$ on the geodesic ray emanating out perpendicularly from $v$. Since $K_r$ is convex with $C^{1,1}$ boundary, $R$ is a homeomorphism onto its image, which is an open subset of $X$. If $y\in A'$ and $v=\rho(y)\in V$, then we have 
		$$R(v,t) = \gamma(y,r+t).$$
		As argued in the previous paragraph, the ray $t\mapsto R(v,t)$ stays inside $CH_t(\Gamma)$ exactly until it passes through $y$, after which it lies outside. So, the intersection of the image of $R$ with $CH_t(\Gamma)$ is exactly $\gamma(\mathring \CI_r)$, which is then open in $CH_t(\Gamma)$.
	\end{proof}
	
	We now claim that the set $$C := CH_t(\Gamma)\setminus \gamma(\mathring \CI)$$
	is convex in $X$. (In the naive proof sketch above, $C$ corresponds to one of the components you get after you cut $CH(\Gamma)$ along $D_1$.) Indeed, if $C$ is not convex, there is some geodesic segment $\alpha$ that has endpoints in $\partial C$, but where $int(\alpha) \subset X\setminus C$. Since $CH_t(\Gamma)$ is convex and contains the endpoints of $\alpha$, we have $\alpha \subset CH_t(\Gamma)$, so the endpoints of $\alpha$ must lie on the frontier of $\gamma(\mathring \CI) \subset CH_t(\Gamma)$, and hence in $\partial K$ by the fact above. Since $K$ is convex, $\alpha$ is contained in $K$, but $K\subset C $, so we have a contradiction.
	
	Since $\overline{\partial A} \cap \partial X =  \overline A \cap \partial X$, the set $C$ above still limits onto the entire limit set of $\Gamma$. Hence, $C \supset CH(\Gamma)$. Now suppose $x\in A$.  To prove the lemma, we want to show that $d(x,K)\leq t$. If $x\in C$, then $x\in K$ and we're done. Otherwise, let $\alpha $ be a shortest path from $x$ to $CH(\Gamma)$. This $\alpha$ must enter $C$ before it extends a distance of $t$ from $x$. Since $x\not \in C$, the first point at which $\alpha$ intersects $C$ cannot lie on $\partial CH_t(\Gamma)$. So it lies on the frontier of $\gamma(\mathring \CI)$ and hence on $\partial K$ by the fact above. This shows $d(x,K)\leq t$, as desired.
	\end{proof}

	Our next goal is to find a lower bound on the areas of small discs in $T$. Later, this will be applied to show that small discs in the thick part of $S_t$ have a similar area lower bound.

	\begin{lem}\label{area lower bound on discs}
		Suppose that $t$ is small and $p\in T$. With respect to the intrinsic metric on $T$, the open metric disc $B_{T}(p,\epsilon/2) \subset T$ has area at least $\epsilon^2/K$ for some constant $K$ depending only on $\kappa^-$.
	\end{lem}

	\begin{proof}
Briefly, the idea here is to use the co-area formula: if one can show that the intrinsic circle of radius $\delta$ around $p\in T$ has length that's at least linear in $\delta$, then we get a lower bound for the area of $B_{T}(p,\epsilon/2) \subset T$ by integrating. The problem is that because of the variable curvature of $M $, we don't have good control over the curvature of $T$, so to get length lower bounds we need to instead use directly that $T$ is a component of $\partial CH_t(\Gamma)$. Also, while there is a well-defined and locally invertible exponential map for $C^{1,1}$ surfaces, there is no guarantee here that the exponential map at $p$ is a homeomorphism onto $B_{T}(p,\epsilon/2)$ for our fixed $\epsilon$, so (at least a priori) the set of points on $T$ at distance $\delta$ from $p$ does not need to be a circle.

\medskip

Fix $p\in T$, let $\delta< \epsilon/2$ and let $B_\delta := B_{T}(p,\delta)$ be the open metric ball in $T$, with respect to the intrinsic metric on $T $. Let $\partial B_\delta$ be its frontier in $T$, which is the set of points that exactly distance $\delta$ from $p$. As mentioned above, the closure $\overline B_\delta$ may not be a submanifold with boundary $\partial B_\delta$.

\begin{claim}\label{connectedsubset}
There is an open disk $D \subset T$ with $B_\delta \subset D$ and where the frontier $\partial D \subset \partial B_\delta$.
\end{claim}

Recall that $T$ itself is a topological disk. As an example, if $\overline B_\delta$ is a pair of pants embedded in $ T$, then $D$ above is the union of the interior $B_\delta$ with two adjacent closed disks.

\begin{proof}[Proof of Claim \ref{connectedsubset}] 
Let $ E \subset  T$ (for `exterior') be the path component of $ T \setminus \overline B_\delta$ containing the single end of $ T$, and let $$ D = T \setminus \overline E.$$ 
If $\gamma$ is any simple closed curve in $ D$, then $ T \setminus \gamma$ is the union of an open disk $D_\gamma $ and a complementary open annulus that contains the end of $ T$. Since $E$ is connected, does not intersect $\gamma$, and intersects this complementary annulus, $E$ is contained in the complementary annulus, hence $\overline E$ is contained in the union of the annulus and $\gamma$, and hence $D_\gamma \subset  D$. In other words, $\gamma$ is nullhomotopic in $ D$, and since $\gamma$ was arbitrary, this means $ D$ is an open disk.
\end{proof}

Let $\CH$ denote $1$-dimensional Hausdorff measure, with respect to the intrinsic metric on $T$. Let $D$ be as in Claim \ref{connectedsubset}, and set $$h_\delta := \CH(\partial B_\delta).$$
The frontier $\partial D$ is connected, so we have $$\diam (\partial D) \leq \CH(\partial D) \leq h_\delta, $$ where the first inequality is an exercise using connectedness and the definition of Hausdorff measure. Here, all metric quantities are computed with respect to the intrinsic metric on $T$. However, the above certainly implies the extrinsic inequality
$\diam_X (\partial D) \leq h_\delta$, where now $\diam_X$ is diameter as a subset of $X$. 

By Lemma \ref{close to ch}, $D$ lies in the $t$-neighborhood of the convex hull $CH(\partial D) \subset X$. Since $\partial D$ has diameter at most $h_\delta$, it lies in a ball of radius $h_\delta$ in $X$. Negative curvature implies balls in $X$ are convex, so $CH(\partial D)$ also lies in an $h_\delta$-ball. It follows that $$\diam_X D \leq 2(h_\delta + t).$$

We now upgrade this inequality to give an upper bound on the diameter of $D$ with respect to the intrinsic metric on $T$. Choose a metric ball $B \subset X$ that contains $\overline  D$ and has radius $2(h_\delta+t)$. If $$\pi : \partial B \longrightarrow CH_t(\Gamma)$$ is the closest point projection, then the image of $\pi$ contains $\overline D$, since a geodesic ray emanating orthogonally out of $CH_t(\Gamma)$ from a point in $\overline D \subset B$ must intersect $\partial B$. So, given $x \in D$, choose some $y\in \partial B$ with $\pi(y)=x$, and join $y$ to $\pi^{-1}(\partial D)$ by a shortest path $\gamma$ in $\partial B$. Standard Jacobi field estimates imply that 
$$\diam \partial  B \leq \pi s \sinh( 4(h_\delta+t)/s), \ \ s= 1/\sqrt{|\kappa^-|},$$
where the right hand side is the intrinsic diameter of a sphere of radius $2(h_\delta+t)$ in a space of constant curvature $\kappa^-$. As long as $h_\delta +t$ is bounded (in our work $t \approx 0$ and $h_\delta\leq \epsilon$) we can use a linear upper bound for $\sinh$, giving $\diam \partial B\leq K(h_\delta + t)$ for some $K$ depending only on $\kappa^-$. The length of the path $\gamma$ above is then also at most $K(h_\delta + t)$. Since the projection $\pi$ is $1$-lipshitz, we then get that $x$ can be connected to $\partial D$ by a path on $T$ of length at most $K(h_\delta + t)$. But $\diam_T \partial D \leq h_\delta$, so after increasing $K$ slightly, we get:

\begin {claim}\label {intrinsic diameter claim}
As long as $h_\delta,t\leq 1$, say, we have $\diam_T D \leq K(h_\delta+t)$ for some $K$ depending only on $\kappa^-$.
\end{claim}

Here is the point of all this. Recall that $p\in T$, that $B_\delta\subset T$ is the metric ball around $p$ with respect to the intrinsic metric on $T$, that $h_\delta = \CH(\partial B_\delta)$ is the one dimensional Hausdorff measure, and $B_\delta \subset D$, a disc with boundary in $\partial B_\delta$. As $p$ is at distance $\delta$ from $\partial B_\delta$, we have
$$\delta = \diam_T B_\delta \leq \diam_T D.$$
So, as long as $t\leq 1$, Claim \ref{intrinsic diameter claim} says that
$$h_\delta \geq \min\{\delta/K-t,1\}.$$
Since $\delta < \epsilon/2$, as long as $\epsilon<1$, say, and $K\geq 1$, we can forget about the minimum and just write $h_\delta \geq \delta/K-t$. Applying the coarea formula for Lipschitz functions on $C^1$-manifolds \cite[Theorem 5.3]{nicolaescu2011coarea} to the function $d(p,\cdot) $ on $T$, we get
$$\area B_T(p,\epsilon/2) = \int_0^{\epsilon/2} h_\delta \, d\delta \geq \int_0^{\epsilon/2} (\delta/K-t) \, d\delta \geq \epsilon^2/4K - t\epsilon/2. $$
So, as long as $t$ is small, after increasing $K$  we get an inequality $$\area B_T(p,\epsilon/2) \geq \epsilon^2/K$$
for some new $K$, still depending only on $\kappa^-$.\end{proof}

We are now ready to prove our version of the bounded diameter lemma for incompressible boundary components of convex cores in variable negative curvature.

\begin{proof}[Proof of Theorem \ref{bdlcc}]
The annuli in the statement of Theorem \ref{bdlcc} are related to those in a `thick-thin decomposition' of $S_t$. Since $S_t$ does not have negative curvature, for complete clarity we begin with some definitions. If $X \subset S_t$ and $\gamma \subset S_t$ is a path, define the \emph{length of $\gamma$ rel $X$} to be the length of $\gamma \cap (S_t \setminus X)$. The \emph{distance between two points $p,q \in S_t$ rel $X$} is the infimal length rel $X$ of a path from $p$ to $q$, and the \emph{diameter of $S_t$ rel $X$} is the supremal distance rel $X$ between two points in $S_t$. The \emph{$\epsilon$-thin part} $(S_t)_{< \epsilon}$ is the set of points in $S_t$ through which there is a loop of length less than $2\epsilon$ that is essential in $S_t$, while the \emph{$\epsilon$-thick part} is $(S_t)_{\geq \epsilon} := S_t \setminus (S_t)_{<\epsilon}$.

As a first goal, we claim:

\begin{claim} $S_t$ has diameter at most $C |\chi(S_t)|/\epsilon$ relative to the thin part $(S_t)_{< \epsilon}$, for some $C=C(\kappa^\pm)$.
\end{claim}
\begin{proof}
If $p\in (S_t)_{\geq \epsilon}$, there are no loops in $ B(p,\epsilon/2)$ that are essential in $S_t$. This follows from an argument of Gromov: if $\gamma$ is a loop in $ B(p,\epsilon/2)$, we can write $\gamma$ as a concatenation of loops based at $p$ of length less than $2\epsilon$, by subdividing $\gamma$ into small arcs of length less than $\epsilon/2$, and joining their endpoints to $p$ with near minimal length paths in $B(p,\epsilon/2)$. So if $\gamma$ is essential in $S_t$, there is some such loop based at $p$ with length less than $2\epsilon$ that is essential, a contradiction.

Hence, $B(p,\epsilon/2)$ is an open, and hence locally path connected, subset of $S_t$ with trivial $\pi_1$-image in $S_t$. By the lifting criterion, $B(p,\epsilon/2)$ lifts to $T$, the chosen component of $\partial CH_t(\Gamma)\subset X$ that covers $S_t$. Then as long as $t$ is small, Lemma \ref{area lower bound on discs} implies that $B(p,\epsilon/2)$ has area at least $\epsilon^2/K$ for some $K=K(\kappa^-)$.

Suppose now that for some $n\in \BN$, the surface $S_t$ has diameter rel $(S_t)_{< \epsilon}$ in the interval $ (n\epsilon,2n\epsilon].$ Choose points $p_0,\ldots,p_n\in T$ such that 
\begin {enumerate}
\item[(a)] $p_i \in (S_t)_{\geq \epsilon}$ for all $i$,
\item[(b)] the distance rel $(S_t)_{<\epsilon}$ from $p_0$ to $p_i$ is $\epsilon \cdot i$. \end {enumerate}
For instance, if $x,y$ realize the diameter of $S_t$ rel $(S_t)_{<\epsilon}$ and $\alpha$ is a path from $x$ to $y$ with nearly minimal $\epsilon$-length, we can construct such $p_i$ along $\alpha$, by choosing $p_i$ to be the first point where (b) holds. Such a $p_i$ is automatically in the $\epsilon$-thick part of $M$, since the $\epsilon$-distance to $x$ does not change when moving in the thin part.

For $i=1,\ldots,n$, the sets $B_{S_t}(p_i,\epsilon/2)$ are all disjoint, and all have area at least $\epsilon^2/K$, from above. So, for small $t$ we have $$n \cdot \epsilon^2/K \leq 2A \cdot |\chi(S_t) |$$ by \eqref{areaform}, where $A=A(\kappa^\pm)$.
But the diameter of $S_t$ rel $(S_t)_{<\epsilon}$ is at most $2n\epsilon$, and we have
$$2n\epsilon \leq 4AK \cdot |\chi(S_t)|/\epsilon,$$
so we're done with $C=4AK$.
\end{proof}

Next, we must construct a collection of disjoint, pairwise non-isotopic embedded, essential annuli $A_i \subset S_t$ such that \begin{enumerate}
	\item each boundary component of each $A_i$ has length at most $3\epsilon$,
	\item if $p\in A_i$, there is a loop $\gamma$ in $M$ that passes through $p$, that is homotopic in $M$ to an essential simple closed curve in $A_i$ and has length at most $3\epsilon + 2t$,
	\item each component of the complement $S_t \setminus \cup_i A_i$ has diameter at most some constant depending only on $\chi(S_t), \epsilon, \kappa^\pm$.
	\end{enumerate}

First, note that if $\alpha,\beta$ are non-isotopic simple closed curves on $S_t$ that both have length less than $2\epsilon$, then we have $\alpha\cap \beta=\emptyset$. Indeed, if $p\in \alpha\cap \beta$, then since $S_t$ is incompressible, $p$ lies in a component $U$ of the thin part $M_{<\epsilon}$.  So, $\alpha,\beta$ generate an abelian subgroup of $\pi_1 M$, and therefore of $\pi_1 S_t$ by incompressibility. This implies that $\alpha,\beta$ have a common power in $\pi_1 S_t$. As both loops are simple, this implies that they are isotopic, a contradiction. 

Now given an essential simple closed curve $\alpha$ on $S_t$, let 
$(S_t)_{<\epsilon}(\alpha)$ be the set of points on $S_t$ through which there is a curve of length less than $2\epsilon$ that is homotopic to $\alpha$. Then $$(S_t)_{<\epsilon} = \cup_\alpha (S_t)_{<\epsilon}(\alpha),$$ since whenever there is an essential loop of length less than $2\epsilon$ through $p$, $C^1$-perturbation and surgery gives a simple such loop. The previous paragraph implies that the $\alpha$ where $(S_t)_{<\epsilon}(\alpha)\neq \emptyset$ are all pairwise disjoint, so finitely many $\alpha$ suffice, and moreover that corresponding sets $(S_t)_{<\epsilon}(\alpha)$ are also disjoint. So, $$(S_t)_{<\epsilon} = \sqcup_{\alpha_i} (S_t)_{<\epsilon}(\alpha_i),$$
for some finite collection $\alpha_i$ of essential simple closed curves on $S_t$.

Fix some $i$, and for concreteness homotope $\alpha_i$ on $S_t$ so that it has length less than $2\epsilon$. We now build an annulus $A_i \subset S_t$ such that $(S_t)_{<\epsilon}(\alpha_i)$ is contained in a $2\epsilon$-neighborhood of $A_i$. To do this, consider the collection of all \emph{open} annuli $A \subset S_t$ that contain $\alpha_i$, and where the frontier of $A$ in $S_t$ is the image of a pair of (possibly non-embedded) curves of length at most $2\epsilon$. Partially order the set of all such $A$ with respect to inclusion. Then for any increasing sequence of such $A$, the union is an annulus of the same type, by Arzela-Ascoli. So, we can pick a maximal element $A_i'$ of this collection. If $\gamma$ is any essential simple closed curve on $S_t$ with length less than $2\epsilon$ that is isotopic to $\alpha_i$, then $\gamma$ must intersect $A_i'$: if not, $\gamma$ bounds an annulus with $\alpha_i$, and the union of this annulus with $A_i'$ is a bigger annulus of the preferred type above, contradicting maximality of $A_i'$. So, it follows that every point in $(S_t)_{<\epsilon}(\alpha_i)$ is within $\epsilon$ of $A_i'$. 

The only problem is that the closure $\overline {A_i'}$ may not be an embedded annulus in $S_t$, and in the statement of the theorem we want our annuli $A_i$ to be embedded compact annuli. To fix this, we just push the boundary components of $A_i'$ slightly inside to create a closed annulus $A_i \subset A_i'$, at the expense of increasing boundary length by a small factor, and slightly increasing the distance to certain points in $(S_t)_{<\epsilon}(\alpha_i)$.  In particular, we can assume each component of $\partial A_i$ has length at most $3\epsilon$, and $(S_t)_{<\epsilon}(\alpha_i)$ is contained in a $3\epsilon$-neighborhood of $A_i$.

Note that as $i$ varies, the annuli $A_i$ constructed above are all disjoint. Indeed, as they are all essential and nonisotopic, no $A_i$ can be contained in any $A_j$, where $i\neq j$. So if $A_i$ intersects $A_j$, their boundaries intersect, in which case we have two essential, intersecting simple closed curves of length less than $3\epsilon$ on $S_t$. As long as $\epsilon$ is sufficiently small, $3\epsilon$ is still smaller than the Margulis constant for $M$, so this violates incompressibility of $S_t$ in the same way as we discussed above.

We next show that each component $K \subset S_t \setminus \cup_i A_i$ has intrinsic diameter at most some constant depending only on $\chi(S_t), \epsilon, \kappa^\pm$. But since each component of $\partial K$ has length at most $3\epsilon$ and the number of components of $\partial K$ is bounded in terms of $|\chi(S_t)|$, if $K$ has huge intrinsic diameter, then there is a point in $K$ at huge distance from $\partial K$. This contradicts that $S_t$ has diameter rel $(S_t)_{<\epsilon}$ at most $C|\chi(S_t)|/\epsilon$  and that $(S_t)_{<\epsilon} $ is contained in a $3\epsilon$-neighborhood of $\cup_iA_i$.

Finally, we show that if $p\in A_i$, there is a loop in $M$ that passes through $p$, is homotopic in $M$ to an essential simple closed curve in $A_i$, and has length at most $3\epsilon+t$. For this, lift $A_i$ to the universal cover $T\subset X$ of $S_t$, giving a strip $\tilde A_i \cong \BR\times [0,1]$, and let $g \in \Gamma$ be a deck transformation such that $\langle g\rangle \backslash \tilde A_i = A_i$. Let $X_{g,3\epsilon} \subset X$ be the set of points that are translated less than $3\epsilon$ by $g$. Since the boundary components of $A_i$ have length at most $3\epsilon$, we have  $\partial \tilde A_i \subset X_{g,3\epsilon} $. Negative curvature implies that $X_{g,3\epsilon} \subset X$ is convex, and then Lemma \ref{close to ch} says that $\tilde A_i$ is contained in the $t$-neighborhood of $CH(\partial \tilde A_i) \subset X_{g,3\epsilon}$. Hence, every point in $\tilde A_i$ is within $3\epsilon+2t$ of its $g$-translate. Projecting to $M$, we are done.
\end{proof}

\subsection{Covering maps}

We begin by stating two results about how ends of $3$-manifolds can cover. The first is a variable curvature version of the Thurston-Canary covering theorem \cite{Canarycovering}.

\begin{named}{The Covering Theorem}
Let $N$ be a complete Riemannian 3-manifold with pinched negative curvature, finitely generated fundamental group and hyperbolic cusps and let $\pi:N\longrightarrow M$ be a Riemannian covering. If $\CE$ is a degenerate end (without cusps) of $N$, and $S$ is the adjacent component of a standard compact core of $N$, then either:\begin{itemize}
\item $M$ is closed and $\pi$ factors as $$N \overset{\pi'}\longrightarrow M' \longrightarrow M,$$ where $\pi'(S)$ is homotopic to a fiber with respect to a fibration of $M'$ over the circle.
\item There is a neighborhood $E\cong S\times\BR$ of $\CE$ in $N$, a degenerate end $\CE'$ of $M$ and a neighborhood $E'\cong S'\times\BR$ of $\CE'$ in $M$ such that $\pi(E)=E'$ and the restriction
$$\pi\vert_E:E\longrightarrow E'$$
is a finite covering.
\end{itemize}
\end{named}

Agol \cite{Agoltameness} proved the above when $\CE$ is incompressible, by proving a version of Canary's Filling Theorem for incompressible ends in variable curvature and then following Canary's argument from \cite{Canarycovering}. One can prove the theorem above in the same way, using our version of the variable curvature filling theorem, see Theorem \ref{filling theorem}, to deal with compressible ends.

The second result is the following easy observation.

\begin{lem}[End embeddings extend to convex core boundary]\label{convex cocompact cover}
Let $\pi: N \longrightarrow M$ be a Riemannian covering of hyperbolic $3$-manifolds. Suppose $\pi$ restricts to an embedding on some neighborhood $U$ of a convex cocompact end $ \CE$ of $N$. Then $\pi$ is an embedding on the component $$  E \subset   N \setminus CC(N)$$ that is a neighborhood of $ \CE$. Moreover, the image $\pi(  E)$  is a component of $M \setminus CC(M)$. \end{lem}
\begin{proof}
Write $  N=\BH^3 / \Delta$ and $M=\BH^3/\Gamma$, where $\Delta \subset \Gamma$. Let $\tilde E,\tilde U$ be the preimages of $E,U$ in $\BH^3$. Let $\overline N  = \BH^3 \cup \Omega(\Delta) / \Delta$ be the associated Kleinian manifold, and let $\partial_\CE N$ be the boundary component of $\overline N$ adjacent to $\CE$. Let $\Omega_\CE \subset \partial_\infty \BH^3$ be the preimage of $\partial_\CE N$. Note that 
\begin{enumerate}
	\item any $\xi \in \Omega_\CE$ has a neighborhood whose intersection with $\BH^3$ lies in $\tilde U$, and
\item any $\xi  \in \BH^3 \setminus \Omega_\CE $ is a limit of a sequence of points in $\BH^3 \setminus \tilde E$.
\end{enumerate}

We claim that the limit set $\Lambda(\Gamma)$ cannot intersect $\Omega_\CE$. If it did, then there would be some hyperbolic type element $\gamma \in \Gamma \setminus \Delta$ with a fixed point $\xi$ in $\Omega_\CE$. If $x \in \BH^3$ is close enough to $\xi$, then (1) says that $x$ and $\gamma(x)$ both lie in $\tilde U$, contradicting that $\pi$ is an embedding on $U$.

From (2), it then follows that $\BH^3 \setminus \tilde E$ is a convex subset of $\BH^3$ limiting onto $\Lambda(\Gamma)$, and therefore contains the convex hull of $\Lambda(\Gamma)$. But the frontier $\partial  \tilde E$ in $\BH^3$ is contained in the convex hull $CH(\Lambda(\Delta))$, by definition of $E$. Since $\Delta \subset \Gamma$, it follows that $\partial \tilde E$ is a collection of components of the boundary of $CH(\Lambda(\Gamma))$, and $\tilde E$ is a collection of complementary components of this convex hull.
Projecting down, this shows that $\pi(E)$ is a component of $M\setminus CC(M)$. 

Suppose now that $\pi$ is not injective on $E$. Then there is some $\gamma \in \Gamma \setminus \Delta$ and some $x\in \tilde E$ such that $\gamma(x)$ is also in $\tilde E$. Let's say that the components of $\tilde E$ containing $x,\gamma(x)$  are $\tilde E_1, \tilde E_2$. Since $\tilde E_1$ and $\tilde E_2$ are components of $\BH^3 \setminus CH(\Lambda(\Gamma))$, we have $\gamma(\tilde E_1)=\tilde E_2$.  But points in $\tilde E_i$ close to $\partial_\infty \BH^3$ lie in $\tilde U$, so this means $\gamma$ maps a point of $\tilde U$ into $\tilde U$. This contradicts that $\pi$ is injective on $U$.
\end{proof}

Finally, we recall a construction from \cite{Biringeralgebraic} that extends an isometric immersion from a manifold with convex boundary to  a covering map. 

\begin{prop}\label {radialextension}
	Suppose that $N $ is a complete, convex hyperbolic $3$-manifold that has compact boundary and $$\pi : N \longrightarrow M$$  is an isometric immersion  into a (complete) hyperbolic $3$-manifold $M $. Then there is a complete  boundaryless hyperbolic $3$-manifold $ N' \supset N$ and an extension $ \pi' : N' \longrightarrow M $ of $\pi$ such that 
\begin {enumerate}
\item $N' \setminus N$ is homeomorphic to $\partial N \times \BR$,
\item $\pi'$  is a covering map.
\end {enumerate}
\end{prop}
\begin {proof}
 The manifold $N'$ is obtained by extending the action of $\pi_1 N$ on  a convex subset of $\BH^3$ to all of $\BH^3$, see  Theorem I.2.4.1 in \cite{Canarynotes} and (1)  follows from properties of the nearest point projection. See \cite[II.1]{Canarynotes}. 

In \cite[Cor 4.2]{Biringeralgebraic}, we show that $\pi$  radially extends to a locally bilipschitz map $ \pi' : N' \longrightarrow M $, which is then necessarily a covering map.
\end {proof}

%
%
%\subsection{Doubly degenerate 3-manifolds}
%Before concluding this section we remind the reader of a few facts on doubly degenerate hyperbolic 3-manifolds. Recall that a hyperbolic 3-manifold $M$ is {\em doubly degenerate} if:
%\begin{itemize}
%\item $M$ is homeomorphic to $S\times\BR$ where $S$ is a closed surface,
%\item $M$ has no cusps, and
%\item both ends of $M$ are degenerate.
%\end{itemize}
%If $M\simeq S\times\BR$ is doubly degenerate we will say that $f:S\to M$ is in {\em the homotopy class of a level surface} if $f$ is a homotopy equivalence. The genus of the surface $S$ is the genus of $M$.
%
%We will need the following three facts:
%
%\begin{named}{Filling theorem}[Canary, Thurston]
%If $M$ is doubly degenerate, then every $x\in M$ is in the image of a simplicial hyperbolic surface in the homotopy class of a level surface.
%\end{named}
%
%\begin{prop}\label{prop:sullivan}
%Suppose that $M$ and $N$ are hyperbolic 3-manifolds, $M$ doubly degenerate. Every bi-Lipschitz map $f:M\to N$ is homotopic to an isometry.
%\end{prop}
%
%\begin{prop}\label{prop:dd-compact}
%Suppose that $(M_i,x_i)$ is a sequence of pointed hyperbolic 3-manifolds which converges geometrically to some $(M,x)$. If $M_i$ is doubly degenerate of genus $g$ independent of $i$ and if $M$ has  no cusps, then $M$ is doubly degenerate of genus $g$ as well.
%\end{prop}
%
%

\section{Simplicial ruled surfaces}\label{sec:shs}

Let $M$ be a complete, oriented, $3$-dimensional Riemannian manifold with pinched negative curvature $\kappa^- \le\kappa\le \kappa^+ <0$.

\subsection{Ruled triangles} Let $\Delta$ be an affine 2-simplex. Here, \emph{affine} means that $\Delta$ is equipped with an \emph{affine structure}, which is a realization of $\Delta$ as an affine simplex in $\BR^2$ up to affine isomorphism. A \emph{geodesic segment} in $\Delta$ is therefore well-defined as a path that is a straight line segment in some/any affine realization of $\Delta$. We usually view geodesics as equipped with constant speed parametrizations. Note that for parametrizations of geodesics in $\Delta$,  `constant speed'  is well-defined by the affine structure, but the actual speed is not.

A continuous map $f:\Delta\longrightarrow M$ is called a {\em ruled triangle} if for some vertex $v \in \Delta$, called the \emph{ruling vertex}, we have that whenever $e$ is 
\begin{itemize}
	\item an edge of $\Delta$, or
	\item a geodesic segment from $v$ to a point on the opposite edge,
\end{itemize}
then $e$ is mapped at constant speed to a geodesic segment in $M$ by $f $.

We say that a ruled triangle $f$ is \emph{degenerate} if its image is a point or a geodesic segment, and \emph{nondegenerate} otherwise. When $M$ is hyperbolic, a nondegenerate ruled triangle is a totally geodesic triangle.

\begin{fact}[Construction of ruled triangles]\label{ruledconstruction} Let $\Delta$ be an affine $2$-simplex and $v$ be a vertex of $\Delta$. Suppose we have a map $$f : \Delta \longrightarrow M$$ that takes each edge to a constant speed geodesic in $M$. Then $f$ is homotopic rel $\partial \Delta$ to a ruled triangle. Moreover, if the image of $f$ is not a geodesic segment, then any such ruled triangle is nondegenerate and is a smooth immersion\footnote{Here, taking a realization $\Delta \subset \BR^2$, the map $f$ is a smooth immersion on $\Delta$ if it is a smooth immersion on an open subset of $\BR^2$ containing $\Delta$.} on $\Delta \setminus v$.
\end{fact} 
\begin{proof}
Since $\Delta $ is simply connected, we can lift it to the universal cover, so we might as well assume $M$ is simply connected. Let $\gamma$ be the side of $\Delta$ opposite to $v$ and let
$$\gamma \times (0,1] \longrightarrow \Delta \setminus v$$ be the natural parametrization, where $p \times [0,1]$ maps to the geodesic in $\Delta$ from $v$ to $p$, at constant speed. Let $$\exp : \BR^3 \longrightarrow M$$ be the exponential map of $M$ based at $f(v)$. Since $M$ is simply connected and negatively curved, $\exp$ is a diffeomorphism. Hence, 
$$\gamma \times (0,1] \longrightarrow M, \ \ (p,t) \longmapsto \exp\big (t \cdot \exp^{-1} \circ f(p)\big )$$
defines a ruled triangle that agrees with $f$ on $\partial \Delta$, and hence is homotopic rel $\partial \Delta$ to $f$, since $M$ is contractible. If $f(\partial \Delta)$ is not a single geodesic segment, the path $\exp^{-1} \circ f(\gamma)$ is never tangent to the radial direction in $\BR^3$, in which case the map above is an immersion.
\end{proof}

If $f:\Delta\longrightarrow M$ is a ruled triangle, we can endow $\Delta$ with a path pseudo-metric, by declaring the length of a path in $\Delta$ to be the length of its image in $M$, and defining distances as infima of connecting path lengths. When $\Delta $ is nondegenerate, this construction gives an actual metric that is smooth except possibly at the ruling vertex $v$. 

\begin{fact}\label{CATtriangle}
	The intrinsic path metric on a nondegenerate ruled triangle has Gaussian curvature at most $\kappa^+$ on $\Delta \setminus v$, and is CAT($\kappa^+$) on $\Delta$. 
\end{fact}

Recall that for $\kappa \leq 0$, a geodesic metric space $X$ is called {CAT($\kappa$)} if for every geodesic triangle $\Delta$ in $X$ the distances between points on $\partial \Delta$ are at most the distances between corresponding points on the boundary of the `comparison triangle' $\Delta' \subset M_{\kappa}^2$. Here,  $M_{\kappa}^2$  is the $2$-dimensional model space of constant curvature $\kappa$, and $\Delta'$ is the triangle with the same side lengths as $\Delta$. See \cite[II.1]{Bridsonmetric} for details.

\begin{proof}
If $p\in \Delta \setminus v$, pick a vector $v \in T\Delta_p$ in the direction of the geodesic ruling. Then $B(v,v)=0$, where $B$ is the second fundamental form of $f(\Delta) \subset M$. The Gaussian curvature upper bound follows from the Gauss Equation \cite[Theorem 2.5]{do1992riemannian}. 

By \cite[Theorem II.1A.6]{Bridsonmetric}, the interior of $\Delta$ is locally CAT($\kappa^+)$. Since it is simply connected, it is globally CAT($\kappa^+$) by \cite[Theorem II.4.1]{Bridsonmetric}. But $\Delta$ is the metric completion of its interior, so $\Delta$ is CAT($\kappa^+$) as well by \cite[Corollary II.3.11]{Bridsonmetric}.
\end{proof}

\subsection{$\Delta$-complexes with CAT($\kappa^+$)-triangles}
A \emph{$2$-dimensional $\Delta$-complex} is an affine gluing of affine $0,1,2$-simplices, where faces of a given simplex are allowed to be glued to each other. See Hatcher \cite[pg 103]{Hatcheralgebraic} for a formal definition. We'll often refer to $0,1$ and $2$-simplices below as vertices, edges and triangles.

In this section we study $\Delta$-complexes equipped with metrics where all triangles are negatively curved, in the following sense.

\begin{defi}\label{catkdef}
	A \emph{CAT($\kappa)$-triangle} is an affine $2$-simplex $\Delta$ equipped with a CAT($\kappa)$ path metric that is smooth on $int(\Delta)$, and where all edges of $\Delta$ are geodesic. 
\end{defi}

For example, Fact \ref{CATtriangle} says that any nondegenerate ruled triangle in a manifold $M$ with curvature at most $\kappa^+$ is a CAT($\kappa^+)$-triangle. For a CAT($\kappa)$-triangle $\Delta$, the \emph{interior angle} of $\Delta$ at a vertex $p$ is defined to be the Alexandrov angle between the two adjacent edges. See \cite[pg 9]{Bridsonmetric} for more information on Alexandrov angles.

Suppose that $X$ is a $2$-dimensional $\Delta$-complex equipped with a path metric $d$. We say that $(X,d)$ \emph{has CAT($\kappa$)-triangles} if for each triangle $\Delta$, the completion of the restriction of $d$ to the interior of $\Delta$ is a CAT($\kappa$)-triangle; note that this complicated phrasing is necessary since a triangle in $X$ is allowed to have two of its sides identified. In other words, $(X,d)$ has CAT($\kappa$)-triangles if it is obtained by gluing together vertices, geodesic edges and CAT($\kappa$)-triangle via isometries. 

A vertex $p$ of $X$ is \emph{regular} if its link is homeomorphic to a circle. The \emph{cone angle} of $X$ at a regular vertex $p$ is the sum over all triangles incident to $p$ of the interior angle at $p$. 

\begin{lem}[Geometry near regular vertices]\label{CATfactX}
If $X$ has CAT($\kappa$)-triangles and the cone angle at a regular vertex $p\in X$ is at least $2\pi$, then $X$ is locally CAT($\kappa$) at $p$.
\end{lem}

Lemma \ref{CATfactX} is well known to experts and has been implicitly assumed in a number of papers on simplicial ruled surfaces, e.g.\ \cite{Agoltameness,Somaexistence,souto2008short}. When all triangles have \emph{constant} curvature $\kappa$, one can cite \cite[(5.6) on pg 251]{Bridsonmetric}, but in the variable curvature case we do not know a reference, so we include a proof here.

\begin{proof}[Proof of Lemma \ref{CATfactX}]
%	In the interior of every triangle, Fact \ref{CATtriangle} says that $S$ is locally CAT($\kappa$). Any point on the interior of an edge of $\mathcal T$ has a neighborhood that is obtained by gluing two CAT($\kappa$) spaces together along a geodesic; such a neighborhood is CAT($\kappa$) by \cite[II.11.1]{Bridsonmetric}. So, it suffices to show that every vertex $p$ has a CAT($\kappa$) neighborhood.
Since the statement is local, it suffices to assume that all triangles of $X$ contain $p$, and that they are only glued along edges incident to $p$, so that $X$ is a homeomorphic to a closed disc. Under this assumption, we will prove that $X$ is (globally) CAT($\kappa$).
	
	If $\Delta$ is a triangle in $X$, the fact that $\Delta$ is CAT($\kappa$) means that there is a unique geodesic from $p$ to any given point $q$ on the opposite side, and that these geodesics and the angles between them depend continuously on $q$. After using some such geodesic to cut some $\Delta$ into two triangles, there are edges $\alpha,\beta$ incident to $p$ that divide $X$ into two topological half-discs $U,V$, in such a way that the intrinsic angle sum at $p$ of all the triangles in $U$ (resp. $V$) is at least $\pi$. Both $U$ and $V$ are CAT($\kappa$), by repeated applications of  \cite[II.11.1]{Bridsonmetric}. 
	
	Working in $U$ for the moment, we claim that the Alexandrov angle between $\alpha, \beta$ at $p$ is $\pi$. Take $x\in \alpha$ and $y\in \beta$ close to $p$. The geodesic $[x,y] \subset U$ is contained in a small neighborhood of $p$, and intersects all edges of triangles in $U$, and only crosses each edge once because geodesics in $U$ are uniquely determined by their endpoints. So, if $$x=z_1,\ldots,z_n=y \in [x,y]$$ are these intersection points, for each $i$ the Alexandrov angle $\angle_p(z_i,z_{i+1})$ between the geodesics from $p$ to $z_i$ and $z_{i+1}$ is an intrinsic angle of one of our triangles. By construction of $U$, we then have
	\begin{equation}\label{anglesumm}
		\sum_{i=1}^{n-1} \angle_p(z_i,z_{i+1}) \geq \pi.
	\end{equation}

	\begin{figure}
		\centering
		\includegraphics{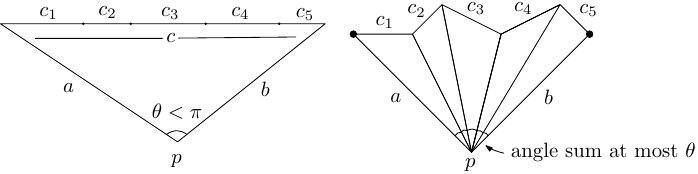}
		\caption{Take a (nondegenerate) triangle in $\BR^2$ with side lengths $a,b,c$, and an angle of $\theta$ opposite to $c$. Subdivide the $c$-edge into segments of length $c_i$, and then take a sequence of triangles glued side to side in the manner of the picture on the right, where all the triangles share the vertex $p$, the opposite sides have length $c_i$, and the initial and terminal adjacent sides have length $a,b$ respectively. Then the angle sum at $p$ is at most $\theta$. For a proof, imagine the picture on the left is made of rods and hinges, and drag the vertices subdividing the $c$ edge one by one so that they are at the correct distance from $p$. This operation decreases the total angle at $p$.}
		\label{comptri}
	\end{figure}

Assume, hoping for a contradiction, that the Alexandrov angle of $\alpha,\beta$ at $p$ is strictly less than $\pi$. By \cite[II.3.1]{Bridsonmetric}, as long as $x,y$ are close to $p$, the comparison triangle $$\Delta(\bar p,\bar x,\bar y) \subset \BR^2$$ has an angle $\theta$ at $\bar p$ that is strictly less than $\pi$. Construct a polygon in $\BR^2$ by gluing together the comparison triangles $\Delta(\bar p,\bar z_i,\bar z_{i+1})$ side by side in $\BR^2$, in the same way that the triangles $\Delta(p,z_i,z_{i+1})$ are glued in $U$. If $x,y$ are close to $p$, all the $z_i$ are also close to $p$, so in light of  \eqref{anglesumm} we can assume that the total angle sum at $\bar p$ of all these comparison triangles is arbitrarily close to $\pi$, in particular, we can assume it is bigger than $\theta$. However, the sides $[\bar z_i,\bar z_{i+1}]$ of these comparison triangles have length that sum to $d(x,y)=d(\bar x,\bar y)$, so the angle sum at $\bar p$ should be smaller than $\theta$. See Figure \ref{comptri}. This is a contradiction.

	The CAT($\kappa$) condition now implies that the union $\alpha \cup \beta$ is a geodesic in $U$, and similarly in $V$. Since $X$ is created by gluing $U,V$ along $\alpha \cup \beta$, it follows that $X$ is CAT($\kappa$).\end{proof}

When $X$ is a locally CAT($\kappa$) metric space and $p\in X$, the \emph{space of directions} $\Sigma_p(X)$ is obtained from the set of all geodesics issuing out from $p$ by identifying pairs of geodesics that intersect at $p$ with Alexandrov angle $0$, see \cite[pg 190]{Bridsonmetric}. The Alexandrov angle gives a natural metric $\angle$ on $\Sigma_p(X)$, and we also consider the induced path metric $d$.  The proof of the following result is in some sense contained in that of Lemma \ref{CATfactX} above, but it makes more sense to state it afterwards.

\begin{lem}[The space of directions at a regular vertex]\label{sofd}
As above, let $p\in X$ be a regular vertex with cone angle $\alpha \geq 2\pi$. With respect to $d$, the space of directions $\Sigma_p(X)$ is a circle with length equal to $\alpha$. Moreover, if $\alpha,\beta \in \Sigma_p(X),$ then $\angle(\alpha,\beta)= \min\{d(\alpha,\beta),\pi \}.$
\end{lem}
\begin{proof}
	If $\Delta$ is a triangle in $X$, the geodesics issuing from $p$ foliate $\Delta$, and are parametrized by their terminal point on the opposite side of $\Delta$. In particular, $\Sigma_p(\Delta)$ is an interval, so $\Sigma_p(S)$ is a circle.
	
	Let's prove the `moreover' statement before calculating the length of $\Sigma_p(S)$. As $d\geq \angle$ and Alexandrov angles are always at most $\pi$, to prove the `moreover' statement it suffices to take $\alpha,\beta \in \Sigma_p(X)$ with $\angle(\alpha,\beta)<\pi$, and show that $d(\alpha,\beta)\leq \angle(\alpha,\beta).$
	
	To do this, fix a finite set $T \subset  \Sigma_p(S) \cong S^1$ such that the angle between any two points in an arc of $\Sigma_p(S) \setminus T$ is at most $\pi/4$. Then \cite[Lemma 9.1.39]{Buragocourse} says that on each such arc, we have $d=\angle$.  Take $x\in \alpha$ and $y\in \beta$ in a small convex CAT$(\kappa^+)$ neighborhood $B \ni p$, whose size will be determined later. Since $\angle(\alpha,\beta)<\pi$, the geodesic $$[x,y] \subset B$$ does not pass through $p$, and each geodesic through $p$ intersects $[x,y]$ at most once. It follows that the induced map $$G : [x,y] \longrightarrow \Sigma_p(X), \ G(t) = [p,t]$$ is an injective path in $\Sigma_p(X)$ from $\alpha$ to $\beta$. Let $$x=t_1,\ldots, t_n=y\in T$$
	be all the points along $[x,y]$ such that $G(t_i) \in \{\alpha,\beta\} \cup T,$ ordered in the way they appear along $[x,y]$. Since $d=\angle$ on (the closure of) each arc of $\Sigma_p(X) \setminus T$, we have $$\length G = \sum_{i=1}^{n-1} \angle(G(t_i),G(t_{i+1})).$$
	As long as our neighborhood $B \ni p$ is small, each angle $\angle(G(t_i),G(t_{i+1}))$ is well approximated by the angle in the $\BR^2$-comparison triangle for $p,t_i,t_{i+1}$, see \cite[II.3.1]{Bridsonmetric}, and a similar statement holds for $\angle(\alpha,\beta)$. So fixing $\epsilon>0$, Figure \ref{comptri} implies that when $B$ is small,
	$$\sum_{i=1}^{n-1} \angle(G(t_i),G(t_{i+1})) \leq \angle(\alpha,\beta) + \epsilon.$$
	Since $\epsilon$ is arbitrary, this shows $d(\alpha,\beta)\leq \angle(\alpha,\beta)$ as required. 
	
	Finally, to calculate the length of $\Sigma_p(X)$, note that for a fixed ruled triangle $\Delta$ incident to $p$, all angles in $\Sigma_p(\Delta)$ are less than $\pi$, so $\angle=d$ on $\Sigma_p(\Delta)$. Hence, the length of $\Sigma_p(\Delta)$ is the sum of the intrinsic angles of all the $\Delta$ incident to $p$, which is the intrinsic cone angle.
\end{proof}

%\begin{defi}[Interior angles]
%	Suppose that $f:\Delta\longrightarrow M$ is a ruled triangle, $v$ is a vertex of $\Delta$, and $\alpha,\beta$ are the adjacent sides of $\Delta$. The \emph{extrinsic angle} at $v$ is the angle at the point $f(v)\in M$ between the geodesics $f(\alpha),f(\beta)$. The \emph{intrinsic angle} at $v$ is the Alexandrov angle at $v$ between $\alpha,\beta$ in $\Delta$.
%\end{defi}
%	
%	See \cite[pg 9]{Bridsonmetric} for more information about Alexandrov angles. The intrinsic angle at $v$ is always at least as big as the extrinsic angle, since the path metric on $\Delta$ is at least the restriction of the distance function from $M$. Note also that if $f $ is smooth at $v$, then the extrinsic and intrinsic angles at $v$ agree. In particular, the extrinsic and intrinsic angles agree when $v$ is not the ruling vertex, or when $M$ is hyperbolic, in which case $\Delta$ is smooth and totally geodesic.

\subsection{Presimplicial ruled surfaces} Let $S$ be a closed surface and let $\CT$ be a \emph{triangulation} of $S$, by which we mean a homeomorphic identification with a $2$-dimensional $\Delta$-complex. 

A continuous map $f:S\longrightarrow M$ is a {\em pre-simplicial ruled surface with respect to $\CT$} if $f$ restricts to a simplicial ruled triangle on every triangle of $\CT$. Here, note that each triangle $\Delta$ of $\CT$ comes equipped with a natural affine structure, and we are requiring then that $f|_{\Delta}$ is ruled with respect to that affine structure.   When we wish to emphasize the triangulation we will write $$f:(S,\mathcal T)\longrightarrow M.$$ We call a triangle $\Delta$ of $\CT$ \emph{degenerate} if its $f$-image is a degenerate geodesic triangle in $M$, and we call $\Delta$ \emph{nondegenerate} otherwise.
 Note that there are three types of degenerate triangles $\Delta$, where either
\begin{enumerate}
\item $f(\Delta)$ is a point,
\item two edges of $\Delta$ are mapped by $f$ to nondegenerate geodesic segments, and one edge is mapped to a point, or
\item all three edges of $\Delta$ are mapped by $f$ to nondegenerate geodesic segments, but the image of one is the concatenation of the image of the other two.
\end{enumerate}
We'll refer to these as type (1), (2) and (3)  degenerate triangles below.

%For each triangle $\Delta$, define an equivalence relation $\sim_\Delta$ on $\Delta$ by setting $x\sim_\Delta y$ if there is a path $\gamma$ in $\Delta$ from $x$ to $y$ such that $f\circ \gamma$ is constant. Since $f(\Delta)$ may be an immersed segment or triangle with nontrivial self-intersection, it is possible that $f(x)=f(y)$ while $x\not \sim_\Delta y$. When $\Delta$ is nondegenerate, $x\sim_\Delta y$ if and only if $x=y$. Otherwise, either
%\begin{enumerate}
%\item $x\sim_\Delta y$ for all $x,y\in \Delta$, or 
%\item[(2, 3)] the $\sim_\Delta$-equivalence classes form a foliation of $\Delta$ by parallel line segments, with respect to some affine structure on $\Delta$,
%\end{enumerate}
%where the labels reflect the three types of degenerate triangles.

Let $d$ be the path pseudometric on $S$ obtained by defining the length of a path $\gamma$ to be the length of the path $f\circ \gamma$ in $M$. Fact \ref{CATtriangle} says that each non-degenerate triangle of $\CT$ is intrinsically a (metric) CAT($\kappa^+$)-triangle, in the sense of Definition \ref{catkdef}. However, within a degenerate triangle there are always distinct points $x,y$ with $d(x,y)=0$.

Let $\sim$ be the equivalence relation on $S$ where $$x\sim y \iff d(x,y)=0,$$ let $S' := S/{\sim}$ be the associated quotient space, and let $$c : S \longrightarrow S'$$ be the quotient map. Note that $f$ factors as $f = f' \circ c$ for some
$$f ' : S'\longrightarrow M.$$
The path pseudometric $d$ descends to a path metric $d'$ on $S'$, where now the length of a path $\gamma$ in $S'$ is the length of $f'\circ \gamma$.

\begin{fact}[Geometry of the quotient] \label{geomquotient}There is a refinement $\CR$ of $\CT$ and a $2$-dimensional $\Delta$-complex structure on $S'$ such that the map $$c : (S,\CR) \longrightarrow S'$$ is simplicial,\footnote{Recall that a map of $\Delta$-complexes is simplicial if it maps vertices spanning a simplex to vertices spanning a simplex, and is affine on simplices.} and where the following hold.
\begin{itemize}
\item The vertex set $\CR^0$ is contained in the $1$-skeleton $\CT^1$.
\item The vertex set of $S'$ is $c(\CT^0)$. 
\item With respect to the path metric $d'$, each edge is a (local) geodesic and $S'$ has CAT($\kappa^+$)-triangles. \end{itemize}
Moreover, for each triangle $\Delta$ of $\CR$, let $\sim_\Delta$ be the equivalence relation on $\Delta$ where $x\sim_\Delta y$ if there is a path $\gamma$ in $\Delta$ from $x$ to $y$ such that $f\circ \gamma$ is constant. Then $\sim$ is generated by all the relations $\sim_\Delta$, and for each $\Delta$, the relation $\sim$ restricts to $\sim_\Delta$ on $int(\Delta)$.
\end{fact}

Intuitively, the $\Delta$-complex structure on $S'$ is obtained by collapsing any degenerate triangles in $S$; the assertion that $c$ is simplicial means that each triangle $\Delta$ in $S$ either maps to a triangle in $S'$, a vertex of $S'$, or is collapsed onto an edge of $S'$ in a way that collapses one edge of $\Delta$ to a point. The reason we need to pass to a refinement of $\CT$ is because of the type (3) degenerate triangles referenced above, see Figure \ref{collapsefig}. 

One word of caution is in order: above, there may be edges $\CR$ that are not mapped to geodesics in $M$ by $f$. So, $f$ is not simplicial ruled with respect to $\CR$, and similarly the quotient map $f' : S' \longrightarrow M$ need not be `simplicial ruled', even in some generalized sense that allows $S'$ to be a  $\Delta$-complex instead of a surface. The reason for all of this is that when constructing $\CR$, it may be necessary to subdivide nondegenerate triangles in $\CT$ in a way that is incompatible with their rulings. Note that this paragraph is compatible with the statement that $S'$ has CAT$(\kappa^+)$ triangles since the latter is just an intrinsic statement, and doesn't reference the map into $M$. 

\begin{figure}
\centering
\includegraphics{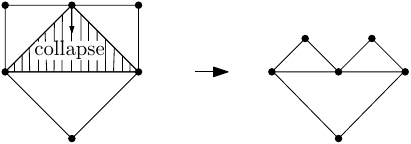}
\caption{Collapsing a type (3) degenerate triangle does not produce a triangulation, since in the result, a single edge of a triangle may be glued to multiple other triangles.}\label{collapsefig}
\end{figure}

\begin {proof}
We begin by constructing $\CR$. Define $\CR^0$ to be the set of all points in $c^{-1}(c(\CT^0))$ that are either in $\CT^0$, or lie on edges $e$ of $\CT$ such that $f(e)$ is not a point. For instance, on the left in Figure \ref{collapsefig}, the set $c^{-1}(c(\CT^0))$ would include all the shown vertices, plus an additional vertex in the middle of the long edge of the shaded triangle. On each edge $e$ that doesn't map to a point under $f$, the map $f$ restricts to an immersion, and $c^{-1}(c(\CT^0)) \subset f^{-1}(f(\CT^0))$, so $\CR^0$ is finite. 

To define the edges of $\CR$, fix a triangle $\Delta$ of $\CT$ and triangulate $\Delta$ with vertex set $\CR^0 \cap \partial \Delta$ as follows, depending on the type of $\Delta$.
\begin{itemize}
	\item If $\Delta$ is non-degenerate, we triangulate $\Delta$ using geodesics in the intrinsic CAT($\kappa^+$)-path metric on $\Delta$ given by Fact \ref{CATtriangle}.
	\item If $\Delta$ is type (1) degenerate, i.e.\ it maps to a point under $f$, the triangulation of $\Delta$ can be arbitrary.
	\item Suppose $\Delta$ is type (2) degenerate, so $f$ is constant on one edge $e_3$ of $\Delta$ and nonconstant on the other two, $e_1,e_2$. Then $\Delta$ is foliated by affine geodesics parallel to $e_3$ on which $f$ is constant. So, elements of $\CR^0 \cap \partial \Delta$ other than the intersection point of $e_1,e_2$ come in pairs, where each pair consists of the two endpoints of an affine geodesic in $\Delta$ that is contained in $c^{-1}(c(\CT^0))$. Use these affine geodesics and diagonal affine geodesics in each complementary quadrilateral to triangulate $\Delta$.
	\item Suppose $\Delta$ is type (3) degenerate, so $f$ is nonconstant on all three edges of $\Delta$ but maps the edge $e_3$ to the union of the other two, $e_1,e_2$. Then there is an affine geodesic $\alpha$ in $\Delta$ from the intersection point of $e_1,e_2$ to a point on $e_3$ such that $f\circ \alpha$ is constant. We declare $\alpha$ to be an edge of $\CR$. The two remaining halves of $\Delta$ are type (2) degenerate, and we triangulate them as in the previous case.
\end{itemize}
This defines a refinement $\CR$ of $\CT$. Note that as promised, $\CR^0$ is contained in the $1$-skeleton of $\mathcal T^1$. 

If $\Delta$ is a triangle of $\CR$, we call $\Delta$ \emph{nondegenerate} if it is contained in a nondegenerate triangle of $\CT$, and \emph{degenerate} otherwise. Note that the edges of a nondegenerate triangle $\Delta$ of $\CR$ might not map to geodesics under $f$, so $f|_{\Delta}$ is not necessarily a simplicial ruled triangle. However, if $\Delta$ is degenerate, then $f|_{\Delta}$ is a degenerate simplicial ruled  triangle in $M$. So, we can classify the degenerate triangles of $\CR $ into types (1)--(3). However, by construction, there are no type (3) degenerate triangles in $\CR$. Indeed, all degenerate triangles of $\CR $ were constructed above to have an edge that is sent to a point by $f$.

\medskip

We now want to construct the $\Delta$-complex structure on $S'$. For technical reasons, it is easiest to replace $\sim$ with an (a priori) finer equivalence relation $\approx$, set $X=S/\approx$, construct a $\Delta$-complex structure on $X$ with the required properties, and then show at the end that $\approx =\sim$, and therefore that $X=S'$. 

On each triangle $\Delta$ of $\CR$, consider the equivalence relation $\sim_\Delta$ where $x\sim_\Delta y$ if there is a path $\gamma$ in $\Delta$ from $x$ to $y$ such that $f\circ \gamma$ is constant. If $\Delta$ is nondegenerate, the $x\sim_\Delta y \iff x=y$. If $f(\Delta)$ is a point, then $x\sim_\Delta y $ for all $x,y\in \Delta$. Otherwise, $\Delta$ is type (2) degenerate, in which case the $\approx$-equivalence classes in $\Delta$ are exactly the affine geodesics parallel to the edge of $\Delta$ that is sent to a point by $f$.

Let $\approx$ be the equivalence relation on $S$ generated by all the $\sim_\Delta$, where $\Delta$ ranges over all triangles of $\CR $, let $X:= S/\approx$ and let 
$$\pi: S \longrightarrow X$$ be the quotient map. Let's define a $\Delta$-complex structure on $X.$ We start by defining the \bf vertex set \rm to be $\pi(\CR^0)\subset X$, and now we have to construct appropriate maps into $X$ from $1$ and $2$-simplices. 

Let's understand how edges of $\CR$ are identified under $\approx$. For each triangle $\Delta$ of $\CR$, the restriction of $\sim_\Delta$ to an edge $e$ of $\Delta$ either identifies all points of $e$ (this happens when $f(e)$ is a point), or identifies no distinct points of $e$. Moreover, if $e_1,e_2$ are two edges of $\Delta$ that are not collapsed to points by $f$, then $\sim_\Delta$ identifies points of $e_1$ with distinct points of $e_2$ only when there is an affine isomorphism $A : e_1 \longrightarrow e_2$ that commutes with $f$, and in that case $$x\in e_1 \sim_\Delta y\in e_2 \iff y=A(x).$$ Note that since $f$ restricts to an immersion on $e_1$ and $e_2$, the map $A$ is uniquely defined by $f$. Since the triangles $\Delta$ of $\CR$ are attached to each other along affine isomorphisms of edges, which may induce possible additional vertex identifications, all these statements translate to the generated equivalence relation $\approx$ as follows: the restriction of $\approx$ to the interior of an edge $e$ of $\CR$ either identifies all points $int(e)$ (this happens if $f(e)$ is a point) or no distinct points of $int(e)$, and if $e_1,e_2$ are two edges of $\CR$, the relation $\approx$ identifies points of $int(e_1),int(e_2)$ only when there is an affine isomorphism $A : e_1\longrightarrow e_2$ as above, in which case
$$x\in int(e_1 ) \approx y\in int(e_2) \iff y=A(x).$$

Pick a maximal collection $\{e_i\}$ of edges of $\CR$ such that no $f(e_i)$ is a point, and where for $i\neq j$, no point in $int(e_i)$ is $\approx$-equivalent to a point in $int(e_j)$. We define the \bf edges \rm of $X $ to be the maps
$$\pi|_{e_i} : e_i \longrightarrow X.$$
Note that as required in a $\Delta$-complex, $\pi|_{e_i}$ maps the endpoints of $e_i$ onto our previously defined vertices of $X$, and  is injective on $int(e_i)$. Injectivity follows since $x\in int(e_1) \approx y\in int(e_1)$ if and only if there is an affine isomorphism $A : e_1\longrightarrow e_1$ commuting with $f$ such that $A(x)=y$, and the only such $A$ is the identity.

Finally, for each non-degenerate triangle $\Delta $ of $\CR$ we define a \bf triangle \rm of $X$, namely the restriction
$$\pi|_{\Delta} : \Delta \longrightarrow X.$$
As the collection $\{e_i\}$ above is maximal, in $X$ the edges of $\Delta$ are attached via affine isomorphisms to the edges of $X$ previously constructed. Since $\Delta$ is non-degenerate, the equivalence relation $\approx$ satisfies $x\approx y \iff x=y$ on $int(\Delta)$, so $\pi|_{int(\Delta)}$ is injective. 

Each point $q\in X$ lies in the interior of a unique simplex. To see this, just take some $p\in S$ with $\pi(p)=q$. If $p$ is a vertex of $\CR$, then $q$ is a vertex of $X$. If $p$ lies in the interior of an edge $e$ of $\CR$, then either $f(e)$ is a point, in which case $q$ is a vertex of $X$, or $e$ is equivalent to some edge $e_i$ above, and $q\in \pi(int(e_i))$. If $p$ lies in the interior of a degenerate triangle $\Delta$ of $\CR$, then it is $\approx$-related to either a vertex or a point on the interior of an edge of $\Delta$, so we are in one of the cases already discussed. If $p$ lies in the interior of a non-degenerate triangle, then $q$ is in the interior of the corresponding triangle of $X$. It is also easy to see from the construction above that all these interiors of simplices are disjoint, so the one containing $p$ is unique.

Finally, we must check that a set $A \subset X$ is open if and only if all its preimages under the simplex maps above are open. (See Hatcher \cite[pg 103]{Hatcheralgebraic} for the definition of a $\Delta$-complex.) As $\pi$ is continuous, the nontrivial direction is to assume that all the preimages are open, and try to show that $A$ is open. Well, look at the way the preimage $\pi^{-1}(A) \subset S$ intersects each triangle $\Delta$ of $\CR$. If $\Delta$ is type (1) degenerate, i.e.\ $f(\Delta)$ is a point, then $\pi^{-1}(A) \cap \Delta$ is either empty or $\Delta$, so is open in $\Delta$. If $\Delta$ is a type (2) degenerate triangle, then $\pi^{-1}(A) \cap \Delta$ is a union of affine geodesics parallel to the collapsed edge of $\Delta$, and is therefore determined by its intersection with either of the non-collapsed edges of $\Delta$. Since the preimages $\pi|_{e_i}^{-1}(A)$ are all assumed to be open, and some $e_i$ is identified by $\approx$ with the two non-collapsed edges of $\Delta$, it follows that $\pi^{-1}(A) \cap \Delta$ is open. Finally, if $\Delta$ is a non-degenerate triangle, the intersection $\pi^{-1}(A) \cap \partial \Delta$ is open by assumption. It then follows that $\pi^{-1}(A)$ is open in $S$, so $A$ is open in the quotient space $X$.

We now have equipped $X$ with a $\Delta$-complex structure, which is constructed exactly so that the map $\pi: (S,\CR) \longrightarrow X$ is simplicial. Indeed, $\pi$ takes vertices to vertices, and if $e$  is an edge of $\CR$, either $f(e)$ is a point, in which case $\pi|_e$ factors through the inclusion of a vertex, or there is an affine isomorphism $A : e\longrightarrow e_i$ for some $i$ such that $\pi|_e = \pi|_{e_i} \circ A$. Similarly, if $\Delta$   is a triangle of $\CR$, then $\pi|_{\Delta}$ factors affinely through one of the simplex maps of $X$, where the dimension of the simplex is either $0,1$ or $2$ depending on whether $\Delta$ is type (1) degenerate, type (2) degenerate, or non-degenerate.

Now since $\approx$  is a finer relation than $\sim$, the path pseudometric on $S$ descends to a path pseudometric $D$ on $X$. Since the edges of $X$ are of the form $\pi|_{e_i} : e_i \longrightarrow X$, where $f|_{e_i}$ is an immersed geodesic in $M$, all edges of $X$ are (non-zero speed) local geodesics with respect to $D$. Similarly, since the triangles of $X$ are of the form $\pi|_{\Delta}$, where $\Delta$  is a nondegenerate triangle of $(S,\CR)$, each such triangle is intrinsically a CAT($\kappa^+)$-triangle. So in fact, $D$ can be constructed by patching together actual path metrics on all the simplices of $X$, implying that $D$ is a metric, not just a pseudo-metric. Hence, $\approx = \sim$, $X=S'$, $\pi=c$ and $D=d'$, so we have  constructed a $\Delta$-complex structure on $S'$, and proved the third bullet point in the fact. Also, the vertex set of $S'$ was defined to be $c(\CR^0)$, and we have $\CT^0 \subset \CR^0 \subset c^{-1}(c(\CT^0))$, so $c(\CR^0) = c(\CT^0)$, which proves the second bullet point in the fact. 

For the `moreover' statement, the fact that the $\sim_\Delta$ generate $\sim$ is just the assertion that $\sim=\approx.$ The fact that $\sim$ restricts to $\sim_\Delta$ on $int(\Delta)$ can be proved in cases. If $\Delta$ is nondegenerate, both relations are just `equality' on $int(\Delta)$. If $f(\Delta)$ is a point, both relations identify all of $int(\Delta)$ to a point. And if $\Delta$ is type (2) degenerate, then this follows from the fact that the interiors of the edges of $\Delta$ that are not collapsed to points both inject under $c$, since they do not admit affine automorphisms commuting with $f$ as mentioned above.
\end{proof}

We can also use the fact above to describe preimages under $c$ of open simplices in $S'$.

\begin{lem}[Simplex preimages]\label{preimageslem}
If $q\in S'$ is a vertex, then $c^{-1}(q)$ is a connected simplicial subcomplex of $\CR$. If $e$ is an edge of $S'$, then
$$c: c^{-1}(int(e)) \longrightarrow int(e)$$ is a fiber bundle with fibers homeomorphic to either a point, an interval or a circle. In the latter two cases, each fiber is a union of affine-geodesic segments in type (2) degenerate triangles $\CR$, with at most one segment per triangle. If $\Delta$ a triangle of $S'$, then $c^{-1}(int(\Delta))$ is the interior of a nondegenerate triangle of $\CR$, on which $c$ is injective. 

In particular, all point preimages under $c$ are connected.
\end{lem}
\begin{proof}
	Let $q\in S'$ be a vertex. Then there is a vertex $p \in \CR^0$ in $c^{-1}(q)$. In each generating relation $\sim_\Delta$, the set of points $\sim_\Delta$-equivalent to a subsimplex of $\Delta$ is a (possibly larger dimensional) subsimplex of $\Delta$. So, if we construct $c^{-1}(q)$ by starting with $p$ and applying the relations $\sim_\Delta$ iteratively, we get that $c^{-1}(q)$ is a subcomplex of $S$. 
	
	Next, let $e$ be an edge of $S'$. Since $e$ is not a point, any triangle of $\CR$ that intersects $c^{-1}(int(e))$ is either nondegenerate or type (2) degenerate. So, start with some edge $e_0$ in $\CR$ with $c(e_0)=e$. If one of the two triangles $\Delta_0$ of $\CR$ containing $e_0$ is degenerate, then if we let $e_1$ be the other edge of $\Delta_0$ that's not collapsed to a point by $c$, we have that $c^{-1}(int(e))$ contains $int(e_0) \cup  int(f_0) \cup int(\Delta_0)$, and on this union $c$ is an interval bundle. (The interval bundle structure comes since $\sim$ restricts to $\sim_{\Delta_0}$ on $int(\Delta_0)$, and hence also on the union above.) We then continue this process. If there's another type (2) degenerate triangle $\Delta_1$ containing $e_1$, we get another interval bundle in $c^{-1}(int(e))$ that concatenates with the one produced above, etc. This process terminates either with  $c^{-1}(int(e))$ an interval bundle, or a circle bundle.
	
	The assertion about preimages of triangle interiors $int(\Delta)$ is immediate, since $\sim=\sim_\Delta$ on $int(\Delta)$, and the latter relation is just `equality'. 
\end{proof}

Define the \emph{regular part} $S'_{reg}\subset S'$ to be the set of points that have a neighborhood homeomorphic to $\BR^2$, and let the \emph{singular part} $S'_{sing}\subset S'$ be its complement. Note that the vertices of $S'$ that are contained in the regular part are exactly the `regular vertices' of Fact \ref{CATfactX}. Then $S'_{sing}$ is a graph, $S'_{reg}$ is homeomorphic to the complement of a finite set of points in a (possibly disconnected) closed surface, and the intersection $\overline{S'_{reg}}\cap S'_{sing}$ is a finite set of vertices.

\begin{kor}\label{regsingkor}
An edge $e$ of $S'$ is in $S'_{sing}$ if and only if the map $c : c^{-1}(int(e))\longrightarrow int(e)$ is a circle bundle.
	\end{kor}

	\begin{proof}
		This follows from Lemma \ref{preimageslem}, and its proof. If $e$ is an edge of $S'$, there are $0$ triangles attached to $e$ when $c^{-1}(int(e))$ is a circle bundle, since then all edges in $S$ that map to $e$ are part of type (2) degenerate triangles. If $c^{-1}(int(e))$ is an interval bundle, there are two triangles attached to $e$, which are the nondegenerate triangles adjacent to the edges forming the $\partial I$-subbundle of the interval bundle $c^{-1}(int(e))$.
	\end{proof}

\subsection{Simplicial ruled surfaces} 
In the previous subsection, we saw that when $f: S \longrightarrow M$ is a presimplicial ruled surface, the pullback path pseudometric on $S$ collapses it to a $\Delta$-complex $S'$ that is a union of a graph $S'_{sing}$ and a possibly punctured surface $S'_{reg}$.

\begin{defi}[SRS]
	We say that a presimplicial ruled surface $f : S \longrightarrow M$ is a \emph {simplicial ruled surface}, or SRS, if every regular vertex in $S'$ (i.e.\ any vertex contained in $S'_{reg}$) has cone angle at least $2\pi$. 
\end{defi}

%Since extrinsic cone angles are at most the intrinsic cone angles, Lemma \ref{CATfact} says that simplicial ruled surfaces are CAT$(\kappa^+)$. Indeed, the reader may wonder why we do not use a more general definition of an SRS, using intrinsic cone angles instead of extrinsic. While almost everything we say about SRSs below applies to surfaces with intrinsic cone angles at least $\pi$, it is technically convenient that the property oThat considered in previous papers.f $f : S \longrightarrow M$ being a simplicial ruled surface depends only on the restriction of $f$ to the $1$-skeleton of the triangulation. Also, in practice one never constructs surfaces with intrinsic cone angle at least $2\pi$, but extrinsic cone angles less than $2\pi$.

Recall from after Definition \ref{catkdef} that the cone angle at a vertex of $S'_{reg}$ is the sum of the Alexandrov angles of the incident $2$-simplices. The definition of an SRS we give above is more general than that considered in previous papers, e.g. \cite{Agoltameness,souto2008short}, in which all triangles in $S$ are assumed to map to nondegenerate triangles in $M$. In that case, $S'=S'_{reg}=S$, so the condition above is just that all cone angles in $S$ are at least $2\pi$. %We use the more general version above since these are the surfaces that arise naturally as maps minimizing the energy of the $1$-skeleton of a given triangulation, as we will discuss in  \S \ref{homotopies}.

As a consequence of Lemma \ref{CATfactX} and \cite[II.11.1]{Bridsonmetric}, which controls the geometry along an edge where two triangles are glued, we have:

\begin{kor}\label{CATfact}
If $f : S \longrightarrow M$ is a simplicial ruled surface, then $S'_{reg}$ is locally CAT($\kappa^+$).
\end{kor}

Simplicial ruled surfaces can be constructed in many ways. Here is one of the most well-known methods.

\begin{example}[`Useful' simplicial ruled surfaces]\label{usefulsrs}
Suppose that  $$f: S \longrightarrow M$$ is a $\pi_1$-injective map from a surface $S$ with genus at least $2$ into a $3$-manifold $M$ with pinched negative curvature as above. Then $f$ can be homotoped to a simplicial ruled surface as follows. First, triangulate $S$ with one vertex in such a way that all edges are homotopically essential loops, fix a distinguished edge $e$, and homotope $f$ so that $e$ maps to a closed geodesic in $M$. Then homotope all other edges rel endpoints to geodesic segments, and homotope the interior of every triangle to be ruled using Fact \ref{ruledconstruction}.  While $f$ cannot send edges to points, it is possible that some triangles are degenerate\footnote{This detail seems to be ignored in some other treatments of the subject, e.g.\ \cite{Canarycovering}. If one really wants all triangles here to be nondegenerate, one can accomplish this with a perturbation, but it's not necessary to do that in our framework.}, since $f$ can take one edge of a triangle to the union of the other two.  However, using the notation of Fact \ref{geomquotient}, the resulting CW-structure on $S'$ has a single vertex, which has cone angle at least $2\pi$ since the geodesic $e$ goes through it. Hence, $f$ is simplicial ruled. Simplicial ruled surfaces with a single vertex and a distinguished edge that is mapped to a closed geodesic are called \emph{useful} in Canary \cite{Canarycovering}. In the hyperbolic setting, this construction was used previously by Bonahon~\cite{Bonahonbouts} and Thurston~\cite[pg 195]{Thurstongeometry}.
\end{example}

%	If $f: S \longrightarrow M$ is a simplicial ruled surface and each ruled triangle of $S$ is isometric to a totally geodesic hyperbolic triangle, $f$ is called a {\em simplicial hyperbolic surface}. Any simplicial ruled surface in a hyperbolic $3$-manifold is simplicial hyperbolic. Note that whenever $f: S \longrightarrow M$ is simplicial hyperbolic, a smooth structure for $S$ can be chosen so that the induced path metric is smooth and hyperbolic except at vertices $p$ of the given triangulation where the cone angle $\alpha(p) > 2\pi$. 

Here is a basic property of simplicial ruled surfaces that we will need in the next section.

\begin{fact}\label{preimagefact}
Suppose $f : S \longrightarrow M$ is simplicial ruled. If $p\in S'_{sing}$, the preimage $c^{-1}(p)$ is non-simply connected and is $\pi_1$-injective in $S$. For each component $X \subset S'_{reg}$, the preimage $c^{-1}(X)$ is an open, $\pi_1$-injective subsurface of $S$ with $\chi(c^{-1}(X)) \leq \chi(X)<0$.
\end{fact}

We'll need the following lemma.

\begin{lem}\label{transversality}
Suppose that $U \subset S'$ is a closed regular neighborhood of a vertex $p\in S'$, and $\alpha$ is a component of $\partial U$. Then $c^{-1}(\alpha)$ is a (possibly inessential) simple closed curve in $S$.
\end{lem}
\begin{proof}
Each $\alpha \subset \partial U$ is either a point in the interior of some singular edge $e$, or a loop in $S'_{reg}$. In the first case, $c^{-1}(\alpha)$ is a simple closed curve by Corollary \ref{regsingkor}. In the second, $\alpha$ is a concatenation of segments in nondegenerate triangles. The preimage of the interior of each such segment is a segment in $S$, and the preimage of each intersection point of $\alpha$ an each edge of $S'$ is either a point or a segment, by Lemma \ref{preimageslem}. So, $c^{-1}(\alpha)$ is a simple closed curve.
\end{proof}

\begin{proof}[Proof of Fact \ref{preimagefact}]
Let $p\in S'_{sing}$ and let $U \ni p$ be an (open) regular neighborhood of $p$ in $S'$.  By Lemma \ref{transversality}, the preimage $T:=c^{-1}(\overline U)$ is compact subsurface with boundary in $S$. Since point preimages of $c$ are connected and $U$ is connected, $T$ is connected. Since $p\in S'_{sing}$, the boundary $\partial U$ has at least two components, so $\partial T$ has at least $2$ components. Hence $T$ is not simply connected. But $T$ deformation retracts onto $c^{-1}(p)$, so $c^{-1}(p)$ is not simply connected either. We'll postpone $\pi_1$-injectivity and now shift attention to $S'_{reg}$.

Let $X \subset S'_{reg}$ be a component. Then $c^{-1}(X)$ is an open subsurface of $S$. To control its Euler characteristic, first note that as all cone angles in $X$ are at least $2\pi$ and the interior angles of a CAT($\kappa^+$)-triangle always sum to less than $\pi$, we have $\chi(X)<0$ by a standard calculation. We claim that $\chi(c^{-1}(X)) \leq \chi(X)$. 

For any vertex $p\in \overline X$, let $U_p$ be the intersection of a closed regular neighborhood of $p$ in $S'$ with $X$. If $p\not \in  X$, then $c^{-1}(U_p)$ is homeomorphic to a collection of once-punctured disks --- this is proved with the same argument as Lemma \ref{transversality} --- and hence has Euler characteristic zero. (Here, the number of disks is the number of components of $U_p \setminus p$.)  If $p\in X$, then $c^{-1}(U_p)$ is a connected subcomplex of a surface that is not a sphere, and hence has Euler characteristic at most $1$. Next, let $e$ be an edge of $\overline X$, and let $U_e$ be obtained by taking a closed neighborhood of $e$ in $X$ and removing $int(U_p)$ and $int(U_q)$, where $p,q$ are the endpoints of $e$. Then $c^{-1}(U_e) $ is homeomorphic to a disk. Finally, for each triangle $\Delta$ of $\overline X$, let $U_\Delta \subset \Delta$ be the subset obtained by removing the interiors of all the $U_p,U_e$, where $p,e$ range over the vertices and edges of $\Delta$. We can now build up $c^{-1}(X)$ as follows. Start with the union of all the $c^{-1}(U_p)$, which has Euler characteristic at most the number of vertices of $X$. Then attach the sets $c^{-1}(U_e)$: each of these is a disk which is attached along two arcs, so attaching it decreases Euler characteristic by $1$. Finally, attach all the sets $c^{-1}(U_\Delta)$; each of these is a disk attached along its boundary, so attaching it increases Euler characteristic by $1$. Using the definition of $\chi(X)$ as the alternating sum of the numbers of vertices, edges, and faces, it follows that $\chi(c^{-1}(X))\leq \chi(X) <0$.

It only remains to show that preimages of singular points and components of the regular part are $\pi_1$-injective. To this end, for every singular vertex $q\in S'_{sing}$, let $U_q$ be a closed regular neighborhood of $q$ in $S'$, and let $\CU=\cup_q U_q$. Then we have
\begin{equation}\label {decomposition of surfaces}S =  \bigcup_{\substack{\text{components } \\X \subset S'_{reg}}} c^{-1}(X \setminus \mathcal U) \ \ \cup \ \bigcup_{\substack{\text{edges } \\e\subset S'_{sing}}} c^{-1}(e \setminus \mathcal U) \ \ \cup \ \ \bigcup_{\substack{\text{vertices } \\ q \in S'_{sing}}}c^{-1}(U_q)
\end{equation}
which by Lemma \ref{transversality} is a decomposition of $S$ into compact, connected subsurfaces glued along their boundaries. The subsurfaces in the first union are deformation retracts of the full preimages $c^{-1}(X)$, and the subsurfaces in the second and third unions deformation retract onto point preimages $c^{-1}(p), p\in S'_{sing}$, where $p\in int(e)$ or $p=q$ depending on which case we are in. By the first three paragraphs of the proof, none of these subsurfaces are disks. It follows from an iteration of Van Kampen that each subsurface is $\pi_1$-injective in $S$. So each $c^{-1}(X)$ is $\pi_1$-injective in $S$ since it deformation retracts onto one of the subsurface is in the first union, while each point preimage $c^{-1}(p), p\in S'_{sing}$ is $\pi_1$-injective since it is a deformation retract of one of the subsurfaces in the second or third union. 
\end{proof}

Here's a useful corollary:

\begin{kor}[Uncollapsed SRSs]\label{pi1injkor}
Suppose that $f : S \longrightarrow M$ is a simplicial ruled surface such that no essential closed curve $\gamma \subset S$ maps to a point under $f$. Then $S'_{sing}$ is empty, so $S'$ is a closed surface. Moreover, the map $c : S \longrightarrow S'$ is a homotopy equivalence.
\end{kor}

For instance, the corollary applies when $f$ is $\pi_1$-injective.

\begin{proof}
If there is some $p \in S'_{sing}$, then Lemma \ref{preimagefact} says that $c^{-1}(p)$ contains a  closed curve $\gamma $ that is essential in $S$. So, $S'_{sing}$ is empty, and hence $S'$ is a surface.

We claim that $c$ is a homotopy equivalence. By a theorem of Youngs \cite{youngs1948homeomorphic} it suffices to show that all point preimages $c^{-1}(p), p\in S'$ are contractible. If $p$ lies in the interior of a triangle, then $c^{-1}(p)$ is a point. If $p$ lies in the interior of an edge, $c^{-1}(p)$ is a segment or a point by Lemma \ref{preimageslem} and Corollary \ref{regsingkor}. 

If $p$ is a vertex of $S'$, then $c^{-1}(p)$ is a deformation retract of $c^{-1}(U)$, where $U \subset S'$ is a small closed regular neighborhood of $p$. As $\partial U$ is connected, Lemma \ref{transversality} says that $c^{-1}(U)$ is a subsurface of $S$ with connected boundary. By our assumption on $S$, the inclusion $c^{-1}(p) \hookrightarrow S$  is trivial on $\pi_1$, so the same is true for the inclusion $c^{-1}(U) \hookrightarrow S$. 

We claim that $c^{-1}(U)$ is a disk. If not, then $\pi_1$-triviality implies that the complement $S \setminus c^{-1}(U)$ has a disk component. As $c^{-1}(U)$ has connected boundary,  $S \setminus c^{-1}(U)$ is a disk. But in a higher genus surface the complement of a disk never has $\pi_1$-trivial inclusion, so this is a contradiction.
\end{proof}
\subsection{Bounded diameter lemmas}
\label{bdlsec}
Simplicial ruled surfaces in manifolds with an upper curvature bound $\kappa^+$ have bounded diameter relative to their thin parts. This idea goes back to Thurston \cite{Thurstongeometry} and Bonahon~\cite{Bonahonbouts}. Essentially, the point is that these surfaces have an area upper bound, and a lower bound on the areas of small balls in their thick parts. While the statements below may be familiar from other references, we remind the reader that our definition of a simplicial ruled surface is more general than usual, and we also include some details that we think are missing from other references.

Let $f: S \longrightarrow M$ be a simplicial ruled surface. We would first like to apply Gauss-Bonnet to give an upper bound on the area of $S$. To control the area of CAT($\kappa$)-triangles, we use the following consequence of Reshetnyak’s Majorization Theorem \cite[9.53 and 9.54]{alexander2019alexandrov}.

\begin{fact}[Majorization of triangles]\label {triangle majorization}
Suppose that $\Delta$ is a CAT($\kappa$)-triangle, in the sense of Definition \ref{catkdef}, and let $\Delta_{\kappa} \subset M^2_{\kappa}$ be (the convex hull of) a triangle with the same side lengths as $\Delta$. Then there is a $1$-lipschitz surjection $\Delta_{\kappa} \longrightarrow \Delta$ that maps each edge of $\Delta_{\kappa}$ isometrically onto the corresponding edge of $\Delta$.
\end{fact}

This has the following consequence.

\begin{kor}[Gauss-Bonnet for SRSs]\label{majorizationkor}
Suppose that $f : S \longrightarrow M$ is a simplicial ruled surface, with $S'=S/\sim$ the quotient. Then 
\begin{equation}\label{eq:gauss-bonnet}
\area(S')\le\frac{2\pi}{|\kappa^+|}\vert\chi(S)\vert.
\end{equation}
\end{kor}

Recall that $S$ is equipped with the pseudo-metric defined so that $f$ preserves path lengths, and $S'$ is the quotient metric space. Here, `area' could mean the scale by $\pi/4$ of $2$-dimensional Hausdorff measure, which agrees with Lebesgue measure on $\BR^2$, but one can also calculate it by adding up the Riemannian areas of the interiors of the triangles in $S'$. The constant $\kappa^+$ is the curvature upper bound for $M$. %If one prefers to work with Hausdorff measure on metric spaces instead of pseudo-metric spaces, one can think of the corollary as applying to the quotient metric space $S'$ discussed in the previous section. 

\begin{proof}
Let $X$ be a component of $S'_{reg}$. Then the closure $\overline X$ is the union of $X$ with some finite number of \emph{boundary vertices}. For each triangle $\Delta \subset \overline X$, take the corresponding triangle $\Delta_{\kappa^+}$ from Fact \ref{triangle majorization}, delete any vertices corresponding to boundary vertices of $X$, and glue all the resulting triangles-with-missing-vertices together as dictated by the way the $\Delta$ are glued, to create a space $X_{\kappa^+}$ homeomorphic to $X$ that comes equipped with a $1$-lipschitz surjection $$ X_{\kappa^+} \longrightarrow X.$$

The Alexandrov interior angles of a CAT$(\kappa^+)$-triangle $\Delta$ are at most the angles of the corresponding $\Delta_{\kappa^+}$, so all vertices of $ X_{\kappa^+}$ have cone angles at least $2\pi$, since we deleted any vertices corresponding to boundary vertices of $X$. Applying Gauss-Bonnet to each $\Delta_{\kappa^+}$ and using that $1$-lipschitz maps cannot increase Hausdorff measure, we obtain $$\area X \leq \area X_{\kappa^+} \leq \frac{2\pi}{|\kappa^+|}\vert\chi(X)\vert$$ 
As $X$ varies over all components of $S'_{reg}$, Fact \ref{preimagefact} says the preimages $c^{-1}(X)$ are disjoint, $\pi_1$-injective subsurfaces of $S$, and that $$\chi^{-1}(c^{-1}(X)) \leq \chi(X)<0.$$ So, we have $$\sum_X |\chi(X)| \leq \sum_X |\chi(c^{-1}(X))| \leq |\chi(S)|,.$$
and the corollary follows.
\end{proof}

Given this area upper bound, the other component of a bounded diameter lemma is a lower bound for the area of small balls in the thick part of $S$. As in the Riemannian case, the \emph{injectivity radius} $\inj_p(S)$ at a point $p\in S$ is half the length of a shortest homotopically essential loop at $p$. Note that since $S$ is only equipped with a path pseudo-metric, it is possible that $\inj_p(S)=0$ for some $p\in S$. 

\begin{lem}\label{injcompare}
Suppose $f: S \longrightarrow M$ is simplicial ruled, $p\in S$, and $q=c(p)\in S'$. If $q\in S'_{sing}$, then $\inj_p(S)=0$. Otherwise, $$\inj_{p}(S) \leq \min\{ d(q,S'_{sing}), \inj_q(S'_{reg})\}.$$
\end{lem}

Note that the reverse inequality is not always true. Namely, it is possible that $c : S \longrightarrow S'$ collapses an entire essential subsurface $X \subset S$ to a vertex in $S'_{reg}$, and then if $p\in X$, we'll have $\inj_p(S)=0$ but the minimum on the right will be nonzero.

\begin{proof}
First, we always have $ \inj_p(S) \leq d(q,S'_{sing})$, since we can create an essential loop in $S$ based at $p$ by following a shortest path from $p$ to $c^{-1}(S'_{sing})$, traversing a zero length essential loop given by Fact \ref{preimagefact}, and then following the path back to $p$.  So, assume that $q\in S'_{reg}$.

To show that $\inj_{p}(S) \leq \inj_q(S'_{reg})$, pick some small $\epsilon>0$ and take an essential loop $\gamma$ in $S'_{reg}$ based at $q$ that has length less than $2\inj_q(S')+\epsilon$. After a small perturbation, we can assume that $\gamma$ is a finite concatenation of paths $\gamma_i$, each of which is contained in some simplex of $S'$. For each $\gamma_i$, there is some path $\alpha_i $ in some simplex of $S$ such that $c\circ \alpha_i = \gamma_i$. (This follows from Fact \ref{geomquotient}, since the restriction of $\sim $ to the interior of a simplex $\Delta$ is the explicit relation $\sim_\Delta$.) As point preimages under $c$ are all path connected, for each $i$ we can choose a path that connects the terminal endpoint of each $\alpha_i$ to the initial endpoint of $\alpha_{i+1}$, without leaving their common preimage. With this method, we can construct a loop $\alpha$ in $S$ based at $p$ with the same length as $\gamma$ such that $c\circ \alpha$ parameterizes $\gamma$, albeit possibly with intervals where the parameterization is constant. Any nullhomotopy of $\alpha$ can be performed in $c^{-1}(S'_{reg})$, since the latter is incompressible in $S$ by Fact \ref{preimagefact}. And any nullhomotopy of $\alpha$ in $c^{-1}(S'_{reg})$ composes with $c$ to a nullhomotopy of $\gamma$ in $S'_{reg}$, so it must be that $\alpha$ is an essential loop in $S$. Hence $$\inj_{p}(S) \leq \length(\gamma) \leq \inj_q(S') +\epsilon,$$ for our arbitrary $\epsilon$.
\end{proof}

\begin{fact}[Convex balls]\label{convexballs}
	Suppose $S$ is simplicial ruled, $p\in S$ and $q=c(p)\in S'$. If $r < \inj_p(S)/2$, the ball $B(q,r) \subset S'$ is convex and CAT($\kappa^+$). Moreover,
	$\area(B(q,r))\geq \pi r^2.$
\end{fact}
\begin{proof}
By Lemma \ref{injcompare}, the metric ball $B(q,2r)$ lies in a component $X \subset S'_{reg}$, and $2r < \inj_q(X)$. 

We claim $B(q,2r)$ lifts homeomorphically to the universal cover $\tilde X$.  Indeed, we can construct such a lift by lifting geodesics issuing out from $q$. This fails only if there is a point $x\in B(q,2r)$ that is connected to $q$ by two geodesics $\alpha,\beta$ with length less than $2r$ that are not homotopic rel endpoints in $X$, which contradicts that $2r < \inj_q(X)$. 

By Fact \ref{CATfact}, $\tilde X$ is CAT($\kappa^+$). Picking a lift  $\ell : B(q,2r)\hookrightarrow \tilde X$, whenever $x,y\in B(q,r)$, the  unique geodesic $[\ell(x),\ell(y)] \subset X$ lies in $\ell(B(q,r))$, by convexity of the distance function, so projects to a geodesic in $B(q,r)$ of length less than $2r$. And since $B(q,2r)$ lifts to $\tilde X$, any other geodesic in $S'$ from $x$ to $y$ has length at least $(2r - d(q,x)) + (2r-d(q,y)) \geq 2r$. So, $B(q,r)$ is convex in $S'$. And as $B(q,r)$ is isometric to its lift in $\tilde X$, we have that $B(q,r)$ is CAT($\kappa^+$).

We now prove the area bound. Working in $X \subset S'_{reg}$, it suffices to prove that $\partial B(q,r)$ has length at least $2\pi r$. Indeed, since $X$ is piecewise Riemannian, the coarea formula \cite{federer1959curvature} applies to the Lipschitz function $d(q,\cdot)$, so we can integrate over $r$ to get the desired area lower bound. 

The space of directions $\Sigma_q(S)$ is intrinsically a circle of length at least some $\alpha \geq 2\pi$; this follows from Lemma \ref{sofd} when $q$ is a vertex, and is immediate from the piecewise Riemannian structure otherwise. Pick a cyclically ordered set $\{\gamma_1,\ldots,\gamma_n\} \subset \Sigma_q(S)$ that divides $\Sigma_q(S)$ into arcs of equal length $\alpha/n$. Each $\gamma_i \in T$ is a geodesic issuing from $q$, and intersects $\partial B(q,r)$ in a point $x_i$. Then we have
\begin{equation}\label{distanglesum}
	\length \partial B(q,r) \geq \sum_i d(x_i,x_{i+1}),
\end{equation} where $i+1$ is interpreted cyclically. As long as $\alpha/n$ is less than $\pi$, the Alexandrov angle $\angle(\gamma_i,\gamma_{i+1}) = \alpha/n$ by Lemma \ref{sofd}. So, $d(x_i,x_{i+1})$ is at least the length of the opposite side in a Euclidean triangle made by two length $r$ edges joined at angle $2\pi/n$. But then the sum in \eqref{distanglesum} is at least the length of a regular $n$-gon inscribed in a circle of radius $r$ in $\BR^2$. As  $n\to \infty$, we get that $	\length \partial B(q,r) \geq 2\pi r$.
\end{proof}

Given $\epsilon>0$, define the \emph{$\epsilon$-thick part} $S_{\geq \epsilon}$ to be the set of points on $S$ where the injectivity radius is at least $\epsilon$, and let the \emph{$\epsilon$-thin part} $S_{<\epsilon}$ be its complement. We can now prove:

\begin{named}{Bounded Diameter Lemma}[Thurston, Bonahon]
If $f:S\longrightarrow M$ is a simplicial ruled surface and $\epsilon>0$, every two points on $S$ are joined by a path $\gamma$ such that
$$\length (\gamma \cap S_{\geq\epsilon}) \leq \frac{32\cdot|\chi(S)|}{| \kappa^+ | \epsilon}.$$
\end{named}

Again, the length of a path on $S$ is the length of its composition with $f$, i.e.\ its length with respect to the pseudometric $d$ on $S$.

\begin{proof}
Given $x,y\in S$, let $\gamma$ be a shortest path from $x$ to $y$, ordered via a parameterization from $x$ to $y$. Take a maximal set of points $x_1<\cdots<x_n$ on $\gamma \cap S_{\geq\epsilon}$ such that the length of $\gamma|_{[x_i,x_{i+1}]}$ is at least $\epsilon/2$ for each $i$. Then $\gamma \cap S_{\geq\epsilon}$ is contained in the union of all the radius $\epsilon/2$ intervals of $\gamma$ centered around the $x_i$, so we have $$\length (\gamma \cap S_{\geq\epsilon}) \leq n \cdot \epsilon.$$ 

Since $\gamma$ is distance minimizing, the distance in $S$ between any two $x_i,x_j$ is at least $\epsilon/2$. So, the $\epsilon/4$-balls around the $x_i$ are disjoint, and in light of Corollary \ref{majorizationkor} and Fact \ref{convexballs} we have $$n \cdot \pi(\epsilon/4)^2 \leq \frac{2\pi}{|\kappa^+|}\vert\chi(S)\vert. $$
The Lemma follows.
\end{proof}

\begin{kor}[A Bers constant for SRSs]\label{berskor}
If $f: S \longrightarrow M$ is a SRS, there is an essential simple closed curve $\gamma$ on $S$ with $$\length(\gamma) \leq \max\{2\epsilon, 64 |\chi(S)| / (|\kappa^+|\epsilon)\}.$$
\end{kor}

Of course, taking $\epsilon$ to be the $\kappa^\pm$-Margulis constant, we get a Bers constant depending only on $\chi(S)$ and $\kappa^\pm$.

\begin{proof}
If $S_{<\epsilon}$ is nonempty, we have a curve of length at most $2\epsilon$. Otherwise, the bounded diameter lemma says that $$diam(S) \leq D:=32 |\chi(S)| / (|\kappa^+|\epsilon).$$ Assume a shortest essential curve $\gamma$ on $S$ has length more than $2D$. Then there are two points $x,y $ on $\gamma$ that cut $\gamma$ into two arcs both of length  at least $D$. If $\alpha$ is an arc with length at most $D$ joining $x,y$ then one of the two $\alpha$-surgeries on $\gamma$ is essential and shorter than $\gamma$, a contradiction. 
\end{proof}

 Note that in general, we have no fixed lower curvature bound for simplicial ruled surfaces in $M$, so a priori there may be simplicial ruled surfaces whose $\epsilon$-thick parts are empty, even when $\epsilon$ is small relative to the pinching constants $\kappa^\pm$ of $M$.   

\begin{defi}\label{NATdef}
A simplicial ruled surface $f : S \longrightarrow M$ is \emph{not accidentally $\epsilon$-thin}, or $\epsilon$-NAT, or just \emph{NAT} when $\epsilon$ is understood, if whenever $\gamma$ is an essential closed curve on $S$ with length less than $\epsilon$, the image $f(\gamma)$ is essential in $M$. 
\end{defi}

Here, as usual we equip $S$ with the pullback path pseudometric.

\medskip

When $f: S \longrightarrow M$ is $\epsilon$-NAT, we have $ f(S_{<\epsilon}) \subset M_{<\epsilon}$. As an example, any $\pi_1$-injective simplicial ruled surface $S \longrightarrow M$ is $\epsilon$-NAT.  More generally, any $S \longrightarrow M$ that is incompressible within an $\epsilon$-neighborhood of its image is also $\epsilon$-NAT. Note that when $f:S \longrightarrow M$ is an $\epsilon$-NAT simplicial ruled surface, no essential closed curve on $S$ is sent to a point by $f$, so by Corollary \ref{pi1injkor} $S'$ is a closed surface homeomorphic to $S$, and the projection map is a homotopy equivalence. 

We then have the following version of the  Bounded Diameter Lemma for NAT  surfaces. 

\begin{named}{Bounded Diameter Lemma rel $M_{< \epsilon}$}\label{rbdl}
Fix $\epsilon >0$. If $f : S \longrightarrow M$  is an $\epsilon$-NAT simplicial ruled surface, then  any two points $x,y\in S$  can be joined by a path $\gamma$ such that $$\length_\epsilon(f(\gamma)) \leq D:=\frac{32\cdot |\chi(S)|}{| \kappa^+ | \epsilon}.$$  In particular,
 we have  that $\diam_\epsilon f(S) \leq D.$
\end{named}

Here, $\length_\epsilon$ and $\diam_\epsilon$ measure length and diameter outside the thin parts of $M$, see \S \ref{relative distances}. So, this version of the bounded diameter lemma is extrinsic, rather than intrinsic as before.

\medskip

The following result shows that when no triangles are degenerate, simplicial ruled surfaces are majorized by hyperbolic surfaces.

\begin{lem}[Uniform lipschitz markings]\label {lipschitzmarkings}
Let $S$ be a $2$-dimensional $\Delta$-complex with CAT($\kappa^+$)-triangles, such that $S$ is homeomorphic to a closed surface, and the cone angle at every vertex is at least $2\pi$. Then there is a smooth hyperbolic surface $S_\text {hyp} $ and a $C$-lipschitz homotopy equivalence $S_\text{hyp} \longrightarrow S $, where $C=C(|\chi(S)|,\kappa^+).$
\end{lem}

Here is a simple corollary of this that we will find useful below.

\begin{kor}\label {lipschitzcorollary}
	Let $S \longrightarrow M$ be a $\pi_1$-injective simplicial ruled surface and suppose $S$ is $\epsilon$-thick. Then there is a minimal size generating set for $\pi_1 S$  that can be realized as a wedge of loops on $S$ that has length at most some $L=L(|\chi(S)|,\epsilon).$
\end{kor}
\begin{proof}
	By Fact \ref{pi1injkor}, we have that the quotient $S'=S/\sim$ is a surface, and Lemma \ref{injcompare} then says that $S'$ is also  $\epsilon$-thick. The surface $S_{hyp}$ given by applying Lemma \ref{lipschitzmarkings} to $S'$ is $\epsilon/C$-thick, so (at worst) by Mumford's compactness theorem there is some such $L$ such that the conclusion of the corollary is satisfied for $S_{hyp}.$ We can then push the wedge of loops on $S_{hyp}$ forward to $S'$ via the lipschitz homotopy equivalence. 
	
	After a perturbation, we can assume each loop on $S'$ is a finite concatenation of arcs in triangles of $S'$, and then we can lift the entire wedge to $S$ without increasing its length, using the same argument as in Lemma \ref{injcompare}. The resulting wedge of loops in $S$ still represents a minimal size generating set for $\pi_1$, since by Corollary \ref{pi1injkor} the map $c: S \longrightarrow S'$ is a homotopy equivalence.
\end{proof}

\begin {proof}[Proof of Lemma \ref{lipschitzmarkings}]
Replacing each triangle in $S$ by the constant curvature $\kappa^+$ triangle with the same side lengths does not decrease cone angles, and the resulting surface admits a $1$-lipschitz homotopy equivalence to $S$ by Fact \ref{triangle majorization}. So, it suffices to assume that all triangles in $S$ have constant curvature $\kappa^+$. Scaling the metric by $\sqrt{|\kappa^+|}$ is an obviously bilipschitz operation, so from now on we assume that $S$ is a gluing of hyperbolic triangles. 

Set $\chi:=\chi(S)$. By the Gauss--Bonnet Theorem, the cone angles of $S$ are all at \emph{most} $2\pi |\chi|$. Using this bound, we will modify the singular metric in a small neighborhood of each cone point, producing a new \emph{smooth} metric with curvature at most $\chi^{-2} < 0$ that is $|\chi|$-bilipschitz to the original one. By the Schwartz-Pick-Ahlfors Lemma and the Uniformization Theorem, the new smooth metric is the image of a $|\chi|^{2}$-lipschitz homeomorphism from a hyperbolic surface $S_{hyp}$. The new metric is $|\chi|$-bilipschitz to the original, so the Lemma will follows.

We now describe how to perform the local smoothing referenced above. Pick a vertex $p \in S $ with cone angle $\alpha  > 2\pi $.  A punctured $\delta$-neighborhood of $p$ on $S$ admits `polar coordinates' $(r,\theta)$, where $ 0 < r < \delta$ and $ \ 0\leq \theta < \alpha$, such that the path metric on $S$ has the form
$$ds^ 2 = dr^2 + \sinh(r)^2 d\theta^2. $$
Formally, we parametrize each triangle incident to $p$ in geodesic polar coordinates, where $\theta$ ranges in an interval of length equal to the interior angle at $p$, then concatenate all these parametrizations.

Pick a smooth, monotone decreasing surjection $$D: [0,\delta] \longrightarrow \left [\frac \alpha{2\pi},1 \right] $$ that is constant in neighborhooods of $0$ and $\delta$. The metric $$dt^2= D(r)^2 dr^ 2 + \sinh(r)^2 d\theta^ 2$$ agrees with the original metric near the boundary of our parametrized neighborhood of $p$, and we claim that $S$ admits a smooth structure wherein $dt$ is smooth at $p$. To see this, note that for small $r$ we have $$dt^2 = (\frac \alpha {2\pi})^2 dr^2 + \sinh(r)^2 d\theta^2.$$ If we change coordinates by setting $R=\frac {\alpha}{2\pi} r$ and $\Theta =  \frac {2\pi}{\alpha} \theta$, then 
$$dt^2 =  dR^2  + \left (\frac{\alpha}{2\pi}\right )^2 \sinh^2 \left (\frac {2\pi}{\alpha} r\right ) d\Theta^2.$$ But this is exactly the expression in geodesic polar coordinates for a hyperbolic metric scaled by $\alpha/(2\pi)$, which extends smoothly to $p$. 

As desired, the new metric $dt^2$ is $\frac \alpha { 2\pi } $-bilipschitz to the original metric. We can also calculate its Gaussian curvature using the familiar formula for curvature in an orthogonal parametrization. Since $D$ does not depend on $\theta$, the formula is simpler than usual:
\begin{align*}
	K(r,\theta) &= \frac{-1}{2D(r)\sinh(r)} \left (\frac { (\sinh^2(r))'}{D(r)\sinh(r)}\right )' \\
&= \frac {\cosh(r) D'(r)}{\sinh(r)D(r)^3} - \frac 1{D(r)^2},
\end{align*}
where primes indicate $r$-derivatives. Since $D$ is positive and monotone decreasing, the first term is negative. As $$D \leq \frac \alpha{2\pi} \leq |\chi|,$$ the curvature $K$ of our new metric is bounded above by $\chi^{-2}$.\end{proof}

\subsection{Minimizing simplicial ruled surfaces}\label {homotopies}  Part of the power of simplicial ruled surfaces is that one can often interpolate between two of them with a homotopy of simplicial ruled surfaces. Canary \cite{Canarycovering} proved a first version of this for hyperbolic $M$;  a concrete statement that one can prove using his arguments is the following.

\begin{named}{Canary's  Interpolation Theorem}[see {\cite[\S 5]{Canarycovering}}]
Let $M$ be a  hyperbolic $3$-manifold with no cusps and let $f_0,f_1: S \longrightarrow M$ be homotopic $\pi_1$-injective useful simplicial hyperbolic surfaces. Then  there is a homotopy $(f_t)$ from $f_0$ to $f_1$  through simplicial hyperbolic surfaces.
\end{named}

Recall from Example \ref{usefulsrs} that a simplicial ruled surface $f:S\longrightarrow M$ is {\em useful} if the associated triangulation of $S$ has a single vertex and there is an edge $e$ such that $f(e)$ is a closed geodesic in $M$.  Canary's proof involves a discussion of hyperbolic tetrahedra that does not immediately translate to the negatively curved setting, and the statement is a bit awkward if $M$ has cusps. Agol \cite{Agoltameness} proved a similar, but weaker, interpolation result in pinched negative curvature, again when $M$ has no cusps. Here, we develop a theory of `minimizing simplicial ruled surfaces', in which there is a natural interpolation theorem for manifolds with pinched negative curvature and possibly cusps. Our theorem also does not require $\pi_1$-injectivity, which is essential in previous versions.

\medskip

   Let $M$ be a complete Riemannian $3$-manifold with pinched negative curvature, let $S$  be a closed  orientable surface of genus $g\geq 2$, let $ G$  be a  finite graph embedded in $S$, and choose a  function
$$\omega: E( G) \longrightarrow \BR_+$$
 assigning every edge $e$ of $ G$ a positive \emph{weight} $\omega(e)$.  For the most part, $G$  will be the $1$-skeleton of a triangulation, but we start  with  arbitrary graphs since we need  Fact \ref {uniquenessfact} below to  apply more generally.%Sometimes, we will speak of the pair $ (\mathcal T,\omega)$ as a \emph {weighted triangulation}.

  We say that a piecewise smooth map $f : S \longrightarrow M$ \emph{minimizes for $(G,\omega)$} in a given homotopy class if it minimizes the weighted sum
\begin {equation}L(f,G,\omega)=\sum_{e \in E(G)} \omega(e) \cdot \big (\length  f(e) \big )^2. \label {lengthsome}
\end {equation}

 To understand why we square the length,  suppose  for simplicity that two edges $e,e'$ of equal weight intersect in a vertex of valence two. The image of $e \cup e'$ under any minimizing map $f$ will be a geodesic, and  the edges $f(e)$ and $f(e')$ will partition $f(e\cup e')$ in half by length.  On the other hand, if there were no square, the partition of $f(e\cup e')$ could be arbitrary. This sort of uniqueness is necessary for the following:

\begin{fact}[Uniqueness of minimizers]\label {uniquenessfact}
Suppose $f_i : S \longrightarrow M$, $i=0,1$, are homotopic, both minimize for $(G,\omega)$, and $(f_i)_*(\pi_1 G) \subset \pi_1 M$  is nontrivial and not cyclic of hyperbolic type. Then $f_0=f_1$ on $G$ up to orientation-preserving reparametrizations of each edge.
%
% A relative version of this is also true: if $f_i$ as above both minimize for $(\mathcal T,\omega)$ among maps where the images of  some vertices of $\mathcal T$  are prescribed, and the $f_i$ are homotopic rel those vertices, then $f_0=f_1$ on $\mathcal T^1$ up to orientation-preserving reparametrizations of each edge.
\end{fact}

Note that if $g :S \longrightarrow M$ and $g|_G$  is homotopically trivial, then any minimizer homotopic to $g$ is constant on $G$, but the image of $G$ can be any point in $M$.  Similarly, if $G$  is a cycle and $g(G)$  is homotopic to a closed geodesic $\gamma$ in $M$, then a one parameter family of minimizers homotopic to $g$ can be obtained by translating along $\gamma$.

\begin {figure}
\centering
\includegraphics {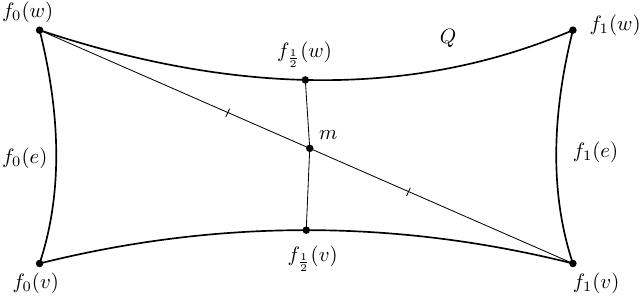}
\caption {Unless the quadrilateral $Q$ is degenerate, $d(f_{\frac 12}(v),f_{\frac 12 }(w)) < \length f_0(e)=\length f_1(e).$   For instance, if $m$ is the midpoint of the appropriate geodesic joining $f_0(w)$ and $f_1(v)$, applying \cite[Lemma 1]{siler2012note} to two appropriate comparison triangles in $\BH^3_{\kappa^+}$ gives that $d(f_{\frac 12}(v),m) < \frac 12 \length f_0(e).$
}	
\label {quad}
\end {figure}

\begin{proof}
% To cut down on notation, we leave the relative version as an exercise for the reader. 
 Our proof is a   variation of what is essentially an argument of Siler~\cite[Theorem 1]{siler2012note}.
Suppose that $f_0,f_1$ are homotopic and both minimize for $(G,\omega)$. Then there is a homotopy $(f_t)$ such that
\begin{enumerate}
	\item  for each vertex $v$ of $G$,  the path $f_t(v), t \in [0,1]$  is a (possibly constant)   constant speed geodesic,
\item for all $t$,  the image $f_t(e)$ of every edge of $G$ is a geodesic.
\end{enumerate} 
Indeed, one can construct such a homotopy from an arbitrary homotopy by first straightening the trajectories of vertices, and then straightening images of edges. By convexity of the distance function, if $e$ is an edge of $G$, the function $t \mapsto (\length f_t(e))^2$  is strictly convex unless it is constant. So, define $L_t  = L(f_t,G,\omega)$. Since $f_0,f_1$ are minimizers,  $$L_{\frac 12} \geq L_0=L_1,$$ so $\length f_t(e)$ is constant. From this, it follows that the  quadrilateral in Figure \ref{quad} is degenerate, so is contained in some geodesic $\gamma$ in $M$.   

There are then two cases. Suppose first that $f_t(v)$ is constant for some vertex $v$. Then if $e$ is an edge connecting $v$ to $w$, the vertical sides of $Q$ in Figure \ref{quad} have the same length and the same bottom endpoint, so degeneracy of $Q$ implies that the vertical sides are equal, so the top side is constant. Hence $f_t(w)$ is constant. Since $G$ is connected, this implies that $f_0=f_1$ on $G$ as desired.

Next assume that no vertex of $G$ is fixed by the homotopy. Picking $e=(v,w)$, degeneracy of $Q$ implies  the paths $t \mapsto f_t(v)$, $t \mapsto f_t(w)$ and $f_t(e)$ all lie along a single geodesic in $M$. By connectivity of $G$, for all $t$ the entire $f_t$-image of the $1$-skeleton of $G$ lies on a single geodesic $\gamma$ in $M$. Hence, $(f_t)_*(\pi_1G)$  is either trivial or is cyclic of hyperbolic type, contrary to our assumption.
\end{proof}

Below, let $\mathcal T$ be a triangulation of the surface $S $. We'll write 
$$L(f,\mathcal T,\omega) := L(f,\mathcal T^1,\omega),$$
 just to cut down a notation, and will similarly say that $f$ \emph{minimizes for} $(\mathcal T,\omega)$  if it minimizes $L(\ \cdot \ ,\mathcal T,\omega)$.

\begin {lem}[Simplicial ruled minimizers]
Suppose that $g : S \longrightarrow M$ is a map such that $g_*(\pi_1 S) \subset \pi_1 M$ is nonelementary, and $(\mathcal T,\omega)$  is a weighted triangulation of $S$. Fixing a ruling vertex for each simplex of $\CT$, there is a unique simplicial ruled surface $$f: (S,\CT) \longrightarrow M$$ homotopic to $g$ that minimizes $L(\cdot ,\mathcal T,\omega)$, and where for each triangle $\Delta$ in $\CT$, the map $f|_{\Delta}$ is ruled using the given ruling vertex. Moreover, the map $\omega \longmapsto f$ is continuous, where the topology on functions $S \longrightarrow M$ is that of uniform convergence.
\label{minimizers}
\end {lem}

Recall also that a triangulation comes equipped with a fixed affine structure on each triangle $\Delta$, and that any ruled triangle $f|_{\Delta}$ is ruled according to this structure.

\begin {proof}
 Take a sequence of maps $f_i : (S,\mathcal T) \longrightarrow M $ homotopic to $g$ such that $L(f_i,\mathcal T,\omega)$  approaches the infimum. We may assume each $f_i$ is constant speed on each edge of the $1$-skeleton $\mathcal T^1$; this implies there is an upper bound independent of $i$ on the Lipschitz constant of $f_i |_{\mathcal T^1}$. Since $\CT^1$ carries the fundamental group of $S$ and $g_*(\pi_1 S)$ is nonelementary, Corollary \ref{distancetogeodesic2} implies that the images $f_i(\CT^1)$  all lie in a compact subset of $M $. So, by Arzela-Ascoli, after passing to a subsequence the maps $f_i$ converge on $\mathcal T ^ 1$ to some  $f : \mathcal T ^ 1 \longrightarrow M$. 
Note that $f$ takes every edge $e$ of $\mathcal T^1$ to a geodesic in $M$. 

If $\Delta$ is a triangle of $\CT$, the image $f(\partial \Delta)$ is nullhomotopic, as it is the limit of the nullhomotopic loops $f_i(\partial \Delta)$. So, by Fact \ref{ruledconstruction} we can extend $f$ to a presimplicial ruled surface 
$$f : S \longrightarrow M$$
that respects the given affine structures and ruling vertices on triangles. Uniqueness of $f$ follows immediately from Fact \ref{uniquenessfact}.

We want to show $f$ is simplicial ruled. So, let $d$ be the induced path metric on $S$, let $\sim$ be the equivalence relation where $x\sim y$ if $d(x,y)=0$, let $S'=S/\sim$, let $c: S \longrightarrow S'$ be the quotient map, and let $f' : S' \longrightarrow M,$ where $f=f'\circ c$. We want to show that the cone angle at every vertex $v\in S'_{reg}$ is at least $2\pi$.

If the cone angle at $v$ is less than $2\pi$, the tangent vectors pointing along the $f'$-images of the edges of $S'$ incident to $v$ form a subset of the sphere $T^1M_{f(v)}$ that is contained in an open hemisphere. Perturbing $f(v)$ in the direction of the center of that hemisphere decreases the lengths of all adjacent edges. At least one edge of $S'$ adjacent to $v$ is the $c$-image of a segment of an edge of $\CT$, as all vertices of $\CR$ lie on $\CT^1$, see Fact \ref{geomquotient}. Hence, such a perturbation of $f$ would decrease $L(f\circ c,\CT,\omega)$, contradicting that $g=f\circ c$ minimizes.

Finally, we prove continuity of $\omega \mapsto f$. It suffices to check continuity on the $1$-skeleton $\mathcal T^1$, since ruling vertices have been specified. But given any convergent sequence $\omega_i\to \omega$, Arzela-Ascoli allows us to pass to a subsequence where the associated maps $f_i$ converge on $\mathcal T^1$ to some $f : S \longrightarrow M$. Then $L(f_i,\mathcal T,\omega_i)\to L(f,\mathcal T,\omega)$. There can be no $h$ with  $L(h ,\mathcal T,\omega)<L(f,\mathcal T,\omega)$, since if there were then we would have $L(h,\mathcal T,\omega_i ) < L(f_i,\mathcal T,\omega_i)$ for large $i$, contradicting that the $f_i$ are minimizers. Hence, $f$  minimizes for $(\mathcal T,\omega)$, and thus is the unique map that minimizes. So, every sequence $\omega_i $  that converges to $\omega$ has a subsequence such that the corresponding maps $f_i \to f$, which shows that the map $\omega \mapsto f$  is continuous.
\end {proof}

 Below, we will also need to understand how  minimizing surfaces degenerate when some of the edge weights tend to zero.

\begin{lem}[Degenerating minimizers]\label{degenerating minimizers}
 Suppose that we have a map $g : S \longrightarrow M$ and that $\mathcal T$  is a  triangulation of $S$.  Let $$\omega : E(\mathcal T) \longrightarrow [0,1]$$ be a system of nonnegative edge weights on $\mathcal T$, and let $G \subset \mathcal T^1$  be the subgraph consisting of those edges $e$ where $\omega(e)>0$.  Fix an affine structure on $\mathcal T$, let $\omega_i$ be  a sequence of positive weights on $\mathcal T$ such that $\omega_i \to \omega$, and let $f_i : S \longrightarrow M$ be SRSs homotopic to $g$ that are minimizing for $(\mathcal T,\omega_i)$. Then either 
\begin{enumerate}
	\item  there is some map $f : S \longrightarrow M$  that is minimizing for  the weighted graph $(G,\omega|_G)$  such that $f_i\to f$ on $G$.
	\item the image $g_*(\pi_1 G) $ is cyclic hyperbolic, and any subsequence of $(f_i)$ contains a further subsequence that converges on $G$ to one of the nonunique minimizers $f : S \longrightarrow M$ for $(G,\omega|_G)$.
\item the image $g_*(\pi_1 G) \subset \pi_1 M$ is cyclic parabolic, and as $i\to \infty,$ we have $\CL(f_i,\mathcal T,\omega_i) \to 0$, and the images $f_i(G)$ exit the corresponding cusp of $M$.
\item  the image $g_*(\pi_1 G)$  is trivial.
\end{enumerate}
\end{lem}

\begin {proof}
 Let's assume that the image $g_*(\pi_1 G)$ is nontrivial. First, suppose $g_*(\pi_1 G)$ is a parabolic subgroup of $\pi_1 M$. Now, $g$ can be homotoped to maps $g_j$ such that $L(g_j,G,\omega)\to 0$ as $j\to \infty$, by taking $g_j(G)$ out the corresponding cusp. For fixed $j$, we have $L(g_j,\CT,\omega_i)\to L(g_j,G,\omega)$. So, given $\epsilon>0$, if we pick some huge $j$ such that $L(g_j,G,\omega)<\epsilon$, then for all sufficiently large $i$ we'll have 
 $$L(f_i,\mathcal T,\omega_i) \leq L(g_j,\CT,\omega_i) \leq L(g_j,G,\omega) + \epsilon \leq  2\epsilon.$$
 Hence, $L(f_i,\mathcal T,\omega_i)\to 0$. It follows immediately that $f_i(G)$ exits a cusp.

So, we can now assume that $g_*(\pi_1 G)$ is not an elementary parabolic subgroup of $\pi_1 M$.  
We claim that the length of $f_i(G)$ is bounded above, independent of $i$. First, if $E$ is the number of edges in $ G$, then
\begin{align}
\length f_i(G)  &\leq L(f_i,G,1) + E \leq \frac{L(f_i,G,\omega_i)}{\min_{e\in G} \omega_i(e)} + E.  \label {someq} \end{align}
For the first inequality, recall that $L$ is defined in  \eqref{someq} as the sum of the \emph {squared} lengths of edges. So if some edge of $f_i(G)$ has length less than $1$, squaring the length gives a smaller value, but we can  compensate for this by adding $1$.  
Since $\omega_i \to \omega$, which is positive on $G$, the denominator  on the right side in \eqref{someq} is bounded away from zero for large $i$. And
\begin{align*}
L(f_i,G, \omega_i) \leq  L(f_i,\mathcal T, \omega_i)  \leq L(g,\mathcal T, \omega_i) \leq L(g,\mathcal T,1),
\end{align*}
at least after perturbing $g$ to have finite length. So, $\length f_i(G)$ is bounded above.

Due to the length bounds discussed in the previous paragraph, the restrictions $f_i |_{G}$ are uniformly lipschitz. Moreover, as $f_i|_G$ are bounded length homotopic graphs in $M$ whose fundamental groups are not elementary parabolic, their images are contained in some compact subset of $M $ by Corollary \ref{distancetogeodesic2}. So, for any  subsequence $f_{i_j}$, Arzela--Ascoli  gives a  further subsequential limit $f$ of $f_{i_j}$ on $G$.  Such an $f$  extends to a map $f : S \longrightarrow M$ (for instance, one can just pick some huge $i_j$ and perturb $f_{i_j}$ to agree with the limit map on $G$)
and this $f$ will minimize the quantity $ L(\ \cdot \ ,G,\omega)$. 

We have shown that every subsequence of $f_i$ has a subsequence that converges on $G$ to some restriction $f|_G$, where $f : S\longrightarrow M$  minimizes $ L(\ \cdot \ ,G,\omega)$.   Any such $f$  must be constant speed on the edges of $G$, since the $f_i$ are constant speed on edges. So, if $g_*(\pi_1 G)$ is not cyclic hyperbolic, Fact \ref{uniquenessfact} says that $f|_G$ is independent of the subsequence chosen.  In other words, picking any such $f$, we have that $f_i|_G$  actually converges to $f|_G$.
\end {proof}

\subsection{Constructions}

Here are some constructions of minimizing SRSs in manifolds with pinched negative curvature.

\begin {kor}[SRSs near closed geodesics and cusps]\label {construction}
Suppose that $g : S \longrightarrow M$ is a map from a closed orientable surface of genus at least two such that $g_*(\pi_1 S) \subset \pi_1 M$ is nonelementary, and $\gamma \subset S $ is a simple closed curve.  Then there is a family of minimizing SRSs $$f_t : S \longrightarrow M$$ homotopic to $g$  such that as $t\to 0$, $f_t(\gamma)$  either converges to a closed geodesic or exits a cusp of $M$, depending on whether $g_*(\gamma)$  is  a hyperbolic or parabolic element of $\pi_1 M$. Moreover, if $g(\gamma)$ has length at most $L$, for small $t$ we have that $d_\epsilon(g(\gamma),f_t(S)) \leq D=D(\kappa^+,\epsilon,L).$
\end {kor}
\begin {proof}
Take a triangulation $\mathcal T$ of $S$ in which $\gamma$ can be realized as a simple cycle in the $1$-skeleton of $T$, and let $f_t : S \longrightarrow M$ be a homotopy of simplicial ruled surfaces
  that minimize the squared-lengths of the weighted triangulations $(\mathcal T,\omega_t)$, where $\omega_t(e)=1$ if $e$  is an edge on $\gamma $, and $\omega_t(e)=t$ otherwise. Then apply Lemma \ref{degenerating minimizers}. The $\epsilon$-distance estimate at the end of the statement follows from Lemma \ref{distancetogeodesic}.
\end {proof}

We can also construct minimizing SRSs near any nonelementary surface in $M$ that satisfies a bounded diameter lemma rel annuli. 

\begin{fact}[Minimizers near bounded surfaces] \label{nearboundary}
Let $S$ be a closed orientable surface of genus $g\geq 2$, let $f_0 : S \longrightarrow M$ be a rectifiable\footnote{This means that any smooth path on $S$ has an image of finite length in $M$. When we apply the fact later, $f_0$ will be the inclusion of a $C^{1,1}$ subsurface of $M$.} map and equip $S$ with the pullback pseudo-metric. Suppose there is a collection $(A_i)$ of pairwise disjoint, pairwise nonhomotopic, essential embedded annuli in $S$ such that 
\begin{itemize}
\item for each $i$, the map $f_0 |_{A_i}$ is homotopically nontrivial, each boundary component of $A_i$ has length at most $\epsilon$, and for each point $p\in A_i$, there is an essential loop $\alpha \subset M$ through $f_0(p)$ that has length at most $\epsilon$ and is homotopic in $M$ to the $f_0$-image of an essential simple closed curve on $A_i$,
\item for each component $Y \subset S \setminus \cup_i int(A_i)$, the map $f_0|_{Y}$ has nonelementary $\pi_1$-image, and $Y$ has intrinsic diameter at most $D$.\end{itemize}
Then $f_0 $ is homotopic to a minimizing simplicial ruled surface, through a homotopy whose image has $\epsilon$-diameter at most $R=R(\kappa^+,D,g)$.  \end{fact}

Recall that $\epsilon$-diameter is diameter with respect to the path pseudo-metric $d_\epsilon$ on $M$ defined by letting the length of a path be the length of its intersection with the thick part. The condition on the annuli $A_i$ is rather awkward: it would be more intuitive to just assume the stronger condition that every point of $A_i$ lies on some essential curve $\alpha \subset S$ of length at most $\epsilon$, but unfortunately our description of the convex core boundary in negative curvature (Theorem~\ref{bdlcc}) doesn't give such curves, so we state the Fact more generally as above.

As an example, the fact applies to any (intrinsically) bounded diameter surface in $M$ with non-elementary $\pi_1$-image. It also applies to $\pi_1$-injective\footnote{Really, what is required is weaker: one needs that the thin parts are $\pi_1$-injective and the thick parts have nonelementary $\pi_1$-image in $M$.} pleated surfaces in hyperbolic $M$, since such surfaces are intrinsically hyperbolic and therefore have a thick-thin decomposition into essential annuli and bounded diameter complementary components. So, it applies to incompressible boundary components of $CC(M)$ when $M$ is hyperbolic. Applying\footnote{Some of the constants produced in Theorem~\ref{bdlcc} are $3\epsilon$ and $3\epsilon+2t$, rather than $\epsilon$, as needed for the Fact. So, when plugging Theorem~\ref{bdlcc} into the Fact, we need to use a smaller $\epsilon,t$ in the former. Also, in the Corollary the reader might expect us to use a $t$-equidistant surface with $t\approx 0$ instead of a $1$-equidistant surface, but as the $1$-equidistant surface is homotopic to the $t$-equidistant surface via a homotopy of tracks less than $1$, there's no difference in the corollary.} Theorem~\ref{bdlcc}, we can generalize this statement to the setting of pinched negative curvature.

\begin{kor}[Minimizing SRSs near $\partial CC$]\label{kor SRSs near partial}
	Take $\epsilon>0$ small. Suppose that $M$ has finitely generated fundamental group and only rank $2$ cusps, all of which are hyperbolic. If $S$ is a genus $g$ incompressible boundary component of the $1$-neighborhood of $CC(M)$, then $S$ is homotopic to a minimizing simplicial ruled surface through a homotopy whose image has $\epsilon$-diameter at most $R=R(\kappa^+,g)$.
\end{kor}

We can also apply Fact \ref{nearboundary} to SRSs that are sufficiently $\pi_1$-injective.

\begin{kor}[Minimizing SRSs near arbitrary SRSs]\label{nearby minimizers}
There are $r,R>0$ depending on $\epsilon,g,\kappa^+$ as follows. Let $S$ be a closed orientable surface of genus $g\geq 2$ and suppose that $f : S \longrightarrow M$ is a simplicial ruled surface that is $\pi_1$-injective within the $d_\epsilon$-neighborhood of radius $r$ around $f(S)$. Then $f$ is homotopic to a minimizing SRS through a homotopy whose image has $\epsilon$-diameter at most $R$.
\end{kor}
\begin{proof}
Equip $S$ with the pullback path pseudometric, and let $S_{<\epsilon} \subset S$ be the $\epsilon $-thin part. As long as $r>\epsilon$, any essential loop of length at most $\epsilon$ in $S$ has essential image in $M$, since otherwise we could nullhomotope the image in its $\epsilon$-neighborhood in $M$, using a coning construction in the universal cover. In other words, $S$ is $\epsilon$-NAT, so by Corollary \ref{pi1injkor} the quotient space $S'$ associated to the pseudometric is homeomorphic to $S$, and the quotient map is a homotopy equivalence. In particular, after replacing $S$ with $S'$, which corresponds to doing a homotopy \emph{within} the image of the SRS, we can assume that $S$ is a locally CAT($\kappa^+$) surface, by Corollary \ref{CATfact}.

If $\alpha,\beta$ are two essential loops of length at most $\epsilon$ that pass through a point $p\in S$, then $\alpha,\beta$ map into a component $T \subset M_{<\epsilon}$, which is contained in the radius $0$ neighborhood of $f(S)$ with respect to the $d_\epsilon$-metric. So by our $\pi_1$-injectivity assumption, the subgroup $\langle \alpha,\beta\rangle \subset \pi_1 (S,p)$ injects into $\pi_1 T$, which is virtually abelian, implying that the subgroup $\langle \alpha,\beta\rangle$ is cyclic. In other words, $S$ satisfies a Margulis Lemma, and then the usual proof\footnote{Briefly, if $S = \Gamma \backslash X$, where $X$ is CAT($\kappa^+$), then for each maximal cyclic subgroup $C \subset \Gamma$ we consider the set $T_C \subset X$ of all points that are translated by less than $\epsilon$ by an element of $C$. If $T_C$ is nonempty, then it is a star-shaped neighborhood of the common axis of the elements of $C$, so $C \backslash T_C$ is an annulus. If $C,D$ are distinct maximal cyclic subgroups, then $T_C\cap T_D=\emptyset$ by our previous discussion with the loops $\alpha,\beta$. So, the components of the $\epsilon$-thin part of $S$ are the projections of the sets $T_C$ that are nonempty. Note that if $g\in \Gamma \setminus C$, then $g(T_C)=T_{gCg^{-1}}$, and $C\neq gCg^{-1}$ since maximal cyclic subgroups of higher genus surface groups have trivial normalizers, so $g(T_C) \cap T_C = \emptyset$ from above, and hence the projection of $T_C$ to $S$ is $C \backslash T_C$, which is an annulus as noted above.} of the thick-thin decomposition for hyperbolic surfaces applies to our locally CAT($\kappa^+$) surface $S$, showing that the components of $S_{<\epsilon}$ are pairwise nonhomotopic essential annuli. 

The Bounded Diameter Lemma says that $S$ has bounded diameter relative to $S_{<\epsilon}$. To use a component $A \subset S_{<\epsilon}$ as one of the annuli in Fact \ref{nearboundary}, we'd have to give a length bound for $\partial A$. Instead it's a bit easier to perturb $A$ so that its boundary obviously has bounded length, as follows.  Write $\partial A = c_1\cup c_2$, and let $\gamma_1,\gamma_2$ be essential simple closed curves in $\partial A$ with length less than $3\epsilon$ (say) based at points of $c_1,c_2$, respectively. For example, take $\gamma_1,\gamma_2$ to be shortest essential curves through $c_1,c_2$. We may assume that $\gamma_1,\gamma_2$ do not intersect the geodesic core of $A$ transversely, since projecting any part that lies on the `wrong side' of the core back onto the core decreases length. So after perturbing them off the core, we can assume $\gamma_1,\gamma_2$ lie on opposite sides of the core, and hence are disjoint. If we let $A' \subset \overline A$ be the subannulus bounded by $\gamma_1,\gamma_2$, then every point of $\overline A$ is within $2\epsilon$ of $A'$: indeed, given $p\in \overline A \setminus A',$ there is a loop of length at most $2\epsilon$ based at $p$, which is then contained in $\overline A$; this loop must intersect $A'$ since otherwise it separates $A'$ from some component of $\partial A$, which $A'$ intersects by construction. So, if we apply the process $A \mapsto A'$ to every component of $S_{<\epsilon}$, we get a collection of $\epsilon$-thin annuli $A_i$ with boundary lengths at most $3\epsilon$, and where $S$ has bounded diameter relative to $\cup_i A_i$. 

It follows from the bounded relative diameter and the bound on the boundary lengths that each component $Y \subset S \setminus \cup_i A_i$ has intrinsic diameter at most some $D=D(\epsilon,g,\kappa^+)$. Moreover, we claim that $f_*(\pi_1 Y)$ is non-elementary. Supposing it is elementary, there is a single component $T \subset M_{<\epsilon}$ such that $f(\partial Y) \subset T$, and then we can homotope $f|_Y$  into $T$ by flowing each $f(x)$ along the shortest path to either the geodesic core of $T$, or out the cusp of $T$. As $Y$ has bounded diameter, $f(Y)$ lies in a bounded neighborhood of $T$, and our homotopy decreases distance to $T$, so the entire homotopy lies in a bounded neighborhood of $T$. But as $Y$ is not an annulus or a disc, the resulting map $Y \longrightarrow T$ is not $\pi_1$-injective. So, if $r$ is larger than all the bounded constants in this paragraph, this contradicts our assumption that $f$ is $\pi_1$-injective inside the $r$-neighborhood of $f(S)$, with respect to the $d_\epsilon$ metric. 

The corollary now follows from Fact \ref{nearboundary}.
\end{proof}

We now prove Fact \ref{nearboundary}. %Throughout the proof, \emph{`bounded' means bounded above by some constant depending on $\kappa^+,D,g,\epsilon$.}

\begin{proof}[Proof of Fact \ref{nearboundary}]
The goal is to eventually produce a minimizing simplicial ruled surface homotopic to $f_0$, so we first focus on producing an appropriate triangulation. First, for each $i $, add a vertex on each boundary component of $A_i$, consider each boundary component as an edge with both endpoints at that vertex, and triangulate $A_i$ by adding two more edges to the interior of $A_i$ so that their endpoints are on the given vertices.  We call the edges intersecting the interior of the $A_i$ \emph{long} edges, since we have no control over their length. Edges constructed outside the interiors of all the annuli will be called \emph{short} edges everywhere below. We claim that we can extend the current collection of vertices and edges to a triangulation of $S$ so that all additional edges have length at most some $C=C(g,D)$ and all have interior in $S\setminus \cup_i A_i$.

The construction is recursive. Suppose that we currently have a graph $G$ embedded in $S$ that includes the vertices and edges discussed above, and let $G_{short}$ be the subgraph of $G$ consisting of all short edges. If there is some component of $S\setminus G$ that is not a disk, we will find an embedded graph $G'\supset G$ with one additional edge in $S\setminus \cup_i A_i$, where $G'_{short}$ has length at most $2(D+\length(G_{short}))$, and where the complexity of the set of components of $S\setminus G'$ is less than it was for $G$. Assuming we can always do this, the process terminates with some $G$ where all complementary components are polygons. The number of steps in the process is bounded in terms of $g$, so all short edges have length at most some constant depending on $D,g$. We can turn $G$ into the $1$-skeleton of a triangulation $\CT$ by adding diagonals that have lengths bounded in terms of $D,g$, by perturbing concatenations of sides into the interior.

We now explain how to construct $G'$. First, it is possible that $G$ is empty, as long as the collection of annuli is empty. In this case, $S$ itself has diameter at most $D$, and if $\gamma$ is a shortest essential closed curve in $S$, then $\length \gamma \leq 2D$. (If not, we could find two points $x,y $ on $\gamma$ that cut $\gamma$ into two arcs both of length  at least $D$; if $\alpha$ is a length $D$ arc in $S$ joining $x,y$ then one of the two $\alpha$-surgeries on $\gamma$ is essential and shorter than $\gamma$, a contradiction.) We then place a vertex on $\gamma $ and regard $\gamma$ as an edge, creating $G'$ as desired. Next, assume $G$ is nonempty and that some component $Y$ of $S\setminus G$ is not a disk. If $\gamma$ be a shortest arc from $\partial Y$ to $\partial Y$ that is essential, i.e.\ not homotopic rel endpoints into $\partial Y$, then $\length \gamma \leq 2D$, by another surgery argument. Moving the endpoints of $\gamma$ along edges of $G$ to its vertices, we have an essential arc $\gamma'$ in $Y$ with length at most $2D+\length(G_{short})$ (at the very worst), and we create $G'$ by adding $\gamma'$ to $G$.

\medskip

Let $\CT$ be the triangulation constructed above. Note that all short edges of $\CT $ have length at most $C=C(g,D)$, by construction.  Let $(\CT,\omega)$ denote the weighted triangulation where all short edges have weight $1$, and all long edges $e$ have weight $1/\length(e)^2$.  Let $$f_1 : (S,\CT) \longrightarrow  M$$ be a SRS that minimizes for $(\CT,\omega)$, as is Lemma \ref{minimizers}.
Then $$L(f_1,\CT,\omega) \leq L(f_0,\CT,\omega),$$
and the latter is bounded in terms of $g,D$, since there are only boundedly many edges, the short edges contribute at most $C^2$ to the sum defining $L(f_0,\CT,\omega)$, and the long edges contribute at most $1$. 
In particular, for each short edge of $\CT$, \emph{we have that $\length (f_1(e))$ is bounded in terms of $g,D$.} Using this, we want to prove that $f_0$ and $f_1$ are homotopic via a homotopy whose image has $\epsilon$-diameter at most $R=R(\kappa^+,g,D,\epsilon)$. % $\length(f_1(\CT^1)) \leq N^2.$ This is the last time that we will use that $f_1$ is a minimizing SRS. Below, all that matters is $f_0,f_1$ are homotopic, have nonelementary $\pi_1$-image, map $\CT^1$ to graphs of bounded length, and have images of bounded diameter\footnote{For $f_1$, this follows since its $1$-skeleton has bounded length and its triangles are ruled.}. 

%Let $\CT_{short}$ be the set of all simplices of $\CT$ that lie outside of the interiors of the annuli $A_i$. So, $\CT_{short}$ is a triangulation of $S \setminus \cup_i int(A_i)$.

 In the following, it will be convenient to be able to assume that $f_0$ is presimplicial ruled. To this end, we construct a homotopy from $f_0$ to a presimplicial ruled surface $g$, where the image of the homotopy has bounded $d_\epsilon$-diameter.  First, define $g$ on $\CT^1$ by homotoping all edges of $\CT$ to geodesics, rel endpoints. By Fact~\ref{ruledconstruction}, we can extend $g$ to the entire surface $S$ so that the restriction to each triangle in $\CT$ is a ruled triangle in $M$. Suppose that $\Delta $ is a triangle in $ \CT$ that is not contained in $\cup_i A_i$. Then $f_0(\partial \Delta)$ and $g(\partial \Delta)$ both have bounded length, and in fact both $f_0(\Delta)$ and $g(\Delta)$ have bounded intrinsic diameter, the former because of the diameter bound on components of $S\setminus \cup_i A_i$, and the latter because it is ruled with bounded diameter boundary. Since these two triangles have the same vertices, it follows that the image of the straight-line homotopy\footnote{Really, this is defined by lifting both triangles to triangles in the universal cover $X\longrightarrow M$ that share the same vertices, and then homotoping through geodesics.} from $f_0( \Delta)$ to $g( \Delta)$  has bounded diameter in $M$. If $\Delta$ is a triangle of $\CT$ that lies in some annulus $A_i$, it may not have bounded diameter, so this argument doesn't work. However, by the assumptions in the statement of our Fact, through every point in the image $f_0(\Delta)$ there is a loop in $M$ homotopic to an essential simple closed curve in $A_i$ that has length less than $\epsilon$. If we lift this picture to the universal cover $X\longrightarrow M$, we get a lift $D$ of $f_0(\Delta)$ and a deck transformation $g$ that translates every point of $D$ by less than $\epsilon$. The set $X_{g,\epsilon} \subset X$ of points translated less than $\epsilon$ by $g$ is convex, by negative curvature, so a straight-line homotopy rel vertices from $D$ to a ruled triangle stays inside $X_{g,\epsilon} $. So downstairs, the image of the homotopy from $f_0( \Delta)$ to $g( \Delta)$ lies inside a component of $M_{<\epsilon}$, and hence has $\epsilon$-diameter zero. Hence, $f_0$ is homotopic to $g$ through a homotopy whose image has bounded $\epsilon$-diameter. 

This new $g$ satisfies all the same hypotheses as $f_0$ in the statement of the Fact, so replacing $f_0$ with $g$, we can now just assume that $f_0$ is presimplicial ruled. Fix a homotopy $(f_t)$ from $f_0$ to $f_1$, where $f_1$ is the minimizing SRS constructed above. For each vertex $v$ of $\CT$, homotope the homotopy so that $t \mapsto f_t(v)$ is a geodesic. Then for each edge $e$, homotope the homotopy rel $\CT^0$ so that $f_t(e)$ is a geodesic for each $t$. Finally, note that when we homotoped $f_0$ to be presimplicial ruled, the ruling vertices could have been specified arbitrarily, so we can assume that they match those of $f_1$. We can then homotope the homotopy rel $\CT^1$ so that $f_t$ is presimplicial ruled for each $t$. 

As in the proof of Fact \ref{uniquenessfact}, see Figure \ref{quad}, for each edge $e$ of $\CT$, the map $t \mapsto \length f_t(e)$ is convex. So if $f_0(e), f_1(e)$ have bounded length, so does each $f_t(e)$. In particular, we get that for every short edge $e$ of $\CT$, we have $\length(f_t(e)) \leq N$ for some $N=N(g,D)$. It follows that for every component $K \subset S \setminus \cup_i A_i$, the diameter of $f_t(K)$ is bounded in terms of $g,D$.

\begin{claim}\label{gammaclaim}
For every vertex $v$ of $\CT$, the path $t \mapsto f_t(v)$ has  length bounded in terms of $g,D,\kappa^+$.
\end{claim}
\begin{proof}
Suppose $v$ is contained in a component $K \subset S\setminus \cup_i int(A_i)$. Pick essential loops $\alpha,\beta \subset K$ based at $v$ that are closed edge paths in $\CT^1$, that repeat edges at most twice, and where $(f_0)_*(\alpha),(f_0)_*(\beta)$ generate a nonelementary subgroup of $\pi_1 M$. This is possible since loops of the type above generate $\pi_1 (K,v)$, which has nonelementary image in $\pi_1 M$. Then as explained above, for all $t$ the loops $f_t(\alpha),f_t(\beta)$ have length bounded in terms of $g,D$.

Lift $f_t |_{\{v\}} $ to a homotopy $\tilde f_t : \{v\} \longrightarrow X$, regarded as a geodesic path, in the universal cover $X\longrightarrow M$. Then the elements $\alpha,\beta \in \pi_1(S,v)$ give deck transformations that for all $t$, translate $\tilde f_t(v)$ by a distance bounded in terms of $g,D$. These deck transformations generate a non-elementary subgroup of isometries of $X$, by construction. It therefore suffices to show that \emph{if $h_1,h_2$ are isometries of $X$ that generate a discrete, nonelementary subgroup, the set  $$S_N(h_1,h_2) := \{ x\in X \ | \ d(h_1(x),x),d(h_2(x),x)\leq N\}$$ has diameter bounded in terms of $N,\kappa^\pm$.} For as $t \mapsto \tilde f_t(v)$ is a geodesic in $X$, a diameter bound implies a length bound. 

The italicized claim above is probably well-known to experts, but we were unable to find a reference, so we will quickly sketch the proof.

\begin{claim}\label{sketchclaim}Given $\delta,R,N>0$, there is some $R'=R'(\delta,R,N,\kappa^+)$ as follows. Suppose $h: X \longrightarrow X$ is an isometry, $x,y\in X$ are points with $d(x,y)\geq R'$, and $d(h(x),x),d(h(y),y)\leq N.$ Let $\gamma : I \longrightarrow X$ be a unit speed parameterization of the geodesic from $x $ to $y$. Then there is a subinterval $J\subset I$ with length at least $R$, and $a \in [-N,N]$ such that for all $t\in J$ we have $d(h(\gamma(t)),\gamma(t+a)) < \delta.$\end{claim}

In other words, if $h$ translates points $x,y$ that are extremely far apart a bounded distance, it acts almost as a (bounded) translation along a long segment of the geodesic from $x$ to $y$. 

\begin{proof}[Proof Sketch of Claim \ref{sketchclaim}] By a hyperbolic trigonometry computation, in the model space of constant curvature $\kappa^+$, two unit speed rays $a(t),b(t)$ with $a(0)=b(0)$ satisfy $0<\delta/2\leq d(a(t),b(t)) \leq N$ only for $t$ in an interval whose length is bounded above in terms of $\delta,N,\kappa^+$. So, applying Toponogov's comparison theorem to the geodesic triangle with vertices $x,y,h(y)$, we get that the geodesics $[x,y]$ and $[x,h(y)]$ are within $\delta/2$ of each other except for a bounded length of each near their terminal endpoints. Applying the same argument to the triangle with vertices $x,h(x),h(y)$, and using the triangle inequality, we get that $[x,y]$ and $[h(x),h(y)]$ are within $\delta$ of each other except in bounded length intervals around their endpoints. Taking the $h$-preimage of a long segment of $[h(x),h(y)]$ that tracks $[x,y]$ gives our interval $J$. We leave the details as an exercise.\end{proof}

So, let $h_1,h_2$ be a pair of isometries of $X$ that generate a discrete, non-elementary subgroup. Given $N$, if $x,y \in S_N(h_1,h_2)$ and $d(x,y)$ is huge, Claim \ref{sketchclaim} implies that both isometries act as-close-as-you-want to a translation along some as-long-as-you-want segment $\ell $ of the geodesic $[x,y]$. It follows that the commutator $[h_1,h_2]$ translates every point on some slightly shorter segment $\ell'$ by a small amount, and similarly, so does $[h_1,[h_1,h_2]]$ on some even shorter $\ell''$. But since $h_1,h_2$ generate a discrete, non-elementary group, so do $h_1,[h_1,h_2]$, for instance by a fixed point analysis. Applying the same logic again, $[h_1,[h_1,h_2]]$ and $[h_1,h_2]$ generate a discrete, non-elementary group. But these two isometries both translate points of $\ell''$ a small amount. So, if we rig the constants correctly, this contradicts the Margulis lemma.
\end{proof}

We now return to the proof of the Fact. Our goal is to show that the image of the homotopy $f_t : S \longrightarrow M$ has $\epsilon$-diameter bounded in terms of $g,D,\kappa^+$. As mentioned above, for every component $K \subset S \setminus \cup_i A_i$ and every $t$, the diameter of $f_t(K)$ is bounded in terms of $g,D$. We claim that \emph{for every $i,t$, the image $f_t(A_i)$ has bounded $\epsilon$-diameter.} If we prove this, then we will have that each $f_t(S)$ has bounded $\epsilon$-diameter, and then the Fact will follow from Claim \ref{gammaclaim}.

So, fix $i$. Let $\hat M \longrightarrow M$ be the cover corresponding to $(f_t)_*(\pi_1 A_i)$, so that $\pi_1 \hat M=\langle\alpha\rangle$ is cyclic.  Let $T \subset \hat M$ to be the subset of points through which there is a closed curve of length less than $2\epsilon$ that is homotopic to the generator $\alpha$ of $\pi_1 \hat M$. Since $\hat M$ has negative curvature, $T$ is convex in $\hat M$, meaning that any geodesic segment with endpoints in $T$ lies in $T$. Note that $T$ projects into the $\epsilon$-thin part of $M$. We can then lift the homotopy $f_t|_{A_i}$ to a homotopy $$\tilde f_t : A_i \longrightarrow \hat M.$$
By the first hypothesis of the Fact, we have $\tilde f_0(A_i) \subset T$. By Claim~\ref{gammaclaim}, if $v$ is one of the two vertices on $\partial A_i$, the path $t \mapsto f_t(v)$ has bounded length in  $M$. Lifting, the same is true of the path $t \mapsto \tilde f_t (v)$. So for all $t$ and for both vertices $v$ of $A_i$, we have that $f_t(v)$ lies in an $r$-neighborhood $\CN_r(T)$, for some $r=r(g,D,\kappa^+)$. Since $\tilde f_t$ is presimplicial ruled and $\CN_r(T)$ is convex, we have $\tilde f_t(A_i) \subset \CN_r(T)$  as well. So projecting, $f_t(A_i)$ lies in the $r$-neighborhood of a component of the $\epsilon$-thin part $M_{<\epsilon}$, and hence has $\epsilon$-diameter bounded in terms of $g,D,\kappa^+$.
\end{proof}

\label{sec:interpolation}

We now prove the central result of this section.

\begin{named}{The Interpolation Theorem}
Let $M$ be a complete 3-manifold with pinched negative curvature, $S$ be a closed surface of genus $g\ge 2$ and $$g : S \longrightarrow M$$ is a map such that $g_*(\pi_1 S) \subset \pi_1 M$ is nonelementary. Suppose that $$f_i: (S,\mathcal T_i) \longrightarrow M,  \ i=0,1$$ are two SRSs homotopic to $g$, where each $f_i$ minimizes for some weighted triangulation $(\mathcal T_i,\omega_i)$, and where the $1$-skeleta of $\CT_0,\CT_1$ are  transverse. 

If $\mathcal T$ is a common refinement of $\CT_0,\CT_1$, then there is a homotopy
$$F:S\times[0,1]\longrightarrow M,\ \ F(x,t)=F_t(x)$$
 with the following two properties:
\begin {enumerate}
\item for $t\in [0,1]$,	each $F_t : (S , \mathcal T) \longrightarrow M$ is a minimizing SRS  with respect to some weight $o_t$ on $\mathcal T$,
\item  for $i=0,1$, there is  a homotopy from $f_i$ to $F_i$ that is contained in some $R$-neighborhood $\CN_R(f_i(S))$, where $R$ depends only on the  pinching constants of $M$.\end {enumerate}
% In particular, (2) implies that for $i=0,1$, there is  a homotopy from $f_i$ to $F_i$ whose tracks have length bounded above by some constant depending only on the  pinching constants of $M$.
  \end{named}

 With a little more work, one can even produce a homotopy from $f_i$ to $F_i$ all of whose tracks have bounded length. However, we do not have an application for that here so we will just prove (2) above.

\begin {figure}
\centering
\includegraphics{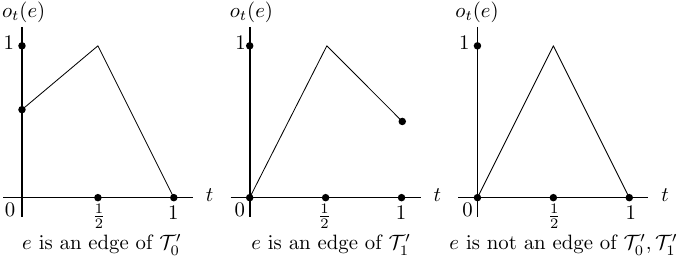}	
\caption {The weights $o_t$ defining the interpolation. In the first case, say, the weights at time zero are determined by distributing the $\omega_0$-weights of edges of $\mathcal T_0$ uniformly over the corresponding edges of $\mathcal T_0'$.}\label {graphs}
\end {figure}

\begin {proof}
 For each $i=0,1$, let $\mathcal T_i'\subset \mathcal T^1$ be the edges that come from edges of $\mathcal T_i$. The $\omega_i$  then induce weight functions $o_i$ on $\mathcal T$ that are supported on $\mathcal T_i'$, where if an edge $e$ of  $\mathcal T_i$  is subdivided into $n$  edges $e_1,\ldots,e_n$ of $\mathcal T_i'$, then we have $o_i(e_j) =\omega_i(e)/n$  for all $j$.  After precomposing $f_i$ with an isotopy, we can assume that the segments $f_i(e_j)$ all have the same length in $M$,  in which case $f_i$ minimizes for the   cell decomposition of $S$ with edge set $\mathcal T_i'$ and weight $o_i$.

For $t\in[0,1]$, let $o_t$ be the  system of weights on $\mathcal T$ described in Figure~\ref{graphs},  remembering that the weights $o_i$ are already defined when $i=0,1$. Fix an affine structure on $\mathcal T$ so that the maps $f_i$ have constant speed when restricted to edges of $\mathcal T_i'$---this is used in the next paragraph.
The weights $o_t$ are positive when $t\in (0,1)$, so after arbitrarily specifying ruling vertices, the  minimizing SRSs $F_t, \ t \in (0,1)$ given by Lemmas \ref{minimizers} form a homotopy.

By Lemma \ref{degenerating minimizers}, as $t\to 0$ the restrictions $F_t |_{\mathcal T_0'}$ converge to $f_0 |_{\mathcal T_0'}$. So, pick some  very small $t_0 > 0$. There is a homotopy with very short tracks from $F_{t_0}$ to a map $F_{t_0}'$ that agrees with $f_0$ on $\mathcal T_0'$.  In fact, we can choose $F_{t_0}'$ to agree with $F_{t_0}$ outside of a small neighorhood of $\mathcal T_0'$, and so that the image of $F_{t_0}'$ is the union of the image of $F_{t_0}$ with the image of a straight line homotopy from $F_{t_0}|_{\mathcal T_0'}$ to $f_0|_{\mathcal T_0'}$. Note that this straight line homotopy has small area, since it has short tracks and is between two bounded length graphs with geodesic edges.

 If $\Delta$ is a  complementary region of $\mathcal T_0'$,  the irreducibility of $M $ implies that the maps $f_0 |_{\Delta}$  and  $F_{t_0}' |_{\Delta}$  are homotopic rel $\partial.$   By Gauss--Bonnet, the areas of the simplicial ruled discs $f_0(\Delta)$  and  $F_{t_0} (\Delta)$ are bounded above in terms of the pinching constant $\kappa^+$, and a slightly increased estimate also holds for $F'_{t_0} (\Delta)$ by the previous paragraph. An argument using the isoperimetric inequality \cite[Lemma 12.1]{Agoltameness} then implies that the two maps $f_0 |_{\Delta}$  and  $F_{t_0}' |_{\Delta}$ are homotopic rel $\partial$ in $M$ within an $R$-neighborhood of the union of their images, for some $R=R(\kappa^\pm)$. 
 
Both $f_0 |_{\Delta}$ and $F_{t_0} (\Delta)$ are contained in a bounded neighborhood of their boundaries; indeed, if some interior point lay far from the boundary, then a large radius disc (which has large area by Fact \ref{convexballs}) would fit in the interior, violating the area upper bound from the previous paragraph. It follows that $F_{t_0}' (\Delta)$ is also contained in a bounded neighborhood of its boundary. So, in light of the previous paragraph, $f_0 |_{\Delta}$  and $F_{t_0}' |_{\Delta}$ are homotopic rel $\partial $  in a bounded neighborhood of $f_0(S)$. Applying this to every  complementary region and then increasing $R$ slightly to replace $F_{t_0}'$ with $F_{t_0}$, it follows that $f_0$ is homotopic to $F_{t_0}$ within an $R$-neighborhood of $f_0(S)\subset M,$ for some $R=R(\kappa^\pm)$.

The same argument  gives some $t_1$ very close to $1$  such that $f_1$  is homotopic  to $F_{t_1}$ within some bounded neighborhood $\mathcal N_R(f_1(S))$.  To finish, just reparametrize the homotopy $F_t$, where $t \in [t_0,t_1]$,  so that the parameter $t$ lies in $[0,1]$.\end{proof}

\subsection{Filling theorems}\label{filling thm sec}

As an application of the Interpolation Theorem, we can prove PNC versions of the Canary/Thurston filling theorems for degenerate ends and convex cores, see  \cite{Canarycovering}. In the following, let $M$ be a complete Riemannian $3$-manifold with pinched negative sectional curvatures $\kappa^- \leq \kappa \leq \kappa^+ < 0$ and hyperbolic cusps.

\begin{sat}[Filling degenerate ends] \label{filling theorem} If $\CE$ is a degenerate end of $M $, where $\CE$ has no cusps. Given a product neighborhood $U \cong S \times (0,\infty)$ of $\CE$, there is a neighborhood $V\subset U$ such that every point $p\in V$ lies in the image of a minimizing SRS $f : S \longrightarrow U$ in the homotopy class of a level surface.
\end{sat}

This theorem generalizes Canary's Filling Theorem \cite{Canarycovering}, proved for hyperbolic $3$-manifolds, to the PNC setting. We note that Agol \cite[Theorem 14.1]{Agoltameness} previously proved a coarse PNC filling theorem for incompressible degenerate ends, showing that every such end has a neighborhood in which every point is at bounded distance from a SRS. 

\begin{proof}
Let $U \cong S \times [0,\infty)$ to be a closed product neighborhood of $\CE$, and set $S_0=\partial U$. By definition, there are simplicial ruled surfaces $$f_i : S \longrightarrow U, \ \ i=1,2,\ldots,$$ in the homotopy class of a level surface whose images exit $\CE$. Fact \ref{fact: dist proper} says that the $\epsilon$-distance to $S_0$ goes to infinity as one exits $\CE$. So by Corollary \ref{nearby minimizers}, we can homotope the $f_i$ to be \emph{minimizing} SRSs without leaving the neighborhood $U$.

By Freedman--Haas--Scott (see Lemma \ref{getembedded}), each $f_i$ is homotopic in $U$ to a level surface $S_i$ that is contained in the $1$-neighborhood of $f_i(S)$. Passing to a subsequence, we can assume that when $i<j$, both $f_j(S),S_j$ are contained in the unbounded component of $U \setminus S_i$, which we call $[S_i,\infty)$. Note that the homotopy from $f_j(S)$ to $S_j$ in $U$ can then be homotoped to lie in $[S_i,\infty)$. When $i<j$, the surfaces $S_i,S_j$ bound a subset $[S_i,S_j]\subset U$ homeomorphic to $S \times [0,1]$.

After discarding finitely many terms of the sequence, we can assume
\begin{enumerate}
\item the $\epsilon$-distance between $S_0,S_2$ is large relative to $g,\epsilon$,
\item $f_1(S)$ is contained in $[S_0,S_2]$, at large $\epsilon$-distance from $S_0,S_2$.
\end{enumerate}
We then set $V = [S_2,\infty)$. Given $p\in V$, pick $i$ large enough so that $p\in [S_2,S_i]$, and then pick $j>>i$ so that $f_j(S)$ has large $\epsilon$-distance from $S_i$. Apply the Interpolation Theorem to the two minimizing SRSs $f_j$ and $f_1$, giving a homotopy as follows. We first homotope $f_j$ in a bounded neighborhood of its image, so in particular within $[S_i,\infty)$. Then we homotope through minimizing SRSs until we enter a bounded neighborhood of $f_1$. This latter homotopy can a priori leave $U$, but since minimizing SRSs in $U$ have bounded $\epsilon$-diameter when they are at least $\epsilon$ away from $S_0$, we can cut off the homotopy the first time we obtain a minimizing SRS contained in $[S_0,S_2]$. (Note that such a time exists, since $f_1(S)$ lies at large $\epsilon$-distance from $\partial [S_0,S_2]$.) The result is a homotopy in $U$ through minimizing SRSs in the homotopy class of a level surface, that starts in $[S_i,\infty)$ and ends in $[S_0,S_2]$. This homotopy must pass through $p$. If not, we can homotope our homotopy to create a map $$S \times [0,1] \longrightarrow [S_2,S_i]$$ that restricts to a homeomorphism on the boundary, but has zero degree. So, there is a minimizing SRS in the homotopy class of a level surface of $U$ whose image contains our arbitrary point $p\in V$.
\end{proof}

Furthermore, we can construct minimizing simplicial ruled surfaces coarsely near every point in the convex core of a manifold $M \cong \Sigma \times \BR$ that has pinched negative curvatures and no cusps. 

\begin{sat}[Filling convex cores]\label{CC filling thm}
	Suppose that $S$ is a closed, orientable surface of genus $g$. Let $M \cong S \times \BR$ have pinched negative sectional curvatures $\kappa^-<\kappa<\kappa^+<0$ and no cusps, and let $\epsilon$ be less than half the Margulis constant, say. If $p\in CC(M),$ then there is a minimizing SRS $f : S \longrightarrow M$ in the homotopy class of a level surface such that $d_\epsilon(p,f(S)) \leq D=D(\kappa^+,g,\epsilon).$
\end{sat}

In particular, fixing $\epsilon$ to be half the Margulis constant, this implies that the injectivity radius of $M$ inside $CC(M)$ is bounded above by some constant depending on $\kappa^+,g$. In the hyperbolic case, this is a result of Thurston, see Canary \cite[Theorem 6.2]{Canarycovering}. Indeed, Theorem~\ref{CC filling thm} is a coarse PNC version of Canary's argument, in which he fills up the convex core of a hyperbolic $M$ with simplicial hyperbolic surfaces.

Below, `bounded' means in terms of $\kappa^+,g,\epsilon$.

\begin{proof}
Fix $p\in CC(M)$. Label the two ends of $M$ as the \emph{positive end} $\CE^+$ and the \emph{negative end}  $\CE^-$, in a way compatible with a fixed parametrization $M\cong S \times \BR$.  By the geometric tameness theorem, each  end $\CE$ of $M$ is either convex cocompact or degenerate. For concreteness, let's assume $\CE^-$ is convex cocompact and $\CE^+$ is degenerate. The other cases use the same arguments.

As in the first paragraph of the proof of Theorem \ref{filling theorem}, there is a sequence of minimizing SRSs $f_i : S \longrightarrow M$ in the homotopy class of a level surface whose images exits $\CE^+$. Applying Freedman--Haas-Scott (see Lemma \ref{getembedded}), for each $i$ there is an embedded level surface $S_i$ in the $1$-neighborhood of $f_i(S)$, and the surfaces $S_i$ must also exit $\CE$. As in the previous proof, we write $(S_i,\infty)$ for the component of $M\setminus S_i$ that is a neighborhood of the end $\CE^+$, and we denote the other component by $(-\infty,S_i)$.  Pick some $i>>0$ such that $p\in (-\infty,S_i)$, and then pick some $j>>i$ such that $f_j(S),S_j\subset (S_i,\infty)$. Note that the homotopy from $f_j(S),S_j$ can then be homotoped to lie in $(S_i,\infty)$.

Let $S^-$ be the boundary component of the $1$-neighborhood of $CC(M)$ that faces $\CE^-$. We construct a homotopy from $S^-$ to $S_j$, by 
\begin{enumerate}
	\item first applying Corollary \ref{kor SRSs near partial} to homotope $S^-$ to a minimizing SRS via a homotopy whose image has bounded $\epsilon$-diameter,
	\item then homotoping in a bounded neighborhood,
	\item then homotoping through minimizing SRSs to a SRS in a bounded neighborhood of $f_i$,
	\item then homotoping to $f_j$ in that bounded neighborhood,
	\item then homotoping to $S_j$ in $(S_i,\infty)$.
\end{enumerate}    Here, homotopies (2)-(4) are as given by The Interpolation Theorem.

Since $p\in CC(M)$, it lies in the submanifold $[S^-,S_i]\subset M$ bounded by  $S^-$ and $S_i$. The homotopy (1)--(5) above must pass through $p$: indeed, it can be retracted into $[S^-,S_i]$, giving a map $$S \times [0,1] \longrightarrow [S^-,S_i]$$ that restricts to a homeomorphism on the boundary, and hence has nonzero degree. Now, homotopy (5) above cannot pass through $p$, and homotopies (1), (2) and (4) all have images of bounded $\epsilon$-diameter, so it follows that $p$ lies at bounded $\epsilon$-distance from one of the minimizing SRSs in homotopy (3), as desired.
\end{proof}

\subsection{Ruled surfaces in link complements} \label {shrinkwrapping section} In \cite{Calegarishrinkwrapping}, Calegari-Gabai showed  how to \emph{shrinkwrap} a surface in  a hyperbolic $3$-manifold $M$ around a geodesic link $\Lambda$. This was the crucial technical tool in their proof, in the same paper, of the Tameness Theorem.   Shrinkwrapping was later refined and greatly simplified by Soma \cite{Somaexistence}, see also Namazi \cite [Section 3]{Namaziquasiconvexity}, who replaced the minimal surfaces used by Calegari-Gabai with simplicial ruled surfaces. In both versions, the idea is to take a surface $S \subset M$ and homotope it \emph{without passing through $\Lambda$} to be `tight', in either the sense of being minimal or being simplicial ruled. 

In this section, we develop an even simpler alternative to shrinkwrapping. Instead of working in $M$ relative to the link $\Lambda$, we just drill out $\Lambda$ and work with simplicial ruled surfaces in $M \setminus \Lambda$.  %Because we do not know a convenient reference for a variant of Lemma \ref{neg-metric} in variable curvature, we work here with hyperbolic $3$-manifolds rather than those with pinched negative curvature.  However, we expect that all the results below apply more generally.

%Suppose  $\Lambda $ is a link in a compact $3$-manifold $M$ and $$f : N \longrightarrow M \setminus\Lambda$$ is an embedding of some compact manifold $N$; here we will always have $\dim N=2$ or $3$. We say that $f$ is \emph{$2$-incompressible} relative to $\Lambda \subset M$ if   there is no essential loop $\gamma $ in $N$ whose image $f(\gamma)$ is either homotopically trivial in $M\setminus \Lambda$ or homotopic to a multiple of a meridian on the boundary of a regular neighborhood of $\Lambda$.  Note that this definition is purely topological, and some authors (e.g.\ Namazi \cite{Namaziquasiconvexity}) use a stronger definition when $M$ is a hyperbolic manifold with cusps. The above suffices for our purposes, though, since in this paper will only need to consider links in manifolds without cusps.

A {\em geodesic link} in a Riemannian $3$-manifold $M$ is a  finite disjoint union of simple, closed geodesics. The {\em tube radius} of  a geodesic link $\Lambda$ is the maximal $\delta$ such that  the metric neighborhood $$\CN_\delta(\Lambda)=\{x\in M\ \vert \ d_M(x,\Lambda)\leq \delta\}$$ is a (closed) regular neighborhood of $\Lambda$. A geodesic link is \emph {$\delta$-separated} if its tube radius is at least $\delta$, and \emph{hyperbolically $\delta$-separated} if in addition, the metric of $M$ is hyperbolic when restricted to $\CN_\delta(\Lambda)$.

\begin{lem}[PNC metrics on link complements]
\label{neg-metric}For all $\delta$ there are $\kappa^-\leq \kappa^+ <0$  as follows. Suppose $\Lambda\subset M$ is a  hyperbolically $\delta$-separated geodesic link in a complete Riemannian 3-manifold $(M,\rho_0)$. Then the manifold $M\setminus\Lambda$ admits a complete Riemannian metric $\rho$ with the following properties:
\begin{enumerate}
\item the sectional curvature of $\rho$ is between $\kappa^-$ and $\kappa^+$ on $\CN_{\delta}(\Lambda)$.
\item the metrics $\rho$ and $\rho_0$ agree outside of  $\CN_{\delta}(\Lambda)$,
\item the ends of $(M\setminus \Lambda,\rho)$ corresponding to the components of $\Lambda$ have neighborhoods on which $\rho$  is hyperbolic.%\item if $\inj(M)> r$, then every closed geodesic in $(M_\Lambda,\rho)$ is also longer than $\delta$.
\end{enumerate}
\end{lem}

In the literature one can find numerous constructions of negatively curved metrics in the complement of a geodesic link; see for instance Kojima \cite[Theorem 1.2.1]{Kojimadeformations} (using an argument of Kerckhoff) and Agol~\cite{agol2002volume}.  The idea is to  work in polar coordinates around $\Lambda$ and stretch the metric in the radial direction to create a cusp. In particular, Kojima's computations imply the lemma above. While his input is a hyperbolic cone manifold, the computation is local, so applies in the setting above. Choosing his unspecified functions $\phi,\psi$ appropriately, one can ensure that the resulting manifold has hyperbolic cusps as in (3). Also, he only proves that his metric is negatively curved, but if we have arranged the cusps to be eventually hyperbolic, a uniform sectional curvature bound comes just from a compactness argument.

Also, the reader may wonder why we are assuming the metric on $M$ is hyperbolic in a neighborhood of $\Lambda$. While hyperbolicity is essential in the explicit metric deformations discussed above, one can probably start with a general manifold of pinched negative curvature, take a $\delta$-separated link $\Lambda \subset M$, perturb the metric in a neighborhood of $\Lambda$ to be hyperbolic using the techniques of \S 3 of Hou \cite{hou2003critical}, and then apply Lemma \ref{neg-metric}. However, for this to be useful one would need to control the tube radius of $\Lambda$ \emph{after} perturbation, which we couldn't easily do using Hou's argument. In fact, in the definition of a product region in \S \ref{sec: WPR} below, we assume that product regions are hyperbolic, and this assumption could be removed if we had a version of Lemma \ref{neg-metric} that worked without any hyperbolicity assumption in $M$.

\medskip

%Define a {\em $\Lambda $-homotopy} of a map $f:S\to M \setminus \Lambda$ to be a homotopy $$F:S \times [0, 1] \to M, \ \ F_t(x):=F(x,t)$$
%such that $f=F_0$ and $F^{-1}(\Lambda)\subset S\times\{1\}.$ Also, we say that $\Lambda$ is \emph{$\delta$-separated}  if the metric $\delta$-neighborhood of $\Lambda\subset M$ is a regular neighborhood.

The following theorem plays a crucial role in this paper. 

\begin{sat}[Shrinkwrapping onto links]\label{shrinkwrapping}
Let $ M$ be a $3$-manifold with a  complete metric $\rho_0$ of pinched negative curvature and let $\Lambda\subset M$ be a hyperbolically $\delta$-separated geodesic link.  Suppose that  $S$ is a closed orientable surface of genus at least $2$ and that $$h:S\longrightarrow M\setminus \Lambda$$ is a $\pi_1$-injective map. Fix a  complete negatively curved metric $\rho$ on $M\setminus \Lambda$  such that the metrics $\rho$ and $\rho_0$ agree outside of  $\CN_{\delta}(\Lambda)$. Then $h$  is homotopic in $M \setminus \Lambda$ to a  minimizing simplicial ruled surface $$f: S \longrightarrow  (M \setminus \Lambda,\rho),$$
and if there is a simple closed curve $\gamma \subset S$  such that $h(\gamma)$ is homotopically trivial in $M$,  we can take $f(S) \cap \CN_{\delta}(\Lambda) \neq \emptyset$.
\end{sat}

 The above is our replacement for the shrinkwrapping theorems of Calegari--Gabai \cite{Calegarishrinkwrapping} and Soma~\cite{Somaexistence}. The version we state above is advantageous for us since it fits in better with our Interpolation Theorem, its hypotheses are weaker (Calegari--Gabai and Soma need a stronger incompressibility condition on $h$), and its proof is short given \S \ref{homotopies}.
 
 \begin {proof}
In general, the existence of such an $f$ is immediate from, say, Lemma \ref{minimizers}. If $h(\gamma)$ is homotopically trivial in $ M $, it is either homotopic in $M \setminus \Lambda$ to a geodesic  that intersects $N_\delta(\Lambda)$  nontrivially, or it is parabolic and can be homotoped completely into $N_\delta(\Lambda)$.  So, for $t\approx 0$, the image of the surface $f_t$ from Corollary \ref{construction} intersects $N_\delta(\Lambda)$. 
\end {proof}
%The existence of $f$ is immediate from Lemma \ref{minimizers}:  one takes a weighted triangulation $(\mathcal T,\omega)$ of $S$ and finds a simplicial ruled surface that minimizes for $(\mathcal T,\omega)$.  Note that a priori, the resulting minimizer may be collapsed, but any collapsed simplicial ruled surface can be homotoped to a (non-collapsed) simplicial ruled surface without changing its image, by collapsing each triangle continuously.
%
%Picking a simple closed curve $\gamma \subset S$,  If $g(\gamma)$ is homotopically trivial in $ M $, it is either homotopic in $M_\Lambda$ to a geodesic  that intersects $N_\delta(\Lambda)$  nontrivially, or $g(\gamma)$  represents a parabolic element of $\pi_1 M_\Lambda$ and can be homotoped completely into $N_\delta(\Lambda)$.  So,
%
%If some $\gamma \subset S$ has nullhomotopic $g$-image, then we can build $\mathcal T$ so that $\gamma$ is  embedded in the $1$-skeleton of $\mathcal T$. Let $\omega_t$ be the system of nonnegative weights on $\mathcal T$ where $\omega_t(e)=1-t$ if $e$ is on $\gamma $ and $\omega(e)=t$ otherwise, and let $$f_t : S \longrightarrow M_\Lambda$$ be simplicial ruled surfaces that minimize for $(\mathcal T,\omega_t)$.
%So,
%   Lemma \ref{degenerating minimizers} implies that for very small $t$, the image $f_t(S)$  intersects $N_\delta(\Lambda)$. 

Perhaps the only drawback of our approach is that we produce simplicial ruled surfaces in $M \setminus \Lambda$, while Calegari--Gabai and Soma produce surfaces in $M$.  In other words, the Bounded Diameter Lemma applies to the surfaces produced in Theorem \ref{shrinkwrapping}, but the metric involved is $\rho,$ not the hyperbolic metric.  However, all of our links in this paper will have bounded length and $\rho=\rho_0$  outside $N_{\delta}(\Lambda)$. So, we can bound the  relative diameter in $M$ with the following trivial  consequence of Theorem \ref{shrinkwrapping} and the Bounded Diameter Lemma\footnote{Note that for the proof of Corollary \ref{shrinkwrapping2}, the complete metric on $M\setminus \Lambda$ will have sectional curvature bounded above by the maximum of $\kappa^+$ and the upper sectional curvature bound in Lemma \ref{neg-metric}, which depends only on $\delta$.}.

\begin {kor}\label {shrinkwrapping2}
Using the notation of Theorem \ref{shrinkwrapping}, if $(M,\rho_0)$ has an upper sectional curvature bound of some $\kappa^+$, the metric $\rho$ is selected using Lemma \ref{neg-metric}, each component of the  link $\Lambda$ has length at most $L$ and we fix $\epsilon>0$, then the $\epsilon$-diameter of $f(S)\subset M$ is at most some constant $D=D(\kappa^+,\epsilon,\delta,g,L)$.  In particular, if there is a simple closed curve $\gamma \subset S $  such that $h(\gamma)$  is null-homotopic in $M$, we can assume the entire image of $f(S)$ lies at $\epsilon$-distance at most $D+\delta$ from $\Lambda \subset M$. \end {kor}

\section{Wide product regions}

\label{sec: WPR}
%\begin{itemize}
%\item 
%\end{itemize}

As the reader might recall, our main theorem is that a closed hyperbolic $3$-manifold with rank bounded above and injectivity radius bounded below decomposes into a union of bounded diameter `building blocks' and wide product regions. In this section, we define \it product regions \rm precisely, and detail their useful properties. Many of the results here are technical lemmas that will be used later.

For simplicity, we will assume everywhere below that \emph {$M$ is a complete Riemannian $3$-manifold with  pinched negative curvature}
$$\kappa^- \leq \kappa(M) \leq \kappa^+ < 0$$
\emph{and hyperbolic cusps}. For the most part, we only care about the case where $M$ is hyperbolic, and when we define product region below, we will actually specify that the metric of $M$ is hyperbolic within a product region. However, it is necessary for \S \ref{Double compression bodies} to allow there to be some variable negative curvature somewhere in $M$, even if  it occurs far from the part of $M$ with which we are concerned.   In particular,  if we phrase all of the following for manifolds with pinched negative curvature, then we can start with a hyperbolic $M$,  drill out some link $\Lambda \subset M$, and then work without comment in the manifold $M\setminus \Lambda$, equipped with the  complete metric of variable  negative curvature  described in Lemma \ref{neg-metric}.

We will fix below some $\epsilon \in (0,1) $ smaller than the Margulis constant and some genus bound $g \in \BN $, and we will often write things like $$L = L ( \kappa^+,\epsilon, g)$$ to indicate that $L $ is a constant depending on $\epsilon, g$ and $\kappa^+.$ Unless otherwise specified, when we talk about NAT surfaces, we mean $\epsilon$-NAT surfaces for our fixed $\epsilon$, i.e. maps $f: S\longrightarrow M$ where essential loops of length less than $\epsilon$ in $S$ map essentially into $M$. See Definition \ref{NATdef}.

\medskip

Intuitively, a `product region' is a subset of $M$ homeomorphic to $\Sigma_g \times [0,1]$  that satisfies a Bounded Diameter Lemma  in the surface direction.  We give two related definitions in this section.  Generally, we prefer the first one because it does not require fixing an arbitrary constant, but the second is needed is order to show that product regions are preserved under bilipschitz maps (Proposition \ref{bilipschitzpreservation}). We will see in Proposition \ref{Bsurfaces} that the two definitions characterize the same objects, up to some boundary error.

\begin{defi}[Product region]\label {prdef}
	A \emph{product region} of genus $g$ in $M$  is the image $U \subset M$  of a proper embedding $$\Sigma_g \times I \longrightarrow M, \ \ I=[0,1],[0,\infty), \text{or } (-\infty,\infty),$$  such that  for some regular neighborhood $\CN(U) \supset U$, we have:
\begin{enumerate}
\item  every point $p\in U$  is in the image of a NAT simplicial ruled surface $\Sigma_g \longrightarrow \CN(U)$  that is a homotopy equivalence,
\item each component $ S \subset \partial U$ lies in the $1$-neighborhood of  another such NAT simplicial ruled surface,
\item the sectional curvature $\kappa$ of $M$  satisfies $\kappa\equiv -1$ on $\CN(U).$
\end{enumerate}
\end{defi}

 For the second definition, if $S$ is a closed surface with some path pseudo-metric, recall from Section \ref{bdlsec} that the \emph{injectivity radius} $\inj_p(S)$ is half the length of the shortest essential loop on $S$ through $p$, and the thick part $S_{\geq \epsilon}$ and thin part $S_{<\epsilon}$ are defined as usual to be the sets of points where the injectivity radius is at least $\epsilon$ and less than $\epsilon$. We say that a rectifiable map $f : S \longrightarrow M$ has \emph {intrinsic $\epsilon$-diameter at most $B$} if when $S$ is equipped with the pullback path pseudo-metric, for every two points $x,y\in S$, there is a path $\gamma$ on $S$ from $x$ to $y$  such that
$$\length (\gamma \cap S_{\geq \epsilon}) \leq B.$$
In other words, these are maps $f$ that satisfy the conclusion of the Bounded Diameter Lemma. Just as in Definition \ref{NATdef} we say that $f$ is ($\epsilon$-)\emph{NAT} if any essential closed curve on $S$ with length less than $\epsilon$ has essential image in $M$.

\begin{defi}[$B$-product region]\label {bprdef}
A  $B$-{\em product region} in $M$ of genus $g$ is the image $U \subset M$  of a proper embedding $$\Sigma_g \times I \longrightarrow M, \ \ I=[0,1],[0,\infty), \text{or } (-\infty,\infty),$$  such that  for some regular neighborhood $\CN(U) \supset U$, we have:
\begin{enumerate} 
\item every $p \in U $ is in the image of a homotopy equivalence $f: \Sigma_g \longrightarrow \CN(U)$  that has intrinsic $\epsilon$-diameter at most $B$,
\item for each component $S \subset \partial U$, we have $\diam_\epsilon(S) \leq B+ 2$,
\item the sectional curvature $\kappa$ of $M$ satisfies $\kappa\equiv-1$ on $\CN(U)$. \end{enumerate}
\end{defi}

We saw in \S \ref{bdlsec} that any genus $g$ NAT simplicial ruled surface in a hyperbolic $M$ has intrinsic $\epsilon$-diameter at most
\begin{equation}
B_0:=\frac{32(2g-2)}{ \epsilon}.
\label{b0}
\end{equation}
So, any product region as in Definition \ref{prdef} is a $B_0$-product region. In light of this, the reader may wonder why we do not work just with $B$-product regions in this paper. Indeed, it would be fine to do so, but there are enough constants in this paper, and morally, there should be no constant in the definition of a product region.

Note that condition (2) in Definition \ref{bprdef} is just a bound on the $\epsilon$-diameter of $S$ as a subset of $M$, rather than a true analogue of condition (2) in Definition \ref{prdef}.  The difference is not important and is just because it's simpler to state. The only reason our first definition requires the boundary components to be close to SRSs is that then we get an $\epsilon$-diameter bound for free without having to specify some arbitrary constant. Also, note that the $2$ in Definition \ref{bprdef} is there because there is a $1$ in Definition \ref{prdef} and $1+1=2$.

\medskip

We expect it is not wrong to think of a ($B$-)product region as a hyperbolic region homeomorphic to $\Sigma_g \times I$  that is foliated by surfaces with bounded intrinsic $\epsilon$-diameter. This stronger  characterization would be a pain to verify, however, and is not needed here. We do refer the reader to White \cite{white2013spectral}, though, who essentially shows that our definition is equivalent  to this stronger one inside $\epsilon$-thick manifolds.

 There are still other definitions of product region in the literature. For instance, Namazi \cite {Namaziquasiconvexity} defines product regions as  subsets of $M$  that are bilipschitz to certain subsets of doubly degenerate hyperbolic $3$-manifolds homeomorphic to $S \times \BR$.   Up to some boundary error, his product regions are also product regions as in (either) of our definitions---one can deduce this from Proposition \ref{Bsurfaces} below. The converse should also be true in some sense, although a proof of this might be rather complicated. %likely involve the machinery developed by Minsky-Brock-Canary in the proof of the  ending lamination conjecture \cite{Minskyclassification1,Brockclassification2}.  

 \medskip

Here is our main source of product regions.

\begin{prop}[Constructing PRs]\label {constructing prs}
Let $M $ be a complete hyperbolic $3$-manifold with no cusps that has a degenerate end $\CE $. Then there is a (closed) neighborhood of $\CE $ that is a product region  homeomorphic to $\Sigma \times [0,\infty)$, for some  closed surface $\Sigma$.%exhausted by a nested sequence $\{U_i\}$ of product regions. Moreover, the $\epsilon$-width of $U_i$ converges to infinity as $i\to \infty$.
\end{prop}

 Conversely, note that if a product region $U$ in a hyperbolic manifold $M$ is homeomorphic to $\Sigma \times [0,\infty)$, then $U$ is a neighborhood of a degenerate end, by definition of `degenerate', see \S \ref{sec:ends}.

\begin{proof}
Canary's Filling Theorem \cite {Canarycovering} (or, see Theorem \ref{filling theorem}) implies that there is a pair of neighborhoods $U \subset U' $ of $\CE $, both homeomorphic to $\Sigma \times (0,\infty) $, such that every point of $U $ lies in the image of a NAT simplicial ruled surface $\Sigma \to U'$ that is a homotopy equivalence.

The end $U$ has  infinite $\epsilon $-diameter, by Fact \ref{fact: dist proper}.  So pick $p\in U$ such that $d_\epsilon(p,M \setminus U)$ is larger than one plus the constant in the Bounded Diameter Lemma, and pick a NAT simplicial ruled surface $\sigma : \Sigma \longrightarrow U'$ through $p$.   Then the $1$-neighborhood $\CN_1(\sigma(\Sigma))$ of the image of $\sigma$ lies in $U$. Since $\sigma $ is incompressible in $U $ and homotopic within $U $ to an embedding, Lemma \ref {getembedded} implies that $\sigma $ is homotopic to an embedded surface $S $ whose image is contained in $\CN_1(\sigma(\Sigma))$. By Waldhausen's Cobordism Theorem, $S$ bounds a (closed) neighborhood of $\CE$   that is homeomorphic to $\Sigma \times [0,\infty)$ and is contained in $U$. It is immediate from the construction that this neighborhood is a product region.\end{proof}

%\begin{bei}[Quasi-fuchsian convex cores]
% Suppose that $M \cong \Sigma_g \times \BR$ be a quasi-Fuchsian hyperbolic $3$-manifold. The convex core $CC(M)$ is then homeomorphic to $\Sigma_g \times [0,1]$, and it follows from work of Canary (see the proof of \cite[Theorem 6.2]{Canarycovering}) that every point $p\in CC(M)$ is within some $D=D(g)$ of  the image of some  simplicial hyperbolic surface $f: \Sigma_g \longrightarrow M$  that is a homotopy equivalence.  
%We have seen that such an $f$ has intrinsic diameter rel $M_{< \epsilon}$ at most the constant $B_0$ from \eqref{b0}. We can then homotope $f$ so that its image contains $p$ and so that its  intrinsic diameter is at most $B_0+D$ rel $M_{< \epsilon}$, say. Finally, each component of $\partial CC(M)$ is an incompressible pleated surface, so has diameter at most $B_0$ rel $M_{< \epsilon}$ by the Bounded Diameter Lemma  for pleated surfaces, see \cite{Bonahonbouts}. So, $CC(M)$ is a $(B_0+D)$-product region.
%\end{bei}
%

\subsection{Width and Topological Product Regions.} \rm A {\em level surface} $S$ in a product region $U$ is an embedded surface such that the inclusion $S\hookrightarrow U$ is a homotopy equivalence. We say that $U$ is a \it subproduct region \rm of $V$ if both are product regions and $U \hookrightarrow  V$ is a homotopy equivalence. By Waldhausen's Cobordism Theorem, if $U \subset int(V)$ is a subproduct region then the pair $(V,U)$ is homeomorphic to the pair $(\Sigma \times [-2,2],\Sigma \times [-1,1])$, for some surface $\Sigma$.

  The {\em width} of a product region is the shortest length of a path in it that connects one of its ends or boundary components to the other.  We also define the \it $\epsilon$-width \rm of a  product region to be the shortest $\epsilon$-length of a path from one of its ends or boundary components to the other, see Section \ref {sec:shs}. (So if $U \cong \Sigma_g\times [0,\infty)$, the  $\epsilon$-width of $U$ is the smallest $\epsilon $-length of a path that starts on $\partial U$ and exits the end of $U$.) Note that the $\epsilon$-width of a noncompact product region may be finite if its ends have rank one cusps. However, mostly we will consider ends without rank one cusps in this paper.

\begin{fact}[Width vs. Diameter]\label {width diameter}  Suppose $U \subset M $ is a  (compact) product region of genus $g$ that has $\epsilon$-width $L < \infty$. Then 
$$L/2 \leq \diam_\epsilon(U) \leq L+D, \text{ where } D = 4 \cdot \frac{32(2g-2)}{| \kappa^+ | \epsilon} + 2.$$
\end{fact}

Here, $\diam_\epsilon(U)$  is the \emph {extrinsic} diameter rel $M_{< \epsilon}$ of $U$  as a subset of $M$. So, the $L/2$ would be  nearly realized, for instance, if the two components of $\partial U$ lie very close to each other in $M \setminus int(U)$.  One would get a lower bound of $L$  using intrinsic diameter, but then there would be no possible upper bound, since we do not impose any local regularity on the boundary components of $U$.

\begin {proof}
Let $\gamma$ be a path in $U$ realizing its $\epsilon$-width. The midpoint of $\gamma$  lies at  a $d_\epsilon$-distance at least $L/2$  from $\partial U$, which proves the lower bound. For the upper bound,  just note that if $p_0,p_1 \in U$ then  we can pick NAT  simplicial  ruled surfaces $$f_0,f_1 : \Sigma_g \longrightarrow \CN(U)$$ that are homotopy equivalences and whose images pass through $p_0,p_1$, respectively. By the Bounded Diameter Lemma for NAT surfaces, if $T$  is a  component of $\partial U$, we have: $$\diam_\epsilon f_i(\Sigma_g) \leq \frac{32(2g-2)}{| \kappa^+ | \epsilon}, \ \ \diam_\epsilon T\leq \frac{32(2g-2)}{| \kappa^+ | \epsilon} + 2.$$ 
But
$f_i(\Sigma_g) \cap \partial U \cup \gamma \neq \emptyset$
 for each $i=0,1$.  So, we  can construct a path from $p_0$ to $p_1$ by concatenating  appropriate paths joining points in the sets $f_i(\Sigma_g)$, $\partial U$, and $\gamma$,  which gives the upper bound above. \end {proof}

 Finally, we will refer to properly embedded copies of $\Sigma_g \times I $ in $M$ as \emph{topological product regions}, in contrast with the (geometric) product regions above.  The terms \emph{level surface} and \emph{topological subproduct region} will be used for topological product regions in the obvious ways.

%We will see in \S \ref{PR unqualified} even the constant $B$  in the definition is not so important. Namely,  for any given $B$, we will show that a sufficiently wide $B$-product region $U$ is  exhausted by simplicial ruled surfaces in $M$, except possibly near its boundary components; in particular, we will see that $U$ is a $B_0$-product region   up to some boundary error, where $B_0$ is as in \eqref{b0}.  This will justify dropping the constant $B$ from our discussion, and we will talk in the rest of the paper about \emph{product regions}, understanding that the implied constant is always $B_0$, say.  Again, see \S \ref{PR unqualified} for details  about this.

\subsection{Level surfaces and subproduct regions}

 Throughout this section, we remind the reader that $M$ is a complete Riemannian $3$-manifold with  pinched negative curvature
$\kappa^- \leq \kappa(M) \leq \kappa^+ < 0.$

\begin{lem}[Level surfaces]\label{level surfaces}
 We have the following  related constructions of level surfaces in product regions and $B$-product regions:
\begin{enumerate}
	\item Given $\epsilon,g,\kappa^+$,  there is some $D$  as follows. Let $U \subset M$  be a genus $g$  product region of $\epsilon $-width at least $D$ and let $p\in U$. Then there is a level surface $S \subset U$ that lies in the $1$-neighborhood of a NAT simplicial ruled surface $f : \Sigma_g \longrightarrow U$, and where
$$d_\epsilon(p,q) \leq D, \ \ \forall q\in S.$$
\item  Given $\epsilon,g,\kappa^+,B$,  there is some $D$  as follows. If $U \subset M$  is a genus $g$  $B$-product region of $\epsilon $-width at least $D$ and $p\in U$, then there is a level surface $S \subset U$  such that  $\diam_\epsilon(S)\leq B+2$ and $$d_\epsilon(p,q) \leq D, \ \ \forall q\in S.$$
\end{enumerate}
\end{lem}
\begin{proof}
For (1), let $B_0$  be the constant in the Bounded Diameter Lemma for NAT simplicial ruled surfaces. The  $\epsilon$-diameter of each component of $\partial U$ is at most $B_0+2$, so if $D$ is large there is a point $q\in U$ with $$d_\epsilon(q,\partial U)> B_0, \ \ d_\epsilon(p,q) \leq D-(B_0+1).$$
Let  $f : \Sigma_g \longrightarrow \CN(U)$  be a NAT simplicial ruled surface in the homotopy class of a level surface whose image passes through a point $q \in U$, and note that since $d_\epsilon(q,\partial U)\geq B_0$, we have $f(\Sigma_g) \subset int(U)$. By work of Freedman-Hass-Scott (see Lemma \ref{getembedded}), the map $f$ is homotopic in $U$ to an embedded surface $S$ that lies in $\CN_1(f(\Sigma_g))$. Since $S$ is incompressible in $U$, it is a level surface. So since $d_\epsilon(p,q) \leq D-(B_0+1)$, we have. $$d_\epsilon(p,q) \leq D, \ \ \forall q\in S.$$

 The proof of (2)  is the same, except that we use $B$  instead of $B_0$ and NAT surfaces with intrinsic $\epsilon $-diameter at most $B$ instead of NAT  simplicial ruled surfaces in $M$.
\end{proof}

Lemma \ref{level surfaces} can be used  to construct subproduct regions $V \subset U $ while prescribing the positions of the two components of $\partial V $ up to a bounded $\epsilon$-distance. For instance, the following is immediate:

\begin{kor}[Chopping PRs]\label{cutting-product}
There is some $D=D(\epsilon,\kappa^+,g)$ (or, $D=D(\epsilon,\kappa^+,g,B)$) as follows. Suppose that $U\subset M$ is a  compact genus $g$ product region (or, a $B$-product region) with $\epsilon$-width $L$. Pick $L_1,\ldots,L_n$ with $\sum_i L_i = L$, and assume that each $L_i \geq D$. Then $U$ can be decomposed as the union of $n$ subproduct regions (or, as a union of sub-$B$-product regions)  with  disjoint interiors  $$U=U_1\cup\dots\cup U_n,$$
where each $U_i$ has $\epsilon$-width in $[L_i-D,L_i+D]$.\end{kor}

\begin{proof}
Let $D_1$ be  at least the constant in Lemma \ref{level surfaces}, and let us also assume that $D_1\geq B+2$ if $U$ is a $B$-product region. Set $D=2D_1$.

Pick a path $\gamma$ realizing the $\epsilon $-distance between the components $\partial_\pm U$ of $\partial U$. For each $i=1,\ldots,n$, let $p_i$ be a point on $\gamma$  such that
$$d_\epsilon(p_i,\partial_- U)=\sum_{j=1}^i L_j.$$  And for each $i=1,\ldots,n$, pick a level surface $S_i \subset U$ by applying Lemma \ref{level surfaces} to $p_i \in U_i$, and also define $S_0 = \partial_- U$, and $S_{n+1}=\partial_+ U$.
Then the $\epsilon $-distance between each $S_i$ and $S_{i+1}$ is between $L_i-2D_1$ and $L_i+2D_1>0$, so in particular $S_i,S_{i+1}$ are disjoint as  we have assumed that each $L_i \geq D=2D_1$. Hence, $S_i,S_{i+1}$ bound a  topological subproduct region $U_i \subset U$. Each $S_i$ is contained in the $1$-neighborhood of a NAT simplicial ruled surface in $U$ (or, has  intrinsic $\epsilon$-diameter at most $B+2$), so each $U_i$ is a subproduct region of $U$ (or, a sub-$B$-product region).  Finally,  since $\gamma $ minimizes the $\epsilon$-distance between $\partial_\pm U$, the surface $S_i$  separates $S_{i+1}$ from $\partial_- U$ in $U$. So, one can see inductively that all the $U_i$ have disjoint interiors, and $\cup_i U_i = U$.
\end{proof}

 Here is a version of the above that works for noncompact product regions.   The proof is essentially the same, so we leave it to the reader\footnote{The industrious reader could also write down a version for $B$-product regions.}.

\begin{kor}[Truncating noncompact PRs]\label {truncating noncompact}
Fix $\kappa^+,\epsilon,g$ and let $D$ be as in Lemma \ref{cutting-product}. Suppose that $U\subset M$ is a noncompact genus $g$ product region. Then for any $L > D$,  we have: 
\begin{enumerate}
\item If $\partial U\neq \emptyset$, we can write 	$U=U_1 \cup U_2$  as a union of two  product regions with disjoint interiors, where $U_2$ is compact and has $\epsilon$-width in $[L-D,L+D]$, while $U_1$ is noncompact.
\item If $\partial U=\emptyset$, we can write $U=U_1\cup U_2 \cup U_3$  as a union of product regions with disjoint interiors, where $U_2$ is compact and has $\epsilon$-width in $[L-D,L+D]$, and $U_1,U_3$ are both noncompact.
\end{enumerate}
\end{kor}

 One advantage of Corollary \ref{truncating noncompact}  is that we can now  state most of our results for compact product regions,  knowing that they translate to the noncompact case using the above.

%
%We now show that when distances are measured rel $M_{< \epsilon}$, 
% product regions are quasi-isometric to line segments.
% 
% \begin {lem}
% Suppose that $U\subset M $ is a product region with $\epsilon$-width $L$, and let $S$ be a component of $\partial U$.  Endow $U$  with the path pseudo-metric $d^U_\epsilon$ obtained from the pseudo-metric $d_\epsilon $ on $M$. Then the map
% $$\rho : (U,d^U_\epsilon) \longrightarrow [0,L], \ \ \rho(x)=d_U(x,S)$$
%  is a $(1,C)$-quasi-isometry for some $C=C(g,\epsilon)$.
% \end {lem}
%\begin {proof}
% The map $\rho$ is surjective, so it suffices to show that for $x,y\in U$, 
% \begin {equation}
% 	d_U(x,y)-C \leq |\rho(x)-\rho(y)| \leq d_U(x,y)+C.\label {qi}
% \end {equation}
% Let $\gamma : I \longrightarrow U$  be  a   path with $\gamma(0)\in S$ that realizes the $\epsilon$-width of $U$. %, and pull back the pseudo-metric $d_\epsilon^U$ to $I$. After collapsing points at zero distance from each other, $I$ becomes the interval $[0,L]$. %Note that $\gamma $ is an isometric embedding, and that $d_U(\gamma(t),S)=t$ for all $t$.
%  By Lemma \ref{bdlpr} (5),  there are level surfaces $S_x,S_y$ in $U$ that are contained in the $D$-neighborhoods  of $x,y$, respectively.  These level surfaces must intersect $\gamma$  at points $\gamma (s)$ and $\gamma(t)$, respectively.
% Since $d^U_\epsilon(x,\gamma(s))$ and $d^U_\epsilon(y,\gamma(t))$  are both at most $D$, and $$d^U_\epsilon(\gamma(s),\gamma(t))=|\rho(\gamma(s))-\rho(\gamma(t))| ,$$ Equation \eqref{qi} follows, with $C=2D$.  \end {proof}
% 
% 
%\vspace{2mm}

\subsection{Incompressibility rel links}

In this section, we prove Theorem~\ref{links}, which is inspired by Namazi \cite{Namaziquasiconvexity}, and which allows us to make a product region $U \subset M$ $\pi_1$-injective by drilling out appropriate geodesic links in $U$ near its boundary components. After proving it, we'll describe some of its many applications.

The first step is the following lemma, which shows how to construct closed geodesics homotopic to simple closed curves on level surfaces at arbitrary locations in a product region. Many of the results in this section will be phrased just for $B$-product regions, but we remind the reader that (unqualified) product regions as in Definition \ref{prdef} are also $B_0$-product regions, where $B_0=32(2g-2)/\epsilon$.

\begin{lem}[Geodesics in PRs]
\label {geosinprs}
There is some $D=D(\kappa^+,g,B,\epsilon)$ as follows. Let $U \subset M $ be a genus $g $ $B$-product region with $\epsilon$-width at least $D$, and let $p\in U$. Then there is a closed geodesic $\gamma \subset U$  that is homotopic in $U$ to a simple closed curve on a level surface, such that $$d_\epsilon(p,\gamma)\leq D, \ \ \ \length \gamma \leq D.$$ \end{lem}
\begin {proof}
%For (1), pick a simplicial ruled surface $f : S \longrightarrow \CN (U) $ with $p \in f (S) $ that is in the homotopy class of a level surface.  If we equip $S$  with the pullback metric, the Bounded Diameter Lemma states that $S $ has diameter at most some $D$ rel $S_{< \epsilon } $.  To have a $\epsilon$-diameter bound on $f (S) $, one only needs to show that $f (S_{<\epsilon })\subset M_{< \epsilon } $.  But if some essential loop on $S $ with length less than $\epsilon $ is mapped outside $M_{< \epsilon } $, then its image is compressible in $M$, and hence bounds a (singular) disk in $M $ with diameter at most $\epsilon $. Since $S$ is incompressible in $U$, such a disk must exit $U$. Working inductively through the components of $S_{< \epsilon } $, one sees that if $d(p,\partial U)\geq D+\epsilon$, at each step the distance to $\partial U $ is large enough that this is impossible. So, $\diam_\epsilon (f (S)) \leq D$, which  proves (1)  with the constant $D+\epsilon$ instead of $D$.

%For (2),  as $f$ is incompressible the Bounded  Diameter Lemma  implies that $\diam_\epsilon f(S)$  is at most some $D=D(\kappa^\pm,g,\epsilon)$. So, if $d(p,\partial U)$  is bigger than this $D$, then $f(S)\subset U$.  Every incompressible surface in a topological product region is homotopic to a cover of a level surface, so as $S$ has genus at most $g$, the map $f$  is in the homotopy class of a level surface.

The $\epsilon$-diameter of each component of $\partial U$ is at most $B$, so if the $\epsilon$-width of $U$ is bigger than $2A$ (for some $A$ to be determined later) there is a point $q\in U$ with $$d_\epsilon(p,q) \leq B+A, \ \ d_\epsilon(q,\partial U)\geq A.$$
 Choose a homotopy equivalence $f : \Sigma_g \longrightarrow \CN(U)$ with intrinsic $\epsilon$-diameter at most $B$ that passes through $q$. 

For convenience in notation, let's assume that $B\geq \epsilon$. Then by the same argument as in Corollary \ref{berskor}, there is a simple closed curve $\alpha $ on $\Sigma_g$  such that the length of $ f(\alpha)$ is at most $2B$. (If the $\epsilon$-thin part of $\Sigma_g$ is nonempty, there's a curve of length $2\epsilon$. Otherwise, $\Sigma_g$ has diameter at most $B$, so has a curve of length at most $2B$.)
If the curve $\alpha$ is null-homotopic in $M$, it bounds a disk in $M$  with diameter $2B$; this disk must then intersect $M \setminus U$, which is impossible if $A\geq 3B$. We  also claim the element $f(\alpha) \in \pi_1 M$ is not parabolic as long as $d(q,\partial U)$ is large enough.  For if it were parabolic, Lemma \ref{distancetogeodesic} would give a noncompact  component of the $\epsilon $-thin part of $M$ at an $\epsilon$-distance of at most $2B/\epsilon$ from $f(\alpha)$. This component would have to intersect $M \setminus U$, which is a contradiction as long as $A \geq 2B/\epsilon$. Therefore, $f(\alpha)$  is homotopic to a geodesic $\gamma$ in $M$. Again by Lemma \ref{distancetogeodesic}, we have $d_\epsilon(f(\alpha),\gamma)\leq 2B/\epsilon$, so $d_\epsilon(p,\gamma) \leq 2B/\epsilon + B + B+A$. So, for any $A \geq \max\{3B,2B/\epsilon\}$, we have $d_\epsilon(p,\gamma) \leq D$ for any $D \geq 2B/\epsilon + 2B+A.$
\end{proof}

We are now ready to state the main theorem of the section.

\begin{sat} [Existence of links]
\label {links}

Given $R$, there is some constant $L=L(\kappa^+,g,\epsilon,R)$ (or, $L=L(\kappa^+,g,\epsilon,R,B)$) as follows. Let $U \subset M $ a compact genus $g$ product region (or, a $B$-product region) that has $\epsilon $-width at least $L$. Then $U $ can be  divided into subproduct regions (or, sub-$B$-product regions) $$U=U_1 \cup U_2 \cup U_3,$$ where $U_i,U_{i+1}$ share a boundary component for $i=1,2$, each of $U_1,U_3$ has  $\epsilon $-width at most $L$, and such that there is a $0.025$-separated geodesic link $\Lambda \subset U_1 \cup U_3$  with the following properties:
\begin{enumerate}
\item $\length(\Lambda) \leq L$ and $d_\epsilon(\Lambda,\partial(U_1\cup U_3)) \geq R$,
\item the inclusions into $M \setminus \Lambda$ of $U_2$ and every component of $M\setminus U$ are all $\pi_1$-injective.\end{enumerate}   
Moreover, there is a $0.025$-separated geodesic link $\Lambda_2 \subset U_2$  such that
\begin{enumerate}
\item[(3)]  there is no $\pi_1$-injective map $f : S^1 \times [0,1]\longrightarrow U_2 \setminus \Lambda_2$  such that $f(S^1 \times \{0\})$ and $f(S^1 \times \{1\})$ lie on different components of $\partial U_2$.
\end{enumerate}
Finally, if we fix a complete metric on $M \setminus \Lambda$ with pinched negative curvature that agrees with the metric of $M$ outside $N_{0.025}(\Lambda)$, as in Theorem \ref{shrinkwrapping}, then $U_2 \subset M \setminus \Lambda$ is a product region (or, a $B$-product region, respectively). \end{sat}

Most of Theorem \ref{links} is essentially due to Namazi \cite{Namaziquasiconvexity}, except that his manifolds are hyperbolic, his definition of product region is stronger than ours (see \S \ref{PR unqualified}), he doesn't discuss (3), and he proves a stronger incompressibility statement, which is overkill here given our shrinkwrapping technology (Theorem \ref{shrinkwrapping}). However, his proof works in our setting,  and is even a bit simpler since our incompressibility conclusion is weaker.   Since Theorem \ref{links} is one of the most crucial   technical tools in this paper,  we include a proof sketch. 

The constant 0.025 is unimportant, although we do need a uniform separation constant so that Lemma \ref{neg-metric} outputs a metric on $M\setminus \Lambda$ with uniformly bounded sectional curvatures. Here, 0.025 comes from an argument of Gabai, and is chosen so that any closed geodesic that is not 0.025-separated can be surgered to give a shorter geodesic. 

Below, `huge' is much bigger than `bounded', and both terms should be interpreted relative to the constants $\kappa^+,B,\epsilon,g$.

\begin{proof}[Proof Sketch] (See Namazi \cite[\S 5]{Namaziquasiconvexity}.)
We'll first work on constructing $\Lambda$ so that (1) and (2) hold.  For (2) it suffices to  do the construction so that  $U_2$ is  incompressible in $M\setminus \Lambda$, since then $M \setminus U$  is incompressible in $M\setminus \Lambda$ by Van Kampen's theorem. Also, by replacing $M$ by the cover corresponding to $\pi_1 U$, we can assume that $U$ separates $M$. 
 
Fix $n=n(g)$, to be determined later. If $U$ has huge $\epsilon $-width, then by Lemma \ref{geosinprs} and Corollary \ref{cutting-product}, we can  decompose $U$  as a union of product regions (or, $B$-product regions) $$U=U_1\cup U_2 \cup U_3$$ and find a union $\Lambda$ of disjoint closed geodesics in $U_1 \cup U_3$,  with $n$ geodesics contained in each of $U_1,U_3$, such that 
\begin{itemize}
	\item each  component $\alpha$ of $\Lambda$ has bounded length,
\item there is a huge $\epsilon $-distance between  any two components $\alpha,\beta$ in $ \Lambda$, and between any component $\alpha$ and $\partial U_1 \cup \partial U_3$.
\end{itemize}
 Moreoever, we can do this so that the widths of $U_1,U_3$ are at most some (huge) constant depending only on $B,\epsilon,g$.
 If any $\alpha$ comes within $0.025$ of itself, perform a surgery to get a shorter closed geodesic. (This argument is due to Gabai \cite{gabai1997geometric}, but see Namazi's proof.) Hence, we can assume each component $\alpha$ of $\Lambda$ is $0.025$-separated.

 Pick a level surface $S \subset U_2$.   Hoping for a contradiction, suppose that $S$ is compressible in $M \setminus \Lambda$. The Loop Theorem then gives a {simple} closed curve on $S$ that bounds an embedded disc  $D \subset M$  disjoint from $\Lambda$. Surger $S$ along $D$ to produce a new surface in $M \setminus \Lambda $. Repeating, we obtain a collection of at most $|\chi(S)|$ surfaces $S_1,\ldots,S_r$ in $M$ that are all $\pi_1$-injective in $M \setminus \Lambda$. (Possibly some are spheres.)

\begin{figure}
	\centering
\includegraphics{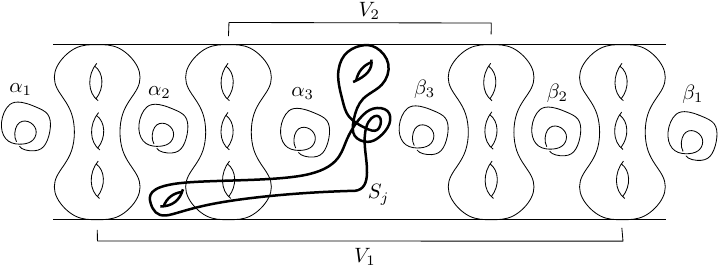}
\caption{The surface $S_j$ in the proof of Theorem \ref{links}. The decomposition $U=U_1\cup U_2 \cup U_3$ is not pictured; if it were, $U_2$ would be a region between $\alpha_3$ and $\beta_3$. }
\label{hosseinpf}
\end{figure}

Whenever  components $\alpha$ and $\beta$ of $\Lambda $  lie in $U_1$ and $U_3$, respectively, they are separated in $M$ by $S$, and hence they are also separated in $M$ by some $S_j$. (One proves this by cutting up $S$ one surgery at a time---remember that $U$  separates $M$.) If the number $n$ of  components of $\Lambda$  we constructed in each of $U_1,U_3$ is very large, there is even some fixed $S_j$ that  separates  three components $\alpha_1,\alpha_2,\alpha_3$  that lie in $U_1$ from three  components $\beta_1,\beta_2,\beta_3$  that lie in $U_3$.  Place four bounded $\epsilon$-diameter level surfaces  between these 6 components of $\Lambda$, so that we have the situation shown in Figure \ref{hosseinpf}. 

Since $S_j$ separates distinct components of $\Lambda$,  it cannot be a sphere. As the components of $\Lambda$ have bounded length,  Corollary \ref{shrinkwrapping2}  implies that we may homotope  $S_j$ in $M\setminus \Lambda$ to a (possibly nonembedded) surface with bounded $\epsilon $-diameter. As  this new $S_j$ still separates $\alpha_3,\beta_3$, it must intersect the subproduct region $V_2 \subset U$ pictured in Figure \ref{hosseinpf}; consequently, it will be contained in the larger region $V_1 \subset U$ as long as the $\epsilon $-distance between any two level surfaces is huge. But since $S_j$ separates $\alpha_1$ from $\beta_1$, it then also separates the components of $\partial U$, which is impossible if the genus of $S_j$ is less than $g$. Hence, no compressions were performed and $S$ was  incompressible in $M \setminus \Lambda$,  contradicting our initial assumption that it was not.

We now   prove (3). At the expense of starting with a slightly larger $L$, subdivide $U_2$ into very wide product regions $$U_2= V_l \cup V_m \cup V_r,$$ where the indices stand for left, middle and right. Let $\partial_l U_2, \partial_r U_2$ be the components of $\partial U_2$, where $\partial_l U_2 \subset \partial V_l$ and $\partial_r U_2 \subset \partial V_r$. Pick bounded length $0.025$-separated geodesics $\gamma,\delta$ in $V_l,V_r$, respectively,  and let $\Lambda_2 = \gamma \cup \delta$. By the annulus theorem (c.f.\ Scott \cite{Scottstrong}), it suffices to show  there is no properly embedded annulus $$f : S^1 \times [0,1]\longrightarrow U_2 \setminus \Lambda_2$$  such that $f(S^1 \times \{0\})\subset \partial_l U_2$ and $f(S^1 \times \{1\}) \subset \partial_r U_2$. So, suppose such an $f$ exists: then there are simple closed curves on $\partial_l U_2, \partial_r U_2$ that are homotopic in $U_2 \setminus \Lambda_2$.  It follows  from Corollary \ref{construction} that $\partial_l U_2, \partial_r U_2$ are homotopic in $M \setminus (\Lambda \cup \Lambda_2)$   to surfaces $\Sigma_l,\Sigma_r$ that have bounded $\epsilon$-diameter with respect to the induced negatively curved metric on $M \setminus (\Lambda \cup \Lambda_2)$, and that intersect along some closed $M \setminus (\Lambda \cup \Lambda_2)$-geodesic. So, $\Sigma_l \cup \Sigma_r$ has bounded $\epsilon$-diameter in $M$ by Corollary \ref{shrinkwrapping2}. 

Recall our assumption that $U$ separates $M$. Let $M_r$ be the component of $M\setminus int(U_2)$ adjacent to $\partial_r U_2$ and define $M_l $ similarly. Since $V_m$ has large $\epsilon$-width, the union $\Sigma_l \cup \Sigma_r$ is contained in either $M_l \cup V_l \cup V_m$ or $V_m \cup V_r \cup M_r$. Let's assume $\Sigma_l \cup \Sigma_r \subset V_m \cup V_r \cup M_r$ without loss of generality. By (2), $\partial_l U_2$ is incompressible in $M\setminus \Lambda$, so the inclusion $U_2 \cup M_r \setminus (\Lambda\cup \Lambda_2) \hookrightarrow M \setminus (\Lambda \cup \Lambda_2)$  is $\pi_1$-injective, and then Lemma~\ref{pushinghomotopieslem} implies that $\Sigma_l$ and $\partial_l U_2$ are homotopic in $(U_2 \cup M_r )\setminus (\Lambda\cup \Lambda_2)$.  A small perturbation of $\partial_l U_2$ separates $\partial_l U_2$ from $\gamma$ in $M$, and the homotopy from this perturbation to $\Sigma_l$ can then be performed in the complement of $\partial_l U_2 \cup \gamma$, implying the $\Sigma_l$ separates $\partial_l U_2$ and $\gamma$ in $M$. This is a contradiction, though, since $\gamma \subset V_l$ and $\Sigma_l \subset V_m \cup V_r \cup M_r$.

\medskip

Finally, we must show that when we endow $M\setminus \Lambda$ with  a  complete negatively curved metric as in Theorem \ref{shrinkwrapping}, we can assume that the subset $U_2 \subset M\setminus \Lambda$  is also a product region (or, a $B$-product region). In  the construction above, the $\epsilon$-distance from $\Lambda$ to $U_2$ in $M$ was chosen to be large. So, we can assume that  the  distance is large enough to use Corollary \ref{cutting-product} to construct a  regular neighborhood $U \supset \CN(U_2) \supset U_2$ that does not intersect $\CN_{0.025}(\Lambda)$.  Moreover, we can assume that this distance is large enough so that no NAT simplicial ruled surface (or, surface with $\epsilon $-diameter at most $B$) can intersect both $U_2$ and $\CN_{0.025}(\Lambda)$. Hence, if use this $\CN(U)$ in the definition of a product region (or, $B$-product region) in $M \setminus \Lambda$, all the required properties come directly from the fact that $U_2$ is a product region  (or, $B$-product region) in $M$.
\end{proof}

Here is an application of Theorem \ref{links} that we'll use many times.

\begin{lem}[Minimizing SRSs in PRs]
\label {minsrsinpr}
There is some constant $D=D(\kappa^+,B,g,\epsilon)$ as follows. Let $U \subset M $ be a genus $g $ $B$-product region with $\epsilon$-width at least $D$, and let $p\in U$. Then there is a minimizing SRS $f : \Sigma_g \longrightarrow U$
 in the homotopy class of a level surface with $$d_\epsilon(p,f(\Sigma_g)) \leq D.$$
 \end{lem}

For product regions, rather than the more general $B$-product regions, Lemma \ref{minsrsinpr} actually follows from Corollary \ref{nearby minimizers}, since at least if $p$ is deep in $U$, we can just take an SRS through $p$ and then homotope it to a nearby minimizing SRS without leaving $U$. One could also rephrase Corollary \ref{nearby minimizers} so it applies to $B$-product regions, but it's a little ugly, so we'll just give a proof using Theorem \ref{links} instead.

\begin{proof}
Suppose that $U$ has large $\epsilon$-width, and let $p\in U$. Write $$U = U_1 \cup U_2 \cup U_3$$ and $\Lambda$ as in Theorem \ref{links}, with $\Lambda$ a link in $U_1\cup U_3$ such that $U_2 \hookrightarrow M \setminus \Lambda$ is $\pi_1$-injective. Since the $\epsilon$-widths of $U_1,U_3$ are bounded, we can apply Lemma \ref{geosinprs} to $U_2$ to get a geodesic $\gamma$ in $U_2$ at bounded $\epsilon$-distance to $p$ that is homotopic to a simple closed curve on a level surface of $U_2$, and which is at least an $\epsilon$-distance of $C+1$ away from $\partial U_2$, where $C$ is the constant in the bounded diameter lemma for SRSs. 

Equip $M\setminus \Lambda$ with the negatively curved metric from Theorem \ref{shrinkwrapping}. Then Corollary \ref{construction} gives a minimizing simplicial ruled surface $$f : \Sigma_g \longrightarrow M \setminus \Lambda$$ whose image intersects a $1$-neighborhood of $\gamma$, and which is homotopic in $M\setminus \Lambda$ to a level surface of $U_2$.  The Bounded Diameter Lemma says that $f(\Sigma_g) \subset U_2$. Since $U_2 $ is $\pi_1$-injective in $M\setminus \Lambda$, $f$ is a $\pi_1$-injective surface in $U_2$ with the same genus as $U_2$, and hence is homotopic \emph{within $U_2$} to a level surface. As $f(\Sigma_g)$ has bounded $\epsilon$-distance to $p$, we are done.
\end{proof}

\begin{kor}[PRs in the convex core]\label {prs in the convex core}	Given $\kappa^+,\epsilon,g$,  there is some $C>0$ as follows. Let $U \subset M$ be a genus $g$ product region with $\epsilon$-width at least $C$. Then  there is a subproduct region $V\subset U$  such that each component of $U \setminus V$ has $\epsilon$-width at most $C$, and where $V \subset int(CC(M))$.\end{kor}
\begin {proof}
By Lemma \ref{minsrsinpr}  and Corollary \ref{cutting-product}, if $U$ is sufficiently wide then there is  a subproduct region $V \subset U$ such that each component of $U \setminus V$ has $\epsilon$-width at most some $C=C(\kappa^+,\epsilon,g)$ and  contains the image of a minimizing simplicial  ruled surface $f : \Sigma_g \longrightarrow U$ in the homotopy class of a level surface. 

The image of any minimizing SRS is contained in $CC(M)$, since the closest point retraction onto $CC(M)$ is distance decreasing on $M\setminus CC(M)$, and no component of $M \setminus CC(M)$ can be  contained in $U$, since the injectivity radius at every point of $U$ is bounded above by Lemma \ref{geosinprs}. So as the images of the  simplicial ruled surfaces $f$ above separate $V$ from $\partial U$,  it follows that $V \subset int(CC(M))$.\end {proof}

As a simple corollary of Theorem \ref{links}, we have the following version of Theorem \ref{links} for unions of multiple product regions.

\begin {kor}\label {linkkor}
 Suppose that $\{U^i\}$ is  a collection of compact product regions (or, $B$-product regions) in $M$ as in Theorem \ref{links} and that $\Lambda^i \subset U^i$ is obtained by applying Theorem \ref{links} to each $U^i$. Then the inclusion into $M \setminus \cup_i \Lambda^i$ of each component of $M\setminus U^i$ is $\pi_1$-injective.
\end {kor}
\begin {proof}
 This follows from Van Kampen's Theorem, since for each $i$ the  middle subproduct region $U^i_2 \subset U^i$ from Theorem \ref{links} will be incompressible in $M\setminus \Lambda$.
\end {proof}

In light of Theorem \ref{shrinkwrapping}, our shrinkwrapping theorem, we will often be homotoping  surfaces in the complement of links. The following observation will be useful, and so is worth mentioning now. 

\begin{fact}\label {homotopyfact}
If $\{U^i\}$ is a collection of product regions (or, $B$-product regions) in $M$ and $\Lambda^i \subset U^i$, $\Lambda^i_2 \subset U^i_2 \subset U^i$ are links as in  Theorem~\ref{links}, then if two maps $f,g : X \longrightarrow M \setminus \cup_i \ int(U^i)$ with nontrivial $\pi_1$-image are  homotopic in $ M \setminus \cup_i (\Lambda^i \cup \Lambda^i_2)$,  they are also homotopic in $ M \setminus \cup_i int( U^i )$.
\end{fact}

Note that this result is very similar to Lemma \ref{pushinghomotopieslem}. Indeed, we could use Lemma \ref{pushinghomotopieslem} in the proof if we know that the components of $M\setminus\cup_i int(U_i)$ are compact, and if the images of $f,g$ lie in the same component. As is, the argument is slightly more complicated.

\begin{proof}
We  first claim that  the images of $f,g$ are contained in the same component $N$ of $ M \setminus \cup_i  int(U^i_2)$.  If they are not, then since they have nontrivial $\pi_1$-image there is some smooth map $g: A \longrightarrow M \setminus \Lambda$ from an annulus where the boundary components map into different components of $ M \setminus \cup_i int(U^i)$.  Since all the boundary components of the $U^i_2$ are incompressible in $M \setminus \Lambda$, we can homotope $g$ so that all intersections with all $\partial U^i_2$ are essential loops in $A$.  Then $g$ restricts to a map on some subannulus of $A$ that maps into some $U_2^i \setminus \Lambda_2^i$ and connects distinct boundary components of $U_2^i$. By (3) in Theorem~\ref{links}, this is a contradiction.

 We now claim that $f,g$ are homotopic in $N$. For this, let $N_{ext} $ be the union of $N$ and all subproduct regions $U^i_j$ that share a boundary component with $N$. Then each component $S \subset \partial N_{ext}$ is a boundary component of some $U^i_2$, and hence is incompressible in $M\setminus \Lambda$ by (2) of Theorem \ref{links}. The cover of $M\setminus \Lambda$ corresponding to $\pi_1 S$ is then homeomorphic to $S \times \BR$, by the Tameness Theorem \cite{Agoltameness} and standard topological arguments, see Hempel \cite{Hempel3-manifolds}. We can identify $S\times \{0\}$ with a lift of $S$, and up to reversing the sign of $t$, we can assume that for small $t>0$, the point $(p,t) \in S \times \BR$ projects into $U_2^i\subset M\setminus \Lambda$.  Construct a covering space $X \longrightarrow M \setminus \Lambda $ by starting with $N_{ext}$, and gluing each component $S \subset \partial N_{ext}$ to the manifold $S\times [0,\infty)$, and then defining the covering map to $M \setminus \Lambda$ to be the identity on $N_{ext}$, and the covering maps considered above on each $S\times [0,\infty)$. The maps $f,g$ and the homotopy between them all lift to $X$, and the homotopy can then be retracted into $N_{ext} \subset X$, since the complementary components are products.  Projecting down to $M$, we get a homotopy between $f,g$ in $N_{ext}$, which can be homotoped in $M$ to a homotopy in $N$.
\end{proof}
\subsection{Reconciling the definitions, and bilipschitz maps} 
 \label {PR unqualified}

 Here, we show that $B$-product regions are the same as product regions, up to some boundary error depending on $B $.  We then use the  equivalence to show that  bilipschitz images of product regions are product regions, again up to some boundary error.

\begin{prop}[$B$-PR $\implies$ PR]
\label{Bsurfaces}
 There is some $C=C(\kappa^+\epsilon,g,B)$ as follows.
Suppose that $U \subset M $ is a compact $B$-product region of genus $g $ and $\epsilon$-width $L+C$, where $L\geq 0$.
Then $U $ contains a topological subproduct region  that is a product region in the sense of Definition \ref{prdef}, and that has $\epsilon $-width  at least $L$.
\end{prop}

Recall  that for the converse, every product region is a $B_0$-product region, where $B_0=32(2g-2)/\epsilon$, by the Bounded Diameter Lemma.  In particular, this shows that up to some boundary error, every $B$-product region is a $B_0$-product region for this fixed $B_0$.

\begin {proof}
 Suppose that the $\epsilon $-width of $U$ is very large. Using Corollary~\ref{cutting-product} and Theorem~\ref{links}, divide $U$ into three topological product regions $$U=U_1 \cup U_2 \cup U_3,$$  satisfying (1) and (2)  in the statement of Theorem \ref{links}, where regions with adjacent indices share a boundary component, where the $\epsilon $-widths of $U_1$ and $U_3$ are  at most some $L_0=L_0(\epsilon,g,B)$, and where there is a $0.025$-separated geodesic link $\Lambda \subset U_1 \cup U_3$ with length at most $L_0$ such that  the region $U_2$ is incompressible in $M\setminus \Lambda$.  Moreover, fixing a complete metric on $M \setminus \Lambda$ with pinched negative curvature that agrees with the metric of $M$ outside $N_{0.025}(\Lambda)$, as in Theorem \ref{shrinkwrapping}, we can assume $U_2 $  is an (incompressible) $B$-product region in $M \setminus \Lambda$. %If $D$ is large, there is a regular neighborhood of $U_2 \subset \subset M\setminus \Lambda$ that  is disjoint from $N_{0.025}(\Lambda)$ and contains every surface with $\epsilon$-diameter at most $B$ that intersects $U_2$. Hence, $U_2$ is a

By Lemma \ref{minsrsinpr}, there are minimizing simplicial ruled surfaces $$f_1,g_1,g_3,f_3: \Sigma_g \longrightarrow U_2 \subset M\setminus \Lambda$$  such that the images of all four surfaces are all separated from each other and from $\partial U_2$ by some to-be-determined bounded $\epsilon$-distance, the images of $f_1,g_1$ and $g_3,f_3$ lie at a bounded $\epsilon$-distance from $U_1$ and $U_3$, respectively, and where $f_1,f_3$ are separated by both $g_1$ and $g_3$ within $U$.  (So in other words, within $U_2$ the four surfaces appear in the order in which they were initially listed.) By choosing these surfaces to have distance at least $\epsilon$ to $\partial U_2$,  we can ensure that they are all $\epsilon$-NAT when included into $M$. Applying Freedman--Hass--Scott (Lemma \ref{getembedded}),  choose  level surfaces $S_1,S_3$ in $U_2$  that are contained in $1$-neighborhoods of  the images of $g_1,g_3$,  respectively. Then $S_1,S_3$  bound a topological product region $V \subset U_2$  whose $\epsilon $-width differs from that of $U$  by only some constant $C=C(\epsilon,g,B)$.  We must show that every point of $V$  is in the image of a NAT simplicial ruled surface $S \longrightarrow U$ in the homotopy class of a level surface.

We can assume that  the images of $f_1,g_1,g_3, f_3$ are all at least $R$ from each other, where $R$ is as in the Interpolation Theorem. Hence, after replacing $f_1,f_3$ by nearby maps, we can assume that there is a homotopy $(f_t), \ t \in [1,3],$ from $f_1$ to $f_3$ in $M_{\Lambda}$  such that $f_t$ is a simplicial ruled surface for all $t$.  One can also homotope $f_1$ to $f_3$ by first homotoping it to $S_1$ within $U_2 \setminus V$, then homotoping $S_1$ to $S_3$ through level surfaces in $V$, and finally homotoping $S_3$ to $f_3$ in $U_2 \setminus V$.  This homotopy  passes through every point of $V $ exactly once.  Concatenating it with the homotopy $(f_t)$ by simplicial ruled surfaces gives a map from $\Sigma_g \times S^1$ to $M $ whose image is disjoint from $\Lambda $. If this map were smooth it would have degree zero, showing that every point of $V$ lies in the image of the homotopy by simplicial ruled surfaces, and we would be done, and the general case follows from the smooth case by smooth approximation. Hence, every point in $V$  is in the image of some $f_t$.  And  as long as $d(V,\partial U_2)$  is large, every such $f_t$ is a NAT simplicial ruled surface in $M$. \end{proof}

As  a consequence of Proposition \ref{Bsurfaces}, we now prove the following proposition, promised at the beginning of \S \ref{sec: WPR}.

\begin{prop}[Bilipschitz maps preserve PRs]\label{bilipschitzpreservation}
Given $K\geq 1$ and $\epsilon,g$, such that both $\epsilon,K\epsilon$ are less than the Margulis constant, there is some $C=C(K,\epsilon, g)$ as follows. 

Let $M_1,M_2$ be  complete hyperbolic $3$-manifolds and let $U_1 \subset M_1 $ be a genus $g$ product region of $\epsilon$-width  at least $L+C$, where $L>0$, and let $$F : U_1 \longrightarrow U_2 \subset M_2 $$ be a $K $-bilipschitz embedding onto some subset $U_2 \subset M_2 $.  Then $U_2 $ contains a product region $V$ of $\epsilon$-width at least $L $ whose inclusion into $U_2$ is a homotopy equivalence.
\end{prop}

Here, we are only assuming our manifolds are hyperbolic (really it only matters that $M_2$ is) because we require the metric on  a product region to be hyperbolic, and hyperbolicity is certainly not preserved under bilipschitz maps.

 \begin {proof}
 Perhaps  shrinking $U_1$  slightly,  we can assume that $F$ maps $$F(U_1 \cap (M_1)_{< \epsilon}) \subset (M_2)_{< K\epsilon}.$$ For if not, some homotopically essential loop $\gamma$ of length at most $\epsilon$ in $U_1$ is mapped  to a homotopically trivial loop $F(\gamma)$ in $U_2$. Since $$ \length F(\gamma) \leq K\epsilon,$$ the loop $F(\gamma)$ bounds a disk with diameter at most $K\epsilon$ in $M_2$. This disk must exit $U_2$,  which can only happen if $\gamma$  was within $K^2 \epsilon$ of $\partial U_1$.

 By the Bounded Diameter Lemma and Corollary \ref{cutting-product}, we can pass to a subproduct region $V \subset U_1 $, with nearly the same $\epsilon$-width as $U_1$, such that every point $p \in V $ is in the image of a NAT simplicial ruled surface $f : S \longrightarrow U_1 $  in the homotopy class of a level surface. These surfaces have  intrinsic diameter at most $B_0 :=32(2g-2)/\epsilon$ rel $M_{< \epsilon}$, by the Bounded Diameter Lemma, and the components of $\partial V$ have $\epsilon$-diameter at most $B_0+2$.
 For any $f$ as above, the   composition $F\circ f$ has   intrinsic $K\epsilon$-diameter at most $KB_0$. A similar $K\epsilon$-diameter bound applies to the components of $f(\partial V)$. Lemma \ref{differentepsilons} transforms these into $\epsilon$-diameter bounds. Therefore, $f( V)$ is a $B$-product region, for some $B= B(K,g,\epsilon)$, and Proposition~\ref{Bsurfaces} finishes the proof. \end{proof}

 From now on, we will always work with product regions instead of $B$-product regions. 

\subsection{Multiple product regions} In this short section,  we  prove two properties of  \emph{pairs} of product regions in $M$.
 First, we  show that if two product regions overlap significantly, they share a  level surface.

\begin{lem}\label{intersectprs}
	Given $g,\epsilon$, there is some $D=D(g,\epsilon,\kappa^+)$  as follows.  Suppose that $U,V$ are two product regions in $M$, both with genus at most $g$,  and suppose that there is some $p \in U \cap V $ such that
$$d_\epsilon(p, \partial U \cup \partial V) \geq D.$$
 Then there is an embedded surface $S\subset M$  that is a level surface in the interiors of both $U$ and $V$.  Moreover, we can assume  every point in $S$ lies at a $d_\epsilon$-distance at most $D$  from $p$.
\end{lem}
\begin {proof}
If $D$ is much larger than the constants from Lemma \ref{level surfaces} (1) and Corollary \ref{cutting-product}, there is a level surface $S$ for $U$  that lies within $D$ of $p$ that is  contained in a subproduct region $V' \subset V$, where $V'$  is a subset of $U$. This $S$ is an incompressible embedded surface in $V'$, and hence is a  level surface.  Certainly, if $D$  is large enough this $S$  will lie in the interiors of $U,V$.\end {proof}

Second,  we show that if wide product regions co-bound a trivial interval bundle, the whole picture is contained in a single product region.  This is important in the proof of Theorem \ref{nonhakenthm}.  Because the metric on a product region is assumed to be hyperbolic, even when $M$ has variable curvature, we must assume  that $M$  itself is hyperbolic here.

 \begin {lem}\label {trivial interval bundles}
Given $g,\epsilon$, there is some $L=L(g,\epsilon)$  as follows. Suppose that $U_0,U_1$ are two disjoint genus $g$ product regions in a complete hyperbolic $3$-manifold $M$, that each have $\epsilon $-width at least $L $, and that $V$  is a component of $M \setminus int(U_0 \cup U_1)$  that is a trivial interval bundle (i.e., a topological product region). Then $U_0 \cup V \cup U_1$  is a product region. 
 \end {lem}
 \begin {proof}
 It suffices to show that every point $p\in V$  is in the image of a NAT simplicial ruled surface $f : S \longrightarrow U_0 \cup V \cup U_1$ in the homotopy class of a level surface. Using Corollary \ref{cutting-product} divide each $U_i$ into two product regions $U_i^1,U_i^2$, of roughly equal width, where $U_i^1$ is adjacent to $V $. As long as $L$ is large, there is a $0.025$-separated geodesic link $$\Lambda \subset U_1^2 \cup U_2^2$$  such that  the inclusion $V \hookrightarrow M \setminus \Lambda$  is $\pi_1$-injective, by Corollary \ref{linkkor}.
  By Lemma \ref{minsrsinpr}, if $L$  is large there is  for each $i $ a minimizing simplicial ruled surface $f_i : S \longrightarrow U_i^1$  in the homotopy class of a level surface.  The maps $f_0$ and $f_1$  are incompressible and homotopic in $M \setminus \Lambda $,  so after replacing them by nearby surfaces,  the Interpolation Theorem gives a homotopy $(f_t)$  between them in $M \setminus \Lambda $  such that $f_t$ is simplicial ruled for all $t$. 
  
    The rest of the  argument is the same as the end of the proof of Proposition \ref{Bsurfaces}.   Briefly, if $L$  is large  we can also assume that there is a homotopy from $f_0$ to $f_1$  that touches each point of $V $ once. So, concatenating that with $(f_t)$  gives a map from some closed mapping torus to $M$ that  misses $\Lambda$, and hence has degree zero.  Therefore, each point of $V$ is in the image of some $f_t$.   And the image of any such $f_t$   is far from the product regions $U_i^2$ if $L$ is large, so any such $f_t$  is a NAT  simplicial ruled surface in $M$. \end {proof}

\subsection{A covering theorem} Recall that our main source for product regions is  degenerate ends of hyperbolic $3$-manifolds.  In Section \ref {sec:ends}, we saw that Canary's Covering Theorem  \cite{Canarycovering} limited the ways in which degenerate ends can cover.  Similar covering theorems can  be made for product regions with sufficiently large width;  below is one version that will find use in later sections.   We state it just for $\epsilon$-thick manifolds, since then we can  deduce it as a corollary of Canary's result via a geometric  limit argument. However, we expect that a similar statement  holds for product regions with large $\epsilon$-width in arbitrary $M$, and that one  may be able to prove it using arguments more or less along the same lines as those in Canary's paper.

\begin{sat}
[A PR covering theorem] \label {PR covering theorem}    There is some $L=L(g,\epsilon)$  as follows. Let $M$  be a  complete hyperbolic $3$-manifold and let $$U,V,W \subset M $$ be $\epsilon $-thick, genus $g $ product regions  that have disjoint interiors,  and where each shares  one boundary component with the next. Assume that $$f : M \longrightarrow \bar M $$ is a covering map, that $f|_U$  is an embedding, and that the widths of $U,W$  are at least $L$.  Then either%   and where either
\begin {enumerate}
\item the restriction $f |_{U \cup V}$ is  an embedding,
\item[(2)] some component of $\bar M \setminus f(int(U))$ is a  compact interval bundle.
\end {enumerate}
\end{sat}

Intuitively, if we restrict a covering map $f$ to a wide product region in $M$, which above is $U\cup V \cup W$, and we know that the $f$ is an embedding on a sufficiently wide subproduct region, then either $f$ is an embedding on almost all of the product region, or the covering map wraps the product region up into a interval bundle in $\bar M$.

 A less brief, but more clear  way to rephrase (2)  is that  either some component of $f(\partial U)$  bounds a twisted interval bundle over a non-orientable surface in $\bar M \setminus f(int(U))$, or  the  entire  complement $\bar M \setminus f(int(U))$  is a trivial  interval bundle,  in which case $\bar M$ fibers  over the circle  in such a way that level surfaces of $f(U)$ are fibers.

Note that with two applications of Theorem \ref{PR covering theorem},  one can prove that if a product region $V$ has a sufficiently wide subproduct region $U \subset V$ that embeds  under a covering map, then $V$  has a subproduct region $W \subset V$  that also embeds, and where  each component of $V \setminus W$  has bounded width.  However, the version stated above is that which we will need in this paper.
\medskip

 Before starting the proof, we need to establish a couple preliminary results about product regions in $\epsilon$-thick hyperbolic $3$-manifolds.  The first is a (slightly more involved) cousin of Corollary \ref{cutting-product}.

\begin{lem}\label {subproducts in balls}
 Given $g,\epsilon$, there is some $D=D(g,\epsilon)$ as follows.
Suppose that $M $ is a  complete,  $\epsilon $-thick hyperbolic $3$-manifold and that $U \subset M$  is a product region with genus $g$. Fix some $r>0$ and  assume that $p\in U$ is a point with $d(p,\partial U) \geq r+D$.  Then  there is a  subproduct region $V \subset U$ with  $ B_{M}(p ,r-D) \subset V  \subset B_{M }(p ,r+D).$
\end{lem}
\begin{proof}
By Lemma \ref{level surfaces} (1), there is some $ D_1= D_1(\epsilon,g)$  such that the $ D_1$-ball around any point in a genus $g$ product region in an $\epsilon$-thick  hyperbolic $3$-manifold contains a level surface  that is contained in the open $1$-neighborhood of a NAT  simplicial  hyperbolic surface  in the homotopy class of a level surface. Set $D=5D_1$.

Assume for concreteness that  the product region $U $ is compact, i.e.\  homeomorphic to $\Sigma_g \times [0,1]$.  The other cases are similar, and we leave them to the reader. Choose a  level surface $S$  that is contained in the $ D_1$-ball around $p $, and let $\partial_\pm U  $ be the two components of $U $. Since $d(p,\partial U )\geq r+5D_1$, there is a unit speed  geodesic
$$\gamma  : \BR \longrightarrow M $$ such that for some $t ^- < - r-4 D_1 < 0 <  r+4 D_1 < t ^+$, we have  $$\gamma (0)\in S, \ \ \gamma ([t ^-,t ^+]) \subset U , \ \  \gamma(t ^-) \in \partial_- U , \ \  \gamma(t ^+) \in \partial_+ U ,$$
and where $\gamma |_{[t ^-,t ^+]}$ minimizes the distance in $U $ between the components $\partial_\pm U $. 

Let $S ^-,S ^+$ be  level surfaces of $U $  that are contained in the $ D_1$-balls around $\gamma (-r)$ and $\gamma (r)$, respectively. So
$$S ^\pm \subset B(p ,r+2 D_1) \setminus B(p ,r-2 D_1),$$ since $\gamma $  is distance minimizing.  Moreover, $S ^-,S ^+$  are disjoint, and thus  bound a topological subproduct region $V  \subset U $. Note that 
\begin {enumerate}
\item $\gamma \big ( [-r+ D_1,r- D_1]\big ) \subset  V , $
\item $\gamma \big ([t ^-,t ^+] \setminus[-r- D_1,r+ D_1]\big ) \subset U  \setminus  V ,$
\end{enumerate}
since  the geodesic $\gamma |_{[t ^-,t ^+] }$ passes through $S ^-,S ^+$ within  the intervals $$[-r- D_1,-r+ D_1], \ \ [r- D_1,r+ D_1],$$ respectively, and $V $ is the set of points in $U $ that are separated from $\partial U $ by the level surfaces $S ^\pm$.

Suppose $q \in B(p ,r-3 D_1)$ and let $T$  be a level surface of $U $  contained in the $ D_1$-ball around $q$. For some $t \in [t ^-,t ^+]$, we have $\gamma (t)\in T$, and since $\gamma $  is minimizing we must have $t \in [-r+ D_1,r- D_1]$. There is then a path of length  at most $ D_1$ from $q$ to $\gamma (t)$, and  this path cannot intersect either $S ^\pm$. Hence, $q \in V $ by (1) above.

Next, suppose $q \in V $. If $d(q,S ^- \cup S ^+) \leq  D_1,$ then $q \in B(p ,r+3 D_1)$ and we are done.  Otherwise, for any level surface  $T \subset U $  contained in the $ D_1$-ball around $q$,  we have $T \subset V $. Pick $t \in [t ^-,t ^+]$  with $\gamma (t)\in T$. Then by (2) above, we have $t \in [-r- D_1,r+ D_1]$. Hence, $d(p ,q) \leq  D_1 +  (r+ D_1) +  D_1\leq r+3 D_1$ and we are done. \end{proof}

 Second, we use Lemma \ref {subproducts in balls} to prove the following:

\begin{lem}\label{limit-double}
Fix $g\ge 2$ and $\epsilon>0$. Let $(M_i)$ be a sequence of  complete, $\epsilon$-thick hyperbolic $3$-manifolds, let $U_i\subset M_i$ be genus $g$ product regions, and $p_i\in U_i$ a sequence of points with $d(p_i,\D U_i)\to\infty$.  Then after passing to a subsequence, $(M_i,p_i)$ converges geometrically to a doubly degenerate hyperbolic 3-manifold $(M,p)$ homeomorphic to $\Sigma_g\times\BR$.  

Moreover, if $(\phi_i)$  is a sequence of almost isometric maps coming from the geometric convergence, as in Definition \ref{geolimits}, then 
\begin{enumerate}
	\item[(a)] for any level surface $S \subset M$,  the image $\phi_i(S)$  is a level surface in $U_i$  for all large $i$, and  similarly,
\item[(b)] if $S_i \subset U_i$  are level surfaces  such that $S_i \subset B(p_i,R)$  for some $R$ independent of $i$, then  the image of $\phi_i$  contains $S_i$  for large $i$, and $\phi_i^{-1}(S_i)$  is a level surface in $M$.
\end{enumerate}  
\end{lem}
\begin {proof}
Since the $M_i$  are all $\epsilon$-thick,  after passing to a subsequence we can assume that $(M_i,p_i)$ converges geometrically  to some  pointed hyperbolic $3$-manifold $(M,p)$, see e.g.\ \cite[Theorem E.1.10]{Benedettilectures}. We claim $$\sup_{x\in M} \inj_M(x)< \infty.$$ To see this, let $(\phi_i)$  be a sequence of almost isometric maps given by the geometric convergence, as in Definition \ref{geolimits}. Given any $x\in M$,  we can take $i$  large enough so that  $\phi_i(x)$ is contained  deep inside $ U_i$. By  Lemma \ref{geosinprs},  there is a closed geodesic $\gamma$ in $U_i$ that lies near $\phi_i(x)$, and the preimage $\phi^{-1}(\gamma)$ is a  closed curve with bounded length that lies near $x$.  Moreover, $\phi^{-1}(\gamma)$ is homotopically essential in $M$, since if it were not, it would bound a disk with bounded diameter, which we could compose with $\phi_i$ to get a nullhomotopy of $\gamma \subset M_i$, which is impossible since $\gamma $ is a geodesic.

 It therefore suffices to show that $M \cong \Sigma_g \times \BR$. Indeed, $M$ is $\epsilon$-thick, and hence has no cusps, and any such $M$ where the injectivity radius is uniformly bounded above is doubly degenerate by the Geometric Tameness Theorem.  To show $M \cong \Sigma_g \times \BR$, let $D$ be as in Lemma \ref{subproducts in balls} and fix $n\in \BN$.  Then  as $d(p_i,\partial U_i)\to \infty $, for all $i \geq i_0=i_0(n)$,  there is a subproduct region $V_{n,i} \subset U_i$  with $$ B_{M_i}(p_i ,4Dn-D) \subset V_{n,i}  \subset B_{M_i }(p_i ,4Dn+D).$$
 If $i$  is large, $\phi_i$ is almost isometric and its image contains $V_{n,i}$, so we then have  a topological product region $\phi_i^{-1}(V_{n,i}) \subset M$ such that
$$  B_{M_i}(p_i ,4Dn-2D) \subset \phi_i^{-1}(V_{n,i})  \subset B_{M_i }(p_i ,4Dn+2D).$$
Here, $2D$ is  just some constant bigger than $D$, which we need to increase because of the bilipschitz distortion in the maps $\phi_i$. Set
$$W_n :=\phi_i^{-1}(V_{n,i_n}),$$
for some  large $i_n$ chosen as above. Then since $$4Dn+2D  \leq 4D(n+1)-2D,$$  we have
$W_n \subset W_{n+1}$  for all $n$.
 Moreover,  for large $n$ each $W_n$ is a subproduct region of $W_{n+1}$. One way to see this is to note that with $i=i_{n+1}$, the image $\phi_{i_{n+1}}(W_n)$  is an almost isometric image of $W_n$ in $M_{i_{n+1}}$, and hence contains a subproduct region with nearly the same width by  Proposition \ref{bilipschitzpreservation}. If $n$  is large,  the point $p_{i_{n+1}}$ is contained in both this subproduct region and in $V_{n+1,i_{n+1}}$, far away from the boundaries of both. So, Lemma \ref{intersectprs}  implies that $\phi_{i_{n+1}}(W_n)$ and $V_{n+1,i_{n+1}}$ share a level surface.  Applying $\phi_{i_{n+1}}^{-1},$  we see that $W_n$ and $W_{n+1}$  share a level surface as well.  This proves that $M$  is a nested union of topological product regions, each a subproduct region of the next. So, $M\cong \Sigma_g \times \BR$. 

 The statements about level surfaces  are immediate from the construction. For we can assume that we have passed to the subsequence $M_{i_n}$, $n=1,2,\ldots,$ and then any  fixed level surface $S \subset M$ lies in $W_n$  for large $n$, and therefore  is a level surface in $W_n$. Hence, $S$ pulls back to a  level surface in $V_{n,i_n} \subset U_{i_n}$.  Similarly, for any sequence of level surfaces $S_n \subset U_{i_n}$  that lie at bounded distance from $p_{i_n}$,  we have that $S_{i_n} \subset V_{n,i_n}$ is a level surface for large $n$.  This implies that $\phi_{i_n}^{-1}(S_{i_n})$  is a level surface of $W_n$,   and hence  is a level surface of $M$. \end{proof}

 Before starting on Theorem \ref{PR covering theorem}, we  first  prove two lemmas.

\begin{lem}[Overlapping embeddings]\label {overlapping embeddings}
 There is some $L=L(\epsilon,g)$  as follows. Suppose that $M$  is a complete hyperbolic $3$-manifold, $$f : M \longrightarrow \bar M$$  is a covering map, and $U_1,\ldots,U_4$  are  compact, $\epsilon $-thick, genus $g$  product regions  with disjoint interiors, each sharing a boundary component with the next, and where $U_1,U_2,U_4$ all have  width at least $L$. Set $$V = U_1 \cup U_2, \ \ W = U_2 \cup U_3 \cup U_4.$$  If both $f|_{V} $  and $f|_{W}$  are embeddings, either \begin{enumerate}
	\item $f|_{U_1 \cup U_2 \cup U_3}$  is an embedding, or
\item  some component of $\bar M \setminus f(int(V))$, say,  is a compact  interval bundle.\end{enumerate} 
 \end{lem}

See Figure \ref{UV}  for an illustration.

\begin{figure}
	\centering 
\includegraphics{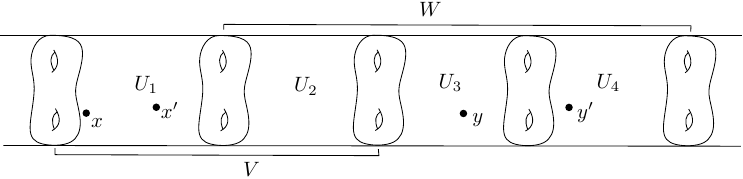}
\caption {The product regions in Lemma \ref{overlapping embeddings},  and the points involved in its proof.}
\label {UV}
\end{figure}
\begin {proof}
Fix some large $L$ to be specified later, and assume $U_1,\ldots,U_4$ are as above. Assume that $f|_{U_1 \cup U_2 \cup U_3}$  is not an embedding.   Then there are points $x \in U_1$ and $y\in U_3$ with $f(x)=f(y)$. 

Choose some point $x' \in U_1$ such that $d(x,x') \leq L/2$ and $d(x',\partial U_1) \geq L/2$; for instance, if we join $x$ to the two components of $\partial U_1$ by shortest paths, the sum of the two path lengths is at least $L$, and we can find such an $x'$ on their union. Since $d(y,\partial W) \geq L$, there is some $y' \in W$ with $f(y')=f(x')$ and $d(y,y') \leq L/2$.

%and e can assume that $d(y, \partial W)\geq D_2$  for some to be determined $D_2=D_2(D,\epsilon,g)$. The $D_2$-ball around $ f(y)\in \bar M$  is then contained in the product region $f(W)$. Since each component of $\partial f(V)$  has diameter  bounded in terms of $\epsilon,g$ and $f(V)$ is very wide, if $D_2=D_2(D,\epsilon,g)$  is large we can assume that there is some point $x'\in V$ with $$f(x') \in B , \ \ d(x',\partial V)\geq D_1.$$ Then we will have $f(x')=f(y')$  for some $y' \in W$.  Note that $$d(x,x'),d(y,y')\leq D_2.$$ So, if $L$  is  much larger than $D_1,D_2$,  we  can also assume that $$x\neq y' \ \text{ and } \ d(y',\partial W) \geq D_1$$ Since $f$  is an embedding on $V,W$, we then have $$x'\in U_1, \ \ y' \in U_3 \cup U_4.$$

As long as $L/2$ is larger than the constant $D$ in Lemma \ref{intersectprs}, the intersection of the two product regions $f(V),f(W) \subset \bar M$  contains a $D$-ball around $f(x')=f(y')$, so from the lemma we get that $f(V),f(W) $ share a level surface $S$ that lies in that $D$-ball. Lifting, we obtain level surfaces $S_x,S_y$ of $V$ and $W$ that lie in $D$-balls  around $x',y'$,  respectively, and that both project to $S$. If $L$ is large enough, the surface $S_x$ lies in $V$,   at least $D$ away from $U_3 \cup U_4$, while $S_y$  lies $W$, in the $D$-neighborhood of $y'\in U_3 \cup U_4$. So, $S_x $ and $ S_y$ are disjoint, and bound a topological product region $I \subset U_1 \cup \cdots \cup U_4$.

The restriction $f|_I$ maps both boundary components to $S$, which is $2$-sided in $\bar M$, given our standard orientability assumption. Moreover, $f(int (I)) $ does not intersect $S$, since $f$ is injective on $V$ and on $W$. Let $N$ be the connected manifold obtained by cutting $\bar M$ along $S$, and taking the connected component that contains $f(I)$. (Note that the two `ends' of $f|_I$ may a priori limit onto the same side of $S$, so if $S$ is separating we cut along $S$ and then take the component that lies to `that' side.) Then the restriction $f|_I$ extends to a covering map $I \longrightarrow N$. It follows that $N $ has incompressible boundary, and $f|_I$ is a nontrivial homotopy between two (possibly identitical) components $\partial N$. So Waldhausen's Cobordism Theorem implies that $N$ is an $I$-bundle. It follows that $\bar M \setminus f(int(V))$, say, is an $I$-bundle, as in the statement of the lemma. \end {proof}

The following is a weaker version of Theorem \ref{PR covering theorem} in which $V$ has bounded width. 

\begin{lem}[Propagating the embedding]\label {propagating}
	Given $L_+$, there is some $L'=L'(g,\epsilon,L_+) $ as follows. Let $M$ be a complete hyperbolic $3$-manifold and let $U,V,W\subset M $ be  compact, $\epsilon$-thick, genus $g $ product regions with disjoint interiors,  which are adjacent, in the sense that each shares a boundary component with the next. Let $$f : M \longrightarrow \bar M $$ be a covering map,  and assume that $f|_U$  is an embedding, that $V$  has width at most $L_+$, and that $U,W$ have width at least $L'$. Then either 
\begin{enumerate}
	\item $f|_{U \cup V}$  is an embedding, or  
\item some component of $\bar M \setminus f(int(U))$  is a (compact, possibly twisted) interval bundle.
\end{enumerate}\end{lem}

So, if $f$ is an embedding on $U$, either it extends to an embedding on the slightly larger product region $U\cup V$ or we see an interval bundle in $\bar M$.

\begin{proof}
Fix $L_+$, and	assume that this is not the case, for any $L'$.  Then there is a sequence of counterexamples,  by which we mean $$f_i : M_i \longrightarrow \bar M_i, \ \ U_i,V_i,W_i\subset M_i, \ \ \width(U_i),\width(W_i) \to  \infty,$$
such that  for all $i$,  the product regions $V_i$ have width  at most $L_+$, the restriction $f_i|_{U_i}$  is an embedding, but the restriction $f_i|_{U_i \cup V_i}$ is not, and where no component of $\bar M_i \setminus f_i(int(U_i))$ is an  interval bundle.  By Corollary \ref{cutting-product}, for sufficiently large $i$ we may also pick subproduct regions $U_i' \subset U_i$  and $W_i' \subset W_i$,  both adjacent to $V_i$,  that have width in the interval $[L,2L]$,  where $L=L(\epsilon,g)$  is a constant at least as large as the constant in Lemma \ref{overlapping embeddings}.

 Pick base points $p_i \in V_i $ for $M_i$.
Since all our manifolds are $\epsilon $-thick, after passing to a subsequence, we can assume by Arzela--Ascoli that $(f_i)$  converges uniformly on compact sets to a covering $$f : (M,p) \longrightarrow (\bar M,f(p)),$$ where $(M,p)$ and $(\bar M,f(p))$ are geometric limits of  the  sequences of pointed manifolds $(M_i,p_i)$ and $(\bar M_i,f_i(p))$, respectively. If $(\phi_i)$ and $(\bar \phi_i)$ are sequences of almost isometric maps coming from these two geometrically convergent sequences, as in Definition \ref{geolimits}, this means that on every  fixed compact subset of $M$, the composition $\bar \phi_i^{-1} \circ f_i \circ \phi_i$  is defined  for large $i$ and  converges uniformly to $ f$ as $i\to \infty$.
By Lemma~\ref{limit-double}, we may assume that $M$  is a doubly degenerate hyperbolic $3$-manifold homeomorphic to $\Sigma_g \times \BR$, and that both
\begin{enumerate}
	\item[(a)] for any level surface $S \subset M$,  the image $\phi_i(S)$  is a level surface in $U_i$  for all large $i$, 
\item[(b)] the image of $\phi_i$  contains $V_i$  for large $i$, and $\phi_i^{-1}(V_i)$  is a topological subproduct region of $M$.
\end{enumerate}  
Here, to get (b) above from Lemma~\ref{limit-double} (b), we are using the  fact that \begin{equation}
 	\sup_i \diam(V_i) < \infty, \label {sssssup} \end{equation}  which follows from Fact \ref{width diameter} since  each $V_i$  is $\epsilon $-thick  and has width at most $L_+$,  which is independent of $i$.  There is a similar diameter bound for the product regions $U_i'$, so we may also assume after excluding  finitely many $i$ that $U_i'$ lies in the image of $ \phi_i$  for every $i$.

%As long as $L_-$ is at  least some constant depending only on $\epsilon,g$,   Lemma \ref{level surfaces}  allows us to choose in each $W_i$ a level surface $S_i$  such that 
%\begin {equation}
%\label {dsw}d(S_i,\partial W_i) \geq 1.
%\end {equation}
   Since  $\bar \phi_i^{-1} \circ f_i \circ \phi_i \to f$, for large $i$  there is a homotopy \begin{equation}\label {homotp} f \circ \phi_i^{-1}|_{U_i'} \sim \bar \phi_i^{-1} \circ f_i|_{U_i'}\end {equation}  %via  a homotopy in $\bar M$  that has arbitrarily short tracks.  
By (b) above, $\phi_i^{-1}|_{U_i'}$  is a  topological subproduct region in $M$.    Moreover, $\bar \phi_i^{-1} \circ f_i|_{U_i'}$  is an embedding, and by  \eqref{homotp}, its image $$\bar U:=\bar \phi_i^{-1} \circ f_i(U_i') \subset \bar M$$ lifts homeomorphically to a topological subproduct region of $M$.

% lifting the homotopy and composing with $\phi_i$, we obtain an embedded surface $S_i' \subset M_i$  that is homotopic to $S_i$ via a homotopy with tracks  of length less than $1$,  such that $f \circ \phi_i^{-1}|_{S_i'}$  is an embedding. This $S_i'$ lies  in the $1$-neighborhood of $S_i$,  and so is a level surface in $W_i$ by \eqref{dsw}. By (b) above, $\phi_i^{-1}(S_i')$  is a level surface in $M$. 

Thurston's Covering Theorem (see \cite{Canarycovering}) then implies that either  $f$  is a homeomorphism, or that some component $$I \subset \bar M \setminus int(\bar U)$$ is a (compact, possibly twisted) interval bundle.  Suppose for a moment that we are in the second case.  Then when $i$  is large, the image $\bar \phi_i(I)$ is defined, and  is an  embedded interval bundle in $\bar M_i$  whose boundary is a subset of $\partial f_i(U_i)$. Hence, $\bar \phi_i(I)$  is a component of $\bar M_i \setminus f_i(int(U_i'))$,  which is a contradiction since $U_i'$  is a subproduct region of $U_i$ and we assumed no component of $\bar M_i \setminus f_i(int(U_i))$ was an interval bundle.  It follows that $f$  is a homeomorphism. 

Let $\CN_i$  be the $1$-neighborhood of $ U_i' \cup V_i \cup W_i' \subset M_i$.  
If $i$  is large,  the composition $\bar g_i :=\bar \phi_i \circ f\circ\phi_i^{-1}$  is defined on $\CN_i$, and since $f$  is a homeomorphism, this $\bar g_i$ will be an embedding of $\CN_i$ into $\bar M_i$.  And since $\bar \phi_i^{-1} \circ f_i \circ \phi_i \to f$,  we know that  for large $i$,   the maps  $$f_i,\bar g_i : {\CN_i} \longrightarrow \bar M_i$$ are homotopic via a homotopy with tracks of length less than $1$. Lifting this homotopy to $M_i$, we get a homotopy  between two maps $$\iota,g_i : \CN_i \hookrightarrow M,$$ where $\iota$  is the inclusion and where $f_i \circ g_i = \bar g_i$. Since this homotopy also  has tracks of length less than $1$,  we have $g_i(\CN_i) \supset U_i' \cup V_i \cup W_i' .$ So as $ \bar g_i$ is an embedding on $\CN_i $, the map $f_i$ is an embedding on $U_i' \cup V_i \cup W_i .$

 We now know that $f_i$  is an embedding on both $U_i$  and $U_i' \cup V_i \cup W_i' .$ So,  applying Lemma \ref{overlapping embeddings} to  the four  adjacent  product regions
$$U_i\setminus int(U_i'), U_i', V_i,W_i',$$
 we see that either  some component of $\bar M_i \setminus f_i(int(U_i))$ is a compact interval bundle or that $f_i$  restricts to an embedding on $U_i \cup V_i$.
\end{proof}

 We are now ready to prove the main result of the subsection.  %We would like to let the reader know that we initially wrote a version of the following that introduced all the relevant constants explicitly at the beginning of the proof, but the result of this was that the first two paragraphs were basically unreadable. We hope that the reader will appreciate the more informal approach adopted below.

\begin{proof}[Proof of Theorem \ref{PR covering theorem}]
 It suffices to prove Theorem \ref{PR covering theorem} when  the product regions $U,V$ are compact. Any noncompact $U$ can be exhausted by compact product regions using  Corollary \ref{truncating noncompact}, and  the reader  can use this to deduce the general case from the compact case.

Using Corollary  \ref{cutting-product}, write $U, V,W$ as a union of subproduct regions  with disjoint interiors $$  U = V_0, \ \ V = V_{1} \cup  \cdots \cup V_{n}, \ \ W =V_{n+1},$$
where each $V_i$  shares a boundary component  with $V_{i+1}$, and where each $V_i, i \in [1,n]$ has width  at most some $L_+=L_+(\epsilon,g)$.  Here, in the language of Corollary  \ref{cutting-product}, $L_+$ is `$2D$', and $n$   is determined by the width of $V$. Take the lower bound $L$ for the widths of $U=V_0$ and $W=V_{n+1}$ discussed in the statement of the theorem to be larger than the constant obtained by applying Claim \ref{propagating} with the inputs $L_+,\epsilon,g$.

Pick  some maximal $k\geq 0$  such that $f$  restricts to an embedding on 
$$V^{k} := V_{0} \cup \cdots \cup V_{k}$$
 We may assume that $k<n$, since otherwise we are done. Apply Lemma~\ref{propagating}  to the three adjacent product regions $$V^{k}, \ \ V_{k+1}, \ \ V_{k+2} \cup \cdots \cup V_{n+1},$$
 noting that $V_{k+1} $  has width   at most $L_+$ and  the other two have width at least $L=L(L_+,\epsilon,g)$.
If some component of $\bar M \setminus f(int(V^k))$ is a compact interval bundle,  it follows from the fact that $U \subset V^k$  that some component of some component of $\bar M \setminus f(int(U))$  is a compact interval bundle, so we are done.  Therefore, $f$  restricts to an embedding on $V^k \cup V_{k+1}$,  violating the maximality of $k$.
\end{proof}
 
\subsection{Wide product regions are unknotted}\label {wide product regions are unknotted}
If a surface $S$ is embedded in a $3$-manifold, the existence of a wide product neighborhood puts strong topological constraints on the embedding. 

A first result in this direction is the following. See also the subsequent Corollary \ref{unknotted}, and especially Theorem \ref{double compression body theorem} in \S \ref{Double compression bodies}.

\begin{prop}[Surfaces compress in complement, if at all]\label{incompressible-incompressible}
There is some $L=L(g,\epsilon,\kappa^+)$  as follows. Assume that $ \mathcal U \subset \mathcal V$ are collections of pairwise disjoint product regions in $M$, where each product region has $\epsilon$-width  at least $L$ and genus at most $g$.

 Let $\Sigma \subset M\setminus \cup_{V\in \mathcal V} V$ be an embedded, closed orientable surface of genus at most $g$. If $\Sigma$ is incompressible in $M\setminus  \cup_{V\in \mathcal V} V$ and  compressible in $M \setminus \cup_{U\in \mathcal U} U$, then $\Sigma$ is  homotopic in  $M\setminus  \cup_{V\in \mathcal V} int(V)$ to a boundary component of some $V \in \mathcal V \setminus \mathcal U$.\end{prop}

Before starting on the proof, we record the following useful lemma.

\begin{lem}\label{its a level surface lem}
Suppose that $M$ is an orientable, irreducible $3$-manifold, $V \subset M$ is a topological product region, $S$ is a closed orientable surface, and $f : S \longrightarrow V$ is a $\pi_1$-injective map that is homotopic in $M$ to an embedding. Then $f$ is homotopic in $V$ to a level surface.
\end{lem}
\begin{proof}
Since $f$ is $\pi_1$-injective in $M$, and hence in $V$, it is homotopic in $V$ to a finite cover of a level surface. From this, we get that the inclusion $V \hookrightarrow M$ is $\pi_1$-injective: indeed, the kernel of the map $\pi_1 V \longrightarrow \pi_1 M$ intersects $f_*(\pi_1 S) \subset \pi_1 V$ trivially, but the image has finite index, and $\pi_1 V$ is a surface group, so it has no nontrivial finite subgroups.

Since $f$ is homotopic to an embedding in $M$, work of Freedman--Hass--Scott (see Lemma \ref{getembedded}) implies $f$ is homotopic in $M$ to an embedded surface $\Sigma \subset V$, which must be a level surface of $V$. And then Lemma \ref{pushinghomotopieslem} implies $f$ is homotopic to $\Sigma$ within $V$.
\end{proof}

\begin{proof}[Proof of Proposition \ref{incompressible-incompressible}]
 In the proof below, a number is `bounded' if it is less than some constant depending on $g,\epsilon, \kappa^+$, and is `huge' if it is much bigger than any of the other bounded constants in the proof.

 Suppose that the $\epsilon$-width of each  product region in $\mathcal V$ is huge and that $\Sigma\subset M\setminus \cup_{V\in \mathcal V} V$ is an  incompressible orientable surface with genus at most $g$ that is compressible in $M\setminus \cup_{U\in \mathcal U} U$ . 

By Corollary \ref{cutting-product},  we can  decompose each $V\in \CV \setminus \CU$  as a union 
$$V=V_1\cup V_2$$
of product regions\footnote{Hopefully the notation is not confusing: there is a product region $V_1$ inside each $V$, so really $( \cdot )_1$ is a function with input $V$.}, each with huge $\epsilon$-width.  Apply Theorem~\ref{links} to all the product regions $U \in \mathcal U$ and all $V_1,V_2$, where $V \in \CV \setminus \CU$, to get $0.025$-separated, bounded length geodesic links $\Lambda_U \in U$ and $\Lambda_{V,i} \subset V_i$ such that for each $i=1,2$, we have
\begin{enumerate}
	\item the link $\Lambda_{V,i}$ has huge $\epsilon $-distance from $\partial V_i$,	
\item the surface $\Sigma$ is incompressible in $$M_i := M \setminus (\bigcup_{V\in \CV \setminus \CU}\Lambda_{V,i} \cup  \bigcup_{U\in \mathcal U}  \Lambda_U ).$$
\item if $\Sigma$ is  homotopic in $M_i$ to a component $S \subset \partial V$, it is homotopic to $S$ within $M\setminus \cup_{V \in \CV} int(V)$.
\end{enumerate}
Here, property (2) comes from Corollary \ref{linkkor} and the fact that $\Sigma$  is incompressible in $M \setminus \cup_{V \in \mathcal V} V$. Property (3) comes from Fact \ref{homotopyfact}. Note that each of the links $\Lambda_{V,i}$ above is really a union of all 3 of the links mentioned in Theorem~\ref{links}.  
 
By (2) above, $\Sigma$ is incompressible in the manifold $$M_{1,2} := M \setminus \Big (\bigcup_{\substack{V\in \CV \setminus \CU,\\ i\in \{1,2\}}}\Lambda_{V,i} \ \cup \  \bigcup_{U\in \mathcal U}  \Lambda_U \Big ),$$ which we equip with a metric of pinched negative curvature as in Lemma \ref{neg-metric}. Since $\Sigma$ is not incompressible in $M \setminus \bigcup_{U\in \mathcal U}  \Lambda_U$, Corollary~\ref{shrinkwrapping2} implies that $\Sigma \hookrightarrow M_{1,2}$ is homotopic within $M_{1,2}$ to an SRS $$f: \Sigma \longrightarrow M_{1,2}$$ whose entire image lies at a bounded  $\epsilon $-distance from one of the links $ \Lambda_{V,i}$. Assume $i=1$ for concreteness, and call the product region containing the link $W \in \CV$, so that $f(\Sigma)$ lies in $W_1$, at bounded $\epsilon$-distance from $\Lambda_{W,1}$. 

By property (2) above, $f$ is $\pi_1$-injective in $M_2$. Since it is homotopic to the embedded surface $\Sigma$, Lemma \ref{its a level surface lem} implies that $f$ is homotopic within $W_1 \subset M_2$ to a level surface of $W_1$. Concatenate the homotopy from $\Sigma$ to $f$ in $M_{2}$ with a homotopy through $W_1$ to the shared boundary component $S =\partial W \cap \partial W_1$, to get a homotopy from $\Sigma$ to $S$ in the complement of $\Lambda_{W,2}$. Applying (3) above, we get that $\Sigma$ is homotopic to $S$ in the complement of the union of $int(W_2)$ and all $int(V)\in \CV \setminus \{W\}$. But then $\Sigma$ is also homotopic to $S$ within the complement of all $int(V), V\in \CV$, as desired.
\end{proof}

 One can interpret Proposition \ref{incompressible-incompressible}  as saying that collections of product regions are `unknotted' with respect to embedded surfaces with bounded genus in their complements. Here is a corollary that  will help to better  explain  this intuition. 
 
\begin{kor}[WPRs in compression bodies are unknotted]\label {unknotted}  There is some $L=L(\epsilon,g,\kappa^+)$  as follows. Suppose that $M$ has pinched negative curvature $\kappa \in [\kappa^-,\kappa^+]$, no cusps, and finitely generated fundamental group, so that by the Tameness Theorem we can write $M=int(\bar M)$ for some compact $3$-manifold $\bar M$. 

Suppose that all components of $\partial \bar M$ have genus at most $g$, that $\mathcal U$  is a (possibly empty) collection of product regions in $M $, each with genus at most $g$ and $\epsilon$-width at least $L$, and that  there is a connected component $$C \subset \bar M \setminus \cup_{U \in \mathcal U} int(U)$$  that is a compression body.   Suppose that $V\subset int(C)$ is a compact product region with genus at most $g$ and $\epsilon$-width at least $L$. Then $C\setminus int (V)$  is the union of two compression bodies $C_1,C_2$, where $\partial_{e} C_1=\partial_{e} C$, one boundary component of $V$ is an interior boundary component of $C_1$, and the other boundary component of $V$ is $\partial_{e} C_2$.
\end{kor}

See \S \ref{compression bodies} for the definition of a compression body. When $C $ is a compression body, a compression body $D \subset C$ such that $\partial_e D=\partial_e C$ is called a \emph {subcompression body} of $C$. So, the conclusion above is that $V$ is a product neighborhood of an interior boundary component of some subcompressionbody of $C$.  The corollary has content even if $\mathcal U$ is empty, as long as $M $ is a compression body.  For  example, if $S$ is the boundary of a knotted compact core in a handlebody $M$, then $S$ cannot be a level surface in a wide product neighborhood.

\begin {proof}
Let $\partial_{e} C$  be the exterior boundary of $C$, and let $\Sigma_- \subset C$ be an interior boundary component of the characteristic compression body of $\partial_{e} C$ within $C \setminus int(V)$. By Proposition \ref{incompressible-incompressible}, if $L$ is large then either 
\begin {enumerate}
\item 	$\Sigma_-$ is isotopic in $C \setminus int(V)$ to a boundary component of $V$, or
\item  $\Sigma_-$ is incompressible in $C$, and hence is isotopic to an interior boundary component $\Sigma_+$ of the compression body $C$. 
\end {enumerate}

In the first case, $V$  is a product neighborhood of an interior boundary component of a subcompressionbody of $C$, and hence $C\setminus int(V)$  is a union of two (possibly trivial) compression bodies of the form desired, say by Corollary 2.4 of \cite{biringer2017automorphisms}.
So we can assume that for every $\Sigma_-$, we are in the second case. Then for some $\Sigma_-$ we have that $\Sigma_-$ and $\Sigma_+$ bound  a topological product region $[\Sigma_-,\Sigma_+]$ with $V \subset [\Sigma_-,\Sigma_+] \subset C$.  We need to show $V$ is a  topological subproduct region of $[\Sigma_-,\Sigma_+]$. 

If $L$ is large, then Corollary \ref{cutting-product} says each $U\in \CU$ is a union $$U=U_1\cup U_2 \cup U_3$$ of adjacent subproduct regions, each with $\epsilon$-width at least $L/4$. Theorem \ref{links} says that for each $U \in \CU$, there is some geodesic link $\Lambda_U \subset U_2$ such that if $\Lambda := \cup_{U\in \CU} \Lambda_U$, the manifold $M_\Lambda := M \setminus \Lambda$ admits a metric with uniformly pinched negative curvature and hyperbolic cusps, that agrees with the metric on $M$ outside a small neighborhood $\BT_\Lambda \supset \Lambda$, and where the inclusion $C \hookrightarrow M_\Lambda$ is $\pi_1$-injective.

Just to simplify notation, let $\Sigma$ be a level surface of $[\Sigma_-,\Sigma_+]$. Let $N \longrightarrow M_\Lambda$ be the cover corresponding to $\pi_1 \Sigma$. Since $N$ has pinched negative curvature and hyperbolic cusps, it is tame by Agol's version of the Tameness Theorem \cite{Agoltameness}. As $\Sigma  \hookrightarrow C$ and $C \hookrightarrow M_\Lambda$ are $\pi_1$-injective, so is $\Sigma \hookrightarrow M_\Lambda$, so we have $N\cong\Sigma  \times \BR$ by Lemma \ref{unique-compact}. We'd like to lift $[\Sigma_-,\Sigma_+]$ to $N$, but technically, $\Sigma_+$ may be a component of $\partial \bar M$. If this is the case, the end of $M_\Lambda$ adjacent to $\Sigma_+$ lifts homeomorphically to $N$, and can then be compactified, giving a bordification $\bar N \supset N$ such that the covering map $N \longrightarrow M_\Lambda$ extends continuously to a homeomorphism from $\partial \bar N$ to $\Sigma_+$. For consistency in notation, when $\Sigma_+$ is not a component of $\partial \bar M$, we set $\bar N := N$. In both cases, then, we can lift the topological product regions $V \subset [\Sigma_-,\Sigma_+]$ to $\hat V \subset [\hat \Sigma_-,\hat \Sigma_+] \subset \bar N$. 

Our goal is to show that there is a map $$f: \Sigma \longrightarrow N \cong \Sigma\times \BR$$ in the homotopy class of a level surface, whose image is contained in $\hat V$. Then $f$ will be homotopic within $\hat V$ to an level surface of $\hat V$ by Lemma~\ref{its a level surface lem}, implying that $\hat V \subset [\hat \Sigma_-,\hat \Sigma_+] \subset \bar N$ are topological subproduct regions, and therefore so is $V \subset [\Sigma_-,\Sigma_+]$.

The map $f$ will be a minimizing SRS produced by the Interpolation Theorem, see \S \ref{sec:interpolation}.   Take $L$ large and first use Corollary \ref{cutting-product} to write $$\hat V = \hat V_1 \cup \hat V_2 \cup \hat V_3$$ as a union of three adjacent subproduct regions, each with $\epsilon$-width at least $L/4.$ To construct $f$, it suffices to find a pair of disjoint level surfaces $S_\pm\subset N \cong \Sigma \times \BR$ such that 
\begin{enumerate}
\item $\hat V_2$ is contained in the interior of the compact submanifold $[S_-,S_+] \subset N$ bounded by $S_-$ and $S_+$, and
\item $S_\pm$ are both homotopic in $N \setminus \hat V_2$ to minimizing SRSs.
\end{enumerate}
Indeed, assuming that we have all this, suppose $L$ is large and take a point $p \in \hat V_2$ with $d_\epsilon(p,\partial \hat V_2) \geq L/8$. Construct a homotopy $(f_t)$ that is a concatenation of a homotopy in $N \setminus \hat V_2$ from $S_-$ to a minimizing SRS, a homotopy within an $R$-neighborhood, a homotopy through minimizing SRSs, a homotopy within an $R$-neighborhood, and a homotopy from the resulting minimizing SRS to $S_+$ in $N \setminus \hat V_2$. Here, the middle three homotopies come from the Interpolation Theorem. The image of $(f_t)$ must contain $p$: if not, we can retract it to a degree zero homotopy equivalence $\Sigma \times [0,1] \longrightarrow [S_-,S_+]$ that is a homeomorphism on the boundary, a contradiction. If $L$ is large relative to $R$, the only part of the homotopy $(f_t)$ that can pass through $p$ is the homotopy through minimizing SRSs given by the Interpolation Theorem. So, we can take $f=f_t$, where $f_t(\Sigma) \ni p$.

\medskip

It remains to construct the surfaces $S_\pm$. Let's fix an identification $N\cong \Sigma \times \BR$ and refer to an end of $N$ as \emph{positive} or \emph{negative}, according to the sign of the second coordinate in a neighborhood of the end. We can assume that the surface $\hat \Sigma_-$ faces the negative end of $N$, and that $\hat \Sigma_+$ either lies in $N$ and faces the positive end of $N$, or lies in $\partial \bar N$ and compactifies the positive end. Note that to ensure that $\hat V_2 \subset [S_-,S_+]$, it suffices to construct each $S_\pm$ so that it  separates $\hat V_2 $ from the $\pm$-end of $N$. That is, we can work on constructing $S_-,S_+$ separately.

\medskip

Let's first work on constructing $S_+$. The surface $\Sigma_+$ is an interior boundary component of $C$, so either it is a component of $\partial \bar M $, or a boundary component of some $U\in \CU$. In the first case, $\Sigma_+$ bounds an end of $M$, and since $M$ has no cusps, this end is either convex cocompact or degenerate. So, there are really three cases to consider: the convex cocompact case, the degenerate case, and the product region case. Note that when $\Sigma_+ \subset \partial \bar M$, then the lift $\hat \Sigma_+$ is a component of $\bar N$ that compactifies the positive end of $N$.

 First, suppose that $\Sigma_+ \subset \partial \bar M$ bounds a convex cocompact end of $M$, so that the positive end of $N$ is convex cocompact. Then we can just set $S_+ \subset \partial \CN_1(CC(N))$ to be the boundary component of the $1$-neighborhood of the convex core that faces $\hat \Sigma_+$. If $L$ is large, then Corollary \ref{prs in the convex core} says that a large $d_\epsilon$-neighborhood of $\hat V_2$ is contained in $ CC(N)$. In particular, $S_+$ separates $\hat V_2$ from the positive end of $N$. Corollary \ref{kor SRSs near partial} says that $S_+$ is homotopic to a minimizing SRS via a homotopy of bounded $\epsilon$-diameter. Since a large neighborhood of $\hat V_2$ is contained in $ CC(N)$, the image of this homotopy lies outside $\hat V_2$.

Next, suppose $\Sigma_+ \subset \partial \bar M$ bounds a degenerate end of $M$, so that the positive end of $N$ is degenerate. By Theorem \ref{filling theorem}, there is a minimizing SRS $f_1 : \Sigma \longrightarrow N$ with image arbitrarily far out the positive end of $N$. Use Freedman--Haas--Scott (see Lemma \ref{getembedded}) to homotope $f_1$ to a level surface $S_1$ contained in the $1$-neighborhood of its image. We can assume that $S_1 $ separates $\hat V_2$ from the positive end of $N$. Use Theorem \ref{filling theorem} again, producing a minimizing SRS $f_2 : \Sigma \longrightarrow N$ and a level surface $S_2 \subset \CN_1(f_2(\Sigma))$ even further out the end, so that both $f_2(\Sigma)$ and $S_2$ lie in the component $(S_1,\infty) \subset N$ of $N\setminus S_1$ that is a neighborhood of the positive end. Then $f_2$ and $S_2$ are homotopic in $(S_1,\infty)$, and hence outside of $\hat V_2$. So, we can set $S_+=S_2$.

Finally, suppose that $\Sigma_+ \subset M$, and is a boundary component of a product region $U\in \CU$. Recalling that $U=U_1 \cup U_2 \cup U_3$, where $\Lambda_U \subset U_2$, we may assume $\Sigma_+ \subset \partial U_1$. Then we can lift $U_1 \subset M_\Lambda$ to a product region $\hat U_1 \subset N$ that is adjacent to $[\hat \Sigma_-,\hat \Sigma_+]$ along $\hat \Sigma_+$. If $L$ is large, we can use Lemma \ref{minsrsinpr} to find a minimizing SRS $g : \Sigma \longrightarrow \hat U_1$, whose image has large $\epsilon$-distance to $\partial \hat U_1$. Work of Freedman--Haas--Scott (see Lemma \ref{getembedded}) gives a level surface $S_+$ contained in a $1$-neighborhood of $g(\Sigma)$. Since $g,S_+$ are homotopic within $\hat U_1 \subset N \setminus \hat V_2$, we are done.

\medskip

We now work on constructing $S_-$. If the negative end of $N$ has no cusps, it is either convex-cocompact or degenerate. Then we can follow the construction of $S_+$ above, and set $S_-$ either to be the corresponding boundary component of $\CN_1(CC(N))$, or a level surface obtained via Freedman-Haas-Scott from a SRS far out the degenerate end. So, we can assume from now on that the negative end of $N$ has cusps. 
	
By Fact \ref{simplecusp} and Corollary \ref{constructing prs}, there is a minimizing SRS $h: \Sigma \longrightarrow N$ in the homotopy class of a level surface whose image intersects deeply some cusp of the negative end of $N$. We can assume that $h(\Sigma)$ intersects the negative end neighborhood $(-\infty, \hat \Sigma_-] \subset N\times \BR$ bounded by $\Sigma_-$. If $\pi : N \longrightarrow M_\Lambda$ is the projection map, then $\pi \circ h(\Sigma)$ intersects an $\epsilon$-cusp neighborhood of $M_\Lambda$. Since $M$ has no cusps, $\pi \circ h(\Sigma)$ intersects the small neighborhood $(\BT_{\Lambda_U}\setminus \Lambda_U) \subset M_\Lambda$ of some link $\Lambda_U \subset \Lambda$, where $U \in \CU$. Since $U=U_1 \cup U_2 \cup U_3$, with $\Lambda_U\subset U_2$, and $U_1,U_3$ have large $\epsilon$-width, we get that a large $d_\epsilon$-neighborhood of $\pi \circ h(\Sigma)$ is contained in $U\setminus \Lambda \subset M_\Lambda$. 

Since $ \Sigma_-$ lies outside $U$, we then have that a large $d_\epsilon$-neighborhood of $h(\Sigma) \subset N$ is disjoint from $\hat \Sigma_-=\emptyset$, and hence lies in $ (-\infty, \hat \Sigma_-).$ In particular, if we use Freedman-Haas-Scott (c.f.\ Lemma \ref{getembedded}) to produce a level surface $S_-$ in a $1$-neighborhood of $h$, then $S_-$ lies in $ (-\infty, \hat \Sigma_-)$ as well, and hence separates $\hat V_2$ from the negative end of $N$. Moreover, $S_-,h$ are homotopic in $ (-\infty, \hat \Sigma_-) \subset N \setminus \hat V_2$, so we're done.\end {proof}

Here is a strengthened version of the above for multiple  product regions lying in a compression body complementary component.

\begin{kor}\label {unknotted2} 	 There is some $L=L(\epsilon,g,\kappa^+)$  as follows. Write $M$  as the interior of a compact manifold $\bar M$, suppose $\mathcal U$ is a collection of product regions in $M $, each with genus at most $g$ and $\epsilon$-width at least $L$, and that  there is a connected component $$C \subset \bar M \setminus \cup_{U \in \mathcal U} int(U)$$  that is a compression body  with exterior boundary $\partial_{e} C$.
Suppose $\mathcal V$  is a finite collection of compact product regions in $int(C)$, each with genus at most $g$ and $\epsilon$-width at least $L$. Then  every component $$N \subset C \setminus \cup_{V \in \mathcal V} int(V)$$  is a compression body.
Moreover, create a graph $\mathcal T$ whose vertices are these $N$, whose edges are product regions $V \in \mathcal V$, and where for a  product region $V $, each of the two components $S \subset \partial V$ determines an adjacency of $V$ with the vertex $N$ with $S \subset \partial N$.

Then $\mathcal T$ is a tree. Root this tree at the component $$N^c \subset C \setminus \cup_{V \in \mathcal V} int(V)$$ whose boundary contains $\partial_{e} C$. Then for every  vertex $N \neq N^c$ that is a nontrivial compression body, the exterior boundary $\partial_{e} N$ faces the product region $V \in \mathcal V$   that connects $N$ to its parent.  Moreover,   there are at most $2g-2$ vertices $N$  that are nontrivial compression bodies.\end{kor}

Here, the \emph{parent} of a non-root vertex $N$ in a rooted tree is the unique vertex adjacent to $N$ that is closer to the root.  

\begin {proof}
Everything except the bound on the number of  nontrivial vertices follows via induction on $|\mathcal V|$, where the base case $\mathcal V=\emptyset$ is trivial.  Namely, if $\mathcal V$ is as above, fix a product region $V_0\in \mathcal V$. By induction, the    claim is true for $\mathcal V_{old}:=\mathcal V \setminus V_0$, so the components $$N \subset C \setminus \bigcup_{V \in \mathcal V_{old}} int(V)$$
are compression bodies that fit together into a tree $\mathcal T_{old}$ satisfying the properties in the statement of the corollary.  In particular, the  component $N_0$
containing $V_0$  is a compression body, and either $\partial_{e} C  = \partial_{e} N_0$, in which $N_0$ is the root of $\mathcal T_{old}$ , or the exterior boundary $\partial_{e} N_0$ faces the product region connecting $N_0$ to its parent.

Let $\mathcal T$  be the graph whose vertices are the components of $$N \subset C \setminus \bigcup_{V \in \mathcal V} int(V).$$
If $L$ is large, then Corollary \ref{unknotted}  implies that $N_0 \setminus int(V_0)$  is a union of two compression bodies $N_1,N_2$, where $\partial_{e} N_1 = \partial_{e} N_0$ and $\partial_{e} N_2$ is a boundary component of $V_0$. In particular, $V_0$  separates $N_0$, so $\mathcal T$  is obtained from $\mathcal T_{old}$ by replacing $N_0$ by an edge and partitioning its adjacent vertices between the two vertices of this edge, where the partition corresponds to the partition of the components of $\partial N_0$ induced by the decomposition $\partial N_0 = \partial N_1 \cup \partial N_2 \setminus \partial V_0$. So, $\mathcal T$ is a tree, and  all of its vertices are  compression bodies.   

If $N_0$  is the root of $\mathcal T_{old}$, then $N_1$ is the root of $\mathcal T$. Otherwise,  the exterior boundary $\partial N_1 = \partial N_0$ faces  the parent of $N_0$ in $\mathcal T_{old}$, which is also the parent of $N_1$ in $\mathcal T$, by the description of $\mathcal T$ in the previous paragraph. Similarly, the parent of $N_2$ is $N_1$, and the exterior boundary $\partial N_2$ faces $N_1$. 

 It remains to verify the bound on the number of vertices that are nontrivial compression bodies. Enumerate the  vertices of $\mathcal T$ as $N_0,\ldots,N_m$ in such a way that the index of a parent is always less than the index of its child, and for $i=0,\ldots,m$ let $C_i\subset C$  be the submanifold that is the union of all vertices and edges in the subgraph of $\mathcal T$   spanned by $N_0,\ldots,N_i$. So, $C_i$ is obtained from $C_{i-1}$ by gluing some product region $V \in \mathcal V$ onto a  component of $\partial C_{i-1}$, and then attaching the exterior boundary $\partial_{e} N_i$ to the free boundary component of $V$.

No $N_i$ is a solid torus, since it shares its exterior boundary with a product region, which by definition has higher genus. So, 
$$\chi(\partial_{int} N_i) > \chi(\partial_{e} N_i).$$
for every \emph {nontrivial} $N_i$, by Lemma \ref{euler decrease}. So, for all such $i$,  we have $$\chi(\partial C_i) > \chi(\partial C_{i-1}).$$  Since $C_0$ is a compression body with exterior boundary of genus at most $g$, we have $\chi(\partial C_0)\geq 2-2g$. So as $\chi(C)\geq 0$, the number of  $N_i$ that are nontrivial compression bodies is at most $2g-2$.  
\end {proof}

%
% Here is a similar result for twisted  interval bundles.
%
% \begin {lem}[Twisted interval bundles]
% Given $g,\epsilon$, there is some $L=L(g,\epsilon)$  as follows. Suppose that $U$ is a genus $g$ product region in $M$ with $\epsilon $-width at least $L $, and that $V$  is a component of $M \setminus int(U)$  that is a twisted interval bundle over a non-orientable surface. Then  every point of $V $ is in the image of a simplicial  ruled surface $f : S \longrightarrow V \cup U$  in the homotopy class of a level surface of $U$.
% \end {lem}
%\begin {proof}
%As in the previous proof, write $U=U^1\cup U^2$ where $ U^1$ is adjacent to $V $, and let $\Lambda $ be a $0.025$-separated geodesic link	in $U^2$  such that the inclusion $V \hookrightarrow M \setminus \Lambda$  is $\pi_1$-injective.  If $\hat M$ is the cover of  $M \setminus \Lambda$  corresponding to $\pi_1 U^1$, then $V \cup U^1$ lifts to a union 
%$$\hat U^1_1 \cup \hat V \cup \hat U^1_2 \subset \hat M,$$
% where $\hat U^1_i$  are isometric lifts of $U^1$ and $\hat V$  double covers $V$.  So, the lemma follows by projecting down the simplicial  ruled  surfaces  through points of $\hat V$   that are given by Lemma \ref{trivial interval bundles}.
%\end {proof}
%
%
%

\section{Double compression bodies}
\label{Double compression bodies}
Suppose that $U$ is a compact topological product region in an irreducible 3-manifold $M$, and assume that $M\setminus int(U)$ is irreducible. 

\begin{bem}
For example, $M\setminus int(U)$ is irreducible when the homomorphism $\pi_1(U)\to\pi_1(M)$ is non-trivial. Indeed, any sphere $S \subset M\setminus int(U)$ bounds a ball $B \subset M$, and if $B$ doesn't lie in $M\setminus int(U)$, it contains $U$, and hence $\pi_1(U)\to\pi_1(M)$ is trivial.
\end{bem}

Let $S_1$ and $S_2$ be the components of $\partial U$. For $i=1,2$, let $$C_i\subset M \setminus int(U)$$ be the characteristic compression body  of $S_i$. Up to isotopy we may assume that $C_1$ and $C_2$ are disjoint. The {\em double compression body} $\mathcal{DC}(M,U)$ of $U$ is then the union 
$$\mathcal{DC}(M,U)=C_1 \cup \CN(U) \cup C_2$$
Existence and uniqueness up to isotopy of the double compression body follow directly from Proposition \ref{prop:rel-compbody}. 

\begin{bem}
With $U\subset M$ as above,  $\mathcal{DC}(M,U)$ is the smallest irreducible submanifold containing $U $ with the property that any disc $$(D^2,\partial D^2) \into (M, U) $$ with interior disjoint from $U $ can be homotoped rel-$\partial$ into $\mathcal{DC}(M,U)$. By `smallest', we mean that $\mathcal{DC}(M,U)$ can be homotoped, relative to $U$, into any other irreducible submanifold having this property.
\end{bem}

 Our main goal in this section is the following theorem.

\begin{sat}[Compressing wide product regions]\label{double compression body theorem}
For all $g\ge 2$ and $\epsilon>0$, there is  some $L=L(\epsilon,g)$ as follows. Suppose $M$ is a  complete hyperbolic 3-manifold  with no cusps  and $U \subset M$ is a compact genus $g$  product region of $\epsilon$-width at least $L$.
\begin{enumerate}
	\item If $\Sigma $ is a  boundary component of $\mathcal {DC}(M,U)$  that is not isotopic to a level surface of $U$, then $\Sigma$ is incompressible in $M$
\item Either the inclusion $\mathcal{DC}(M,U) \hookrightarrow M$ is $\pi_1$-injective, or  $U$ compresses to one side in $M$ and bounds a twisted interval bundle over a nonorientable surface to the other side. 
\end{enumerate}
%Moreover, suppose that some component $S \subset \partial U$ is  incompressible in $M\setminus int(U)$. Let $\hat M \longrightarrow M$ be  the cover corresponding to $\pi_1 U $  and take isometric lifts in $\hat M$ of all the subsets of $M$ above, giving $$\hat S \subset \partial \hat U, \ \ \hat U \subset \hat {DC} \subset \hat M.$$  Then either some component of $I \subset M \setminus int(U)$ is a (possibly twisted) interval bundle with $S \subset \partial I$, or the following hold:
%\begin {enumerate}
%	\item[(2)] $\hat{DC}$ is  a standard compact core for $\hat M$, so in particular  the inclusion $DC \hookrightarrow M$ is $\pi_1$-injective,
%\item[(3)] there is a product region $\hat U^e \subset \hat M$  containing $\hat  U$,  
%
%as a subproduct region such that for some $R=R(\epsilon,g)$, we have $$U^e \subset \CN_R(CC(\hat M)),
%\item $\partial_{ E } \CN_R(CC(\hat M))$ and $\hat S_3$ bound a topological product region in $ \hat M$ with diameter at most $D$.
%\end {enumerate}
\end{sat}

 These conclusions do not hold for general  topological product regions in $3$-manifolds.  For instance, for (2) suppose that $K$  is a nontrivial knot in $\BS^3$ and $U$ is a  topological product region obtained by taking a  regular neighborhood $\mathcal N(K) \supset K$ and removing the interior of a smaller regular neighborhood. Then $\mathcal{DC}(\BS^3,U) = \mathcal N(K)$, which does not $\pi_1$-inject in $\BS^3$, and $\BS^3 \setminus \mathcal N(K)$ is not a twisted interval bundle.
 In some sense, (2) says that a wide product region separates its two sides enough that compressions one performs on one side do not affect the other side.

 As a further illustration of Theorem \ref{double compression body theorem}, note the following corollary, which will never be used in this paper, but which we think is particularly instructive. 

\begin{kor}\label{non-haken-handle}
For all $g\ge 2$ and $\epsilon>0$, there is $L$ as follows. Suppose $M$ is a \emph{non-Haken} hyperbolic $3$-manifold and $S \subset M$ is a genus $g$ surface that has a product neighborhood of $\epsilon$-width at least $L$. Then at least one of the components of $M\setminus S$ is a handlebody $H$. Moreover, either $\pi_1(H)\longrightarrow\pi_1(M)$ is injective or $M\setminus H$ is homeomorphic either to a handlebody or to a twisted interval bundle.
\end{kor}
\begin{proof}
The non-Haken assumption means that $M$ does not have incompressible surfaces. So, $U$  compresses to at least one side. Part (1) of the previous theorem then implies that the corresponding compression body is a handlebody $H$. If $U$ also compresses to the other side, the part (1) gives us another handlebody, and we are done. If $U$ is incompressible to the other side, part (2) of the theorem says that either $H$ is $\pi_1$-injective, or we have a twisted interval bundle on the other side.
\end{proof}

The rest of the  section is devoted to the proof of Theorem \ref{double compression body theorem}. It would be possible, albeit very messy, to give a direct bound for $L$ in terms of all the constants we have met so far, but we will not attempt this, to keep the argument as transparent as possible.% In order to keep the proof of Theorem \ref{double compression body theorem} as transparent as possible we will start with some constant $L$ and then enlarge it at every step by amounts depending only on $g$ and $\epsilon$ as we see fit; we do so only finitely many times.

\subsection{Proof of Theorem \ref{double compression body theorem}, (1)} Assume that $U$ is wider than the constant $L$ provided by Proposition~\ref{incompressible-incompressible}. By construction, 
$$\mathcal{DC}(M,U)=C_1\cup U\cup  C_2$$
where $C_1,C_2\subset M\setminus U$ are the characteristic compression bodies in $M\setminus U$ associated to the boundary components $S_1$ and $S_2$ of $U$. If $\Sigma $ is an interior boundary component of $C_i$, then either $\Sigma$ is isotopic to a level surface of $U$ or it has at most genus $g-1$ and is incompressible in $M\setminus U$. In the latter case, $\Sigma$ is incompressible in $M$ by Proposition~\ref{incompressible-incompressible}.

\subsection{Proof of Theorem \ref{double compression body theorem}, (2)}  The second half of the proof is much more involved.  As in the previous case, let $U$ be a wide compact product region,
let $S_1,S_2$ be the two components of $\partial U$, and write 
$$\mathcal{DC}(M,U)=C_1\cup U \cup C_2.$$

If $C_1$ and $C_2$ are trivial compression bodies, i.e.\ interval bundles over $S_1$ and $S_2$, the Loop Theorem \cite{Hempel3-manifolds} implies that $U$ is $\pi_1$-injective in $M$. Since in this case $\mathcal{DC}(M,U)$ is isotopic to $U$, it is also $\pi_1$-injective. On the other hand, if both $C_1,C_2$ are nontrivial, then (1) says that $\D \mathcal {DC}(M,U)$ is incompressible in $M$, and hence $\mathcal{DC}(M,U) \hookrightarrow M$ is again $\pi_1$-injective.
This reduces Theorem \ref{double compression body theorem} (2) to the case that $C_2$, say,  is a trivial compression body and $C_1$ is not.

\vspace{2mm}

To limit notation, \emph{we now disregard the trivial compression body $C_2$, write $C:=C_1$ and  $DC:=\mathcal {DC}(M,U) = C \cup U$, and set $S$ to be the boundary component of $M \setminus DC$ adjacent to $U$.} See Figure \ref{2claimfig}.  

Our goal now is to prove:

 \begin{claim}\label{2claim}
If $U$ has $\epsilon $-width at least some $L_0=L_0(\epsilon,g)$, then either $DC \hookrightarrow M$ is $\pi_1$-injective, or $S$ bounds a twisted interval bundle over   a non-orientable surface in $M \setminus DC$.
\end{claim}

\begin {figure}
\centering
\includegraphics{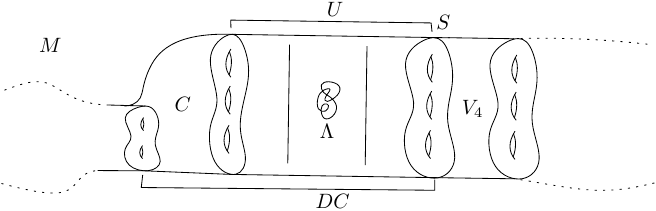}
\caption{The difficult case in (2), see Claim \ref{2claim}.}\label{2claimfig}
\end {figure}

Suppose that $U$ has huge $\epsilon $-width. By Corollary~\ref{cutting-product} and Theorem \ref{links} we can decompose $U$ as the union of three subproduct regions 
$$U=V_1\cup V_2\cup V_3$$
with disjoint interiors, such that the following properties hold:
\begin{itemize}
	\item $V_i$ and $V_{i+1}$ are adjacent to each other, and the boundary component $S$ of $U$ is contained in $\D V_3$,
	\item where $V_1,V_3$ have $\epsilon$-width at least some huge $L$,
	\item $V_2$ has $\epsilon$-width bounded above by some $L_0=L_0(g,\epsilon)$, but is wide enough to contain the $1$-neighborhood of a $0.025$-separated link $\Lambda\subset V_2$ such that $M\setminus U$ is incompressible in $M\setminus\Lambda$.
\end{itemize}
Perhaps starting out with a smaller $U$, we may also assume that there is a product region $V_4$ on the other side of $S$ from $V_3$, and that $V_4$ also has $\epsilon $-width at least $L$. (This $V_4$ is used in Lemmas \ref{product-in-convex} and \ref{lemma:hastaloshuevos}.) See Figure \ref{2claimfig} again.

 The rest of the proof will operate under the assumption that the lower bound $L$ for the $\epsilon$-widths of $V_1,V_3,V_4$ is large with respect to $g,\epsilon $.  In particular, we will feel free to increase $L$ a finite number of times to make it bigger than various constants depending on $g,\epsilon $ that appear in Sections \ref{sec:shs}  and \ref{sec: WPR}.  We think that this approach makes the section more readable (and writable) than if we were to try to identify an appropriate $L$  at the start.  That said, $L=(100+g+\frac 1\epsilon)!$ works.

\vspace{2mm}

%Assuming that $L$ is large enough so that Theorem \ref{links} holds, there is a $0.025$-separated link $\Lambda\subset V_2$ such that $M\setminus U$ is incompressible in $M\setminus\Lambda$. For later use, note that we can even take $\Lambda$ to lie in some subproduct region $V_2' \subset V_2$ whose $\epsilon$-width is bounded (if large) in terms of $g,\epsilon$. Then each component of $\D DC$ is incompressible in $M\setminus\Lambda$. 

Let $\rho$ be the negatively curved metric on $M\setminus\Lambda$  that was described in Lemma~\ref{neg-metric}. We set from now on 
$$M_\Lambda=(M\setminus\Lambda,\rho).$$
Moreover, if $W\subset M$ is a 3-dimensional submanifold containing $\Lambda$ we will write $W_\Lambda=W\setminus\Lambda$ when we want to think of it as a submanifold of the Riemannian manifold $M_\Lambda$. For instance, we will encounter $DC_\Lambda$ and $\BT_\Lambda$, where $\BT_\Lambda$ is the $0.025$-tubular neighborhood of $\Lambda$ in $M$. For the convenience of the reader we recall a few facts about $M_\Lambda$:

\begin{itemize}
\item $\rho$ agrees with the  hyperbolic metric on $M$ outside of $\BT_\Lambda$.
\item $M_\Lambda$ has hyperbolic cusps and sectional curvature pinched by
$$\kappa^-\le\kappa_{M_\Lambda}\le \kappa^+,$$
where $\kappa^\pm$ were  fixed at the beginning of \S \ref{sec: WPR}.  In the following, we will assume that $L$ is huge when compared to any previously seen constants depending on $\kappa^\pm$.
\end{itemize}

As we mentioned above, the components of $\D DC$ are incompressible in $M\setminus\Lambda$. So, the inclusion $DC_\Lambda \hookrightarrow M_\Lambda$ is $\pi_1$-injective. Let $$ p:  N_\Lambda \longrightarrow M_\Lambda$$  be the cover corresponding to the subgroup $\pi_1(DC_\Lambda)\subset\pi_1(M_\Lambda)$. By construction, the submanifold $DC_\Lambda\subset M_\Lambda$ lifts homeomorphically to a submanifold of $  N_\Lambda$, which we will also denote by $DC_\Lambda$, hoping that no confusion will arise. Since $  N_\Lambda$ has finitely generated fundamental group and hyperbolic cusps, $  N_\Lambda$ is tame by the Tameness Theorem \cite{Agoltameness}. Since compact cores are standard when the boundary is incompressible (Lemma \ref{unique-compact}),  it follows that $  N_\Lambda\setminus DC_\Lambda$ is homeomorphic to $\D DC_\Lambda\times\BR$.

  For later use, we record the following result. 
 
\begin{lem}\label{product-in-convex}
If $L $ is at least some $L_1=L_1(\epsilon,g)$, then $S \subset CC(N_\Lambda)$.
\end{lem}

\begin{proof}
$S$ is contained within the product region $U \cup V_4$, at large $\epsilon$-distance from the boundary. So, this follows from Corollary \ref{prs in the convex core}. \end{proof}

 Let $E$ be the component of $ N_\Lambda\setminus DC_\Lambda$ adjacent to $S$, so $E$ is a product neighborhood of an end $ \CE$ of $ N_\Lambda$. Let $\CN_1(CC(N_\Lambda))$ be the radius one neighborhood of the convex core in $  N_\Lambda$. The two cases of Claim~\ref{2claim} will depend on whether  there are points in $E$  that lie in  $\CN_1(CC(N_\Lambda))$  and project into the regular neighborhood $\BT_\Lambda$ of our link $\Lambda$.

\vspace{2mm} 

\noindent \bf Case 1. \rm $p(E \cap \CN_1(CC(N_\Lambda)))$  does not intersect $\BT_\Lambda$.

\vspace{2mm}

We will prove that $DC\hookrightarrow M$ is $\pi_1$-injective  by building a  hyperbolic $3$-manifold $\hat M$  with convex boundary and an isometric immersion $$\rho : \hat M \longrightarrow M$$  such that $DC$ embeds as a compact core in $\hat M$ on which $\rho$ restricts to the inclusion $DC \hookrightarrow M$. By  Proposition \ref{radialextension}, $\rho$ will then extend to a covering map from a complete hyperbolic $3$-manifold in which $DC$ is a compact core,  implying that  $\pi_1 DC$ injects into $\pi_1 M$. 

Let $\Sigma\subset\D DC$ be a boundary component different from $S$ and recall that any such $\Sigma$ is $\pi_1$-injective in $M$. Let $M_\Sigma \cong \Sigma \times \BR$ be the cover of $M$ associated to  $\Sigma$. Then $\Sigma$ lifts to a level surface $\tilde \Sigma \subset M_\Sigma$; let $U_\Sigma$ be the component of $M_\Sigma \setminus \tilde \Sigma$  such that points in $U_\Sigma$ sufficiently close to $\tilde \Sigma$ project under the cover $M_\Sigma\longrightarrow M$ to points \emph{outside} $DC$. 

%Define 

 Since we are in Case 1 and our original manifold $M$ had no cusps, the end $ \CE$ also has no cusps. If $ \CE$ is  degenerate, let $W=E$. If $ \CE$   is convex cocompact, let $W\subset E$ be  the topological product region bounded by $S$ and the component $\partial_{ \CE}\CN_1(CC(  N_\Lambda))) \subset \partial \CN_1(CC(  N_\Lambda)))$ that faces $\CE$. Note that Lemma~\ref{product-in-convex}  implies that $\partial_{ \CE}\CN_1(CC(  N_\Lambda)))$ lies in $E$. % implies  that $CC(  N_\Lambda))$ facing $ \CE$ lies in $E$, and thus  bounds a topological product region $W \subset E$ with $S$. In both cases, we have $int(W) \cap DC_\Lambda=\emptyset$.

 In both cases, Lemma~\ref{product-in-convex} says that $S \subset CC(N_\Lambda)$, which  implies that $W \subset \CN_1(CC(N_\Lambda))$.  So  since we are in Case 1, the restriction of the metric of $  N_\Lambda$ to $W$ is hyperbolic. 
Define
$$\hat M = \left(\bigcup_{\Sigma \subset \partial DC\setminus S} U_\Sigma \right ) \cup DC \cup_S W.$$ Here, each $U_\Sigma$ is  attached to $DC$ along $\Sigma$, and $DC$ is glued to $W$ along $S$. Note that the gluing respects the metrics, so $\hat M$  is a hyperbolic manifold  with one convex boundary component. 
Each of the pieces defining $M_\Sigma$ comes with an isometric immersion into $M$: for $DC$ it is the inclusion, for $U_\Sigma$ it is the restriction of $M_\Sigma \longrightarrow M$, and for $W$ it is the restriction of $p : N_\Lambda \longrightarrow M_\Lambda$, since  $p(W)$ does not  intersect $\BT_\Lambda$.  These maps combine to give an isometric immersion $$\rho : \hat M \longrightarrow M$$
as desired, which finishes the proof.
\vspace{2mm} 

\noindent \bf Case 2. \rm $p(E \cap \CN_1(CC(N_\Lambda)))$  intersects $\BT_\Lambda$.

\vspace{2mm} 

 Our goal now is to prove that $S$  bounds a twisted interval bundle in $M \setminus U$. The first step is to prove the following Lemma.

\begin{lem} \label{lemma:hastaloshuevos}
If $L$ is large,
there is a minimizing SRS $f : S \longrightarrow N_\Lambda$ in the homotopy class of the inclusion, such that $f(S)$ is contained in $E$ and $p\circ f(S) $ intersects $\mathbb T_\Lambda$.
\end{lem}

\begin{proof}
 Suppose first that the end $ \CE$ has cusps, and using Fact \ref{simplecusp}, pick a simple closed curve $\gamma \subset S$  that is homotopic into a cusp of $ \CE$. By Corollary \ref{construction}, there is a family of minimizing SRSs $f_t : S \longrightarrow N_\Lambda$  in the homotopy class of the inclusion such that $f_t(\gamma)$ exits that cusp as $t\to 0$. In particular, for small $t$ the  surface $f_t(S)$   intersects $E \setminus V_4$. If $L $ is large, $V_4$ has large $\epsilon$-width, so the Bounded Diameter Lemma implies that $f_t(S) \subset E$ for small $t$. Moreover, for small $t$ the projection $p \circ f_t(\gamma)$ then lies deep inside a cusp neighborhood of $M_\Lambda$. But $M$ had no cusps, so this implies that for small $t$, $$p \circ f_t(S) \cap \BT_\Lambda \neq \emptyset.$$  
  Hence, we can take $f=f_t$ for some small $t$.
  
   Suppose now that $ \CE $ has no cusps. The Geometric Tameness Theorem  implies that  $ \CE$ of $N_\Gamma$ is either convex-cocompact or degenerate.  Suppose for simplicity that $ \CE$  is convex-cocompact; we will deal with the degenerate case at the end of the proof.  %Since  $M_\Lambda$ is non-compact but has finite volume, the Covering Theorem \cite{Canarycovering} implies that $ \CE$  cannot be degenerate. In other words, $\D_{ \CE}CC(  N_\Lambda)$  is homotopic to $S$. 
By Lemma~\ref{product-in-convex}, if $L$ is large the boundary component $\D_{ \CE} \CN_1(CC(  N_\Lambda))$ is  contained in $E$. %Hence, $S=\DE$ and $\D_{ \CE}CC(  N_\Lambda)$ bound a topological product region in $E$.

By Corollary \ref{kor SRSs near partial}, the surface $\partial_\CE \CN_1(CC(  N_\Lambda))$ is homotopic to a minimizing SRS $f_0 : S \longrightarrow N_\Lambda$ via a homotopy whose image has bounded $\epsilon$-diameter. Choose a minimizing SRS $f_1 : S\longrightarrow  N_\Lambda $ in the homotopy class of the inclusion such that the $R$-neighborhood of its image is contained in $ V_3$, say, where $R$ is as in  the Interpolation Theorem. Perhaps after homotoping $f_0,f_1$ within $R$-neighborhoods of their images, the Interpolation Theorem gives a homotopy $(f_t)$  between $f_0,f_1$ through minimizing SRSs.  

We claim every point $x\in \CN_1(CC(  N_\Lambda))\cap E$ lies within a bounded $\epsilon$-distance of the image of the homotopy $(f_t)$. Indeed, concatenating the 4 homotopies discussed in the previous paragraph gives a homotopy $H$ from $f_1$ to the surface $\partial_{\CE} \CN_1(CC(  N_\Lambda))$. It suffices to show that any $x\in \CN_1(CC(  N_\Lambda))\cap E $ lies in the image of $H$, since the three homotopies that are not $(f_t)$  all have images of bounded $\epsilon$-diameter. 

So, assume that some $x\in \CN_1(CC(  N_\Lambda))\cap E $ does not lie in the image of $H$. The image of $H$ is contained in $U \cup E$, which is a topological product region that contains $\CN_1(CC(  N_\Lambda))\cap E  $ as a topological subproduct region. Hence, we can retract $H$ to a homotopy from $S$ to $\partial_{\CE} \CN_1(CC(  N_\Lambda))$ that is contained in $\CN_1(CC(  N_\Lambda))\cap E $, and whose image still misses $x$. In other words, we have a map $$S \times [0,1] \longrightarrow \CN_1(CC(  N_\Lambda))\cap E \cong S \times [0,1]$$ that restricts to a homotopy equivalence between the boundaries. This map still misses a point of $\CN_1(CC(  N_\Lambda))\cap E$, so has zero degree, which is a contradiction.

As we are in Case 2, there is some $f_t$ whose image comes within a bounded $\epsilon $-distance of some $x\in E$ such that $p(x) \in  \BT_\Lambda$. 
 As $L$  is large and $x \not \in V_4$, the Bounded Diameter Lemma  implies that  $f_t(S) \subset  E $. We are then done with Lemma \ref{lemma:hastaloshuevos}, with $f=f_t$.

\medskip

The case where $ \CE$  is degenerate is  exactly the same, except that instead of constructing $f_0$ near the convex core boundary, we fix $x \in  E $ with $p(x) \in \BT_\Lambda$ and take $f_0$ deep enough in the end so that any homotopy from it into $U_\Lambda$  passes through $x.$ \end{proof}
 
We now want to use the map $f$ above to show that $S$ bounds a twisted interval bundle in $M\setminus U$. To this end, let $R$ be the constant given in the Interpolation Theorem. If $L$ is large, we can use Corollary~\ref{cutting-product} to find subproduct regions $V_1' \subset V_1$ and $V_3' \subset V_3$, both sharing a boundary component with $V_2$, such that for $i=1,3$, 
\begin{enumerate}
	\item $V_i'$ is $\epsilon$-wide enough to contain the $R$-neighborhood of a minimizing SRS in the homotopy class of a level surface, see for instance Lemma~\ref{minsrsinpr},
	\item $V_i'$ is $\epsilon$-wide enough to contain a level surface $S_i'$ such that any incompressible SRS in $N_\Lambda$ intersecting $S_i'$ is contained inside of $V_i'$, and where $S_i'$ lies at least $R$ away from $\partial V_i'$, see Lemma~\ref{level surfaces} and the Bounded Diameter Lemma for SRSs,  
	\item $V_i'$ has $\epsilon$-width bounded \emph{above} in terms of $g,\epsilon$.
\end{enumerate}
So in other words, in (3) we are saying that $V_i'$ should be chosen just wide enough that (1) and (2) hold.

Using (1) and the Interpolation Theorem, there is a homotopy $(f_t)$ in $N_\Lambda$ through minimizing SRSs such that $f_0$ is an SRS in $V_3'$ in the homotopy class of a level surface, and where $f_1$ lies in an $R$-neighborhood of the map $f$ from the Lemma above. Note that $f(S) \subset E \setminus V_4$, since $p$ is injective on $U_\Lambda \cup V_4$ and $p\circ f$ intersects $\BT_\Lambda$. So, if $L$ is large, $f_1(S) \subset E$. The projection $(p \circ f_t)$ is a homotopy in $M_\Lambda$ through SRSs, where $p\circ f_0$ lies in $V_3'$ and is homotopic to a level surface, and where $p\circ f_1$ lies at bounded $\epsilon$-distance from $\BT_\Lambda$, and hence in the topological product region $[S_1',S_3'] \supset V_2$ bounded by the level surfaces $S_1',S_3'$.

Let $m \in [0,1]$ be the maximal parameter such that $$p\circ f_m(S) \not \subset [S_1',S_3'].$$ Here $m>0$, since as $f_0 \subset U_\Lambda$ and $f_1(S) \subset E$, some $f_t(S)$ has to intersect $S$ and hence cannot project into $[S_1',S_3']$ since $p$ is injective on $U_\Lambda$. Similarly, we can show that $f_m(S) \subset E$: if not, there is some $t\in (m,1)$ such that $f_t(S) $ intersects $S$, and then $p\circ f_t(S) \not \subset [S_1',S_3']$, contradicting maximality. Also, note that the image $p\circ f_m(S)$ must intersect  $S_i'$ for either $i=1$ or $i=3$, and hence is contained in $V_i'$, by (2) above. Since $p \circ f_m$ is incompressible in $M_\Lambda$ and is homotopic to an embedding, it is homotopic to an embedding in $V_i'$ by Freedman--Hass--Scott (see Lemma \ref{getembedded}). Hence, it is homotopic to a level surface in $V_i'$, and in particular to $S_i'$.

 Let's now construct a new homotopy $(g_t)$ from $(f_t)$ as follows. On $[0,\frac 13]$, say, define $(g_t)$ to be a homotopy in $V_3'$ such that $g_0(S)=S_3'$ and where $g_{\frac 13} = f_0$ is as before. On $[\frac 13,\frac 23]$, define $(g_t)$ as a reparametrization of the restriction of $(f_t)$ to the interval $[0,m]$. On $[\frac 23, 1]$, use the homotopy lifting property to define $(g_t)$ to be homotopy with $g_{\frac 23}=f_m$ such that $p\circ g_t$ is a homotopy within $V_i'$ from $p\circ f_m$ to a parametrization of $S_i'$, where here $i$ is as in the previous paragraph.
 
 There are now two cases. If $p \circ g_1(S)=S_1'$, then $(p \circ g_t)$ is a homotopy in $M_\Lambda$ from $S_3'$ to $S_1'$, which are incompressible embedded surfaces. Waldhausen's cobordism theorem (see \S  \ref{intervalbundle}) then implies that $S_1'$ and $ S_3'$ together bound a trivial interval bundle in $M_\Lambda$. However, this is impossible since no component of $M_\Lambda \setminus (S_1'\cup S_3')$ is precompact. 
 
 We may therefore assume that $p \circ g_1(S) = S_3'$, so that $(p \circ g_t)$ is a homotopy from $S_3'$ to itself in $M_\Lambda$. We  claim that this homotopy is `essential', in the sense that it cannot be homotoped rel $p\circ g_0,p\circ g_1$ to a homotopy whose image lies in $S_3'$. To see this, note that on $[\frac 23, 1]$, we have defined $(g_t)$ so that $g_{\frac 23}(S)=f_m(S) \subset E$, as noted above, and where $p \circ g_t$ is always contained within $V_3'$. In particular, injectivity of $p$ on $U_\Lambda$ means that no $g_t(S)$ can intersect $S \subset \partial U_\Lambda$. So, $g_t(S) \subset E$ for all $t\in  [\frac 23, 1]$. But by construction, $g_0(S) =S_3'$, which does not lie in $E$, so $g_0$ and $g_1$ disjoint images in $N_\Lambda$. It follows from the homotopy lifting property that $(p\circ g_t)$ is essential.

Applying Waldhausen's cobordism theorem, $S_3'$ bounds a twisted interval bundle in $M_\Lambda$. There are two components of $M_\Lambda \setminus S_3'$, and the one whose closure is a twisted interval bundle must be the one containing $S$, since the other one is not pre-compact due to the drilled out link $\Lambda$.  It follows that $S \subset M_\Lambda$ bounds a twisted interval bundle outside $DC_\Lambda$, and hence that $S\subset M$ bounds a twisted interval bundle outside $DC$, as desired.

\vspace{2mm}

 This completes the proof of Theorem \ref{double compression body theorem}.

\subsection{Some applications}
 In this section, we give two applications of Theorem \ref{double compression body theorem}.  The first is a corollary of Theorem \ref{double compression body theorem} and the unknotting results of \S \ref{wide product regions are unknotted}; it and its corollaries will be used in \S \ref{splitting along barriers}, and to prove our characterization of strong convergence in Theorem \ref{strong limits}.

\begin{sat}\label {complementary compression bodies} 	Given $g,\epsilon$,  there is some $L=L(g,\epsilon) $ as follows.  Suppose that $M $ is a  complete hyperbolic $3$-manifold  and that  $N \subset M$ is a (compact, connected) $3$-submanifold such that
\begin {enumerate}
\item  the inclusion $N \hookrightarrow M$ is $\pi_1$-surjective,
\item each component $S \subset \partial N$ has genus at most $g$ and bounds a  product region $U \subset M \setminus int(N)$ with $\epsilon$-width at least $L$.
\end {enumerate}
Then every component of $M \setminus int(N)$ is a compression body with missing interior boundary.
\end{sat}
\begin {proof}
 Enumerate  the components of $\partial N$ as $S_1,\ldots,S_n$  and let $U_1,\ldots,U_n$  be the adjacent product regions.  For each $i=1,\ldots,n$, let $S_i'$ be the component of $\partial U_i$  that is not $S_i$, and let
$$C_i = C(M\setminus int(U_i), S_i')$$
be the associated characteristic compression body.
Since $N \hookrightarrow M$ is $\pi_1$-surjective, every component of $\partial N$ separates $M$. Hence, all the $C_i$ are disjoint, from each other and from $N$. Let $$N^{ext} = N \ \cup \ \bigcup_{i=1}^n 
(U_i \cup C_i), \ \ \text{ so } \ \partial N^{ext} = \bigcup_i \partial_{int} C_i,
$$

\begin {claim} \label {inj}The inclusion $N^{ext} \hookrightarrow M$ is $\pi_1$-injective. 
\end {claim}

 Assuming this for a moment, let's finish the proof of the theorem. Since $N^{ext} \hookrightarrow M$ is also $\pi_1$-surjective, we have that $N^{ext} $ is a compact core for $M$.  It suffices to show that $N^{ext} $  is a \emph{standard} compact core, since certainly every component of $int(N^{ext}) \setminus int(N)$ is a compression body  with missing interior boundary.

  By the Tameness Theorem, $M$  is the interior of a manifold with boundary $\bar M$. Let $T$  be a component of some $\partial C_i \subset \partial N^{ext}$, let $ T_\partial$  be the component of $\partial \bar M$ that $T$ faces, and let $$E \subset \bar M \setminus int(N^{ext})$$  be the component bounded by $T$ and $ T_\partial$.
By uniqueness of compact cores up to homotopy, see Fact \ref{coreuniqueness}, $T$ is homotopic to $ T_\partial$ in $\bar M$. 
 
  We want to show that $E$  is a trivial interval bundle. This follows from Waldhausen's cobordism theorem when $T$ and $ T_\partial$ are incompressible in $\bar M$, so we can assume that they are both compressible.  
Let  $$\pi: \hat M \longrightarrow \bar M$$  be the cover corresponding to $\pi_1  T_\partial$. Then $\hat M $ has a $\pi_1$-surjective boundary component $\hat { T_\partial}$ that projects homeomorphically to $ T_\partial$. Lift the homotopy from $T_\partial$ to $T$ to a homotopy from $\hat T_\partial$ to some $\hat T \subset \hat M$. 

%We claim that \emph {$E$ lifts homeomorphically to $\hat M$.}It suffices to show that any loop $\gamma \subset E$ can be homotoped in $M$ to a loop on $T$. Since $N^{ext} \hookrightarrow M$ is $\pi_1$-surjective, $\gamma$  can be realized as one   component of $f(\partial A)$, where $$f : A \longrightarrow M$$  is a smooth map from an annulus $A$, and where the other component of $f(\partial A)$ lies in the  compact core $N^{ext}$. If we take $f$  to be transverse to $T$, which separates the two  components of $f(\partial A)$ in  $M$, then some component of $f^{-1}(T)$ must  separate the boundary components of $A$, and hence must be an essential loop in $A$. It follows that $\gamma $ is  homotopic to a curve on $T$ as desired.

Since $T$ is compressible in $M$, Theorem \ref{double compression body theorem} (1)  implies that  as long as $L$ is large, the compression body $C_i$ with $T \subset \partial_{int} C_i$ is in fact a trivial interval bundle. So, $U_i \cup C_i$ lifts homeomorphically to $\hat U_i \cup \hat C_i \subset \hat M$, where $\hat T \subset \partial \hat C_i$. 
By Corollary \ref{unknotted}, the component $\hat E $ of $\hat M \setminus int(\hat U_i \cup \hat C_i) $ that has $\hat T$ in its boundary is a compression body. But $\hat E$ has two boundary components, $\hat T$ and $\hat T_\partial$, and these have the same genus! So $\hat E$ is a trivial interval bundle. 

Given $x\in \bar M$, let $n(x) = |\pi^{-1}(x) \cap \hat E|.$ Then $n\equiv 1$ on $T_\partial$, since $\hat E \setminus \hat T_\partial$ projects into $M$, not $\partial \bar M$. The function $n$ is continuous except on $T$, where it can jump by $1$. So, $n\equiv 1$ on the component of $\bar M \setminus T$ containing $T_\partial$, and either $n\equiv 0$ or $n\equiv 2$ on the other component. In the first case, $\pi$ restricts to a homeomorphism from $\hat E$ to $E$, and we're done. So, assume we're in the second case\footnote{Concretely, in this case you should imagine that $\bar M$ is a twisted interval bundle over a nonorientable surface, that $T_\partial=\partial \bar M$ and $T$ is parallel to $T_\partial$, that $\hat M$ is a trivial interval bundle that double covers $\bar M$, and $\hat T$ is the lift of $T$ that is farther from $\hat T_\partial$. This can't happen, since we're assuming $T$ is compressible, and the argument we give in the proof is essentially just saying that.}. By compressibility of $T$, we can take a disc $D \subset M \setminus int(E)$ with $\partial D \subset T$ an essential curve, since by definition $T$ is incompressible in $E$. Since we're in the second case, the inward normal direction of $\hat T \subset \hat E$ projects to the inward normal direction of $T \subset M \setminus int(E)$, and we can lift $D$ to a disc $\hat D \subset \hat E$ whose boundary is an essential curve in $\hat T$. This is impossible, since $\hat E$ is an interval bundle.

\medskip

 We now prove Claim \ref{inj}, i.e.\ that $N^{ext} \hookrightarrow M$ is $\pi_1$-injective. It suffices to show that whenever $\gamma \subset \partial N^{ext}$  is an  essential closed curve that is  null-homotopic in $M$, then $\gamma$  is also null-homotopic in $N^{ext} $. For if this is true, and $N^{ext} \hookrightarrow M$ is still not $\pi_1$-injective, one can take a smooth map $f : (D^2,\partial D^2) \longrightarrow (M,N^{ext})$ such that $f|_{\partial D^2}$ is essential in $N^{ext}$,   and such that $f$  is transverse to $\partial N^{ext}$ and minimizes the number of components of the collection of circles $$\mathcal I_f := f^{-1}(\partial N^{ext}) \subset int(D^2).$$ Taking an outermost component $c \subset \mathcal I_f$, we can redefine $f$ on the disk bounded by $c$ to be a null-homotopy of the curve $\gamma = f(c) \subset \partial N^{ext}$ in $N^{ext}$. After a perturbation, this reduces the number of components of $\mathcal I_f$, a contradiction to the minimality condition.

 So, let $\gamma \subset \partial N^{ext}$  be an  essential closed curve that is  null-homotopic in $M$. Then 
  we have that $\gamma \subset \partial_{int} C_i$ for some $i$. Let $$D_i := C(M \setminus int(U_i),S_i), \ \ \mathcal {DC}(M,U_i) = D_i \cup U_i \cup C_i.$$

 Suppose first that $\mathcal {DC}(M,U_i) \hookrightarrow M$ is $\pi_1$-injective. Then $\gamma$ is null-homotopic in $\mathcal {DC}(M,U_i)$, so it suffices to show that $D_i$ is homotopic rel $S_i$  to a submanifold of $N^{ext}$, for then we can isotope our null-homotopy to lie in $N^{ext}$ too. This is  obvious if the characteristic compression body $C(N,S_i)$ is isotopic to $D_i$, so we may assume this is not the case.  In other words, $C(N,S_i)$ has an interior boundary component $X \subset \partial_{int} C(N,S_i)$ that is compressible in $M \setminus int(U_i)$. As $X$ is incompressible in $N$,  Proposition \ref{incompressible-incompressible} says that $X$ is homotopic in $N$ to some  component $S_j \subset \partial N$, where $j =j(X) \neq i$. Since $X,S_j$ are embedded and incompressible in $N$, they bound a trivial interval bundle $I_X$ in $N$ by Waldhausen's cobordism theorem. Then  $$D_i ' := C(N,S_i ) \ \ \ \cup \hspace{-7mm} \bigcup_{\substack {X \subset \partial_{int} C(N,S_i ) \\ X \text{  compressible in } M \setminus int(U_i) }}  \hspace{-14mm}  I_X \cup U_{j(X)} \cup C_{j(X)}$$
is a  compression body in $M \setminus int(U_i)$ with exterior boundary $S_i$.  By definition of the compression bodies $C_{j(X)}$, no interior boundary component of $D_i'$ compresses in $M \setminus int(D_i')$. Hence, $D_i'$ is a characteristic compression body for $S_i \subset \partial (M \setminus int(U_i))$, and hence $D_i$  is homotopic rel $S_i$ to $D_i'$. But $D_i' \subset N^{ext}$  by construction, so we are done.

Now assume that $\mathcal {DC}(M,U_i)\hookrightarrow M$  is not $\pi_1$-injective.
Theorem \ref{double compression body theorem}  says that some component of $M \setminus int(U_i)$  is a twisted interval bundle $I$ over a non-orientable surface, and that $U_i$ compresses nontrivially to the other side. Let's first suppose that the component $T_i \subset\partial U_i$  defined above bounds such a component $I$, so that $$N \cap I = \emptyset.$$ Now $I$ has a double cover which is a trivial interval bundle over the orientation cover of the  given non-orientable surface. Attaching two copies of $M \setminus int(I)$ to this double cover, we can construct a nontrivial connected cover of $M$ to which $N$ lifts homeomorphically, contradicting the fact that $N \hookrightarrow M$ is $\pi_1$-surjective. 

Next, suppose that the component $S_i \subset\partial U_i$ bounds  a twisted interval bundle $I \subset M \setminus int(U_i)$, so that $N \subset I$.  As in the previous paragraph, we can construct a double cover $$\hat M\longrightarrow M$$  and a pair of product regions $\hat U_i^1, \hat U_i^2 \subset\hat M$, each of which projects homeomorphically onto $U_i$, and which together bound a trivial interval bundle $\hat I \subset \hat M$ that double covers $I \subset M $. Now any  product region $U_j \subset I$  is covered by a product region $\hat U_j$ in $\hat I$. A priori, $\hat U_j$ could be a double cover of $U_j$, and  could have genus at most\footnote{Here, $2g$ is an overestimate, since it is $\chi$ that multiplies by two.} $2g$, but if $L$  is large, then Corollary \ref{unknotted2} implies that any product region in $\hat I$ with genus at most $2g$ and $\epsilon$-width at least $L$ is  actually a topological subproduct region of $\hat I$, and hence $\hat U_j$ is a homeomorphic lift of $U_j$, while $I \setminus int(U_j)$ is the union of a trivial interval bundle adjacent to $U_i$, and a smaller twisted interval bundle. From this,  it follows that we are in one of the following two situations:
\begin{enumerate}
	\item $N$  is a trivial interval bundle bounded by $U_i$ and some $U_j$, and  the component of $M\setminus int(U_j)$ that does not contain $N$ is a twisted interval bundle over  a non-orientable surface,
\item $N$ is a twisted interval bundle over a non-orientable surface.
\end{enumerate}
Since $N \hookrightarrow M$ is $\pi_1$-surjective, case (1) is impossible, as shown  in the previous paragraph. In case (2),  the manifold $N^{ext}$ is just $N \cup U_i \cup C_i$, and as mentioned above $C_i$ is a nontrivial compression body. So by Theorem \ref{double compression body theorem} (1), $N^{ext}$  has incompressible boundary in $M$, and hence is $\pi_1$-injective.\end {proof}

 In Theorem \ref{PR covering theorem}, we showed that product regions in $\epsilon$-thick hyperbolic $3$-manifolds obey a `covering theorem' analogous to Canary's Covering Theorem \cite{Canarycovering} for degenerate ends.  Combining this with Theorem~\ref{double compression body theorem}  gives the following application.

\begin{kor}[Extending product regions]\label {extending product regions}
 For all $g\geq 2$ and $\epsilon>0$,  there  are $L,D$ as follows.  Let $M$ be a complete $\epsilon $-thick hyperbolic $3$-manifold and $U\subset M $ be a product region of genus at most $g$ and width at least $L $, such that no component of $M \setminus int(U)$  is a compact (twisted) interval bundle. Suppose that $S\subset\partial U$  is a component that is incompressible in $M\setminus int(U)$. Then $U$ is a subproduct region of some  product region $$U^e \supset U$$ such that one of the two following possibilities holds. 
\begin{enumerate}
\item $S$ bounds a component of $U^e \setminus U$  that is a product neighborhood of a degenerate end of $M$.
	\item if $S$ and $S^e \subset \partial U^e$ bound a component of $U^e\setminus U$, letting $$\hat M \longrightarrow M$$ be the cover corresponding to $\pi_1 U$ and lifting $S^e\subset \partial U^e$ homeomorphically to $\hat S^e \subset \partial \hat U^e \subset \hat M$, the  component $ \hat E^e\subset \hat M \setminus int(\hat U^e)$  adjacent to  $\hat S^e$ is a product neighborhood of a convex cocompact end  of $\hat M$,  and 
$
\diam(\hat E^e \cap \CN_1(CC(\hat M))) < D.$
\end{enumerate} 
\end{kor}

 Informally, the corollary says that the product region $U$ can be extended  almost all the way to the boundary of its `convex hull' in $M$, while staying embedded.% See Figure \ref{??}.

\begin {proof}
 As in the statement of the corollary, let $$\pi : \hat M \longrightarrow M$$ be the cover corresponding to $\pi_1 U$.  Suppose that $S \subset \partial U$  is a component that is incompressible in  $M\setminus int(U)$, and lift it to $$\hat S \subset \hat U \subset \hat M.$$
If $L$  is large, then  Theorem \ref{double compression body theorem} (2)  implies that $\mathcal{DC}(M,U) \hookrightarrow M$ is $\pi_1$-injective, so $\mathcal{DC}(M,U)$  is a compact core for $\hat M$.   But since $S$  is incompressible in $M\setminus int(U)$,  the double compression body $\mathcal{DC}(M,U)$  is just  the union of $U$ with a compression body, so $\hat M $  is a compression body by the  Tameness Theorem \cite{Agoltameness,Calegarishrinkwrapping} and the uniqueness of compact cores up to homeomorphism \cite{mccullough1985uniqueness}.  Corollary \ref{unknotted} then implies that $\mathcal{DC}(M,U) \subset \hat M$  is a \emph {standard} compact core, so the component $$\hat E \subset\hat M \setminus int(\hat U)$$ adjacent to $\hat S$  is a product neighborhood of some end $\hat \CE$ of $\hat M$. This $\hat \CE$  is either convex cocompact or degenerate, since $\hat M$ has no cusps.

\smallskip
\noindent \it Case (1), when $\hat \CE$ is degenerate. \rm By Proposition \ref{constructing prs}, $\hat \CE$ has a neighborhood that is a product region,  which we can take to be contained in $\hat E$ by Corollary \ref{cutting-product}. This product region  bounds a trivial interval bundle with $\hat U$, so Lemma \ref{trivial interval bundles}  implies that $\hat U \cup \hat E$  is itself a product region,  as long as $L$  is large.  Since no component of $M \setminus int(U)$  is a compact interval bundle, Theorem~\ref{PR covering theorem}  says that there is a subproduct region $$\hat U_{emb} \subset \hat U \cup \hat E$$  that embeds in $M$, and  where  the difference $ (\hat U \cup \hat E) \setminus U_{emb}$ is a subproduct region  with bounded width  that is  adjacent to the single component of $\partial( \hat U \cup \hat E)$. Now if $L$ is  large, $\hat U$  is wide enough  so that $$(\hat U \cup \hat E) \setminus U_{emb} \subset \hat U,$$
 in which case $\pi(\hat E)$ is an embedded product region in $M$  that is a product neighborhood of a degenerate end of $M$, so we are done with $$U^e=U \cup \pi(\hat E).$$ 
\smallskip

\noindent \it Case (2), when $\hat \CE$  is convex cocompact. \rm Let $\partial_{\hat \CE} CC$  be the component of $\partial CC(\hat M)$ that faces the end  $\hat \CE$.   Using  Theorem \ref{links} and Corollaries~\ref{cutting-product} and \ref{prs in the convex core}, if we are given any $L_->0$,  there are $L_+,L$  depending on $L_-,\epsilon,g$  such that as  long as $U$  has width at least $L$, we can write
$$\hat U = \hat U_1 \cup \hat U_2 \cup \hat U_3\cup \hat U_4$$  as a union of adjacent product regions, where $\hat S \subset\partial \hat U_4$, where \begin{equation}
	\label{incc}\hat U_2 \cup \hat U_3 \subset int(CC(\hat M)),
\end{equation} where $\hat U_1,\hat U_3,\hat U_4$  all have width in $[L_-,L_+]$, and  where there is some $0.025$-separated geodesic link $\Lambda \subset \hat U_1$ with length at most $L_+$, say, such that the inclusion  of the trivial interval bundle $$\hat U_2 \cup \hat U_3 \cup \hat U_4\cup \hat E \hookrightarrow \hat M \setminus \Lambda$$ is $\pi_1$-injective.   Note that as long as $L_- \geq \epsilon$, say, we have that the $\epsilon$-ball around every point of $\hat U_2 \cup \hat U_3 \cup \hat U_4 \cup \hat E$ is contained in $\hat U \cup \hat E$, so the following condition holds: \begin {enumerate}
\item[($\star$)]	for any map $f : S \longrightarrow\hat U_2 \cup \hat U_3 \cup \hat U_4 \cup \hat E$ in the homotopy class of a level surface, there is no  essential closed curve on $S$ whose image has length less than $\epsilon$.
\end {enumerate}

Since $\hat M$ is hyperbolic, $\partial_{\hat \CE} CC$ is a pleated surface, and hence is intrinsically hyperbolic. As $\hat M$ is $\epsilon$-thick and $\partial_{\hat \CE} CC$ is incompressible in a large neighborhood of itself, the surface $\partial_{\hat \CE} CC$ is $\epsilon$-thick. Hence, $\partial_{\hat \CE} CC$ has bounded geometry in the sense of Fact \ref{nearboundary}. So, Fact \ref{nearboundary} gives a homotopy $$f_t : S \longrightarrow \hat U_2 \cup \hat U_3 \cup \hat U_4\cup \hat E \subset \hat M \setminus \Lambda, \ \ t\in[0, 1]$$  where $f_0$ parameterizes $\partial_{\hat \CE} CC$ and $f_{1}$ is a minimizing simplicial ruled surface, and where  the entire homotopy is contained in $\CN_R(\partial_{\hat \CE} CC)$, for some $R=R(\epsilon,g)$. By Lemma \ref{minsrsinpr}, there is a minimizing SRS
 $$f_2 : S \longrightarrow \hat U_2 $$
 in the homotopy class of a level surface. After  modifying  $f_1$ and $f_2$ by homotopies supported within balls of  bounded radius,  we may assume that the Interpolation Theorem applies, giving a  homotopy $(f_t), \ t\in [1,2]$  through minimizing simplicial ruled  surfaces.  Increasing $R$ appropriately, we may assume by the Bounded Diameter Lemma that for every $t\in [0,2]$, the image $f_t(S)$ has  diameter at most $R$.

Pick  some point $p \in CC(\hat M)$ with $$2R+1 \leq d(p,\partial_{\hat \CE}CC) \leq 2R+2.$$ By \eqref{incc}, the surface $\partial_\CE CC$ lies in $ \hat U_4 \cup \hat E$. So as long as $L_-$ is large enough,  we can assume that $$p\in \hat U_3 \cup \hat U_4 \cup \hat E.$$  So, $p$  is in the image of some $f_{t_0}, \ t_0\in [0,2]$, by the same argument as in the last paragraph of the proof of Proposition \ref{Bsurfaces}. Work of Freedman--Hass--Scott, see Lemma \ref{getembedded},   implies that the map $f_{t_0} $  is homotopic to an embedding $\hat T^e \subset \CN_1(f_{t_0}(S))$. If $L_-$ is large,  we  can assume $$\hat T^e \subset \hat U_3 \cup \hat U_4 \cup \hat E;$$ note that $\hat T^e$  is a level surface therein. Note also that since $f_{t_0} (S)$ has  diameter at most $R$, we have \begin{equation} \label{sintno}\hat T^e \cap \CN_R(\partial_{\hat \CE}CC) = \emptyset,\end{equation} by our definition of the point $p$. Let $$\hat V^e \subset \hat U \cup \hat E$$ be  the topological subproduct region bounded by $\hat T^e$ and containing  the product regions $\hat U_1, \hat U_2$. 

We claim that every  $x \in \hat V^e $  is in the image of some NAT simplicial ruled surface in the homotopy class of a level surface. If $x \in \hat U_1\cup \hat U_2$, this is clear since $\hat U$  is a product region.  Otherwise, $x$  is in the image of some $f_{t}, t\in [0,2]$. By \eqref{sintno}, we have $x \not \in \CN_R(\partial_{\hat {\CE}} CC)$, so  in fact $t\in [1,2],$ in which case $f_t$ is a  simplicial ruled surface. Now, a priori  all these surfaces are in the variably curved manifold $\hat M \setminus \Lambda$, but since $\Lambda \subset \hat U_1$ and $L_-$ is large, all the surfaces we are considering  are very far from $\Lambda$, and hence are NAT  simplicial  hyperbolic surfaces in $\hat M$. Finally, as $\hat T^e$  was constructed to be in the $1$-neighborhood  of a NAT simplicial ruled surface,  we have that $\hat V^e$ is  a product region,  not just a topological one.

If $L_-$  is large, we can write $$\hat V^e = \hat U^e \cup W, $$  in such a way that $\hat U_1 \subset \hat U^e$, say, and $W$ has width  in the interval $[L,2L]$,  where $L=L(\epsilon,g)$  is the constant in  Theorem \ref{PR covering theorem}.  
Let $$\hat S^e\subset\partial \hat U^e$$  be the component facing $\hat \CE$, and let $$\hat I \subset\hat M$$ be the topological product region bounded by $\hat S^e$ and $\partial_{\hat \CE}CC$.  The distance $d(\hat S^e,\hat T^e) = \width(W)\leq 2L$, and by construction of our point $p$, the distance from $\hat T^e$  to $\partial_{\hat \CE} CC$ is also bounded, say by $3R+3$. Hence, $\hat I$  has bounded width. Let $\gamma$ be a  path in $\hat I$  realizing this width.   Since $\partial_{\hat \CE} CC \subset\hat U_4 \cup \hat E$, as long as $L_-$ is large,  we can assume that $$\hat I \subset \hat U_3 \cup \hat U_4 \cup \hat E,$$ so any point $x \in \hat I$  is in the image of some $f_t, t\in [0,2]$. Each such image $f_t(S)$ has  diameter at most $R$, and we must have $$f_t(S) \cap (\hat S^e \cup \gamma \cup \partial_{\hat \CE}CC) \neq \emptyset.$$ So, since each of $\hat S^e$,  $\gamma, $ and $\partial_{\hat \CE}CC$ has  bounded diameter, $\hat I$ also has  diameter at most some  constant depending only on $\epsilon,g$. In  particular, this implies that if $\hat E^e \subset \hat M \setminus \hat U^e$  is the component facing $\hat \CE$,  the intersection $\hat E^e \cap \CN_1(CC(\hat M))$ has  bounded diameter as desired.

Suppose for the moment that $\hat U \subset \hat U^e$.    As long as $L_-$ is large, we can assume that the width of $\hat U$  is at least the constant $L$  from Theorem~\ref{PR covering theorem}. As  we are assuming that  no component of $M \setminus int(U)$ is a compact interval bundle,  if we apply Theorem \ref{PR covering theorem}  to the triple
$$\hat U, \hat U^e \setminus int(\hat U), W$$
 of adjacent product regions, we see that the covering map $\pi$  restricts to an embedding on $\hat U^e$.  So,  setting $U^e = \pi(\hat U^e)$, we are done.

If  $\hat U \not \subset \hat U^e$,  then we have
$$d(\hat U,\partial_{\hat \CE}CC) \leq d(\hat S^e,\partial_{\hat \CE}CC),$$
 which is bounded above by some constant depending on $\epsilon,g$.  We can then   discard the $\hat U^e$  produced above, set $U^e = U$,  and the required diameter bound will follow from the same arguments as above.
 \end {proof}

\section{Carrier graphs}

\label {carrier}

A \emph{carrier graph} in a hyperbolic manifold $M $ is a $\pi_1$-surjective map $f : X \longrightarrow M$, where $X$ is a connected finite graph.   For instance, any generating set for $\pi_1 M $ gives a carrier graph whose domain is a wedge of circles.  We say that carrier graphs $f: X\longrightarrow M $ and $g: Y\longrightarrow M $ are \emph {equivalent} if there is a homotopy equivalence $h:X\longrightarrow Y $ such that $g\circ h$ is homotopic to $f $.
We will assume everywhere below that all carrier graphs are rectifiable, i.e.\ they have finite length.

If $f: X \longrightarrow M$  is a carrier graph, the hyperbolic metric on $M$ pulls back to a path (pseudo)-metric on $X$, which we use to measure the lengths of the edges of $X$.  We say that $f $ has \emph {minimal length} if its total edge length is at most that of any other equivalent carrier graph $Y \longrightarrow M$.  It follows from Arzela-Ascoli that any equivalence class of carrier graphs has a {minimal length} representative.  

A carrier graph $f : X \longrightarrow M$ is \emph{degenerate} if there is an edge of $X$ that is mapped to a point in $M$, and \emph{nondegenerate} otherwise. When $f$ is degenerate, let $X'$ be the graph obtained from $X$ by collapsing all such edges to points, let $\pi : X \longrightarrow X'$ be the projection, and let $f' : X' \longrightarrow M$ be the nondegenerate carrier graph with $f = \pi \circ f'$. If there are entire loops in $X$ that are mapped to points in $M$, then $f'$ will not be equivalent to $f$. However, if $f$ has minimal length, then $f'$ does too, since any equivalent carrier graph can be blown up to a carrier graph with the same length that is equivalent to $f$. So, in the following, \emph{whenever we consider minimal length carrier graphs, we will always assume they are nondegenerate.}

White \cite {Whiteinjectivity} has shown\footnote{Technically, White has a standing minimal rank assumption in his paper, but his argument here is just that if $X$ isn't trivalent with $2\pi/3$ angles in the image, there is a length decreasing perturbation of $f$. So, this result applies to any (nondegenerate) minimal length $f$.} that if $f: X \longrightarrow M $ is a  minimal length carrier graph, then $X$ is trivalent and the edges of $X$ map to geodesic segments in $ M$ that connect at $2\pi/3$-angles.  He used this to show that any minimal length carrier graph has a cycle with total edge length less than some constant depending on $\rank(\pi_1 X)$. In fact, we will see that minimal  length carrier graphs  have a lot of extra structure.

\begin {defi}[Relatively connected]\label {relatively connected}
 A carrier graph $f: X \longrightarrow M$ with geodesic edges is $D$-\emph{relatively connected} if whenever $Y\subset X$  is a proper subgraph, there is an edge $e\subset X \setminus Y$ with $
\length^f_{Y}(e) \leq D.
$
\end{defi}

Here, $\length^f_{Y}(e)$ is the \emph{length of $e$  relative to $Y$}, which we will define below. Informally (and  slightly inaccurately) it is the length of the part of $e$ that lies outside the convex hull of $Y$ in $M$. 

Extending earlier results of ours in \cite{Biringergeometry,Soutorank}, we prove:

\begin{sat}
\label{chainsprop}
Fixing $\epsilon>0 $, there is  some constant $D=D(\epsilon) $ as follows. Let $M$ be a  complete $\epsilon$-thick hyperbolic $3$-manifold  and assume that every essential closed curve on $\partial CC(M)$ that is nullhomotopic in $M$ has length at  least $\epsilon$. Then  any minimal length carrier graph $f: X \longrightarrow M$ is $D$-relatively connected. \end{sat}

%Here, the conclusion of Theorem \ref{chainsprop} means the following.
%First, though,   we should mention that Theorem \ref{chainsprop} is  a bit different from that in \cite{Biringergeometry}. There, we give a chain of nested subgraphs of $X$, where all edges of one subgraph have bounded length relative to the previous subgraph, but the two  versions are easily equivalent. 

 We should mention that there is also a  relative version of  Theorem~\ref{chainsprop} that works without the $\epsilon $-thick assumption,  compare with the version in  the appendix of \cite{Biringerranks}, but we will not need it in this paper. 
 
\medskip

We now give  the definition of  relative length. Let $f : X \longrightarrow M $ be a carrier graph with geodesic edges in a hyperbolic $3$-manifold $M $. 

\begin{defi}
A \emph{rooted edge} of $X$ is a pair $(e,v)$  where $e$  is a directed edge of $X$ and $v$  is the initial vertex of $e$.
\end{defi}

 For the most part, one can ignore the direction of $e$. We only mention it since  our carrier graphs can have edges that are closed loops, and in a rooted edge we want to specify a `side' of $e$ that is attached to $v$.

Given a rooted edge $(e,v)$ and a subgraph $Y \subset X$, we want  to define a number called  the \emph{convex hull length}
$$CHL^f_Y(e,v) \in [0,\infty).$$
This will be zero if $v\not \in Y$, and if $v \in Y$, it will informally measure the amount of length $f(e)$ expends near the convex hull of $f(Y) \subset M$ just after leaving the vertex $v$.
However, to make this precise we need to pass to universal covers. So, assume $v\in Y$ and let
$$ \pi : \tilde X \longrightarrow X, \  \  \pi : \BH^3\longrightarrow M,  \ \ \tilde  f: \tilde X \longrightarrow \BH ^ 3, \ \ \pi \circ \tilde f = f \circ \pi$$ be the appropriate universal covers and lifts. Lift $e $ to an edge $\tilde e$ in $\tilde X$ and let $\tilde v$ be the initial endpoint of $\tilde e$,  with respect to the lifted orientation.  If the component of $\pi^{-1}(Y)$ containing $\tilde v$ is $$\tilde Y_0 \subset \pi^{-1}(Y) \subset\tilde X,$$
we define the convex hull length of $(e,v)$ to be 
\begin{equation}CHL^f_Y(e,v) := \length \left (\tilde f(\tilde e) \cap \CN_1(CH(\tilde f(\tilde Y_0)))\right).\label {chl}	
\end{equation}

Here, $\CN_1(\, \cdot \, )$ denotes the $1$-neighborhood in $\BH^3$. Note that the intersection above is a subsegment of $\tilde f(\tilde e)$  starting at the vertex $\tilde v$, and that $CHL^f_Y(e,v)$ does not depend on any of the choices of lifts.

\begin{defi}[Relative length]
\label {relative length} The \emph{length of $e$ relative to $ Y$} is 
$$\length^f_{Y} (e) \ =\ \max \Big \{0, \   \length (e) \  - \sum_{\substack{\text{ vertices } \\ v \text{ of } e}} CHL^f_Y(e,v)\Big \}.$$ 
  \end{defi}
  
There are only two terms in the summation.   Recall that each term $CHL^f_Y(e,v)$ is the length of some subsegment of $f(e)$. These segments may overlap, in which case the difference above will be negative.  This is the reason for the maximum. Note that as defined, relative length is  then just the length of the part of $f(e)$ that is not contained in either of the subsegments defining convex hull length.

 Before beginning the proof of Theorem \ref{chainsprop}, note that  the definition of  convex hull length (and hence, that of relative length) would change drastically if we used the convex hulls themselves  instead of their $1$-neighborhoods in \eqref{chl}. For taking the convex set  to be a plane, a geodesic can run for a long time in the $1$-neighborhood before intersecting the plane. However, the  following standard lemma shows that convex hull length  only changes by a uniform constant if the radius is changed from $1$ to another fixed constant.

\begin{lem}[Changing the radius]\label{convlem1}
If $C$  is a convex set in $\BH^3$ and $\gamma$  is a geodesic segment contained in $\CN_b(C) \setminus \CN_a(C)$, where $b>a$, then
$$\length \gamma \leq \frac {2b}{1-\frac 1{\cosh a}}.$$
\end{lem}
\begin{proof}
Since $\gamma \subset \BH^3 \setminus \CN_a(C)$,  we have $$\length \gamma \geq \cosh a \cdot \length \pi(\gamma), $$ by Fact \ref{expdecrease}, where $\pi$ is the nearest point  projection to $C$. But  also $$\length \gamma \leq 2b+\length \pi(\gamma),$$ since $\gamma$ minimizes the distance between its endpoints. Putting these two estimates together proves the lemma.
\end{proof}

\subsection{Proof of Theorem \ref{chainsprop}}
 The proof is  somewhat similar to that in the appendix of \cite{Biringerranks}. Before starting the proof proper, we identify the  constant $D=D(\epsilon)$  in the statement of the theorem. It will depend on two constants, $L$ and $A$, which are described in the two claims below.

\begin{claim}[Long segments, bounded angles]\label {long segments}
Given $\epsilon$, there is some $L=L(\epsilon)>0$  as follows.	 Suppose that $\gamma$ is a path in $\BH^3$  that is a  concatenation of  geodesic segments $\gamma_i = [p_i,p_{i+1}]$, where $i = 0,\ldots,n-1$, such that for all $i=0,\ldots,n-2$,  we have
\begin {enumerate}
\item  the angle made by $\gamma_i$ and $\gamma_{i+1}$ at $p_{i+1}$ is at least $\pi/3$, 
\item $\length \gamma_i \geq L$.
\end {enumerate}
Let $P\subset\BH^3$ be the plane through $p_0$  that is perpendicular to $\gamma_0$,  and let $$\pi_P : \BH^3 \longrightarrow P$$  be the  closest point projection. Then we have $$d(p_0,\pi_P(p_n)) < \sinh^{-1}(\frac{\epsilon}{2\pi}) \ \text{ and } \ \gamma \cap P = \{p_0\}.$$
\end{claim}

Note that $i=0,\ldots,n-2$ in (2),  as well as in (1). So, it is not a problem if the \emph{last} segment we are concatenating is short. Also,  $\pi/3$ and $\sinh^{-1}(\frac{\epsilon}{2\pi})$  could be replaced by arbitrary positive constants, at the expense of increasing $L$. We phrase it as above  just to cut down on notation.

\begin {proof}
 For each $i=0,\ldots,n-2$, let $P_i$  be the plane that perpendicularly bisects  the segment $\gamma_i$. A hyperbolic geometry exercise shows that if $L$ is at least some  constant depending on $\alpha$, then each $P_i$  is disjoint from $P_{i+1}$, when defined.   Similarly, we can assume that $\gamma_{n-1}$ is disjoint from $P_{n-2}$.  It follows that each $P_i$  intersects the path $\gamma$  exactly once. 

Now if if $L$  is  also at least some constant depending on $\epsilon$, then $$\diam(\pi_P(P_{0})) \leq \sinh^{-1}(\frac{\epsilon}{2\pi}).$$ 
The first part of the claim follows since $P_0$  separates $p_n$ from $P$. For the second part,  just note that $P_0$  separates every point on $\gamma \setminus \gamma_0$ from $P$, so $(\gamma \setminus \gamma_0) \cap P = \emptyset$, and  certainly $\gamma_0 \cap P =\{p_0\}.$ \end {proof}

We also need the following, which we leave to the reader. 
\begin{claim}\label {2ndconstants}
There is some $A=A(\epsilon)>L$ such that
\begin {enumerate}
\item[(i)] if an isometry $g : \BH^3 \longrightarrow \BH^3$  has translation  length at least $\epsilon$,  we have that $d(g(x),x) \geq L$ for every point $x\in \BH^3$ that lies at least $A$ away from the axis of $g$, and
\item[(ii)]any geodesic ray emanating from a convex subset $K \subset \BH^3$  that leaves $\CN_A(K)$ meets  $\partial \CN_A(K)$ in an angle of at least $\pi/3$.
\end {enumerate}
\end{claim}

 To prove Theorem \ref{chainsprop},  it would suffice to show that if $D= 2L$ and $$X\longrightarrow M$$  is a minimal length carrier graph that is not $D$-relatively connected, then there is  an essential closed curve on $\partial CC(M)$ with length at most $\epsilon$ that is nullhomotopic in $M $. We will do this, but \emph{whenever we speak of convex hull length or relative length below, we will use the constant $A$ instead of $1$ in \eqref{chl}.}  One can then use Lemma \ref{convlem1} to translate this back into a statement about relative connectedness with the usual definition, and the $D$ produced will differ from $2L$ only by a constant depending on $A$, which depends only on $\epsilon$.

\medskip

So,  assume that $Y\subset X $  is a proper subgraph such that every edge $e \subset X \setminus Y$ has relative length
$\length_Y(e) \geq D.$ (Again, we are using $A$  instead of $1$  in the definition of convex hull length \eqref{chl}  throughout the rest of the proof.) In a further abuse of power, \emph{we will assume throughout the proof that our carrier graph is  embedded in $M$,} just so that we  can cut down on notation by not naming the map.  Let
$$  \pi : \BH^3\longrightarrow M$$ be a  universal covering map, and let $$\tilde X = \pi^{-1}(X), \ \ \tilde Y = \pi^{-1}(Y).$$
 So, if $(e,v)$  is a directed edge of $X \setminus Y$ with $v\in Y$, we compute its convex hull length by lifting it to a directed edge $(\tilde e,\tilde v)$ in $\tilde X$, letting $\tilde Y_0 \subset \tilde Y$ be the component  containing $\tilde v$,  and setting
\begin{equation}CHL_Y(e,v) = \length \left (\tilde e \cap \CN_A(CH(\tilde Y_0))\right).\label {chl2}	
\end{equation}
 Then $\length_Y(e)$ is defined as in Definition \ref{relative length}.

\medskip

Fix an edge $e_0 \subset \tilde X \setminus \tilde Y$. Since $\length_Y(\pi(e_0))\geq D=2L$, there is some $p \in e_0$ such that  for each component $\tilde Y_0 \subset \tilde Y$  containing a vertex of $e_0$, the point $p$ lies at a distance of at least $L$ along $e_0$ from $\CN_A(CH(\tilde Y_0))$. So for instance, we can take $p$ to be the midpoint of the subsegment of $e_0$ whose length defines $\length_Y(\pi(e_0))$. We claim:

\begin{claim}\label {projectionclaim}
If $P \subset \BH^3$ is the plane through $p$  that is perpendicular to $e_0$,  and $\pi_P : \BH^3 \longrightarrow P$ is the orthogonal projection, then $$\pi_P(\tilde X) \subset B:= B_P\left (p,\sinh^{-1}\big (\frac{\epsilon}{2\pi}\big )\right ) \ \text{ and } \  P \cap \tilde X = \{p\}.$$
\end{claim} 

 Postponing the proof of the claim for a moment, we finish the proof of Theorem \ref{chainsprop}.  Since $\pi_P$ takes (unparameterized) geodesics to geodesics and the disk $B \subset P$  is convex,  the preimage $\pi_P^{-1}(B) $ is a  convex subset of $\BH^3$ that contains $\tilde X$. Hence, we have  a closed curve $$c:= \partial  B \subset \BH^3 \setminus CH(\tilde X), \ \ \length c = 2\pi \sinh(\sinh^{-1}\big (\frac{\epsilon}{2\pi}\big )) = \epsilon.$$

There is a bi-infinite edge path in $\tilde X$ that  passes through $e_0$  exactly once, and whose inclusion in $\BH^3 $ is proper. (If not, one  of the two components $Z \subset \tilde X \setminus int(e_0)$ is compact, and we can create a shorter equivalent carrier graph by collapsing $Z \cup e_0$ to a point and propogating the change $\pi_1 M$-equivariantly,  contradicting that $X \longrightarrow M$ has  minimal length.) This path intersects $B$  exactly once, so $c=\partial B$ cannot bound a disk in the complement of this path,  implying that $c$ is  homotopically  essential in $\BH^3 \setminus CH(\tilde X)$.

The convex core $CC(M)$ is  the quotient of the convex hull in $\BH^3$ of the limit set $\Lambda(M) \subset S^2_\infty.$ Since $X$  is a carrier graph for $M$, and is compact, we have the two inclusions
$$CH(\Lambda(M)) \subset CH(\tilde X) \subset \CN_r(CH(\Lambda(M)))$$
for some $r>0$.  It follows that $\BH^3 \setminus CH(\Lambda(M))$  deformation retracts onto $\BH^3 \setminus CH(\tilde X)$, since we can map $x \in CH(\tilde X) \setminus CH(\Lambda(M))$ to the unique point in $\partial CH(\tilde X)$ that lies along the ray through $x$ that emanates perpendicularly out from $\partial CH(\Lambda(M))$. So, $c$ is homotopically essential in $\BH^3 \setminus CH(\Lambda(M))$ as well.  The projection $\pi(c)$ is then a loop in $M$ with length $\epsilon$ that is essential in $M \setminus CC(M)$, but not in $M$. Projecting $\pi(c)$ onto $\partial CC(M)$ gives an essential closed curve on $\partial CC(M)$ with length less than $\epsilon$ that is nullhomotopic in $M$, contradicting our assumption in Theorem \ref{chainsprop} that there is no such loop.

\begin{proof}[Proof of  Claim \ref{projectionclaim}]
Let $q\in \tilde X$  be a point that does not lie on the edge $e_0$.  We will construct a path $\gamma'$ from $p$ to $q$ that is a concatenation of geodesic segments connected with angles at least $\pi/3$, such that 
\begin{enumerate}
	\item  except possibly for the segment incident to $q$, all our geodesic segments have length  at least $L$,
\item the segment incident to $p$ is a subsegment of $e_0$.
\end{enumerate}
Claim \ref{projectionclaim}  will then follow from our definition of $L$ in Claim \ref{long segments}.

\medskip

Since $X\longrightarrow M$  is $\pi_1$-surjective, $\tilde X$  is connected. Hence,  there is an embedded path $\gamma$ in $\tilde X$ from $p$ to $q$. Note that $\gamma$  certainly starts at $p$ with a subsegment of $e_0$. We will  modify $\gamma$ to give a path $\gamma'$  as above.

Index the maximal subpaths of $\gamma$ that lie in $\tilde Y$ as $\gamma_i \subset \gamma$, where $\gamma _i$ lies  in the component $\tilde Y_i \subset \tilde Y$, and set $$C_i := \CN_A(CH(\tilde Y_i)).$$ 
For each $i$, let $\gamma_i^{ext} \supset \gamma_i$ be the maximal subpath contained in $C_i$, and let $x_i,y_i$ be  the endpoints of $\gamma_i^{ext}.$ As $i$ varies, the paths $\gamma_i^{ext}$ are disjoint:  if two intersect, they will do so along a lift of an edge $e \subset X \setminus Y$ with $\length_Y(e) =0$, contradicting  our choice of $Y$.  

 Create $\gamma'$ from $\gamma$ by replacing every subpath $\gamma_i^{ext}$ with the geodesic $[x_i,y_i] \subset\BH^3$ connecting its endpoints.   The result is a concatenation of geodesic segments of the following two types, which appear in alternation as one traverses $\gamma'$:
\begin{enumerate}
	\item $(e \cap \gamma)\setminus \cup_{i \in I} C_i$, where $e $ is an edge of $\tilde X \setminus \tilde Y$ and $I$ is the set of  (at most two) indices of components $\tilde Y_i$  that are adjacent to $e$,
\item  segments $[x_i,y_i] \subset C_i$  constructed above.
\end{enumerate}
In (1), note that the intersection $e \cap \gamma=e$  unless $e$ contains  one of the endpoints $p,q$ of $\gamma$,  in which case $e \cap \gamma\subset e$  is some subsegment.  
Before the statement of Claim \ref{projectionclaim}, we chose $p$ to lie at a distance of at least $L$  along $e_0$ from any $C_i$, so $\gamma'$  begins at $p$ with a type (1) subsegment of $e_0$  of length at least $L$. 
We want to show that 
\begin{itemize}
\item all our segments above connect  with angles at least $\pi/3$, and
\item  all the rest of them (except possibly the one incident to $q$) also have length at least $L$.
\end{itemize}

The type (1) segments where $e=e\cap \gamma$ are exactly those  we use to compute relative length, so each of these has length at least $D\geq L$ by our  assumption that there every edge $e \subset X \setminus Y$ has $\length_Y(e)\geq D$. We have already dealt with the segment  incident to $p$, and we do not care about the length of the segment incident to $q$.

Now suppose there is a type (2) segment with length less than $L$, so for some $i$,  we have $d(x_i,y_i) < L.$ Assume that $[x_i,y_i]$ is not incident to $q$. (It can't be incident to $p$, as explained above.) Then $x_i,y_i\in \partial C_i$, since $\gamma_i^{ext}$ was a \emph{maximal}  segment of $\gamma$  contained in $C_i$.   It follows that $x_i,y_i$ lie on edges $e_i,f_i \subset \tilde X \setminus \tilde Y$,  respectively, and these edges have endpoints $x_i',y_i'$ that lie on $\tilde Y_i \subset \tilde Y$. See Figure \ref{surgery!}.

\begin{figure}
\centering\includegraphics{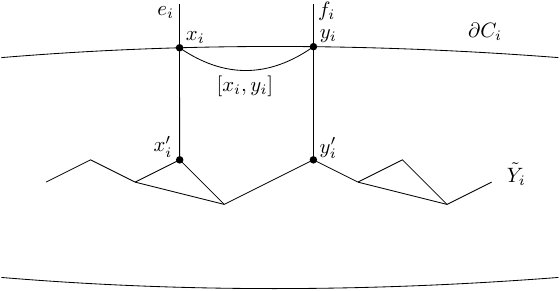}
\caption{}
\label{surgery!}
\end{figure}

Let $[x_i,x_i'] \subset e_i$ and $[y_i,y_i'] \subset f_i$ be the segments with the indicated endpoints. We claim that in $M$, we have $$\pi([x_i,x_i'])\cap  \pi([y_i,y_i']) = \emptyset.$$ If not, there is an element $\gamma \in \Gamma$ with $\gamma([x_i,x_i']) \cap [y_i,y_i']\neq \emptyset$, where $\Gamma $ is the deck group of the covering map $\pi : \BH^3 \longrightarrow M$. In particular, $\gamma(e_i)=f_i$. There are then two cases, depending on the orientation with which $\gamma$ identifies $e_i$ and $f_i$.

First, assume that $\gamma(x_i')=y_i'$. Since the components of $\tilde Y$ are permuted by the $\Gamma$-action, and $x_i',y_i' \in \tilde Y_i$, we have $\gamma(\tilde Y_i)=\tilde Y_i$, and hence the axis of $\gamma$ lies inside the convex hull of $\tilde Y_i$. So, Claim \ref{2ndconstants} (i) implies that the translation distance of $\gamma$ is at least $L$ on $\partial C_i$. However, $\gamma$ preserves $\partial C_i$, so $\gamma(x_i) \in f_i \cap \partial C_i$, and we must have $\gamma(x_i)=y_i$. Since $d(x_i,y_i)<L$ we have a contradiction. 

In the other case, we have $\gamma(e_i)=f_i$, but $\gamma(x_i')$ is the other endpoint of $f_i$, not $y_i'$. By $\Gamma$-invariance of $\tilde Y$, the image $\gamma([x_i,x_i'])$ lies in the $A$-neighborhood of some \emph{other} component of $\tilde Y$. Since we're assuming $\gamma([x_i,x_i'])$ intersects $[y_i,y_i']$, the entire segment $f_i \setminus [y_i,y_i']$ has to be contained in $\gamma([x_i,x_i'])$. But then $f_i$ actually has zero length relative to $\tilde Y$, contrary to our standing assumption.

We have now shown that $\pi([x_i,x_i'])\cap  \pi([y_i,y_i']) = \emptyset$. So, we can replace $[x_i,x_i']$ with the geodesic $[x_i,y_i]$ and propagate this change equivariantly to create a new carrier graph that is equivalent to $X$. Then $$\length [x_i,y_i] < L < A \leq \length [x_i,x_i'].$$ So the new carrier graph is shorter, contradicting  the fact that $X$ has minimal length. 

 Finally, we must check that all the angles in $\gamma'$ are at least $\pi/3$. There are two types of angles: those at vertices of $\tilde X$, and those at the endpoints of the segments $[x_i,y_i]$. Since $X$  has minimal length, the first  kind are all equal to $2\pi/3$, by an observation of White \cite{Whiteinjectivity}.  If some endpoint, say $x_i$, of some $[x_i,y_i]$ is not also an endpoint of $\gamma$, then it is adjacent in $\gamma'$ to a subsegment of an edge $e \subset \tilde X \setminus \tilde Y$ with one vertex on $\tilde Y_i$. By Claim \ref{2ndconstants} (ii), $e$ makes an angle of at least $\pi/3$ with $\partial C_i$, and hence an angle of at least $\pi/3$ with $[x_i,y_i]\subset C_i$.
\end {proof}

\subsection{Further observations about relative connectivity}
Here is another way to phrase the definition of convex hull length, without passing all the way to the universal cover. Suppose $(e,v)$  is a rooted edge and $v$  lies in a component $Y_0\subset Y$. Let $$\pi : \hat M \longrightarrow M$$ be the cover corresponding to the subgroup $f_*(\pi_1 Y_0) \subset \pi_1 M$. Let $$Y_0' = Y_0 \cup_{v} e$$
 be the graph obtained by attaching $e$ to $Y_0$ at $v$, and lift $f | _{Y_0'}$ to 
$$\hat f : Y_0' \longrightarrow \hat M.$$ Let  $C \subset \hat M$ be the smallest convex subset of $\hat M$ containing $\hat f( Y_0)$. Here, we say that a subset of $\hat M$ is \emph {convex} if every geodesic segment with endpoints in the subset lies in the subset. So  for instance, $C$  contains the convex core of $\hat M$.   Then  we have that
\begin{equation}
CHL^f_Y(e,v)= \length ( \hat f(e) \cap \CN_1(C)). \label{rellength2}
\end{equation}

 To reconcile this with the previous definition, just note that the convex hull in \eqref{chl} projects to $C$ under the covering map $\BH^3 \longrightarrow \hat M$.

\vspace{2mm}

 We will need the following consequence of relative connectedness during the main construction in this paper, see \S \ref{short graph subsection}.

\begin {lem}Suppose that $f : X \longrightarrow M$ is a	$D$-relatively connected carrier graph with $k$ edges in a  hyperbolic $3$-manifold $M\neq \BH^3$  with injectivity radius at least $\epsilon>0$. Then the image $f(X)$ lies in a $D'=D'(D,k,\epsilon)$-neighborhood of the convex core $CC(M)$.\label{cgincc}
\end {lem}
\begin {proof}
For each $i\leq k$, we'll show that for every $f$ as above, there is a subgraph $Y \subset X$ with $i$ edges such that for some $D_i'=D_i'(D,i,\epsilon),$ \emph{every edge $e\subset Y$ with $f_i(e) \not \subset \CN_{D_i'}(CC(M))$ is contained in a component $Y_0 \subset Y$ such that $\length f(Y_0) \leq D_i'$ and the restriction $f |_{Y_0}$ is null-homotopic.} Taking $i=k$ will prove the lemma, since $X$ is connected and $f$ is not null-homotopic.

 The base case $i=0$ is trivial. So, assume the above holds for some $i$ and let $f : X \longrightarrow M$ be as in the statement of the lemma. Pick a subgraph $Y \subset X$ with $i$ edges as given by  the induction hypothesis, and pick an edge $e \subset X \setminus Y$  with $\length^f_{Y}(e) \leq D.$  We want to show that $Y \cup e$ is a subgraph of $X$ with $i+1$ edges that satisfies our claim.

First, if $e $ does not share any vertices with $Y$ then $\length e \leq D$ and we are done if $D_{i+1}'$ is set to be bigger than $D$ and $D_i'$.

Next, assume that $e$  is adjacent to two components $Y_0 \subset Y$ that are not contained in $\CN_{D'_i}({CC(M)})$. Lift $Y_0$ to $\tilde Y_0 \subset \tilde X$, as in Definition \ref{relative length}. Since $f
_{Y_0}$  is null-homotopic and $\length f(Y_0) \leq D_i'$, the convex hull of $\tilde f(\tilde Y_0) \subset \BH^3$ has diameter at most $D_i'$. Hence, at most a segment of $\tilde e$ of length $D_i' + 1$   can lie inside $\CN_1(C_{\tilde Y_0})$. This shows that $\length e$ is bounded above by $D + 2(D_i'+1)$. So, if $Z \subset Y \cup e$ is the subgraph consisting of $e$ and the two adjacent components of $Y$,  then $\length (Z )$  is at most some $D_{i+1}'$  depending on $D,D_i'.$ If the  restriction $f |_{Z}$ is null-homotopic, we are done. If not, there is a simple cycle $\gamma$ in $Z$ whose image is essential in $M$. This $\gamma$ has length at most $D_{i+1}'$, so by Lemma \ref{distancetogeodesic} the distance from $f(\gamma)$ to $CC(M)$ is bounded above by some constant depending on $D_{i+1}'$ and $\epsilon$.  So, $f(Z)$ lies in a  bounded neighborhood of $CC(M)$ and we are done after increasing $D_{i+1}'$.

Next, assume that $e$ is adjacent to two components $Y_0 \subset Y$  that are contained in $\CN_{D'_i}({CC(M)})$. It follows from the definition of relative length that at most a segment of $e$ of  length $D$ can be contained outside $\CN_{D'_i}({CC(M)})$. Hence, we are done with $D_{i+1}'=D_i' + D$.

We leave the mixed case---where exactly one vertex of $e$ lies on a component of $Y$ that is not contained in $\CN_{D'_i}({CC(M)})$---to the reader as the argument is a combination of the previous two paragraphs.
\end {proof}

\begin {figure}
\centering 
\includegraphics{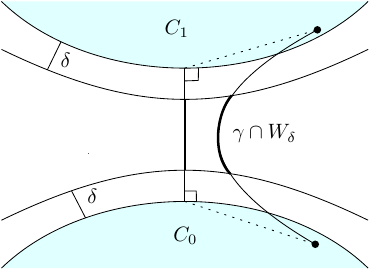}
\caption {The length of $\gamma$  outside a neighborhood of convex sets $C_0,C_1$ containing its endpoints is approximately the distance between $C_0$ and $C_1$.}
\label {lengthquad}	
\end {figure}

 In fact, using the notation of Definition \ref{relative length}, the  following lemma implies that the length of $e$  relative to $Y$ does not depend so sensitively on the geodesic $\tilde e$, but just on the positions of the adjacent components of $\pi^{-1}(Y) \subset\BH^3$.
 
   \begin {lem}
   \label{convexlem}
   Let $C_0,C_1$ be convex sets in $\BH^3$, let $\gamma$  be a geodesic with endpoints $v_i\in C_i$, where $i=0,1$, and let $$W_\delta =  \BH^3 \setminus (\CN_\delta(C_0) \cup \CN_\delta(C_1) ) .$$
   Then $\length (\gamma \cap W_\delta)$ is within some   $D=D(\delta)$ of $d_{\BH^n}(C_0,C_1).$\end{lem}
  \begin {proof}
  Let $\alpha$ be a geodesic minimizing the distance between $C_0$ and $C_1$, and form the geodesic quadrilateral where $\alpha,\gamma$ are opposite sides. The other two sides  are contained in $C_0$ and $C_1$, respectively, so as geodesic quadrilaterals in $\BH^3$ are $4$-thin, we have that $$\alpha \cap W_8 \subset \CN_4(\gamma \cap W_4), \ \ \gamma \cap W_8 \subset \CN_4(\alpha \cap W_4).$$
But by Lemma \ref{convlem1}, the  intersections of $\alpha $ and $\gamma$  with $ W_4 \setminus W_8$ have length at most $17$, so this means that the Hausdorff distance between $\alpha \cap W_8$ and $\gamma \cap W_8$  is at most $2 \cdot 17 +4=38$. See Figure \ref{relative length}. So,\begin{align*}
	38 \cdot 2 & \geq 	|\length (\gamma \cap W_8) -\length (\alpha \cap W_8)| \\
	&\geq 	|\length (\gamma \cap W_8) - d(C_0,C_1) - 16 |,
	\end{align*}
 which proves the theorem for $\delta=8$, and the result for arbitrary $\delta$ follows  from  Lemma \ref{convlem1}.
  \end {proof}
 
%  
% \begin{lem}\label{convexlem}
% Let $C$  be a convex set in $\BH^3$ and let $\alpha,\beta$ be geodesic segments in $\BH^3$, where  the Hausdorff distance $d_{Haus}(\alpha,\beta)\leq D.$
%  Then    $$|\length (\alpha \cap \CN_1(C)) - \length (\beta \cap \CN_1(C))| \leq L=L(D).$$
% \end{lem}
%
%
%\begin {proof}
%Let $x,y$  be the endpoints of  $\alpha $ and let $\pi : \BH^3 \longrightarrow C $ be the  nearest point projection. It  follows from $\delta$-hyperbolicity and the exponential distortion of $\pi$ away from $C$ that for some uniform  $L'>0$,  $$|\length (\alpha \cap \CN_1(C)) -d_{\BH^3}(\pi(x),\pi(y))| \leq L'$$ See Figure \ref{lengthquad}. Since the endpoints of $\alpha $ are within $D$ of the endpoints of $\beta$ and $\pi$ is $1$-lipschitz, it is obvious that $d_{\BH^3}(\pi(x),\pi(y))$ differs by at most $2D$ from the corresponding quantity associated to $\beta$.
%\end {proof}

 The following corollary of Lemma \ref{convexlem}  will be used in \S \ref{shortgraphssec}.
 
\begin {kor}[Edge homotopies]\label {edge homotopies}
Suppose that $M$  is a complete hyperbolic $3$-manifold, $f : X \longrightarrow M$ is a carrier graph with geodesic edges, $Y \subset X$  is a subgraph and $e=(v,w) \subset X \setminus Y$, with $v\in Y$. Let 
$$g : X \longrightarrow M$$
 be a carrier graph such that $g=f$ on $X \setminus e$ and $g|_e$ is a geodesic that is homotopic rel endpoints to $f|_{\gamma  \cdot e}$, where $\gamma \cdot e$ is the concatenation of a loop $\gamma$ in $Y$ based at $v$ and the edge $e$. Then for some uniform $D$, 
$$|\length_Y^f(e) - \length_Y^g(e)| \leq D.$$Moreover, we also have that 
$$|(\length f(e) - CHL_Y^f(e,v)) - (\length g(e) - CHL^g_Y(e,v))| \leq D. $$
\end{kor}
\begin {proof}
For concreteness, let's assume that $w\in Y$ as well. The other case is easier. Suppose $\pi : \tilde X \longrightarrow X$ is the universal cover, that an edge $\tilde e \subset \tilde X$ has vertices $\tilde v,\tilde w$, and $\tilde Y_v,\tilde Y_w$ are the component of $\pi^{-1}(Y)$ that contains $\tilde v,\tilde w$. Construct lifts $\tilde f,\tilde g: \tilde X \longrightarrow \BH^3$ that agree on $\tilde v$. Then $\tilde g = \tilde f$ on $\tilde Y_v$, while $\tilde g = \gamma \circ \tilde f$ on $\tilde Y_w$, where abusively we identify the loop $\gamma$ with the deck transformation of $\BH^3 \longrightarrow M$ that takes $\tilde f(\tilde v)$ to the terminal endpoint of the lift of that loop starting at $\tilde f(\tilde v)$. In particular, since $\gamma$ is a loop in $Y$, we have an equality of images $\tilde g(\tilde Y_w) = \tilde f(\tilde Y_w)$. So, the  first part of corollary follows from Lemma \ref{convexlem} with $\delta=1$.  The second part also follows from Lemma~\ref{convexlem} if we take $C_0$ to be the convex hull of $\tilde Y_v$  and $C_1=\{\tilde f(\tilde w)\}.$
\end {proof}

\section{Algebraic and geometric limits} \label {convergence section}
 In this section, we discuss algebraic and geometric limits of sequences of representations $\rho_n : \Gamma \longrightarrow \PSL_2\BC$.  The results  here will be applied to representations of free groups coming from bounded length graphs in  hyperbolic $3$-manifolds.  So while the majority of the available literature on algebraic and geometric limits concerns \emph{faithful} representations, it is crucial that our results hold more generally.

 There are three parts below.  First, we recall some well-known facts about geometric and algebraic convergence, and describe how to extract such limits from sequences of hyperbolic $3$-manifolds with bounded length carrier graphs. Second, we recall and extend some of our work from \cite{Biringeralgebraic}, where we compared algebraic and geometric limits of sequences of (unfaithful) representations and gave a criterion that ensures they coincide, i.e.\  that the convergence is \emph{strong}.  Finally,  we use techniques from \cite{Biringeralgebraic} and  Theorem~\ref{double compression body theorem} to   understand quite precisely the topology of strong limits.

\subsection{Geometric convergence}

A sequence $(G_i)$ of closed subgroups of $\PSL_2\BC$ converges {\em geometrically} to a subgroup $G < \PSL_2\BC$ if every $g\in G$ is the limit of a sequence $g_i\in G_i$, and every accumulation point in $\PSL_2\BC$ of every sequence $g_i \in G_i$ lies in $G$. The induced topology on the space of closed subgroups of $\PSL_2\BC$ is called the \emph{Chabauty topology}, see e.g.\ \cite[Ch E]{Benedettilectures} or \cite{Canarynotes}. It is well-known that the Chabauty topology is compact and metrizable, see e.g.\ \cite[Lemma E.1.1]{Benedettilectures}.

Most of our arguments are based on an interpretation of geometric convergence in terms of the quotient manifolds $\BH^ 3 / G_i $. Below, a \emph{pointed hyperbolic manifold} is a manifold together with a base point.

\begin{defi}\label {geolimits}
A sequence $(M_i,p_i)$ of pointed hyperbolic 3-manifolds converges {\em geometrically} (or \emph{smoothly}) to a pointed manifold $(M_\infty,p_\infty)$ if there is a sequence $(K_i)$  of compact sets with $K_i \subset int(K_{i+1})$ and $\cup_i K_i=M_\infty$ such that for large $i$, there are smooth embeddings $$\phi_i:K_i \longrightarrow M_i, \ \ \phi_i(p_\infty)=p_i, $$ such that $(\phi_i)^* g_i \to g_\infty$ in the $C^\infty$ topology, where $g_i$ is the Riemannian metric tensor on $M_i$.  We will refer to the maps $\phi_i$ as the {\em almost isometric embeddings provided by geometric convergence}. 
\end{defi}

When $A \subset M_\infty$  is compact, we will often talk about {the image $\phi_i(A)$,} understanding implicitly that this only makes sense for large $i$. Relatedly, in notation we will often suppress the domains $K_i$ of the $\phi_i$, since in applications the precise domains are unimportant, and we will almost always be restricting the $\phi_i$ to some specified compact set.

\begin{fact}[Bilipschitz embeddings]
Suppose $(M_i,p_i) \to (M_\infty,p_\infty)	$ geometrically and $(\phi_i)$ is the associated sequence of almost isometric embeddings. If $B \subset M_\infty$ is compact and $\lambda >1$ is given, then for large $i$, the restrictions $\phi_i |_B$ are $\lambda$-bilipschitz. 
\end{fact}
\begin{proof}
The fact that for $\lambda>1$, the $\phi_i$ are \emph {locally} $\lambda $-bilipschitz for large $i$ follows immediately from the $C^\infty$-convergence $(\phi_i)^* g_i \to g_\infty$. 
To upgrade this to a global statement, it suffices to assume that $B$ is a closed metric ball of some radius $R$ around $p_\infty$. Let $B'$ be the ball of radius $3R+1$ around $p_\infty$. The length of any path from $B$ to $\partial B'$ in $M_\infty$ is at least $2R+1$, while any two points in $B$ are joined by a path in $B$ of length at most $2R$. For any given $\lambda>1$, for large $i$ the restrictions $\phi_i|_{B'}$ are locally $\lambda$-bilipschitz, so as long as $\lambda$ is close to $1$, any near-minimal length path between two points of $\phi_i(B)$ has to be contained in $\phi_i(B')$. From this it follows that $\phi_i|_{B}$ is $\lambda$-bilipschitz, not just locally so.
\end{proof}

Fixing some $p \in \BH^3$, every discrete, torsion free subgroup $G < \PSL_2\BC$ gives a pointed hyperbolic $3$-manifold $M=\BH^3 / G$, where we take the projection of $p$ as the basepoint. This gives a surjection
$$\{ \text {discrete, torsion free } G \} \longrightarrow \{ \text{pointed } (M, p_M) \}/\text{pointed isometry},$$
 where the fiber containing $G $ consists of all groups $\gamma G \gamma^{-1}$, where  $\gamma $ is a hyperbolic isometry fixing $p$. We then have the following relationship between geometric convergence of groups and manifolds:

\begin{prop}\label{convergence-convergence-geom}
Let $G_1,G_2,\dots,G_\infty$ be discrete and torsion-free subgroups of $\PSL_2\BC$ and consider for all $i=1,\dots,\infty$ the pointed hyperbolic 3-manifold $(M_i,p_i)$ where $M_i=\BH^3/G_i$ and $p_i$ is the projection of a fixed point $p \in{\BH^3}$. Then the following hold. 
\begin{enumerate}
	\item If $G_i$ converges geometrically to $G_\infty$, then $(M_i,p_i)$ converges geometrically to $(M_\infty,p_\infty)$.
	\item If $(M_i,p_i)$ converges geometrically to $(M_\infty,p_\infty)$, then there are  elements $\gamma_i \in \PSL_2\BC$ with $\gamma_i(p)=p$ such that  the conjugates $\gamma_i G \gamma_i^{-1} $  converge geometrically to $ G$.
\end{enumerate}
Moreover, if $G_i \to G_\infty$ and $M_i \to M_\infty$ one can take the almost isometric maps $\phi_i$  provided by geometric convergence to satisfy
\begin {equation}
	(\phi_i)_*(\gamma)\to \gamma \text{ in } \PSL_2\BC, \label{compatibility}
\end {equation}
for any $\gamma \in G_\infty$. Here, for $i=1,2,\ldots,\infty$ we identify $G_i$ with $\pi_1(M_i,p_i)$ by taking $g\in G_i$ to  be the projection to $M_i$ of a path in $\BH^3$ from $p$ to $g(p)$, so  there is a map $(\phi_i)_* : G_\infty \longrightarrow G_i$  above.

\end{prop}

This result is well known, and something like it appears in Theorem E.1.13 of \cite{Benedettilectures}. To get the exact statement above, one can apply Lemma 3.7 and Proposition 3.10 of \cite{abert2017growth}, compare with Proposition 2.9 in \cite{Abertunimodular}, noting that the conjugating elements in (2) are necessary because here we are using pointed manifolds instead of framed manifolds.

We will also need the following well-known compactness criterion, which follows from the proposition above and \cite[Theorem E.1.10]{Benedettilectures}. 

\begin{fact}[Compactness]\label{geomcompact}
Suppose that $(M_i,p_i)$ is a sequence of complete hyperbolic $3$-manifolds such that $\inj_{M_i}(p_i)\geq \epsilon$, where $\epsilon >0$ is independent of $i$. Then there is a subsequence of $(M_i,p_i)$ the converges geometrically to a complete hyperbolic $3$-manifold $(M_\infty,p_\infty)$. Moreover, if all the $M_i$ are $\epsilon$-thick, so is $M_\infty$.
\end{fact}

For the `moreover' statement, one can either use the fact that injectivity radius is continuous with respect to geometric convergence of manifolds, see \cite{Erlichcontinuity}, or (more directly) one can view the injectivity radius criterion in the universal cover and use the perspective of geometric convergence of subgroups of $\PSL_2 \BC$.

It will be convenient below to be able to pull back minimizing SRSs in a geometric limit to minimizing SRSs in the approximating manifolds, at least after a small homotopy.

\begin{fact}[Pulling back minimizing SRSs]\label{pullback minimizing}
	Suppose that $(M_i,p_i) \to (M_\infty,p_\infty)$ geometrically, $(\phi_i)$ are the associated almost isometric embeddings,  $ f : (S,\CT) \longrightarrow M_\infty$ is a minimizing SRS with non-elementary $\pi_1$-image, and $U \subset M_\infty$ is a precompact open neighborhood of $f(S)$.
	 Then for large $i$, the map $\phi_i \circ f : (S,\CT) \longrightarrow M_i $ has non-elementary $\pi_1$-image, and is homotopic within $\phi_i(U)$ to a minimizing SRS.
\end{fact}
\begin{proof}
	First, we show that $\phi_i \circ f$ has non-elementary $\pi_1$-image. Pick closed curves $\alpha,\beta$ in $S$ such that $f(\alpha),f(\beta)$ are homotopic in $M_\infty$ to distinct closed geodesics $\hat \alpha,\hat \beta$. Let $\epsilon>0$ be less than a third of the Hausdorff distance between $\hat \alpha,\hat \beta$. For large $i$, the curves $\phi_i \circ f(\alpha) $ and $\phi_i \circ f(\beta)$ are homotopic to $\phi_i(\hat \alpha)$ and $\phi_i(\hat \beta)$, which are loop in $M_i$ with arbitrarily small geodesic curvature, that lie at Hausdorff distance bigger than $2\epsilon$ from each other. By Lemma 1 in the appendix of \cite{leininger2006small}, say, for large $i$ each of these can be homotoped to a closed\footnote{Technically, the cited reference is only about homotopies of paths rel endpoints, so only produces a path in $M_i$ that is geodesic except for one corner. However, one can apply the cited Lemma twice to say that the angle at the corner must be close to $\pi$ for large $i$, in which case the homotopic closed geodesic also lies at small Hausdorff distance.} geodesic in $M_i$ that lies within a Hausdorff distance of $\epsilon$ of the original loop. So, the Hausdorff distance between the two resulting closed geodesics is positive, and $\phi_i \circ f$ has non-elementary image.
	
	We now want to homotope $\phi_i \circ f$ to a minimizing SRS. Pick some $\delta>0$ such that $\CN_{10 \cdot \delta}(f(S)) \subset U$, say. First, we show how to homotope $ \phi_i \circ f$ in $M_i$ rel the vertices of $\CT$ to a pre-simplicial ruled surface $g_i$ via a homotopy  with tracks of length at most $4\delta$. 
	
	For each edge $e$ of $\CT$, the map $f|_e$ is a constant speed geodesic, so for large $i$ the map $\phi_i \circ f|_e$ is a $\lambda$-quasigeodesic with $\lambda \approx 1$. By Lemma 1 in the appendix of \cite{leininger2006small}, say, the image $\phi_i \circ f(e)$ lies in the $\delta$-neighborhood of the geodesic connecting its endpoints, and since $\phi_i \circ f|_e$ is a $\lambda$-quasigeodesic for $\lambda\approx 1$, we can even assume for large $i$ that the straight-line homotopy\footnote{Really, all of this discussion should happen in the universal cover of $M_i$.} from $\phi_i \circ f|_e$ to the constant speed geodesic joining its endpoints has tracks with length at most $2\delta$. Doing this for every edge defines a map $g_i : \CT^1 \longrightarrow M_i$ with constant speed geodesic edges that is homotopic to $\phi_i \circ f|_{\CT^1}$ via a homotopy with tracks of length at most $2\delta$.
	
	 Next, take a triangle $\Delta$ of $\CT$ with ruling vertex $v$ and a point $p$ on the opposite edge, and let $[v,p ] \subset \Delta$ be the geodesic from $v$ to $p$. Then $f|_{[v,p]}$ is a constant speed parameterized geodesic, so the same argument as above says that $\phi_i \circ f$ is homotopic rel endpoints to a constant speed geodesic, via a straight line homotopy with tracks of length at most $2\delta$. This geodesic can then be straight line homotoped with tracks of length at most $\delta$ to the nearby geodesic joining $g_i(v)$ to $g_i(p)$. Varying $\Delta,v,p$, we get a pre-simplicial ruled surface $g_i : (S,\CT) \longrightarrow M_i$ extending our previous definition of $g_i$ on $\CT^1$, that is homotopic to $\phi_i\circ f$ via a homotopy with tracks of length at most $4\delta$. 
	 
	 \medskip

Using Lemma \ref{minimizers}, find a homotopy $(g_i^t)$ with $g_i^0=g_i$ such that $g_i^1$ is a minimizing SRS in $M_i$, with respect to the same edge weights $\omega$ on the triangulation $\CT$ as with $f$. As in the proof of Fact \ref{uniquenessfact}, we may assume that the track $t \mapsto g_i^t(v)$ of each vertex $v$ of $\CT$ is a constant speed geodesic, and that each $g_i^t$ is presimplicial ruled, so that by convexity of the distance function of $M_i$, the function $t \mapsto L(g_i^t,\CT^1,\omega)$ is convex. In particular, since $g_i^1$ is minimizing, we have \begin{equation}
	L(g_i^t,\CT^1,\omega) \leq L(g_i,\CT^1,\omega) \leq L(\phi_i \circ f,\CT^1,\omega),\label{lengthboundhomotopy}
	\end{equation}
for all $t$, where the last inequality uses that on $\CT^1$, the map $g_i$ is obtained by straightening $\phi_i \circ f$ rel vertices.  

If for large $i$, the homotopies $g_i^t$ are all contained in the neighborhoods $\phi_i(U)$, we are done. Hoping for a contradiction, assume that for infinitely many $i$ this is not the case. If all vertices of $\CT$ move only a little bit under the homotopies $g_i^t$, then by convexity of the distance function so do all edges, and then so do all triangles. So under our assumption, for each $i$ there is some smallest $t_i \in [0,1]$ such that if $h_i := g_i^{t_i}$, we have $d(h_i(v),g_i(v)) \geq 5\delta$ for some vertex $v$ of $\CT$. By convexity of the distance function, our specially constructed homotopy $g_i^{t}$ has tracks of length at most $5\delta$ on $[0,t_i]$. In particular, this restricted homotopy is contained in the closed $9\delta$-neighborhood of $\phi_i\circ f(S)$, which is contained in $\phi_i\circ f(U)$ for large $i$. So, the maps $$\phi_i^{-1} \circ h_i : S \longrightarrow U \subset M_\infty$$ are well-defined. By \eqref{lengthboundhomotopy}, the fact that $h_i$ is presimplicial ruled, and the fact that $\phi_i$ is almost isometric, the maps $\phi_i^{-1} \circ h_i$ are uniformly lipschitz (say, with respect to an appropriate piecewise Euclidean metric) and all map into a bounded neighborhood of $f(S)$. So after passing to a subsequence, Arzela-Ascoli implies that they converge to a limit map $h : S \longrightarrow M_\infty$. The same Arzela-Ascoli argument also applies to the concatenation of the homotopies from $\phi_i \circ f$ to $g_i$, and from $ g_i$  to  $h_i$, since all these homotopies are constructed from straight-line homotopies of geodesics, and they all have tracks of small length. So, $h$ is homotopic to $f$ in $U \subset M_\infty$. Moreover, since the $\phi_i$ are almost isometric for large $i$, we have by \eqref{lengthboundhomotopy} that $$L(h,\CT^1,\omega) = \lim_i L(h_i,\CT^1,\omega) \leq \lim_i L(\phi_i \circ f,\CT^1,\omega) = L(f,\CT^1,\omega).$$
In other words, $h$ and $f$ both minimize $L(\cdot ,\CT^1,\omega)$. However, since we chose $h_i$ so that $d(h_i(v),g_i(v)) \geq 5\delta$ for some vertex $v$ of $\CT$, and we have $d(g_i(v),\phi_i\circ f(v)) \leq 4\delta$, it follows that $d(h_i(v),\phi_i \circ f(v)) \geq \delta$ for all $i$, and hence that $h(v)\neq f(v)$. This contradicts uniqueness of $L(\cdot ,\CT^1,\omega)$-minimizers on the $1$-skeleton, as stated in Fact \ref{uniquenessfact}.
\end{proof}
% Note that using base frames instead of base points, one could get a more aesthetically pleasing `if and only if' criterion for convergence above, since  discrete, torsion free groups $G < \PSL_2\BC$  correspond exactly to framed hyperbolic $3$-manifolds. However, for the most part we would rather not fuss with base frames below.

\subsection{Algebraic convergence}

Let $\Gamma$ be a finitely generated group. Recall that a sequence $(\rho_i)$ of representations $\rho_i:\Gamma\longrightarrow\PSL_2\BC$ converges {\em algebraically} to a representation $\rho$ if for every $\gamma\in\Gamma$ we have $\rho_i(\gamma)\to\rho(\gamma)$ in $\PSL_2\BC$.  When a sequence $ (\rho_i) $ converges algebraically to a representation $\rho$ and the images $\rho_i (\Gamma) $ converge geometrically to some group $G$, it is easy to see that $\rho(\Gamma) \subset G$. In other words, the algebraic limit quotient manifold $\BH^3/\rho(\Gamma)$ covers the geometric limit manifold $\BH^3/G$. 

All algebraic and geometric limits we  consider in this paper will be  constructed from sequences of manifolds that are marked by short carrier graphs.  In the following example, we  explain in detail how to extract such limits, and show that the carrier graphs `converge' to a carrier graph in the algebraic limit. We include this here so that later in the paper, we do not have to pollute our main arguments with references to base points and representations.

\begin {bei}[Algebraic limits via short carrier graphs] \label {limit example}  Suppose  that $(M_i)$   is a sequence of complete $\epsilon$-thick hyperbolic $3$-manifolds and that  for each $i $, we have a carrier graph $$f_i : X \longrightarrow M_i$$
where $X$ is a fixed finite graph and $\sup_i \length f_i(X) < \infty.$ Then after picking a base point $ x \in X$ and universal covering maps $$(\BH^3,p) \longrightarrow (M_i,f_i(x))$$  taking a fixed $p\in \BH^3$ to $f_i(x)$, we get holonomy representations
$$\rho_i : \Gamma \longrightarrow \PSL_2\BC, \ \ \ \ \Gamma := \pi_1 (X,x), \ \ \ \ \BH^3 / \rho_i(\pi_1 (X,x)) \cong M_i$$
such that a path from $p$ to $\rho_i(\gamma) (p)$ in $\BH^3$ projects down to a loop in $\pi_1(M_i,f_i(x))$ representing $(f_i)_*(\gamma)$.

Each $f_i(X)$  has bounded length in $M_i$.  So for any fixed finite generating set for $\pi_1 (X,x) $, the image in $ \pi_1 (M,f_i(x))$ of each generator $\gamma$ can be represented by a loop with bounded length. Hence, the isometries $f_i(\gamma)$ will all translate $p \in \BH^3$ a bounded amount, so after passing to a subsequence we can extract an algebraic limit 
$$\rho : \Gamma \longrightarrow \PSL_2\BC, \ \ \ \  M_A := \BH^3 / \rho(\pi_1 (X,x)).$$
By Fact \ref{geomcompact}, after passing to a subsequence, we may also assume that the groups $\rho_i(\Gamma)$ converge geometrically to some subgroup $G < \PSL_2 \BC$. Since the $M_i$ are all $\epsilon $-thick, $G$ is discrete, so by Proposition \ref{convergence-convergence-geom} the pointed manifolds $(M_i,f_i(x))$ converge geometrically to  the quotient  $$M_G:= \BH^3 / G,$$  which we consider based at the projection $p_G\in M_G$ of $p\in\BH^3$. Let $$\pi : M_A \longrightarrow M_G$$  be the associated covering map and let $$\phi_i : K_i \longrightarrow M_i$$
be the almost isometric maps given by geometric convergence, chosen  to satisfy \eqref{compatibility} in Proposition \ref{convergence-convergence-geom}. By Arzela-Ascoli, after passing to a subsequence we may assume that the maps $$(\phi_i)^{-1} \circ f_i : X \longrightarrow M_G, $$ which are defined for large $i$, converge uniformly to some map $$f_G : X \longrightarrow M_G.$$ 
Note that $f_G (x)=p_G$. When we identify $\pi_1(M,p_G)$ with $G$, it follows  from  \eqref{compatibility}  that the image $(f_G)_*(\pi_1(X,x))$  is equal to the subgroup $\rho(\Gamma) $. Hence $f_G$  lifts to a carrier graph
$$f : X \longrightarrow M_A,$$
 such that the composition $\phi_i \circ \pi \circ f$ is homotopic to $f_i$ via a homotopy with tracks of lengths at most $\epsilon_i$, where $\epsilon_i\to 0$ as $i\to \infty$.
\end{bei}
  
\subsection{Strong convergence} A sequence of representations $\rho_n : \Gamma \longrightarrow \PSL_2\BC$ converges \emph {strongly} to $\rho$ if $\rho_n \to \rho$  algebraically and the images $\rho_n(\Gamma)$  converge to $\rho(\Gamma)$ geometrically. In \cite{Biringeralgebraic}, we gave  the following criterion for strong convergence.

\begin{sat}[see \cite{Biringeralgebraic}]\label{max-cyclic-ugly}
Let $\Gamma$ be a finitely generated group and $\rho_i : \Gamma \longrightarrow\PSL_2\BC$ be discrete, torsion-free representations.  Assume that $ (\rho_i) $ converges algebraically to $\rho$ and geometrically to $G <\PSL_2\BC$, and set $M_A=\BH^3 / \rho(\Gamma)$ and $M_G=\BH^3 / G$. If 
\begin{itemize}
\item $\rho(\Gamma)$ does not contain parabolic elements, and 
\item every degenerate end of $M_A$ has a neighborhood which embeds under the covering $\pi: M_A \longrightarrow M_G$,
\end{itemize}
then $G=\rho(\Gamma)$. 
\end{sat}

 We used Theorem \ref{max-cyclic-ugly} in \cite{Biringeralgebraic}  to prove that geometric limits of algebraically convergent sequences are  finitely generated,  provided that there are no parabolics.  Here is a more precise version of this that we will need here, but which  is unfortunately not stated in \cite{Biringeralgebraic}.
 
\begin{prop}
\label{limits-fg}
Let $\Gamma$ be a finitely generated group and $$\rho_i:\Gamma\longrightarrow \PSL_2\BC$$ be a sequence of discrete, torsion free representations converging geometrically to a group $G\subset\PSL_2\BC$ with infinite co-volume and no parabolic elements. Assume that $\Gamma' \subset \Gamma $ is a rank $k$ subgroup such that the restrictions $\rho_i |_{\Gamma' } $ converge algebraically.

Then  for some $n \leq 2k$ there are maps $\sigma_i : F^n \longrightarrow \Gamma$  with $\Gamma'\subset \sigma_i(F^n) $
such that after passing to a subsequence, the  representations $$\rho_i \circ \sigma_i : F^n \longrightarrow \PSL_2\BC $$  converge algebraically to some $\rho$, and every degenerate end of $\BH^3 / \rho (F^n)$ has a neighborhood  that embeds under the covering $$\BH^3 / \rho(F^n) \longrightarrow \BH^3 / G.$$
 In particular, if  $\Gamma'=\Gamma$  then $G=\rho(F^n)$ by Theorem \ref{max-cyclic-ugly}, and hence $G$ is finitely generated.
\end{prop}

 Note that even if $\sigma_i(F^n)=\Gamma'$  for all $i $, the algebraic limit of $\rho_i \circ \sigma_i$ may be bigger than that of $\rho_i |_{\Gamma'}$. The point is that  for large $i$ the generators of $F^n$ may map to very complicated words in $\Gamma'$, words which may otherwise be lost in the algebraic limit.

 \begin {proof}The proof of Proposition \ref{limits-fg} is essentially contained in \cite{Biringeralgebraic}, but here is a quick sketch. 
 Assume that $\rho_i |_{\Gamma'}$  converge algebraically to $\rho' : \Gamma' \longrightarrow\PSL_2\BC$  and that some degenerate end $\CE$ of $M_A=\BH^3/\rho'(\Gamma')$ does not have a neighborhood that embeds under the covering map $$M_A \longrightarrow M_G : =\BH^3/G.$$
 By the Covering Theorem \cite{Canarycovering},  the end $\CE$  then has a neighborhood that maps finite to one onto (but does not embed into) a neighborhood of a degenerate end $\CE_G$ of $M_G$. 

%Take a surjection $F^k \longrightarrow \Gamma'$ and extend it to a map $$F^{k+1} = F^k \star \langle a\rangle \longrightarrow G$$

Pick generators $\gamma^1,\ldots,\gamma^k$ for  $\Gamma' \subset G$ and let $g  \in G$ represents a loop in $\pi_1 M_G$ that is homotopic into $\CE_G$, but does not lift homeomorphically into $\CE$. Note that $g^m \in \rho(\Gamma)$  for some $m$; we say $g$ is a \emph{root} of an element of $\rho(\Gamma)$. Since $G $ is discrete  and $\rho_i(\Gamma) \to G$  geometrically,  if we choose a sufficiently small open  neighborhood $U\subset \PSL_2\BC$ of $g$, then for large $i $ there is some $\gamma_i \in \Gamma$ such that $\rho_i(\gamma_i)$  is the unique element of $\rho_i(\Gamma)$  that lies in $U$.  Then for large $i$ we define 
$$ \sigma_i : F^{k+1} \longrightarrow \Gamma$$
 by sending the generators to $\gamma^1,\ldots,\gamma^k,\gamma_i$.   The  sequence $\rho_i \circ \sigma_i$  then algebraically converges to some $\rho : F^{k+1} \longrightarrow \PSL_2\BC$ with $$\rho(F^{k+1}) = \langle \rho'(\Gamma'), \, g\rangle \subset G.$$

If every degenerate end of $\rho(F^{k+1})$ has a neighborhood that embeds in $M_G$,  we are done, but this may not be the case. However,  we obtained $\rho(F^{k+1})$ from $\rho(\Gamma)$ by adjoining a root of some element, and the end $\CE$ of $M_A$  has a neighborhood that finitely covers a neighborhood of an end of $M_A^1:=\BH^3/\rho(F^{k+1})$, but no neighborhood that embeds.  One can use these two properties to show that  Euler characteristics satisfy
$$ |\chi(M_A^1)| < |\chi(M_A)|.$$
 This is done during the proof of Proposition 5.1 in \cite{Biringeralgebraic}; the point is to show that the dimension of the $\PSL_2\BC$-character variety decreases.

We now iterate the process above,  each time adding a single root, constructing new algebraic limits, and decreasing the absolute value of the Euler characteristic.  This process must terminate in at most $|\chi(M_A)| \leq k$ steps, so as in each step we added a single generator to our representations, the end result is a sequence of maps 
$$\sigma_i : F^n \longrightarrow \Gamma, \ \ \rho_i \circ \sigma_i \to \rho$$
where $n\leq 2k$ and where every degenerate end of $\BH^3 / \rho (F^n)$ has a neighborhood  that embeds under the covering \[\BH^3 / \rho(F^n) \longrightarrow \BH^3 / G.\qedhere\]
\end{proof}

\subsection{The structure of strong limits}  We conclude with the following result, whose proof combines techniques from \cite{Biringeralgebraic} with  Theorem \ref{complementary compression bodies}.

\begin {sat}\label {strong limits}
 Suppose that $\Gamma$  is a finitely generated group and $$\rho_n : \Gamma \longrightarrow \PSL_2\BC$$ are discrete, torsion free representations that  converge strongly to a representation $\rho_\infty$, and assume that for $i=1,2,\ldots,\infty$ and $\gamma \in \Gamma,$ no element $\rho_i(\gamma)$ is  parabolic. Let $\epsilon>0$, let
$$M_i = \BH^3 / \rho_i(\Gamma), \ \ i=1,2,\ldots,\infty,$$ and let $(\phi_i)$ be a sequence of almost isometric maps given by the geometric convergence $M_i \to M_\infty$, as  in Definition \ref{geolimits}.
Given a compact $C_0 \subset M_\infty$ and $r>0$, there is a standard compact core $$ C \subset M_\infty, \ \ C \supset C_0, $$ and a sequence $L_i \to \infty$ such that after excluding finitely many $i $, the following holds for every component $S \subset \partial C$.
\begin{enumerate}
	\item[(a)]  If $S$  faces a degenerate end   of $M_\infty$,  then $\phi_i(S)$  is a boundary component of a product region $$V_i^S  \subset CC(M_i) \setminus int(\phi_i(C))$$  that has $\epsilon $-width at least $L_i$.  Moreover, the  component of $M_i \setminus int(\phi_i(C))$ adjacent to $\phi_i(S)$ is a compression body with missing interior boundary.
 \item[(b)] If $S$  faces a convex-cocompact end  of $M_\infty $, then $\phi_i(S)$ bounds a component $E_i^S \subset M_i \setminus int(\phi_i(C))$ that  is a  product neighborhood of a  convex cocompact end of $M_i$, where $$E^S_i \subset M_i \setminus \CN_r(CC(M_i)).$$
%$$S= \partial_\CE \CN_s(CC(M_\infty))$$ for some radius $s\geq s_0$, and $\phi_i(S)$ bounds a  component of $M_i \setminus int(\phi_i(C))$ that is a product neighborhood of a convex cocompact end $\CE_{i}$ of $ M_i$. Also, $\phi_i(S)$ is homotopic to $\partial_{\CE_i} \CN_s( CC(M_i))$ via a homotopy with tracks of length at most some $\epsilon_i$, where $\epsilon_i \to 0$.
\end{enumerate} 
\end {sat}

 Above, $CC$  always denotes the convex core and $\CN_r(\ \cdot \ )$ is a  closed metric $r$-neighborhood. 
%A compression body with \emph{missing interior boundary}  is a manifold obtained from a compression body by  removing the interior boundary. In (1), the missing interior boundary components correspond to ends $M_i$ and the exterior boundary is a component of $\partial V_i$.
Also, note that $\pi_1 M_\infty$ is  finitely generated, so $M_\infty$ is tame by the Tameness Theorem (see \S \ref{sec:ends}). Since $M_\infty $ has no cusps, every end of $M_\infty $ is either convex cocompact or degenerate.

\begin {proof}
We'll construct the compact core $C \subset M_\infty$ by constructing its boundary components $S^\CE$, one for each of the ends $\CE$ of $M_\infty$. 

First, fix some $R \geq r+2$  large enough so that $C_0 \subset \CN_R(CC(M_\infty))$. When $\CE$ is a convex cocompact end of $M_\infty$, we let $$S^\CE \subset \partial \CN_R(CC(M_\infty))$$ be the component that faces $\CE$. 

Next, let $\CE$  be a degenerate end of $M_\infty$.  Using  Proposition \ref {constructing prs} and  Corollary \ref{truncating noncompact}, pick a product region $U^\CE \subset M_\infty$ that is a neighborhood of $\CE$, where $U^\CE \cap C_0=\emptyset$ .   Fixing some to-be-determined $\ell_->1$, Lemma \ref{minsrsinpr} gives a minimizing simplicial ruled surface $$f^\CE : S \longrightarrow U^\CE$$   in the homotopy class of a level surface whose $\epsilon$-distance from $\partial U^\CE$ lies  in the interval $[\ell_-,2\ell_-]$.  Using Freedman--Hass--Scott (see Lemma \ref{getembedded}), let $S^\CE$ be a level surface of $U^\CE$ with 
$$S^\CE  \subset \CN_{\frac 12}(f^\CE(S)) \subset U^\CE. $$

 We now let $C \subset M_\infty$ be the standard compact core bounded by the surfaces $S^\CE$, where $\CE$  ranges over all the ends of $M_\infty$. By construction, $C_0 \subset C$. Our goal now is to verify that parts (a) and (b) of the Theorem hold for the surfaces $S_\CE$, at least for sufficiently large $\ell_-$.

\medskip

We first work on (b). If $\CE$  is a convex cocompact end of $M_\infty$, let $$S^\CE_1 \subset \partial \CN_1(CC(M_\infty))$$  be the component facing $\CE$, and  consider the component 
$$F_i^\CE \subset M_i \setminus \phi_i(S^\CE_1)$$
on the side of $\phi_i(S^\CE_1)$  corresponding to the side of $S^\CE_1$  facing $\CE$. Let $$N_i =  M_i \setminus \cup_\CE F_i^\CE.$$
By strong convergence, the inclusion $N_i \hookrightarrow M_i$ is $\pi_1$-surjective. Moreover, $N_i$ has convex\footnote{While it is true that $N_i$ has convex boundary, it's a little bit annoying to write down a proof. Indeed, one would like to say that $S_1^\CE$ is strictly convex, and use the $C^\infty$-convergence $\phi_i^* g_i \to g_\infty$ to say that $\phi_i(S_1^\CE)$ is strictly convex as well, since its principal curvatures approaches those of $S_1^\CE$. However, $S_1^\CE$ is only $C^2$-a.e., so this is hard to make precise. To fix it, one could either cinch one's belt and work in the lower regularity setting, or one could replace $S_1^\CE$ in the argument above by a smooth approximation of it that is still strictly convex.  One way is to find such an approximation is to use Theorem 2 (a) of \cite{greene1976c}, which says that any strictly convex function $f$ on a Riemannian manifold can be uniformly approximated by a smooth strictly convex function $g$. One applies this to the function $f(x) = d(x, CC(M_\infty))$, noting that by negative curvature, $f$ is strictly convex on $X\setminus CC(M_\infty)$. %(Here, if $\alpha$ is a geodesic segment that lies outside of $C$ and $\beta$ is the geodesic in $C$ between the projections of the endpoints of $\alpha$, then an easy comparison geometry argument shows that $t \mapsto d(\alpha(t),\beta(t))$ is strictly convex, and it follows that $f \circ \alpha$ is strictly convex.) 
	Pick some $g$ that approximates $f$ closely and choose a regular value $a \approx 1$ for $g$. The level set $g^{-1}(a)$ is a smooth strictly convex surface that can be used in place of $S_1^\CE$ above.}  boundary in $M_i$  for large $i$, since the  $S^\CE_1$ are strictly convex and the $\phi_i$ are almost isometric.
 So, Proposition~\ref{radialextension}  implies that every component of $M_i \setminus N_i$ is  a product neighborhood of a convex cocompact end of $M_i$, and every $\phi_i(S^\CE_1)$ bounds one of these product neighborhoods. Moreover, since $N_i$  is convex, we have $CC(M_i) \subset N_i$. Since $R \geq r+2$ and $S^\CE\subset \CN_R(CC(M_\infty))$, we have that $S^\CE$ lies at least $r+1$ away from $S^\CE_1$. It follows that  for large $i$  the surfaces $\phi_i(S^\CE)$ will lie at least $r$ away from $\phi_i(S^\CE_1)$, and therefore at least $r$ from $N_i$. So, the component $E_i^{S_\CE} $ of $M_i \setminus int(\phi_i(C))$ bounded by $\phi_i(S^\CE)$ lies at least $r$ outside the convex core of $M_i$, as desired in (b).

\medskip

We now work on part (a), where $\CE$ is degenerate. Fix some large $\ell>0$. By Corollary \ref{truncating noncompact}, there is a compact subproduct region $$U^\CE_\ell \subset U^\CE,$$
that has $\epsilon$-width at least $\ell$ such that $ \partial U^{\CE} \subset \partial U^\CE_{\ell}$. 
 If $\ell$ is  at least some constant depending only on $\epsilon,g,\ell_-$ and $i$ is large, Proposition~\ref{bilipschitzpreservation}  implies that the topological product region $\phi_i(U^\CE_\ell)$  contains a topological subproduct region $$W_{i,\ell}^\CE \subset \phi_i(U^\CE_\ell)$$
 that is a (geometric)  product region of nearly full $\epsilon$-width.   Moreover,  if $\ell_-$  is at least some constant depending only on $\epsilon,g$, and $\ell$ is much larger than $\ell_-$, we can assume that $\CN_2(\phi_i(S^\CE)) \subset W_{i,\ell}^\CE$. And in light of Corollary \ref{prs in the convex core}, we can also assume each $W_{i,\ell}^\CE \subset CC(M_i)$. Now  the surface $\phi_i(S^\CE) $ divides $W_{i,\ell}^\CE$ into two  topological subproduct regions. In fact, these are actually (geometric) subproduct regions: since $S^\CE$ is contained in the $1/2$-neighborhood of a minimizing SRS in the homotopy class of a level surface, Fact \ref{pullback minimizing} implies that $\phi_i(S^\CE) $ is contained in the $1$-neighborhood of a minimizing SRS $g_i$ in the homotopy class of a level surface  of $W_{i,\ell}^\CE$, and this $g_i$ is $1$-NAT, say, since $$\CN_1(\mathrm{Im}(g_i)) \subset \CN_2(\phi_i(S^\CE)) \subset W_{i,\ell}^\CE.$$
 
Let $V_{i,\ell}^\CE \subset W_{i,\ell}^\CE \subset CC(M_i)$ be the one  of these two  subproduct regions that faces the side of $\phi_i(S^\CE) $  that is the $\phi_i$-image of the side of $S^\CE$ facing $\CE$.  Then as long as $\ell$  is much larger than $\ell_-$,  we will have, say,  that the  $\epsilon$-width of $V_{i,\CE}^\ell$ is at least $ \ell/4.$
 So, taking a  sequence of  constants $\ell$  converging to infinity, and discarding finitely many $(M_i)$ and reindexing, we can assume that for each $i$,  the surface $\phi_i(S^\CE) $  is a boundary component of some product region $V_i^\CE \subset CC(M_i)$ that lies on the side of $\phi_i(S^\CE) $  that is the $\phi_i$-image of the side of $S^\CE$ facing $\CE$, and where the $\epsilon$-widths of the product regions $V_i^\CE$ go to infinity.
\end {proof}

\section{Splitting along barriers}

\label {splitting along barriers}

Let $M$  be a complete hyperbolic $3$-manifold.

\begin {defi}[Barrier] Let $U \subset M$ be a compact product region and let $ \mathcal {DC}(M,U)$ be  the associated double compression body, as defined in \S \ref{Double compression bodies}. A \emph{barrier} of $U$ is a component of $\partial \mathcal {DC}(M,U)$ that is incompressible in $M$.% for some product region $U $ with genus at most $g$ and width at least $L$.
\end{defi}

Recall that Theorem \ref{double compression body theorem} (1) states that if $U $  has genus $g$ and $\epsilon $-width at least some $L=L(g,\epsilon)$, then any boundary component of $\mathcal{DC}(M, U) $  that is not isotopic to a level surface of $U$ is incompressible in $M$. Also, the reader should note that $ \mathcal {DC}(M,U)$ is only defined up to isotopy. So really, we'll call any incompressible surface isotopic to a component of $\partial \mathcal {DC}(M,U)$ a `barrier' of $U$.

\medskip

\noindent \it Convention: \rm In the rest of this paper a \emph{$(k,L)$-product region}  will be a product region with genus at most $k$ and $\epsilon$-width  at least $L$. A barrier of a $(k,L)$-product region  is called a \emph{$(k,L)$-barrier.} 

\medskip

% Often, it will be convenient to isotope a barrier so that it has diameter bounded above by some $C=C(g,\epsilon)$, which one can do by homotoping it to a simplicial hyperbolic surface and then using Freedman--Hass--Scott (Theorem \ref{getembedded}). However, in the definition we do not require anything about the particular embedding.

Essentially, the goal of this section is to bound the complexity of a decomposition of $M$ along a maximal collection of disjoint, nonisotopic barriers. There is some necessary complication in our statements, however, because we need a priori choices of constants.

\begin{sat}[Splitting along barriers]\label{cutting along barriers} \label{maxsplitting}
Fix $k$ and $\epsilon$ and a function $K:\BN\longrightarrow \BN$. Then there are constants $k'=k'(k,K),$ $g=g(k,K)$, and $L=L(k,\epsilon,K)$ such that in every  hyperbolic 3-manifold $M$ with no cusps and $\rank(\pi_1(M))\le k,$ there is an embedded incompressible surface $\Sigma\subset M$, with at most $15k$ connected components, each of which has genus at most $g$, such that
\begin{enumerate}
\item for every component $N \subset M \setminus \Sigma$,  we have $\rank(\pi_1(N))\le k'$,
\item every $(K(g),L)$-barrier in $M$ is isotopic to a component of $\Sigma$.
\end{enumerate}
\end{sat} 

Here, remember that we fix $\epsilon$ in the theorem statement so we can talk about the $\epsilon$-width $L$ of the barriers in (2).

\medskip

If $\mathcal U$ is a collection of disjoint compact product regions, we say that a surface $\Sigma \subset M$ \emph{represents the barriers of $\mathcal U$} if \begin{itemize}
	\item  every component of $\Sigma$  is a barrier of some $U \in \mathcal U$,
\item every barrier of every $U\in \mathcal U$ is isotopic to a component of $\Sigma$.
\end{itemize}
The key step  in the proof of Theorem \ref{maxsplitting} is the following. 

\begin {lem}[Acylindricity]\label {acylindrical}
Given $g,\epsilon>0$,  there is some $L>0$  as follows. Let $M$ be a   hyperbolic $3$-manifold and let $\mathcal U$  be a collection of disjoint, compact $(g,L)$-product regions in $M$.  Then there is a surface $\Sigma \subset M$  that represents the barriers of $\mathcal U$ such that either
\begin{enumerate}
	\item $\Sigma$ is connected, and is a regular  fiber with respect to some fibration of $M$ over $S^1$ or $S^1 / (z \mapsto -z)$, or
\item the splitting of $\pi_1 M$  associated to $\Sigma$ is $7$-acylindrical.
\end{enumerate} 
\end {lem}
 
 We'll defer the proof of Lemma \ref{acylindrical} for a moment.  First, we will recall some key facts about acylindrical splittings,  make some remarks about the constant $7$, and in \S \ref{pfmaxsplit} we show how Lemma \ref{acylindrical} implies Theorem \ref{maxsplitting}. In \S \ref{preliminary}, we  essentially show how to construct the barriers $\Sigma$ above, and then in  \S \ref{acylpf} we will prove the dichotomy in Lemma \ref{acylindrical}.

Also, it is worth mentioning that for the purposes of Theorem \ref{maxsplitting},  it would suffice if the $7$ in Lemma \ref{acylindrical}  were replaced by any function of $g$.  We expect that  one could prove this weaker version of  Lemma \ref{acylindrical} using an extension of the topological `vegematic' argument used in \cite[Theorem 4.1]{Debloisexplicit}, for example. See also \cite[\S 7]{boyer2005characteristic}.

\medskip

Recall that a splitting of a group $G$ as a  graph of groups is \emph {$\ell$-acylindrical} if in the associated action of $G$ on the Bass-Serre tree,  there is no element that fixes a segment of length $\ell+1$.  One  can interpret this in our situation more concretely as follows.  If $f : A \longrightarrow M $  is a map from an annulus  such that $f(\partial A) \subset  \Sigma ,$
then after a homotopy rel $\partial$, the map $f$ is transverse to $\Sigma$ and   decomposes as a concatenation of essential annuli in $M_\Sigma$, the manifold with boundary obtained by  splitting $M$ along $\Sigma$. If the number of such essential subannuli is called the `length' of $f$, then the splitting of $\pi_1 M$  determined by $\Sigma$ is $\ell$-acylindrical if and only \emph{there are no such annuli with length $\ell$.} This  geometric interpretation of $\ell$-acylindricity was first described by Sela \cite{sela1997acylindrical}; see also DeBlois \cite{Debloisexplicit} for a very detailed  proof of the correspondence in the case that $\Sigma$ is connected.

The  reason we  care about $\ell$-acylindricity is that Weidmann \cite{Weidmannrank} proved an analogue of Grushko's theorem for such splittings,  a weakened version of which is as follows: if a torsion-free group $G$ splits as a minimal $\ell$-acylindrical graph of groups with vertex groups $\{A_v \ |\ v \in V\} $ and (nontrivial) edge groups $\{A_e \ |\  e \in E\}$, then  we have
\begin {equation}\rank G \geq \frac 1{2\ell+1} \left ( \sum_{v \in V} \mathrm{rrank}(A_v) + |E|\right ).\label{rrankeq}\end{equation}
Here, $\mathrm{rrank}(A_v)$ is the \emph {relative rank} of $A_v$  with respect to its edge groups, i.e.\ the minimal number of elements of $A_v$ that, together with a conjugate of each adjacent edge group, will generate $A_v$.  Note that\footnote{Our graphs of groups do not have directed edges. If one takes edges to be directed, as Weidmann does in \cite{Weidmannrank}, then  there is no $2$ in Equation \eqref{rrank2nd}.}
\begin {equation}\sum_{v \in V} \mathrm{rrank}(A_v) \geq \sum_{v \in V} \rank(A_v) - 2\sum_{e \in E} \rank (A_e).\label{rrank2nd}	
\end {equation}

 The reader may wonder whether the $7$ in Lemma \ref{acylindrical} is optimal, and in fact it is not.  The optimal number should be $4$, and it is true that using $4$ one could produce slightly  better bounds for the $k'$  in Theorem \ref{maxsplitting}, if desired. %However, for our purposes such an improvement is irrelevant.% if one really wants such bounds (which we do not give in Theorem \ref{main}), we expect that one could do much better after the fact using the techniques of  our earlier paper \cite{Biringerranks}. 
And while the arguments to come could also be used to prove $4$-acylindricity,  everything is a bit easier to organize if we are content with $7$.  Note, however, that the splitting may not be $3$-acylindrical, see Figure \ref{not-2-acyl}.

\begin{figure}
	\centering
\includegraphics{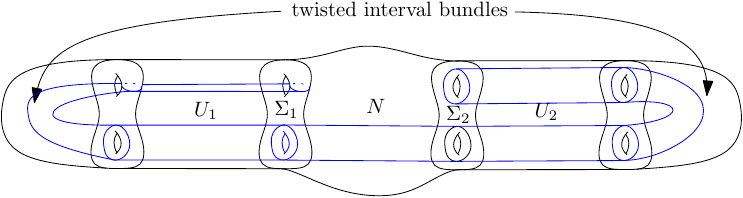}
\caption {One can produce a $3$-manifold by gluing two twisted interval bundles to a hyperbolizable $3$-manifold $N$ that  is not a trivial interval bundle, but where there is an embedded annulus joining its two boundary components. Using the results of BMNS~\cite{Brocksplittings}, one can show  that if the gluing maps are sufficiently complicated, the resulting manifold $M$ is hyperbolic and the gluing surfaces  $\Sigma_1,\Sigma_2$ are adjacent to wide product regions.  We can take the barriers to be the gluing surfaces, and then there will be a length $3$ annulus in $(M,\Sigma_1 \cup \Sigma_2)$.}
\label {not-2-acyl}
\end{figure}

\subsection{Proof of Theorem \ref{maxsplitting}, given Lemma \ref{acylindrical}} \label {pfmaxsplit}
The proof will be algorithmic. In order to present it in a way that is as understandable as possible, we will first give a fairly detailed sketch without worrying about things like the appropriate widths of product regions, and the exceptional cases in Lemma \ref{acylindrical}. Then we will give a  rigorous version of the argument.  This two-step process is necessary, since in order to write down a rigorous proof one has to know in advance how large the genera of the surfaces produced during the algorithm will be, and one cannot know this without going through the algorithm.

 Before starting the first pass, let's briefly discuss how to apply the Weidmann inequality \eqref{rrankeq}, which makes the whole argument work. Suppose that $\mathcal U$  is a collection of $i$ disjoint  product regions in $M$,  each with genus at most $g$, that $\Sigma$ is a surface that represents the barriers of $\mathcal U$, and that the splitting of $\pi_1 M$  given by $\Sigma$  is $7$-acylindrical, as in Lemma \ref{acylindrical}. Then \eqref{rrankeq} implies that 
\begin{equation}\sum_{\substack{\text{ components } \\ N \subset M \setminus \Sigma}}\hspace{-5mm} \mathrm{rrank}(\pi_1 N) \leq (2\cdot 7 + 1) \rank(\pi_1 M) \leq 15k.\label{weidconseq}\end{equation}
Similarly, the number of components of $\Sigma$, which is $|E|$ in \eqref{rrankeq}, is at most $15k$; we'll use this later. Combining \eqref{weidconseq} with \eqref{rrank2nd}, we get \begin{equation}
  \label {rrank3}
\sum_{\substack{\text{ components } \\ N \subset M \setminus \Sigma}} \hspace{-5mm} \rank \pi_1 N \leq 15k + 2 \hspace{-5mm} \sum_{\substack{\text{ components } \\ X \subset \Sigma}}\hspace{-5mm}   \rank \pi_1 X \leq 15k + 4\cdot i\cdot g,
  \end{equation}
where the last inequality comes from the fact that  for a surface $X$, we have $\rank \pi_1 X = 2 genus(X)$, and the genera of all the barriers of a product region sum to at most twice the genus of the product region. 

\medskip

\noindent \it The first pass. \rm Start with $\Sigma=\emptyset$, $k^{(0)}:=k$ and  iteratively define
\begin{equation}\label{step}
	k^{(i)} = 15 k + 4\cdot i \cdot g^{(i)}, \ \ g^{(i)} = K(k^{(i-1)})
\end{equation}
  Since $\rank \pi_1 M = k$,  we are done  unless there is some product region $U^{(1)} \subset M$ with genus at most $g^{(1)}$ such that  the union $\Sigma^{(1)}$  of all its barriers is nonempty.  Hopefully,  \eqref{rrank3} applies, in which case  each component of $M \setminus \Sigma^{(1)}$ has  rank at most $k^{(1)}$.  We are then done unless there is some  product region $U^{(2)} \subset M$ with genus at most $g^{(2)}$ that has a barrier that is not isotopic to  a component of $\Sigma^{(1)}$.  Hopefully the two product regions are disjoint, and we let $\Sigma^{(2)} \subset M$ represent all the barriers of $U^{(1)} \cup U^{(2)}$. Hopefully,  \eqref{rrank3} tells us that  each component of $M \setminus \Sigma^{(2)}$ has  rank at most $k^{(2)}$, and then we are done unless there is some $U^{(3)} \subset M$ with  genus at most $g^{(3)}$ that has a  barrier that is not isotopic to a component of $\Sigma^{(2)}$.  Continue this process.

  At each step, the surface $\Sigma^{(i)}$  has at least $i$  components.  And assuming the associated splitting is always $7$-acylindrical, the number of components of $\Sigma^{(i)}$ must always be at most $15k$; as remarked earlier, this follows from \eqref{rrankeq}. So, our process must terminate in at most $15k$ steps. This (more or less) proves Theorem \ref{maxsplitting}, with $k'=g=k^{(15k)}$.
  
\medskip

\noindent  \it The real proof. \rm Let's assume for simplicity of notation that all the relevant constants/functions are increasing. For instance, it suffices to prove Theorem \ref{maxsplitting} when $K$ is an increasing function, just by replacing $n \mapsto K(n)$ by $n \mapsto \max_{i\leq n} K(i)$. And if $K$ is increasing, so are the sequences $k^{(i)}$ and $g^{(i)}$. Using a similar maximum trick,  we can assume that the $L$ in Lemma~\ref{acylindrical} increases with $g$.

In the proof sketch above we never consider product regions that have genus bigger than\footnote{One may think that $g^{(15k)}$  suffices, but to say that the process terminates after $15k$ steps we take $15k+1$ steps and then say we have a contradiction, so we do consider product regions with genus  $g^{(15k+1)}$, if briefly.} $$g_{max} := g^{(15k+1)}.$$
Let $L_{max}$ be the  constant given by Lemma \ref{acylindrical},  with respect to the inputs $\epsilon$ and $g_{max}.$ The main step of the  argument  is as follows.

\begin {claim}[The repeated step]\label {stepclaim}
Suppose $\mathcal U^{(i)}$  is a collection of $i$ disjoint $(g^{(i)}, L_{max})$-product regions in $M$, where $i \leq 15k$. Then  there is a surface $\Sigma^{(i)} \subset M$  that represents the barriers of $\mathcal U^{(i)}$, and  where for every component $N \subset M\setminus \Sigma^{(i)}$  we have
$$\rank \pi_1 N \leq k^{(i)}.$$
 Moreover, there is some $L=L(\epsilon,g_{max},L_{max})$ such that if $V \subset M $ is a $(g^{(i+1)},L)$-product region  that has a barrier  that is not isotopic to a component of $\Sigma^{(i)}$,  there is a subproduct region $V' \subset V$ of $\epsilon$-width at least $L_{max}$ that is disjoint from every $U\in \mathcal U^{(i)}$.
\end {claim}

 This claim fulfills all our hopes from the first pass.  Namely, the constants $k',L$ in  Theorem \ref{maxsplitting} will just be $k^{(15k)}$ and the $L$ in Claim~\ref{stepclaim}, and the proof proceeds as follows.  Starting with $\mathcal U^{(0)} = \emptyset$, assume that we have constructed a  collection $\mathcal U^{(i)}$ of disjoint genus $(g^{(i)},L_{max})$ product regions in $M$  with at least $i$ barriers. Claim \ref{stepclaim} shows that the components of $M \setminus \Sigma^{(i)}$ have rank at most $k^{(i)}$, so we are done with $\Sigma=\Sigma^{(i)}$ unless there  is some $(g^{(i+1)},L)$-product region $V \subset M$ that has a barrier that is not isotopic to a barrier from $\mathcal U^{(i)}$. If there is such a $ V$, we apply Claim \ref{stepclaim}, and set $\mathcal U^{(i+1)}= \mathcal U^{(i)} \cup \{V'\}$. Since our  product regions  all have $\epsilon$-width at least $L_{max}$, the process terminates in at most $15k$ steps, as explained in the first pass. (So, the requirement $i\leq 15k$ in Claim \ref{stepclaim}  is not an issue.) 

\begin{proof}[Proof of Claim \ref{stepclaim}]
Since $i\leq 15k$,  we have $g^{(i)} \leq g_{max}$, so we know that Lemma \ref{acylindrical} applies to the collection $\mathcal U^{(i)}$, giving a surface $\Sigma^{(i)} \subset M$ that represents the barriers of $\mathcal U^{(i)}$. If the  associated splitting of $\pi_1M$ is $7$-acylindrical,  each component of $ M \setminus \Sigma^{(i)}$ has rank at most $k^{(i)}$, by \eqref{rrank3}.  If $\Sigma^{(i)}$  is a fiber in  a fibration of $M$  over the circle,  $$\rank \pi_1( M \setminus \Sigma^{(i)}) = 2\cdot\mathrm{genus} (\Sigma^{(i)}) \leq 2g^{(i)} \leq k^{(i)}$$
as desired. If $\Sigma^{(i)}$ is a regular fiber in a fibration over $S^1 / (z \mapsto -z)$, then each component of $ M \setminus \Sigma^{(i)}$  compactifies to a twisted interval bundle over  a non-orientable surface with Euler characteristic
$$\chi = \chi(\Sigma^{(i)})/2 \leq (2g^{(i)} -2)/2=g^{(i)} -1.$$
This surface is a connected sum of at most 
$g^{(i)} -1$ copies of $\BR P^2$, and hence its fundamental group has rank at most $g^{(i)} -1 \leq k^{(i)}$. 

 Assume now that $V \subset M $ is a $(g^{(i+1)},L)$-product region that has a barrier that is not isotopic to a component of $\Sigma^{(i)}$, for some huge constant $L=L(\epsilon,g_{max},L_{max})$ soon to be determined. (Note that $g^{(i+1)} \leq g_{max} := g^{(15k+1)}$, since $i\leq 15k$.) Then $V$  cannot share a level surface with any $U\in \mathcal U^{(i)}$, so by Lemma \ref{intersectprs}, there is some $D=D(\epsilon,g_{max})$  such that for any $U\in \mathcal U^{(i)}$,
$$p \in U \cap V \Longrightarrow d_\epsilon(p,\partial U \cup \partial V) \leq D.$$
By the  Bounded Diameter Lemma and Fact \ref{width diameter}, each 
$$B_U = \{p \in U \ | \ d_{\epsilon}(p,\partial U) \leq D\}, \ \ U \in \mathcal U^{(i)}$$
 is a union of $2$ sets, one for each component of $\partial U$, each of which has  $\epsilon $-diameter at most some $D'=D'(\epsilon,g_{max})$. Since $\mathcal U^{(i)}$ only has $i$ product regions in it,  this means that if $L$ is huge with respect to $i$ and $D'$,  there is a point in $V$ that lies at an $\epsilon $-distance much bigger than $L_{max}$ from  both $\partial V$ and from  every $U$.  We  can then use Corollary \ref{cutting-product}   to find a subproduct region $V' \subset V$  with $\epsilon$-width at least $L_{max}$  that is disjoint from every $U \in \mathcal U^{(i)}$.
\end{proof}

\subsection{Setting up the barriers}\label {preliminary}
Working towards Lemma \ref{acylindrical}, we first need to discuss how to best represent the barriers of a collection $\mathcal U$ of product regions.   Note that the barriers of each $U\in \mathcal U$  are only well-defined up to isotopy, so that there is some ambiguity about how to best place the barriers of one product region in relation to those of another.  One would  ideally like to isotope all the double compression bodies $\mathcal{DC}(M,U)$, where $U\in \mathcal U$ to be pairwise disjoint, and then just take the barriers to be the corresponding incompressible boundary components.  In general this may not be possible, since product regions can be `nested', in the sense that some $V \in \mathcal U$ is contained in one of the compression bodies used in constructing $\mathcal{DC}(M,U)$.  However:

\begin{lem}[Non-nested subcollections]\label {non-nested}
	 There is some $L=L(g,\epsilon)$  as follows. Suppose $\mathcal U$  is a finite collection of  disjoint $(g,L)$-product regions in $M$. Then  there is some sub-collection $\mathcal V \subset\mathcal U$   such that
\begin{enumerate}
	\item every barrier of every $U\in \mathcal U$  is also a barrier of some $V\in \mathcal V$,
\item no two  distinct product regions $V_0,V_1 \in \mathcal V$ are isotopic in $M$,
\item all the double compression bodies $\mathcal{DC}(M,V)$, where $V \in \mathcal V $,  can be realized in $M$  in such a way that $$V_0 \neq V_1 \ \implies \ \mathcal{DC}(M,V_0) \cap \mathcal{DC}(M,V_1) =\emptyset.$$
\end{enumerate} 
\end{lem}

 The proof of Lemma \ref{non-nested} will  appear below, and will be a quick consequence of the following lemma and Proposition \ref{incompressible-incompressible}.

\begin{lem}[Recognizing barriers]\label {recognizing barriers}
	There is some $L=L(g,\epsilon)$   such that if $U$ is a $(g,L)$-product region in $M$, an embedded incompressible  surface $\Sigma \subset M $ is a barrier of $U$ if and only if it can be obtained from a level surface of $U$ via a sequence of compressions and isotopies. 
\end{lem}

To understand the point of the lemma, note  that Fact \ref{compressions isotopies}  implies that  the barriers of $U $  are exactly the  incompressible surfaces that are obtained from some  component $S \subset\partial U$ via  a sequence of compressions and isotopies in $M \setminus int(U)$, rather than $M$.  So, the point is that if $U$ is wide enough, one does not produce any extra surfaces if one allows the compressions and isotopies to pass through $U$.

\begin {proof}
 The obvious direction is obviously obvious, so let $\Sigma$  be an incompressible surface in $M$  obtained from a level surface of $U$ via a sequence of compressions and isotopies in $M$.  We want to show that $\Sigma$  is a barrier of $U$. To do this, we work in a cover of $M$.

Let $\pi : \hat M \longrightarrow M$ be  the cover corresponding to the  image  of the map $\pi_1 U \longrightarrow \pi_1 M$ induced by inclusion.  Then $U$  lifts  homeomorphically to a  product region $\hat U$  such that the inclusion $\hat U\hookrightarrow \hat M$  is $\pi_1$-surjective.  Using Corollary \ref{cutting-product}, split $\hat U$ into three  adjacent product regions $$\hat U = V_1 \cup V_2 \cup V_3.$$ Taking $L$  to be large, we can assume  that each $V_i$ has $\epsilon$-width at least any given $L'=L'(g,\epsilon)$. With $V_2$ as the $N$ in Theorem \ref{complementary compression bodies}, it follows that  for large $L'$, each component of $\hat M \setminus\hat U$ is a compression body with missing interior boundary. 

 Regard $\hat M$ as the interior of a manifold with boundary $\bar M$, so that  each component of $\bar M \setminus int(\hat U)$ is an actual compression body. 
\begin {claim}
	Every  incompressible surface in $\hat M$ that is homotopic to a component of $\partial \bar M$  projects to a barrier of $U$ in $M$.\label {projectbar}\end {claim}

 Postponing the proof the claim for a moment, let's finish the lemma. Compressions and isotopies lift to covers, so $\Sigma$  lifts homeomorphically to an incompressible surface $\hat \Sigma \subset\hat M$. By Lemma \ref{minimizers}, we may homotope the  inclusion of $\hat \Sigma$ to a simplicial ruled surface $f : \hat \Sigma \longrightarrow \hat M.$

If $f(\hat \Sigma)$  intersects $V_2$, the  Bounded Diameter Lemma implies that $f(\hat \Sigma) \subset \hat U$ as long as $L'$ is large.  But then $f$ is a $\pi_1$-injective map from a surface into a product region, where the surface has genus at most that of the product region,  implying that $f : \hat \Sigma \longrightarrow \hat U$ is a homotopy equivalence.  In other words, $U \subset M$ is incompressible and $\Sigma$ is homotopic to a level surface in $U$, and hence $\Sigma$ is a barrier of $U$.

So,  suppose that $f(\hat \Sigma)$   does not intersect $V_2$. Then $f(\hat \Sigma)$ is an incompressible surface in a component of $\bar M \setminus int(V_2)$, which is a compression body  whose interior boundary is a union of components of $\partial \bar M$. So by Fact \ref{incincomp}, $f$ is homotopic to a  component of $\partial \bar M$. Hence,  Claim \ref{projectbar}  implies that  the (embedded) surface $\Sigma$  is a barrier of $U$.

\medskip

  It remains to prove Claim \ref{projectbar}. To do this, it suffices to construct a compact core for $\hat M$ and  prove that every incompressible boundary component of this core projects to a  barrier of $U$ in $M$. So, let $$DC:= \mathcal {DC}(M,U)$$ be the double compression body of $U$. If $DC\hookrightarrow M$ is $\pi_1$-injective, then $DC$ lifts homeomorphically to a compact core $\hat {DC} \subset \hat M$, and the claim is clear. We can therefore  assume that $DC \hookrightarrow M$ is \emph {not} $\pi_1$-injective.
If $L$  is large, Theorem \ref{double compression body theorem} (2)  then implies that some component $I \subset M \setminus int(U) $ is a twisted interval bundle and that if $S \subset\partial U$ faces the \emph{other}  component of $M \setminus int(U) $, the compression body $C:= C(M \setminus int(U),S)$ is nontrivial. 

Now $I$ is double covered by a trivial  interval bundle $\hat I \longrightarrow I$, and we can construct a double cover $M' \longrightarrow M$ by gluing two copies of $M\setminus int(I)$ to the boundary components of $\hat I$. This $M'$ contains a union $$K' := C_1' \cup U_1' \cup \hat I \cup U_2' \cup C_2',$$
where each subset shares a boundary component with the next, each $C_i'$  projects homeomorphically onto $C$, and each $U_i'$  projects homeomorphically onto $U$.  Moreover, we can  assume that  $\pi$  factors as $$\hat M \longrightarrow M' \longrightarrow M$$  in such a way that $\hat U$ projects homeomorphically onto $U_1'$, say. Since the inclusion $U_1' \hookrightarrow K'$ is $\pi_1$-surjective, we can then lift $K'$  homeomorphically to a subset $$\hat K := \hat C_1 \cup \hat U_1 \cup \hat I \cup \hat U_2 \cup \hat C_2 \subset\hat M,  \text{ where } \hat U_1 = \hat U.$$
But the compression body $C$ was nontrivial, so Theorem \ref{double compression body theorem} (1) implies that $\partial_{int} C$ is incompressible in $M$, and hence $\hat K \subset\hat M$ has incompressible boundary.  Therefore $\hat K \hookrightarrow \hat M$ is $\pi_1$-injective,  in addition to being $\pi_1$-surjective, implying that $\hat K$  is a compact core for $\hat M$.  And by construction,  every boundary component of $\hat K$  projects to a component of $\partial_{int} C$, and hence to a barrier of $U$ in $M$.\end{proof}

 We are now ready to prove Lemma \ref{non-nested}.

\begin{proof}[Proof of Lemma \ref{non-nested}]
Let $\mathcal V \subset\mathcal U$ be a  minimal sub-collection satisfying (1). Then (2) holds automatically. Let $$ M_{\mathcal V} := M \setminus \cup_{V \in \mathcal V} int(V).$$
 It suffices to show that whenever $V \in \mathcal V$ and $S \subset \partial V$  is some boundary component,  the characteristic compression body $C(M_{\mathcal V}, S)$ is also a characteristic compression body for $S$ in  the larger manifold $M\setminus int(V)$.  If this is true, then  for all $V$, say with $\partial V = S_1 \cup S_2$,  we will have
$$\mathcal {DC}(M,V) = C(M_{\mathcal V}, S_1) \cup V \cup C(M_{\mathcal V}, S_2).$$
And since the characteristic compression bodies of the boundary components of $M_{\mathcal V}$ can all be taken to be disjoint, see Proposition \ref{prop:rel-compbody}, this will prove property (3) above.

So, let $V\in \mathcal V$  and let $S \subset \partial V$  be a boundary component. Assume $$C := C(M_{\mathcal V},S)$$ is  not a characteristic compression body for $S$ in $M \setminus int(V).$ Then some component $X \subset\partial_{int} C$ is compressible in $M \setminus int(V)$. If $L$  is large, then Proposition \ref{incompressible-incompressible} implies that $X$  is homotopic in $M_{\mathcal V}$  to a boundary component of some $W \in \mathcal V$, where $W \neq V$. In particular, a level surface of $W$ can be obtained from a  level surface of $V$ via by a sequence of compressions and isotopies in $M$, where we use  a theorem of Waldhausen \cite{Waldhausenirreducible}   to convert the homotopy given by Proposition \ref{incompressible-incompressible} to an isotopy. It follows from Lemma \ref{recognizing barriers} that  any barrier of $W$ is also a barrier of $V$. So, we can eliminate $W$ from our collection of product regions without changing the collection of associated barriers, contradicting  the minimality of $\mathcal V$.
\end{proof}

\subsection{Proof of Lemma \ref{acylindrical}} \label {acylpf} As in the statement of the lemma, let $M$ be a  complete hyperbolic $3$-manifold, let $\mathcal U$  be a collection of disjoint $(g,L)$-product regions in $M$. Since the conclusion of Lemma \ref{acylindrical} does not involve the collection of product regions, but only the associated collection of barriers, we may assume via Lemma \ref{non-nested} that no two  distinct $U,V \in \mathcal U$ are isotopic in $M$, and the double compression bodies $\mathcal {DC}(M,U)$, where $U\in \mathcal U$, have been chosen so that $$U \neq V \ \implies \ \mathcal {DC}(M,U) \cap \mathcal {DC}(M,V)= \emptyset.$$

Let $\Sigma \subset M$ be the surface representing the barriers of $\CU$ whose components are defined explicitly as follows. For each incompressible product region $U \in \CU$, we include in $\Sigma$ a level surface of $U$ that lies at an $\epsilon$-distance at least $L/3$, say, from $\partial U$. Here, recall that $L=L(\epsilon,g)$ is a lower bound for the $\epsilon$-widths of our product regions. For each compressible $U\in \CU$, we include in $\Sigma$ all incompressible boundary components of $\mathcal {DC}(M,U)$, realized as in the previous paragraph. %The statement of Lemma \ref{acylindrical}  refers to an (orientable, hence two-sided) surface $\Sigma \subset M $.  Instead of  working with surfaces, it is slightly more convenient to work with disjoint unions of trivial interval bundles in $M$. There is a clear dictionary between the two points of view, since we can pass from a two-sided surface to its regular neighborhood, and from a trivial interval bundle to a level surface therein. For each $U \in \mathcal U^c$ and each  component $ S\subset \partial \mathcal{DC}(M,U)$  that is incompressible in $M$, let $I_S\cong S \times [0,1]$ be a  trivial interval bundle in $M$  such that $$I_S  \cap \mathcal{DC}(M,U)=S.$$
%We may choose these $I_S$  so that as $S$ and $U$ vary, the $I_S$ are all disjoint, and so that $I_S \cap \mathcal{DC}(M,V) \neq \emptyset \implies V=U.$ Let 
%$$I = \bigcup_{U \in \mathcal U^{inc}} U \ \ \ \  \cup \  \hspace{-5mm} \bigcup_{\substack{S \subset \partial \mathcal{DC}(M,U) \\ S \text{  incompressible}, \  U \in \mathcal U^c}} \hspace{-10mm} I_S.$$

We want to show that if $L=L(\epsilon,g)$ is large, then either
\begin{enumerate}
	\item $\Sigma$ is connected, and is a regular fiber with respect to a fibration of $M$ over $S^1$ or $S^1 / (z \mapsto -z)$, or
\item the splitting of $\pi_1 M$  associated to $\Sigma$ is $7$-acylindrical.
\end{enumerate} 

Here,  the splitting of $\pi_1 M$ in (2) has  vertices corresponding to the fundamental groups of components of the manifold with boundary $$M_\Sigma := \overline{M \setminus \Sigma}$$ obtained by cutting $M $ along $\Sigma$, and edges corresponding to the fundamental groups of components of $\Sigma$. Note that all these components $\pi_1$-inject into $M$. 

\begin{defi}
A \emph{length $n$ annulus in $(M,\Sigma)$} is a map $$f : S^1 \times [0,n] \longrightarrow M$$ that is the concatenation of `essential' maps
$$f_i : (S^1 \times [i,i+1], S^1 \times \{i,i+1\} )\longrightarrow (M_\Sigma, \partial M_\Sigma),$$
such that for each $i$, the loops $f(S^1 \times \{i-\epsilon\})$ and $f(S^1 \times \{i+\epsilon\})$ lie on opposite sides of $\Sigma$ for small $\epsilon>0$. Above, `essential' means $\pi_1$-injective and not homotopic rel boundary into $\partial M_\Sigma$.
\end{defi} 

Then the splitting of $\pi_1 M$ associated to $\Sigma$ is  $n$-acylindrical exactly when there is no length $n$  annulus in $(M,\Sigma)$. See \cite{Debloisexplicit} for a very detailed treatment of this in the case that $\Sigma$ is connected.

\medskip

We'll prove the dichotomy (1) or (2)  in a moment, but first we prove the following Lemma, which will allow us to ignore compressible product regions $U\in \mathcal U^c$.

\begin{lem}[Reducing to the incompressible case]\label {reducing to incompressibility}
Suppose that a component $\Sigma_0 \subset \Sigma$ is an incompressible boundary component of $\mathcal{DC}(M,U)$, where $U \in \CU$ is compressible. If $L=L(\epsilon,g)$ is large, then there is no essential annulus $$f : (A,\partial A)\longrightarrow (M_\Sigma,\partial M_\Sigma), \ \ A = S^1 \times [0,1]$$ such that $f(S^1 \times \{0\}) \subset \Sigma_0$ and $f(S^1 \times \{t\}) \subset \mathcal {DC}(M,U)$ for small $t$. %  and for each $i=0,1,$ let $\gamma_i := f(S^1 \times \{i\})$, and let $\Sigma_i$  be the component of $\partial M_I$  containing $\gamma_i$. For each $i$, if $L=L(\epsilon,g)$ is large then $\Sigma_i$ cannot be a component of any $\partial \mathcal{DC}(M,U)$, where $U\in \mathcal U^c$.
\end{lem}

\begin {proof}[Proof of Lemma \ref{reducing to incompressibility}]
 Let $N \subset M_\Sigma$
 be the component containing the image of $f$, and let $\Sigma_1 \subset \Sigma$ be the component containing $f(S^1 \times \{1\})$.  By the Annulus Theorem (c.f.\ Scott \cite{Scottstrong}) we can assume that $f$ is an \emph{embedded} annulus in $N$. In particular, this means that the two curves
 $$\gamma_i := f(S^1 \times \{i\}) \subset \Sigma_i$$
 are \emph{simple}. By Corollary~\ref{cutting-product}, if $L$  is large we can write $U$  as a union of three adjacent product regions
$$U = U_1 \cup U_2 \cup U_3,$$
each of which has $\epsilon$-width at least some large $L'=L'(\epsilon,g)$, to be determined. Write 
$$\mathcal {DC}(M,U) = C \cup U_2 \cup D,$$
where $C,D$ are  compression bodies containing $U_1,U_3$, and assume that $S \subset\partial C$. Since $U$ is not $\pi_1$-injective, but $S$ is, $C$ is a \emph {nontrivial} compression body.  Let $\gamma$ be the geodesic in $M$ homotopic to $\gamma_0,\gamma_1$.

Assume first that $\gamma$ intersects $ U_2$. By Corollary \ref{construction}, the inclusion of $\Sigma_0$  is then homotopic to a simplicial ruled surface $f : \Sigma_0 \longrightarrow M $ whose image intersects $U_2$. Since $d_\epsilon(U_2,\partial U)$ is large, the Bounded Diameter Lemma implies that $f(\Sigma_0) \subset U$. As $\Sigma_0$ is incompressible in $M$, it follows from Lemma~\ref{its a level surface lem} that $f$ is homotopic within $U$ to a level surface. Hence, $U$ is incompressible in $M$, a contradiction.

Now assume that $\gamma \subset M\setminus U_2$. 

\begin {claim}\label{not both homo}
The curves $\gamma_0,\gamma_1$ are not homotopic within $M \setminus U_2$.
\end {claim}
\begin {proof}
 %Hoping for a contradiction, suppose that for each $i$,  there is a homotopy in $M \setminus U_2$  from $\gamma_i$ to $\gamma$. (So, in particular $\gamma \subset M \setminus U_2.$) 

Suppose there is such a homotopy, which we can regard as a map $g: A \longrightarrow M\setminus U_2$ from an annulus $A$ that agrees with $f$ on $\partial A$. Since $M$ has no cusps, $M$ is atoroidal, so annuli that agree on their boundary curves are homotopic rel $\partial$. In other words, $f$ is homotopic rel $\partial A$ to an annulus $g$ in $M\setminus U_2$.
 
We claim that actually, $f$ is homotopic rel $\partial A$ \emph{in $N$ to an annulus in $N \setminus U_2$}. To see this, consider the cover $\tilde M \longrightarrow M$ corresponding to $\pi_1 N$. Since $N \subset M$ has incompressible boundary, it lifts to a standard compact core $\tilde N \subset \tilde M$, and $f$ lifts to an annulus $\tilde f  : A \longrightarrow \tilde N$. We can then lift the homotopy $H$ that takes $f$ to $g$ to a homotopy $\tilde H$ from $\tilde f $ to some new annulus in $\tilde M $. Compose $\tilde H$ with a retraction $\tilde M \longrightarrow \tilde N$ that takes $\tilde M \setminus \tilde N$ into $\partial  \tilde N$ and then project the result to $N \subset M$. This gives a homotopy rel $\partial A$ in $N$ from $f$ to an annulus in $N\setminus U_2$.

   %The annulus $f$ is also a homotopy Let $\hat M \longrightarrow M$ be the  cover corresponding to  the subgroup $\langle \gamma \rangle \subset\pi_1 M$, and lift $g$ to a homotopy $\hat g$ starting at  the geodesic $\hat \gamma \subset\hat M$. Then $\hat g$ ends at another lift of $\gamma$, which must be $\hat \gamma$  again, since  it is geodesic. But $\hat M$  deformation retracts onto $\hat \gamma$, so the homotopy $\hat g$ can be homotoped rel  boundary to have image $\hat \gamma$. Projecting to $M$,  this means that  the homotopy $g$  can be homotoped rel boundary  so that its image is  just $\gamma$, which lies in $M \setminus U_2$. But  $f$ is homotopic to the concatenation of $g$ with two homotopies in $M \setminus U_2$,  so this means $f$ is homotopic rel boundary in $M$  to an annulus whose image lies in $M\setminus U_2$.

We  can therefore assume  after a homotopy rel  boundary that the image of $f$  does not intersect $U_2$. All  the interior boundary components of $C$  are boundary components of $N$. So, $C$ is the entire component of $N \setminus int(U_2)$   that contains the image of $f$.  In other words, $f$ is a  properly embedded annulus in $(C,\partial_{int} C)$. A quick surgery argument then proves that $f$ is inessential, a  contradiction. (Cut $C$ into trivial interval bundles via a system of discs, and note that $f$  cannot intersect the discs essentially.) 
\end {proof}

\medskip

By Claim \ref{not both homo}, one of $\gamma_0,\gamma_1$ is not homotopic to $\gamma$ in $M\setminus U_2$. Let's assume it's $\gamma_1$ that has this property. As long as the $\epsilon $-width of $U_2$ is large, there is by  Theorem \ref{links} and Fact \ref{homotopyfact} a  bounded length, $0.025$-separated geodesic link\footnote{In the  notation of Theorem \ref{links}, this  $\Lambda$ is really $\Lambda \cup \Lambda_2$.} $\Lambda \subset U_2$ such that $\gamma_1$ is not homotopic to $\gamma$ in $M \setminus \Lambda$, which we equip with the associated negatively curved metric. Since $\gamma_1 \subset \Sigma_1$ is simple, we can use Corollary \ref{construction}  to homotope $\Sigma_1$ to a map $f : \Sigma_1 \longrightarrow M\setminus \Lambda$ that maps $\gamma_1$ nearly to a geodesic $\gamma'$ or far out a cusp. In the first case, since $\gamma_1$ is not homotopic to $\gamma$ in $M\setminus \Lambda$, the geodesic $\gamma'$ in $M\setminus \Lambda$ must intersect the neighborhood of $\Lambda$ in which curvature is variable, and in the second case the cusp must correspond to a component of $\Lambda$. Hence, in either case $f(\Sigma_1)$ lies in a bounded-radius neighborhood of $\Lambda$, with respect to the metric on $M$. Since $\Lambda \subset U_2$ and $d_\epsilon(U_2,\partial U)$ is large, $f(\Sigma_1) \subset U$. As in the previous case, Lemma~\ref{its a level surface lem} implies that $f$ is homotopic to a level surface in $U$, so $U$ is incompressible, a contradiction.
\end {proof}

Next, we show that for $n$-acylindricity, it is sufficient to consider length $n$ annuli between \emph{simple} closed curves on $\Sigma$. Note that we cannot just cite the Annulus Theorem here, like we did at the beginning of the proof of Lemma \ref{reducing to incompressibility}, since a priori a length $n$ annulus may pass through a component of $\Sigma$ multiple times, and therefore there is no chance of making the entire annulus embedded in $M$.

\begin{lem}\label{simplesuffices}
If there is a length $n$ annulus $f$ in $(M,\Sigma)$, then there is also a length $n$ annulus $g: S^1 \times [0,n]\longrightarrow M$ in $(M,\Sigma)$ such that for all $i=0,\ldots,n-1$, the curve $g(S^1 \times \{i\}) \subset M $ is simple. 
\end{lem}
\begin{proof}
	The idea here is simple. We use the characteristic submanifold theory of Jaco--Shalen and Johannsson to say that after a homotopy, the annulus $f$ factors through a `length $n$ $X$-homotopy' for some surface $X$. Any such $X$ will contain an essential simple closed curve $\gamma$, and restricting the $X$-homotopy to $\gamma$ gives a length $n$ annulus $g$ as desired.
	
Fix a compact surface $X$ with boundary. If $N$ is a $3$-manifold with boundary, an \emph{essential $X$-homotopy} in $N$ is a map $$F : (X \times [0,1], X \times \{0,1\}) \longrightarrow (N,\partial N)$$
	that is $\pi_1$-injective, an embedding\footnote{Really, we just want $F(X \times \{i\})$, $i=0,1$, to be essential subsurfaces of $\partial N$. However, it will be technically convenient to take these subsurfaces to have geodesic boundary, with respect to some hyperbolic metric on $\partial N$, in which case the right condition is that the \emph{interiors} embed.}  on both $int(X)\times \{0\}$ and $int(X) \times \{1\}$, and not homotopic rel $X \times \{0,1\}$ into $\partial N$.	We then define a \emph{length $n$ $X$-homotopy} in $(M,\Sigma)$ to be a map $$F: X \times [0,n] \longrightarrow M$$ that is the concatenation of essential $X$-homotopies $F_i := F|_{X \times [i,i+1]}$ in $ M_\Sigma,$ and where for each $i$, the images $F(X \times \{i-\epsilon\})$ and $F(X \times \{i+\epsilon\})$ lie on opposite sides of $\Sigma$ for small $\epsilon>0$.

A \emph{vertical annulus} in a length $n$ $X$-homotopy $F$ as above is a length $n$ annulus of the form $$S^1 \times [0,n] \longrightarrow M, \ \ (x,t) \longmapsto F(\gamma(x),t),$$
where $\gamma : S^1 \longrightarrow X$ is a closed curve. Two length $n$ annuli $f,g$ are \emph{homotopic in $(M,\Sigma)$} if there is a homotopy from $f$ to $g$ that is a concatenation of proper homotopies from each $f_i$ to $g_i$ in $M_\Sigma$.
To prove the lemma, it suffices to show that \emph{any length $n$ annulus in $(M,\Sigma)$ is homotopic in $(M,\Sigma)$ to a vertical annulus in  a length $n$ $X$-homotopy $F$ as above}. For then any vertical annulus corresponding to a \emph{simple} closed curve in $X$ is a length $n$ annulus $g$ as desired in the statement of the lemma.

The proof of the italicized statement is by induction. The base case, where $n=1$, is an easier version of the inductive case, and we leave that to the reader. So, suppose that we have a length $n+1$ annulus $$f : S^1 \times [0,n+1] \longrightarrow M.$$ By induction, the restriction of $f$ to $S^1 \times [0,n]$ is homotopic in $(M,\Sigma)$ to a vertical annulus in some length $n$ $X$-homotopy $$F : X \times [0,n] \longrightarrow M.$$ Extending the homotopy to a homotopy of $f$, we may assume that $f|_{S^1 \times [0,n]}$ is vertical in $F$, so that there is a closed curve $\gamma : S^1 \longrightarrow X$ such that $f(x,t) = F(\gamma(x),t)$ for all $t\leq n$. 
	
	Let $N \subset M_\Sigma$ be the component containing the image of the restriction $f_{n} := f|_{S^1 \times [n,n+1]}$. By work of Jaco--Shalen~\cite{Jacolectures} and Johannsson~\cite{johannson2006homotopy} (JSJ), \emph{there is some compact surface with boundary $Y$ and a $Y$-homotopy $G : Y \times [n,n+1] \longrightarrow N$ such that $f_{n}$ is properly homotopic to some  vertical annulus in $G$.} Here, the real result in the cited references is that $f_{n}$ is homotopic into 
	 a component $C$ of the characteristic submanifold of $(N,\partial N)$. Such a component $C$ is either a trivial interval bundle, a twisted interval bundle or a solid torus fibered by circles; JSJ show that in the first two cases our annulus is homotopic to a vertical annulus in $C$, and in the last case it is homotopic to a fibered annulus in $C$. So, in the first case, we take $G$ to parametrize $C$. In the second case, we take $G$ to be a trivial interval bundle double covering $C$. And in the third case, $f_{n}$ is homotopic to a cover of an embedded annulus, and then we can take $G$ to be a parametrization of a regular neighborhood of this embedded annulus. Just for extra clarity, at this point we are only assuming that $f_{n}$ is \emph{homotopic to a vertical annulus in $G$}, not that $f_{n}$ \emph{is} a vertical annulus in $G$.
	
 Let $\Sigma_n$ be the component of $\partial N$ containing $f(S^1 \times \{n\})$, and equip $\Sigma_n$ with a hyperbolic metric. After homotopies in $(M,\Sigma)$, we can assume that $f(S^1 \times \{n\})$ is a geodesic and $F(X \times \{n\})\subset \Sigma_n$ has geodesic boundary. Isotope $G$ so that $G(Y \times \{n\}) \subset \Sigma_n$ has geodesic boundary as well. Then we have that $$f(S^1 \times \{n\}) \subset G(Y \times \{n\}) \cap \Sigma_n.$$ Modifying the abritrary homotopy from $f_{n}$ to a vertical annulus in $G$ within a collar neighborhood of $\Sigma_n$, one can show that $f_{n}$ is properly homotopic to a vertical annulus in $G$ \emph{relative to $S^1 \times \{n\}$}. This homotopy of $f_{n}$ then extends to a homotopy of $f$ in $(M,\Sigma)$, so that we can now assume that $f$ has the property that $f|_{S^1 \times [n,n+1]}$ is vertical in $G$.
	
 Finally, since both $F(X \times \{n\})$ and $G(Y \times \{n\})$ are subsurfaces of $\Sigma_n$ with geodesic boundary, and there is a closed geodesic contained in their intersection, there is an essential subsurface $Z \subset \Sigma_n$ (not a disc) that is contained in both. We can then construct a length $n+1$ $Z$-homotopy in which $f$ is vertical by concatenating the restrictions of $F,G$ to the essential subsurfaces of $X,Y$ corresponding to $Z$.
\end{proof}

We are now ready to start the main part of the proof of Lemma \ref{acylindrical}. So, with $\Sigma$ as above, we want to show that either the splitting of $\pi_1 M$ induced by $\Sigma$ is $7$-acylindrical, or $\Sigma$ is a fiber. Assume the splitting is not $7$-acylindrical and let $$f : S^1 \times [0,7] \longrightarrow M$$  be a length $7$  annulus in $(M,\Sigma)$. 
We claim that $\Sigma$  is a regular fiber in a fibration of $M$ over $S^1$ or $S^1 / (z \mapsto -z)$.   For later use, we set $$f_i := f|_{S^1 \times [i,i+1]}, \ \ \gamma_i := f(S^1 \times \{i\}), \ \ i=0,\ldots,6 $$
and let $\Sigma_i$ be the component of $\Sigma$ that contains  $\gamma_i$. By Lemma~\ref{simplesuffices}, we can assume that each $\gamma_i$ is simple. For each $i$, let $N_i$ be the component of $M_\Sigma$ that contains the image of $f_i$.

Suppose that for some $i=1,\ldots,6$, the component $\Sigma_i \subset \Sigma$  is an incompressible boundary component of $\mathcal {DC}(M,U)$, where $U \in \CU$ is compressible. Then either $f_{i-1}$ or $f_{i}$, which end and start on opposite sides of $\Sigma_i$, must enter $\mathcal {DC}(M,U)$, so we get a contradiction by applying Lemma \ref{reducing to incompressibility}. This implies that \emph{for $i=1,\ldots,6$, the surface $\Sigma_i$ is the level surface we constructed deep inside an incompressible product region $U_i \in \CU$}. We'll now  restrict everything we do to the interval $[1,6]$, so that we can completely forget about components of $\Sigma$ that are barriers of compressible product regions.

Since $ M $ has no cusps, for $i=1,6$ the simple closed curves $\gamma_i \subset \Sigma_i$ are both homotopic in $M$ to a single closed geodesic $\hat \gamma$. Using our minimizing SRS technology, we will now homotope $\Sigma_1,\Sigma_{6}$ to surfaces that both realize $\hat \gamma$. Specifically, fixing some arbitrary\footnote{At this point, it really doesn't matter what $a$ is. Later on, it will be notationally convenient to readjust the parametrizations of our homotopies so that $a$ is a specific value, to be determined.} $a\in (1.5,5.5)$, we can construct a pair of homotopies $$h_t : \Sigma_1 \longrightarrow M, \ t\in [1,a], \ \ h_t' : \Sigma_{6} \longrightarrow M, \ t\in [a,6],$$ with the following properties:
\begin{enumerate}
	\item[(a)] $h_1$, $h_{6}'$ are the inclusions of the surfaces $\Sigma_1$ and $\Sigma_{6}$, 
	\item[(b)] $h_t(\Sigma_1) \subset U_1 \ \forall t\in [1,1.5]$, and $h_t'(\Sigma_{6}) \subset U_{6} \ \forall t\in [5.5,6],$
	\item[(c)] for $t\in [1.5,a]$, the surfaces $h_t(\Sigma_1)$ have $\epsilon$-diameter at most some $D=D(\epsilon,g)$, and similarly for $h_t'(\Sigma_{6})$, where $t\in [a,5.5]$.
	\item[(d)] we have $h_a(\gamma_1)=\hat \gamma = h_a'(\gamma_6)$, and also the parametrizations agree, namely $h_a \circ f(x,1) = h_a' \circ f(x,6)$ for all $x\in S^1$.
\end{enumerate}

Let's explain the construction of $h_t$. First, one uses Lemma~\ref{minsrsinpr} to construct a homotopy in $U_1$ from the inclusion $h_1 : \Sigma_1 \longrightarrow M$ to a minimizing SRS $h_{1.5}$, and uses Corollary~\ref{construction} to construct a minimizing SRS $\sigma : \Sigma_1 \longrightarrow M$ such that $\sigma(\gamma_1) $ is homotopic to $ \hat \gamma$ in $M$ with tracks of length at most $1$, say. The Interpolation Theorem then applies, giving a homotopy from $h_{1.5}$ to $\sigma$ through surfaces with bounded (extrinsic, say) $\epsilon$-diameter. We then concatenate the homotopy from $h_1$ to $h_{1.5}$ in $U_i$, the homotopy through bounded $\epsilon$-diameter surfaces from $h_{1.5}$ to $\sigma$, and a homotopy with tracks of length at most $1$ from $\sigma$ to a map $h_a$ that takes $\gamma_1$ to $\hat \gamma$ exactly. The construction of  $h_t'$ is the same, but with a backwards parameterization. A priori, we only have 
$$h_a(\gamma_1)= \hat \gamma  = h_a'(\gamma_{6})$$
but we can then adjust the parametrization of $h_a'$, say, so that (d) holds.

From the two homotopies above we now obtain a map
$$g : S^1 \times [1,6] \longrightarrow M, \ \ g(x,t) = \begin{cases}
h_t \circ f(x,1) & t\leq a \\ 
h_t' \circ f(x,6) & t \geq a,
\end{cases}$$
 which is continuous because of (d). Then $g(x,i) = f(x,i)$ for all $x\in S^1$ and $i=1,6$, so $g$ and $f|_{S^1 \times [1,6]}$ are two (singular) annuli in $M$ that agree on their boundaries. Since $M $ has no cusps, it is atoroidal, and hence $g$ and $f|_{S^1 \times [1,6]}$ are homotopic rel boundary.
 %
%For $i=2,3,4,5$, $f(S^1 \times \{j\}) \subset S_j$, while otherwise the image of $f$ is disjoint from the surfaces $S_j$. Using that $f,g$ are homotopic rel boundaries, we would like to say that a weaker version of the same statement is true for $g$.

\begin{claim}\label{trackingclaim}
After adjusting $a\in (1.5,5.5)$ and reparametrizing the time coordinates in $h_t,h_t'$, we can assume that
$$g(S^1 \times \{i\}) \cap \Sigma_i \neq \emptyset, \ \ \forall i=2,3,4,5$$
\end{claim}
\begin{proof}
We work in the universal cover $\pi: \BH^3 \longrightarrow M$. Since all the surfaces $\Sigma_i$ are incompressible, each component of each preimage $\pi^{-1}(\Sigma_i) \subset \BH^3$ is a topological plane that separates $\BH^3$. The union of all such topological planes models the Bass-Serre tree of the splitting of $\pi_1 M$ induced by the $\Sigma_i$, see \cite[pg 7]{shalen2002representations}, where vertices correspond to complementary components and edges correspond to the planes.

Fix $x\in S^1$ and lift the path $t \mapsto f(x,t), \ t\in [1,6]$ to a path $$t \mapsto \alpha(t), \  t\in [1,6]$$ in $\BH^3$. Then at time $t=i\in \{1,\ldots,6\}$, the path $\alpha$ passes through a topological plane $\tilde \Sigma_i$ that is a component of $\pi^{-1}(\Sigma_i)$, while otherwise $\alpha$ is disjoint from the topological planes. Moreover, the fact that $f$ is a concatenation of \emph{essential} annuli means that in the universal cover, subsequent intersections of $\alpha$ with the union of topological planes cannot occur on the same plane! In other words, $\alpha$ proceeds through the Bass-Serre tree with no backtracking, and when we write $$\alpha(1)\in \tilde \Sigma_1, \tilde \Sigma_{2},\tilde \Sigma_{3},\tilde \Sigma_{4},\tilde \Sigma_{5},\tilde \Sigma_6 \ni \alpha(6)$$
every plane $\tilde \Sigma_i$ separates all the planes on its left in the list from the planes on its right, in $\BH^3$.

Since $f$ is homotopic to $g$ rel boundary on $S^1 \times [1,6]$, the path $t\mapsto g(x,t), \ t\in [1,6]$ lifts to a path $t\mapsto \beta(t), \ t\in [1,6]$ in $\BH^3$ with the same endpoints as $\alpha$. Hence, $\beta$ must pass through all the topological planes $\tilde \Sigma_{2},\tilde \Sigma_{3},\tilde \Sigma_{4},\tilde \Sigma_5$ and we can choose intersections of $\beta$ with those planes that happen in increasing order. (There may be additional intersections, which we ignore.) Adjusting $a \in (1.5,5.5)$ and reparametrizing the time coordinates in $h_t,h_t'$, we can assume that for $i=2,3,4,5$, the path $\beta$ intersects $\tilde \Sigma_i$ at time $i$, which proves the claim.
\end{proof}

We now divide into two cases, depending on whether $a>3$, or $a<4$. The two arguments are exactly the same, up to reversing parameterizations, so let's assume $a>3$. By Claim \ref{trackingclaim}, for $i=2,3$ the surface $h_{i}(\Sigma_1)$ intersects the level surface $\Sigma_{i}$, which lies in $U_i$ at large $\epsilon$-distance from $\partial U_i$. Since $h_{i}(\Sigma_1)$ has bounded $\epsilon$-diameter, if $L$ is large enough we can assume that $h_{i}(\Sigma_1)$ is contained in $U_i$. But then since $h_{i}(\Sigma_1)$ is incompressible, it is homotopic in $U_i$ to a level surface, by Lemma \ref{its a level surface lem}. In other words, $\Sigma_1$ is homotopic to $\Sigma_i$ in $M$. So, $\Sigma_1,\Sigma_2,\Sigma_3$ are all homotopic.

Recall that $N_i \subset M_\Sigma$ is the component containing the image of $f_i$. We claim that  $N_1,N_2$ are both (possibly twisted) interval bundles. Let's focus first on $N_1$. The argument for $N_2$ is exactly the same, so we leave it to the reader. Let $$\hat M \longrightarrow M$$ be the cover corresponding to $\pi_1 N_1$, and lift $N_1$ homeomorphically to $\hat N_1  \subset \hat M$.  Lift the surfaces $\Sigma_1, \Sigma_{2} \subset M$ to the corresponding boundary components $\hat \Sigma_1,\hat \Sigma_2 \subset \partial \hat N_1$, lift the annuli $f,g$ on $S^1 \times [1,2]$ to $$\hat f,\hat g :S^1 \times [1,2]\longrightarrow \hat M,$$ and lift the homotopy $h_t$, for $t\in [1,2]$, to a homotopy $\hat h_t$, making sure that all these lifts start on the surface $\hat \Sigma_1 \subset \hat M$. Note that we have $\hat f(S^1 \times \{2\}) \subset \hat \Sigma_{2}$.   Returning to the notation of the proof of Claim~\ref{trackingclaim}, for our fixed $x\in S^1$ the paths $t \mapsto \hat f(x,t)$ and $t\mapsto \hat g(x,t)$ lift to paths in $\BH^3$ starting at the same point when $t=1$, and where the endpoints at $t=2$ are both on the \emph{same} topological plane $\tilde \Sigma_{2}$. Hence, in $\hat M$ we have that $\hat g(x,2) \in \hat \Sigma_{2}$, so applying the argument from the previous paragraph, the surface $\hat h_{2}(\Sigma_1) \subset \hat U_2$, and is homotopic to a level surface therein. So, $\hat h_t, t\in [1,2]$ is a homotopy from a level surface of $\hat U_1$ to a surface in $\hat U_2$ that is homotopic to a level surface. 

Suppose for a moment that $\hat U_1\neq \hat U_2$. Then Waldhausen's Cobordism Theorem implies $\hat U_1$ and $\hat U_2$  bound a trivial interval bundle in $\hat M$. Now, $ \hat N_1 \subset \hat M$ is a (standard) compact core that intersects $\hat U_1$ and $\hat U_2$ in subproduct regions bounded by $\hat \Sigma_1,\hat \Sigma_2$. So, $\hat N_1 \setminus (\hat U_1\cup \hat U_2)$ is also a standard compact core of $\hat M$, obtained by removing collar neighborhoods of two boundary components of $\hat N_1$.  The difference $\hat M \setminus (\hat U_1 \cup \hat U_2)$ consists of $\hat N_1\setminus (\hat U_1 \cup \hat U_2)$ and two neighborhoods of ends that are only adjacent to one of the product regions $\hat U_1,\hat U_2$. So, $\hat N_1$ is a trivial interval bundle, and hence the same is true for $N_1$. 

On the other hand, suppose $\hat U := \hat U_1=\hat U_2$. The homotopy $\hat h_t, t\in [1,2]$ cannot be homotoped rel boundary into $\hat U$, since when we pass to $\BH^3$ as in the proof of Claim \ref{trackingclaim}, the two points $\beta(1),\beta(2)$ lie on \emph{different} topological planes $\tilde \Sigma_1\neq \tilde \Sigma_2$. So, even in this case Waldhausen's theorem still applies, showing $N_1$ is a possibly twisted interval bundle.

We now know that $N_1,N_2$ are both (possibly twisted) interval bundles. When we initially chose our product regions, we assumed that no two distinct product regions $U \in \CU$ were isotopic. So, we actually have $$U:=U_1=U_2=U_3.$$  If $U$ is nonseparating, then since $N_1$ is an interval bundle and $M=U\cup N_1$, we have that $M$ fibers over the circle with $\Sigma$ a fiber. If $U$ is separating, then it bounds twisted interval bundles $N_1,N_2$ to both sides, and hence $M$ fibers over $S^1/(z\mapsto -z)$, with $\Sigma$ a regular fiber.

This proves Lemma \ref{acylindrical}.

\section {Short carrier graphs}
\label {shortgraphssec}
%In the last section, we saw (Theorem \ref{maxsplitting}) that a thick hyperbolic $3$-manifold $M$ with  bounded rank can be cut along a maximal collection of barriers (of some genus) into pieces with bounded rank.  Pick one of these pieces, and pass to the cover $M$ corresponding to its fundamental group. In this $M$, all barriers are `peripheral', i.e.\ they bound product neighborhoods of ends of $M$.  Our goal in this section is to prove, as a  step toward Theorem~\ref{main}, that such $M$ admit carrier graphs with uniformly bounded length. In the next section, we will use these short graphs to prove a version of Theorem \ref{main} for these $M$.

Fix $\epsilon>0$  and assume that \emph{$M $ is an $\epsilon$-thick complete hyperbolic $3$-manifold such that every essential closed curve on  $\partial CC(M)$  that is nullhomotopic in $M$ has length at least $\epsilon $.} Recall from Theorem~\ref{chainsprop} that the assumption on $\partial CC(M)$ implies that any minimal length carrier graph for $M$ is $D$-relatively connected for some $D=D(\epsilon)$.

%Whenever we introduce a new constant below, this constant depends  (only) on $k,\epsilon$, unless otherwise specified. 

 \label {short graph subsection} 

\medskip

Here is the main result of this section.

\begin{prop}[No barriers implies short carrier graph]\label {no barriers implies}
Given $\epsilon,k,L_0$, there are  $K_0,r_0,D$, where $K_0,r_0$ depend only on $k$, as follows.  
If $\rank \pi_1 M \leq k$ and every barrier of a $(K_0,L_0)$-product region in $M$ is peripheral, there is a carrier graph $f : X \longrightarrow M$ with $\rank \pi_1 X \leq r_0$ and $\length f(X) \leq D$.
\end{prop}

Here, a surface $S \subset M$ is \emph{peripheral} if it bounds an end neighborhood of $M$ homeomorphic to $S \times [0,\infty)$.

 We first prove Proposition \ref{no barriers implies} in a special case. We say $M$ is $(K,L)$-\emph{exceptional} if  there is a $(K,L)$-product region $V\subset M$  such that each component of $M\setminus int(V)$ is either  a (possibly trivial) compression body with missing interior boundary or a  twisted interval bundle over a non-orientable surfaces.  Note that if $M$ is exceptional and all barriers of the associated $V$ are peripheral, then it cannot be that both components of $M\setminus int(V)$ are twisted interval bundles, so in this case either both components are compression bodies, or one is a compression body and the other is a twisted interval bundle.

\begin {claim}[The exceptional case]\label {exceptional}
Given $K,\epsilon$,  there are $L,D$ such that if $M$ is $(K,L)$-exceptional and every barrier of the associated product region $V\subset M$ is peripheral, then $M$ admits a  carrier graph $f : X \longrightarrow M$ with $\rank \pi_1 X \leq 2K$ and $\length f(X) \leq D$.
\end {claim}
\begin{proof}
 Suppose that $V\subset M$ is a $(K,L)$-product region, that $C,N$  are the components of $M \setminus V$, that $C$ is a compression body with missing interior boundary, and that $N$ is either a compression body with missing interior boundary or a twisted interval bundle.

If $N$  is a compression body with missing interior boundary, then the inclusion $V \hookrightarrow N$ is $\pi_1$-surjective. As long as $L$ is large, we can find a NAT simplicial ruled surface $f : S \longrightarrow V$ in the homotopy class of a level surface. Since $M$ is $\epsilon$-thick, Corollary \ref{lipschitzcorollary}  says that when $S$  is endowed with the pullback metric, there is a  minimal rank carrier graph $X\subset S$  with bounded length. Since then $\rank(X) = \rank \pi_1 S \leq 2K$,  restricting $f$ to $X$  proves the claim.

Now suppose $N$  is a twisted interval bundle over a non-orientable surface $S$. By Theorem \ref{links}, $V$  contains a geodesic link $\Lambda $  such that $$M \setminus V \hookrightarrow M  \setminus \Lambda $$ is $\pi_1$-injective and $\Lambda $ lies far from $\partial V$. Fix a metric with pinched negative curvature on $M  \setminus \Lambda $ that agrees with the hyperbolic metric outside of $\CN_{0.025}(\Lambda )$, as  in  Lemma \ref{neg-metric}, and homotope the base of the twisted interval bundle $N \longrightarrow S $ to a simplicial ruled surface $$f  : S  \longrightarrow M  \setminus \Lambda .$$
Here, $|\chi(S )| = K-1$, so for instance we know that $S $ is a connected sum of at most $K+1$ copies of $\BR P^2$. 

%By the Bounded Diameter Lemma, the surface $S $ has diameter  bounded  by a function of $\epsilon,k$ when we equip it with the path metric induced from  the $\pi_1$-injective map $f $.

 Since the $\pi_1$-image of $f $ is the same as that of the twisted interval bundle, which $\pi_1$-injects in $M  \setminus \Lambda $, the image $f (S )$ cannot be completely contained in $V$. So, since $\Lambda $ is far from $\partial V$, the Bounded Diameter Lemma (applied in $M \setminus \Lambda $)  says that the image $f (S )$ lies (much) further than a distance of $\epsilon$ from $\Lambda $. So, the fact that $M \setminus \Lambda $ is $\epsilon$-thick away from $\Lambda$, and $S$ is incompressible in $M\setminus \Lambda$, implies that $S $ is $\epsilon$-thick as well when considered with the pullback metric.

Now $f  : S  \longrightarrow M $ is $\pi_1$-surjective, since the other component $C$ of $M  \setminus int(V)$ is a compression body  with missing interior boundary. By Corollary \ref{lipschitzcorollary},  there are carrier graphs $X \longrightarrow S $ with rank at most $K+1$ and bounded length. Composing with $f $ gives a rank $K+1$ bounded length carrier graph for $M $.
 \end {proof}

%
% To prove  Proposition \ref{no barriers implies}, it now suffices to  show that there are constants $K_0,r,D$ such that the conclusion of the proposition holds for $M$ that are not $(K_0,L)$-exceptional, where $L$ is as in Claim \ref{exceptional}.

  The idea in general is to start with a minimal length carrier graph $$f : X \longrightarrow M$$and  iteratively modify it until it has bounded length.   Recall from \S \ref{carrier} that our assumption that there are no $\epsilon$-short compressible curves on $\partial CC(M)$ implies that any such $f$ is `$D$-relatively connected' for some uniform $D=D(\epsilon)$, meaning that for every proper subgraph $Y\subset X$ there is an edge $e \subset X \setminus Y$  that has \emph{relative length} $$\length_Y^f(e) \leq D.$$ Here, relative length is obtained from the actual length by subtracting the  \emph{convex hull length} $CHL^f_Y(e,v)$ of each of the two associated \emph {rooted edges} $(e,v)$, and then taking a max with zero, see Definition \ref{relative length}.

When $f : X \longrightarrow M$ is $D$-relatively connected, and the constant $D$ is specified, we say that an edge $e$ of $X$ is \emph{short} if $\length f(e) \leq D,$ and \emph{long} otherwise. Let $$X^{sh} \subset X$$ be  the subgraph  consisting of all short edges in $X$.  The key step  that is iterated in the proof of Proposition \ref{no barriers implies}  is the following lemma.  Essentially, it says that if $e$ is a long edge that has bounded length relative to $X^{sh}$, and $v$ is a vertex of $e$, then we can replace the part of $e$ lying in the convex hull adjacent to $v$, i.e.\ the part whose length is $CHL^f_{X^{sh}}(e,v)$, with a collection of new short edges.

\begin{lem}[Modifying carrier graphs]\label {kernel}
Given $\epsilon,k,D,L$, there are $K=3k$ and $D'$ as follows.  Suppose that  $M$ is not $(K,L)$-exceptional and every barrier  of a $(K,L)$-product region is peripheral in $M$.  

Let $f : X \longrightarrow M$  be a $D$-relatively connected carrier graph with at most $k$ edges, and suppose there is an edge $e =(v,w)\subset X $ with $$\length_{X^{sh}}^f(e) \leq D.$$ Then there is a $D'$-relatively connected carrier graph $f' : X' \longrightarrow M$ with at most $9k$ edges, and an edge $e'=(v',w')$ in $X'$ such that 
\begin{enumerate} \item $\length f'(e') - CHL^{f'}_{(X')^{sh}}(e',w') \leq D'$
%\item $CHL^{f'}_{(X')^{sh}}(e',v') \leq D'$,
\item $ \length f'(e') \leq \length f(e) - CHL^f_{X^{sh}}(e,v) + D'$, and \vspace{1mm}
\item the number of long edges in $X' \setminus e'$ is at most the number of long edges in $X \setminus e$. 
\end{enumerate}
\end{lem}

In (3), remember that `short' and `long' are determined by the constant of relative connectedness, so `long' in $X$ means length at least $D$, while in $X'$ it means length at least $D'$.

Note that in Lemma \ref{kernel}, instead of setting $k$ directly to be an upper bound on $\rank \pi_1 M$, as in Proposition \ref{no barriers implies}, we set $k$ to be an upper bound for the number of edges of a given carrier graph, which is a bit easier to work with. When a carrier graph with $k$ edges exists, we have of course that $\rank \pi_1 M \leq k$. Conversely, if $\rank \pi_1 M \leq k$, then any minimal length, minimal rank carrier graph has at most $3(k-1)$ edges. So, if one does not care about precise constants, a rank bound is equivalent to a bound on the number of edges of some carrier graph.

  Really, the point of Lemma \ref{kernel} is the following  consequence.

\begin{kor}
Given $\epsilon,k,D,L$, there are $K=27k$ and $D''$ as follows. Assume $M$ and $f : X \longrightarrow M$ are as in Lemma \ref{kernel}, and that not all edges of $X $ are short. Then there is a $D''$-relatively connected carrier graph $f'' : X'' \longrightarrow M$ with at most $81k$ edges, where the number of long edges in $X''$ is strictly less than the number in $X$.\label {less long edges} \end{kor}

\begin {proof}
Since $f$ is $D$-relatively connected and $X^{sh} \neq X$, there is a long edge $e=(v,w)$ in $X$ such that $\length_{X^{sh}}^f(e) \leq D.$ Let $f' : X' \longrightarrow M$ and $e'=(v',w') \subset X'$  be as given by Lemma~\ref{kernel}. If $e'$ is short, we are done by property (3) of the lemma. If $e'$ is long, we apply Lemma~\ref{kernel} to $f'$ and $e'$, reversing the roles of the two vertices of $e'$, to get  a $D'' $-relatively connected carrier graph $$f'' : X'' \longrightarrow M$$
 and an edge $e''=(v'',w'')$  such that
$$\length f''(e'') \leq \length f'(e') - CHL^{f'}_{(X')^{sh}}(e',w') + D''.$$
(Since $K=27k = 3 \cdot 9k$, the pair $(M,f')$  still satisfies the assumptions of the lemma.) But  $f'$ satisfies property (1) in Lemma \ref{kernel}, so $$\length f''(e'') \leq D'+D''.$$ So, after replacing $D''$ by $D'+D''$, the edge $e''$ is short.  It follows from property (3)  in the lemma that the number of  long edges in $X''$ is strictly less than the number in $X$.\end {proof}

Here is how  Corollary \ref{less long edges} implies  Proposition \ref{no barriers implies}.

\begin {proof}[Proof of Proposition \ref{no barriers implies}]
Suppose that $\rank \pi_1 M\leq k$. Fix a constant $K_0=K_0(k)$, whose value we will make explicit at the end of the proof. Given this $K_0$, pick $L\geq L_0$ big enough so that Claim \ref{exceptional} applies. Assume that all barriers of $(K_0,L)$-product regions are peripheral. If $M$ is $(K_0,L)$-exceptional, then we are done by Claim \ref{exceptional}, since $K_0=K_0(k)$. So, we may assume as well that $M$ is not $(K_0,L)$-exceptional.

 Pick a minimal length rank $k$ carrier graph $f : X \longrightarrow M$. By Theorem~\ref{chainsprop}, since there are no essential curves of length less than $\epsilon$ on $\partial CC(M)$ that are nullhomotopic in $M$, the carrier graph $f$ is $D^{(0)}$-relatively connected  for some $D^{(0)}=D^{(0)}(\epsilon)$. Since $ X $ is trivalent, the  total number of edges in $X$ is at most $n=3(k-1)$.  So, after  iterating  Corollary \ref{less long edges} a total of $n$ times, we obtain a $D^{(n)}=D^{(n)}(k,\epsilon)$-relatively connected carrier graph $$f^{(n)}  :  X^{(n)} \longrightarrow M$$ with at most $ n \cdot 81^{n} $ edges, and no long edges.  In other words, $$\length f^{(n)}(X^{(n)}) \leq D^{(n)} \cdot n \cdot 81^{n}$$
 is bounded as desired. Similarly, $\rank \pi_1 X^{(n)} \leq r_0 := n \cdot 81^{n}.$

The above works as long the constant $K_0$ in Proposition \ref{no barriers implies} is  initially chosen large enough so that we {can} apply the corollary $n$ times.  In particular, all  carrier graphs that we produce  above have at most $ n \cdot 81^{n} $ edges, so it suffices to take \[K_0 = 3 \cdot n \cdot 81^{n}  =  3 \cdot  3(k-1) \cdot 81^{3(k-1)}.\qedhere\]
 \end {proof}

 It remains to prove the lemma. This following argument is in some sense the kernel of the entire paper, so it merits its own subsection.

\subsection {Proof of Lemma \ref{kernel}}
The constant $D'$ in the statement of the lemma will not be explicitly constructed. Rather, we reduce the lemma to the following claim about sequences of carrier graphs in manifolds. Below, a sequence of carrier graphs $f_i$ is \emph{uniformly relatively connected} if all the $f_i$ are $D$-relatively connected for some $D$ independent of $i$.

\begin {claim} \label {thatclaim}Let $M_i$ be a  sequence of pairwise-nonisometric hyperbolic $3$-manifolds satisfying the assumptions of the beginning of the section, i.e.\ all $M_i$ are $\epsilon$-thick, and all essential curves on $\partial CC(M_i)$ that are nullhomotopic in $M_i$ have length at least $\epsilon$.
Set $K=3k$ and assume that every barrier of a $(K,L)$-product region in $M_i$ is peripheral, and no $M_i$ is $(K,L)$-exceptional.

Let $f_i : X \longrightarrow M_i$ be a sequence of uniformly relatively connected carrier graphs with at most $k$ (geodesic) edges,  and assume the limit
\begin{equation}\label {length limit}\lim_{i \to \infty} \length(f_i(e)) \in [0,\infty]\end{equation}
 exists for every edge $e \subset X$. Call an edge $e$ \emph{short} if this limit is finite, and \emph {long} otherwise. Let $X^{sh} \subset X$ be the subgraph consisting of all short edges, and fix an edge $e_0=(v_0,w_0) \subset X$  such that $$\sup_i \length^{f_i}_{X^{sh}}(e_0) < \infty.$$
Then after passing to a subsequence of $(M_i)$, there is  a sequence of uniformly relatively connected carrier graphs $$f_i' : X' \longrightarrow M_i$$ with at most $9k$ edges, and an edge $e_0' = (v_0',w_0')$ in $X'$ such that 
\begin{enumerate} \item $sup_i \big ( \length f_i'(e_0') - CHL^{f_i'}_{(X')^{sh}}(e_0',w_0') \big )<\infty.$
\item $\sup_i \big ( \length f_i'(e_0')- \length f_i(e_0) +CHL^{f_i}_{X^{sh}}(e_0,v_0) \big ) < \infty$, and\vspace{1mm}
\item  the number of `long' edges in $X' \setminus e_0'$, i.e.\ those $e'$ for which $$\sup_i \length f_i'(e') = \infty,$$ is at most the number of long edges in $X \setminus e_0$.
\end{enumerate}
\end {claim}

 This claim implies the lemma, by the following contradiction argument. One supposes that there is no uniform $D'$ that satisfies the  conclusion in Lemma \ref{kernel} and takes a sequence 
$$f_i : X_i \longrightarrow M_i, \ \ e_i =(v_i,w_i) \subset X_i$$ of  manifolds, carrier graphs  and distinguished edges satisfying the  hypotheses of Lemma \ref{kernel} for which larger and larger $D'$ fail. Note that after passing to a subsequence, we can assume all the $M_i$ are pairwise non-isometric, since if infinitely many of them are isometric we can just take the $f'$ in Lemma \ref{kernel} to be a single fixed carrier graph, independent of $i$, which satisfies the properties (1)-(3) if we set $D'=\length(f')$, say.  Since all the $X_i$ have at most $k$ edges, we can  assume after passing to a subsequence that $X_i=X$ is some fixed graph, the  distinguished edges are all the same, and that  the limit \eqref{length limit}  exists for every edge. Apply Claim~\ref{thatclaim} and take $D'$ larger than all the suprema in its conclusion, and large enough so that the $f_i'$ are $D'$-relatively connected. For this $D'$, the conclusion of Lemma \ref{kernel} then holds for all $i$, contradicting that it was supposed to fail for larger and larger $D'$ as $i\to \infty$.

\medskip

 We now start on the proof of the claim.  Set
$$l_i := CHL^{f_i}_{X^{sh}}(e_0,v_0).$$
If $\sup_i l_i < \infty$, we can just set $f_i'=f_i$ and $e_i'=e$ and properties (1)--(3) above will be trivially satisfied.  So, we assume from now on that $$ l_i  \to \infty.$$
Our goal is to modify the carrier graphs $f_i$ without adding too many edges, and in particular without adding \emph{any} long edges, so that  we can chop off the part of $f_i(e_0)$  whose length gives $l_i$.

\vspace{2mm}

 Let's recall the definition of the convex hull length $l_i$ given in \S \ref{carrier}, specifically the version in \eqref{rellength2}.  If $v_0 \not \in X^{sh}$, then $l_i=0$  for all $i$ by definition, contrary to our assumption that $l_i \to \infty$. So,  suppose $v_0$ lies in a connected component $Y \subset X^{sh}$. Let $$\pi_i : \hat M_i \longrightarrow M_i$$ be the cover corresponding to $(f_i)_*(\pi_1Y) \subset \pi_1 M_i$, and let
$$Y' = Y \cup_v e_0$$ 
be the graph obtained by attaching $e_0$ to $Y$ along $v_0$, and let $$\hat f_i : Y' \longrightarrow \hat M_i,$$ be a lift of $f_i |_{Y'}$ into $\hat M_i$. Then
\begin{equation}
	l_i = \length(\hat f_i(e_0) \cap \CN_1(CH(\hat f_i(Y)))).\label {len3}
\end{equation}

\medskip

\noindent \it Step 1:  extracting limits and simplifying the problem. \rm The maps $$\hat f_i |_Y : Y \longrightarrow \hat M_i$$ are all $\pi_1$-surjective and bounded length, so  as in Example \ref{limit example} in \S \ref{convergence section}  we can pass to a subsequence  so that when marked by these graphs, the manifolds $(\hat M_i)$  algebraically converge to some $\hat M_A$. Since all our manifolds are $\epsilon $-thick, we can also assume that when a basepoint $y\in Y$ is fixed, we have a geometric limit
$$(M_i,f_i(y)) \to (M_G,\star).$$ Let $(\phi_i)$ be a sequence of almost isometric maps given by this geometric convergence. By Arzela-Ascoli, we have $$ ( f_i|_Y : Y \longrightarrow M_i) \overset{i\to \infty}{\longrightarrow} (  f_G : Y \longrightarrow  M_G)$$
 after passing to a subsequence, where here  convergence means that $(\phi_i)^{-1} \circ f_i$ is defined for large $i$ and converges to $f_G$ as $i\to \infty$.  Moreover, as long as our maps $(\phi_i)$  were selected to be compatible with the Chabauty convergence of the images of the algebraically convergent sequence of representations we picked (see \eqref{compatibility} in Proposition \ref{convergence-convergence-geom}), we can assume that  there is a covering map $$\pi : \hat M_A \longrightarrow M_G $$
and a $\pi_1$-surjective map $\hat f_A : Y \longrightarrow \hat M_A$ such that $f_G = \pi \circ \hat f_A$. Again, we refer the reader to Example \ref{limit example} in \S \ref{convergence section} for details. Note that if we were following Example \ref{limit example} exactly, we would be taking a geometric limit of the manifolds $\hat M_i$. However, all the same arguments work here.

\begin{claim}\label {get strong}
Replacing the  carrier graphs $f_i$ by new carrier graphs with at most $3k$ edges, and replacing the covers $\hat M_i$ and limiting objects $\hat f_A, \hat M_A$ with their analogues coming from the new carrier graphs, we can assume  after passing to a subsequence that
\begin{enumerate}
	\item $(\hat M_i,\hat f_i(y))$  converges geometrically to $(\hat M_A,\hat f_A(y)),$
\item  every degenerate end of $\hat M_A$  has a neighborhood on which the covering map $\pi : \hat M_A \longrightarrow M_G$ restricts to an embedding.
\end{enumerate}
\end{claim}

Here, by `we can assume' we mean that the new graphs have all the same properties that we required of the old graphs (except that the edge bound is $3k$) and that proving Claim \ref{thatclaim} with the new carrier graphs in place of the $f_i$ there, will imply it for the original $f_i$. 

\begin{proof}
 Since the $M_i$ are pairwise non-isometric and $\epsilon$-thick, the geometric limit $M_G$ has infinite volume. So by Proposition \ref{limits-fg}, the  covering map $\pi$  factors through an intermediate cover $\hat M_A \longrightarrow N \longrightarrow M_G$, where $$\rank(\pi_1 N) \leq 2\rank(Y) \leq 2k,$$  and where every degenerate end in $N$ has a neighborhood that embeds in $M_G$.  
Create a new graph 
 in $M_G$ by letting $Z$ be a wedge of $2k$ circles, letting $Y \vee_y Z$ be the graph obtained by attaching the wedge point of $Z$ to our basepoint $y\in Y$, and extending $ f_G$ to $$\bar f_G : Y \vee_y  Z \longrightarrow M_G$$ in such a way that $(\bar f_G)_*(\pi_1 (Z,y)) = \pi_1 (N) \subset \pi_1 (M_G,f_G(y)).$  Define $$\bar f_i : X \vee_y Z \longrightarrow M_i, \ \ \bar f_i |_{X} = f_i$$ so that the image of each edge of $X$ is  a geodesic segment in $M_i$, and so that  the compositions $(\phi_i)^{-1} \circ  \bar f_i |_{Y \vee_y Z}$  are defined and converge to $\bar f_G$ as $i\to \infty$.
Then the sequence of covers $$\hat M_i' \longrightarrow M_i, \ \ \pi_1(M_i') = (\bar f_i)_*(\pi_1 Y \vee_y Z) \subset \pi_1(M_i)$$  converges algebraically to $N$, when each $\hat M_i'$ is marked by a lift of the bounded length graphs $\bar f_i |_{Y \vee_y Z}$.  
After passing to a subsequence, we can also assume that $(\hat M_i',\bar f_i(y))$  converges geometrically to some $\hat M_G'$. 
There are then covering maps $$N \overset{p}{\longrightarrow} \hat M_G' \longrightarrow M_G.$$ Since degenerate ends of $N $ embed in $M_G$, they also embed in $M_G'$, and hence $p : N \longrightarrow \hat M_G'$ is an isometry by Theorem \ref{max-cyclic-ugly}. So,  the required properties (1) and (2) hold for the  augmented maps $\bar f_i$.

The $\bar f_i$ have at most $3k$ edges, and  otherwise have all the same properties we required of the $f_i$.  For instance, the graphs $\bar f_i(Z)$  have bounded length and $\bar f_i|_X$ is a slight perturbation of $f_i$ for large $i$, so  the carrier graphs  $\bar f_i$  are uniformly relatively connected. And including $Z$ into the short subgraph $(X \vee_y Z)^{sh}$,  we still have that: $$\sup_i \length^{\bar f_i}_{(X \vee_y Z)^{sh}}(e_0) \leq \sup_i \length^{f_i}_{X^{sh}}(e_0)  < \infty$$
Finally, note that proving Claim \ref{thatclaim} for the $\bar f_i$ will  prove it for the original $f_i$.  The only thing to  be careful with is that property (2) in Claim \ref{thatclaim} references convex hull length in relation to the short subgraph, and this could change when we add in $Z$. However, the only way that it can change is to increase, and proving (2) with a larger value for the convex hull length  certainly proves it for a smaller value.
 \end{proof}

 Recall that $e_0$  is adjacent to $Y \subset X$ and the maps $f_i|_{Y \cup_v e_0}$ lift to 
\begin{equation}\hat f_i : Y \cup_v e_0 \longrightarrow \hat M_i, \ \ \ l_i =\length(\hat f_i(e_0) \cap \CN_1(CH(\hat f_i(Y)))).\label{lengtheq}	
\end{equation}

\begin{claim}\label {geominimizing}
 After replacing the $f_i$ by new carrier graphs satisfying all the  same properties, we may assume that $\hat f_i (e_0)$  minimizes the distance in  $\hat M_i$ between its endpoints.
\end{claim}
\begin{proof}
Replace $\hat f_i(e_0)$ by  a distance minimizing geodesic in $\hat M_i$ joining its endpoints, and modify $f_i$ on $e_0$ accordingly. Since  
$$\hat f_i |_Y: Y \longrightarrow \hat M_i$$ is $\pi_1$-surjective, the new $ f_i (e_0)$ is  homotopic to a concatenation of the old one and the image under $ f_i$ of a loop in $Y$. In particular, $f_i$ is still a carrier graph for $M_i$.  We claim it has all the same properties as before.

First, Corollary \ref{edge homotopies} implies that $\length_{X^{sh}}^{f_i}(e_0)$ changes by a bounded amount when $f_i$ is modified. So, we still have
\begin {equation}\sup_i \length_{X^{sh}}^{f_i}(e_0) <\infty.\label {bddrellength}\end{equation}
 Second, we claim that we can still assume that the  $f_i$ are  uniformly relatively  connected.  It suffices to show that if $Z \subset X$ is a proper subgraph,  there is always an edge
$e \subset X \setminus Z$ with $$ \sup_i \length_Z^{f_i}(e)<\infty.$$
 We do this in three cases.
\begin {enumerate}
\item   If $Z$  does not contain $X^{sh}$, just pick $e \subset X^{sh}$.
 \item  If $Z$ contains $X^{sh}$ but not $e_0$, set $e=e_0$ and use \eqref{bddrellength}.
\item Suppose $Z$  contains $X^{sh}$ and also $e_0$. Note that the convex hull of $f_i(Y \cup_v e_0)$ in $\hat M_i$ is unchanged when $f_i$ is modified. So, for any edge $e \subset X \setminus Z$ we have that $\length_Z^{f_i}(e)$ is unchanged, so there is still some $e$ with bounded relative length.
\end {enumerate}
Finally, if after the modification $l_i$ is bounded, then we are done, setting $f_i'=f_i$ and $e_0'=e_0$ as in the beginning of the proof of Claim~\ref{thatclaim}. So, we can assume  after passing to a subsequence that $l_i \to \infty$. 

In summary, all the properties that held before still hold and  {we also have that $\hat f_i(e_0)$  minimizes the distance in $\hat M_i$ between its endpoints.}  The only thing left to mention is that by Corollary \ref{edge homotopies},  the difference
$$\length(f_i(e_0)) - CHL_{X^{sh}}^{f_i}(e_0,v_0)$$
 only changes by a bounded amount when we do the  above modifications. So, proving (2) in Claim~\ref{thatclaim}  for the new carrier graphs will prove it for the old ones.
\end{proof}

 Abusing notation, let's define new maps $$ f_i : Y \cup_v [0,\infty) \longrightarrow M_i,$$
where $v\in Y$ is glued to $0\in [0,\infty)$,
by identifying $e_0$ by arclength with $$[0,\length   f_i(e_0)] \subset [0,\infty)$$ and then defining $  f_i$ to be constant on  $[\length   f_i(e_0),\infty)$.  After  choosing appropriate parametrizations and passing to a subsequence, the maps $(\phi_i)^{-1} \circ   f_i $  converge on compact sets\footnote{Note that one has to first restrict to a compact subset of $Y \cup_v [0,\infty)$ before one can even say that the maps $(\phi_i)^{-1} \circ   f_i $ are  defined for large $i$.}
to a map 
$$f_G : Y \cup_v [0,\infty)\longrightarrow M_G$$
 that extends our original $f_G : Y \longrightarrow M_G$. Lift this $f_G$ to a map
$$\hat f_A : Y \cup_v [0,\infty) \longrightarrow \hat M_A$$
 that also extends our original map with the same name. Then because the edges $\hat f_i(e_0)$ are distance minimizing in $\hat M_i$, the map $\hat f_A$ takes $[0,\infty)$ to a distance minimizing geodesic ray in $\hat M_A$.

The ray $\hat f_A([0,\infty))$ must exit an end $\hat \CE$ of $\hat M_A$.  As $\hat M_A$ is  an algebraic limit, its fundamental group is  finitely generated, and hence it is  topologically and geometrically tame by the Tameness Theorem (or  more historically, by \cite{Canarylimits} or \cite{Brockalgebraic} in this case). Note that since $\hat M_A$ has a carrier graph with at most $3k$ edges,  we have $ \rank \pi_1 \hat M_A \leq 3k.$ In particular,  the first homology of a $3$-manifold compactification of $\hat M_A$ has rank at most $3k$, so by half-lives-half-dies, we have
\begin {equation}
genus(\hat \CE) \leq 3k.	
\end {equation}

\begin{claim}
	The end $\hat \CE$  is degenerate.\end{claim}
\begin{proof}
Suppose not. Since $\hat M_A$ is geometrically tame and has no cusps, the end $\hat \CE$ must then be convex cocompact. 
Choose a lift $$\hat \phi_i : \hat M_A \longrightarrow \hat M_i, \ \ \pi_i \circ \hat \phi_i = \phi_i \circ \pi$$  that takes the basepoint $\hat f_A(y) \in \hat M_A$ to $\hat f_i(y) \in \hat M_i$. So, $$(\hat \phi_i)^{-1} \circ \hat f_i\overset{i\to\infty}{\longrightarrow} \hat f_A.$$
Since $\sup_i \length \hat f_i(Y)<\infty$ and each $\hat f_i(Y)$ has nontrivial image in $\pi_1 \hat M_i$, Lemma \ref{distancetogeodesic} gives some $R$ independent of $i$ such that \begin{equation} \label{inr} \hat f_i(Y) \subset \CN_{R}(CC(\hat M_i)).\end{equation}

 Since the convergence $\hat M_i \to \hat M_A$ is strong, Theorem \ref{strong limits} gives a standard compact core $C \subset \hat M_A$ as follows. Suppose that $S \subset \partial C$ faces the end $\hat \CE$. Then for large $i$, the image $\hat \phi_i(S)$ faces a component $$E_i\subset \hat M_i \setminus \hat \phi_i(C)$$ that is a neighborhood of a convex cocompact end of $\hat M_i$, and where \begin{equation}E_i \subset \hat M_i \setminus \CN_{R+2}(CC(\hat M_i)),  \label{abbbove} \end{equation} say.  Since $\hat f_A([0,\infty))$ exits $\hat \CE$, we can pick some $t\in [0,\infty)$ such that  $\hat f_A(t) \in S \subset \partial C$. As $i\to \infty$, we have $(\hat \phi_i)^{-1} \circ \hat f_i(t) {\longrightarrow} \hat f_A(t),$ so using this, equation \eqref{abbbove}, and the fact that the $\phi_i$ are almost isometric, for large $i$ we have that $\hat f_i(t) \subset \hat M_i \setminus \CN_{R+1}(CC(\hat M_i))$. In light of \eqref{inr}, this means that $\hat f_i(t) \in \hat M_i \setminus \CN_1(CH(\hat f_i(Y)))$. So, $l_i \leq t$ for large $i$, contradicting that $l_i \to \infty$.  
\end{proof}

\medskip

\noindent \it Step 2: Finding product regions. \rm The surgeries we will perform on the carrier graphs $f_i$ will be directed by wide product regions in the $M_i$, which we produce using the degenerate end $\hat \CE$. By Claim \ref{get strong}, there is a  product neighborhood of $\hat \CE$ that embeds under the covering $$\pi : \hat M_A \longrightarrow M_G,$$ whose image in $M_G$ we can choose to be disjoint from $f_G(Y)$. 
By Proposition~\ref{constructing prs} and Corollary \ref{truncating noncompact}, there is a very wide, compact product region $\hat U \subset \hat M_A$ contained as a topological subproduct region of this end neighborhood.  Then $\pi$  restricts to a homeomorphism from $\hat U$  onto some product region $$U =\pi(\hat U)\subset M_G.$$
Proposition \ref{bilipschitzpreservation}   then implies that if $U$ is sufficiently wide and $i$ is large, the topological product regions $ \phi_i(U) \subset M_i$ and $\hat \phi_i(\hat U) \subset \hat M_i$ contain very wide (geometric) subproduct regions $$W_i\subset M_i, \ \ \hat W_i \subset \hat M_i, \ \ \pi_i(\hat W_i)=W_i.$$
We are being intentionally vague at the moment about how wide we want our product regions to be. While we could write down all the properties that we want their widths to satisfy right now, we think it will be easier for the reader if we introduce them one by one. So, below we will repeatedly say things like \emph{`as long as the $W_i$ were chosen wide enough...'} understanding that this can be arranged by starting with a larger $\hat U$, as long as the desired widths are independent of $i$.

 As a first example,  we may assume that for large $i$, $W_i$ is wide enough relative to $\epsilon$ and the genus bound $3k$ so that Theorem \ref{double compression body theorem} applies, and so that $W_i$ is wider than the $L$ in the statement of the current lemma, guaranteeing that all barriers of $W_i$ are peripheral in $M_i$, and that  it is not the case that both components of $M_i \setminus int(W_i)$ are  either compression bodies with missing interior boundary or twisted interval bundles over non-orientable surfaces.
By Theorem \ref{double compression body theorem} (1), any boundary component of the double compression body $ DC(W_i) \subset M_i$ that has smaller genus than $W_i$ is incompressible, hence a barrier, and therefore is peripheral in $M_i$. Since  $M_i$ is non-exceptional, one can then see that $M_i \setminus int(W_i)$ is the disjoint union of a (possibly trivial) compression body $C_i'$ with missing interior boundary, and a $3$-manifold with incompressible boundary that is not a twisted interval bundle over a non-orientable surface.

%$$M_i \setminus int(W_i) = N_i \cup C_i$$ where \emph{$C_i$  is a (possibly trivial) compression body with missing interior boundary, $N_i$ is not a twisted interval bundle, and the surface $ W_i \cap N_i$ is  incompressible in $N_i$.}

Second, we can assume that $W_i$ is wide enough so that Corollary \ref{extending product regions} applies.  This means that $W_i$ is a subproduct region of some $V_i \subset M_i$ such that  the following paragraph holds.  Let $$C_i :=C_i' \cup V_i,$$
which is a compression body with missing interior boundary in which $V_i$ is a  product neighborhood of the exterior boundary $\partial_{e} C_i$. Let $$\tilde M_i \longrightarrow M_i$$ be  the cover corresponding to $\pi_1 V_i$. If $\tilde C_i \subset \tilde M_i$  is a homeomorphic lift of $C_i$, Corollary \ref{extending product regions} says that the difference $$\tilde E_i := \tilde M_i \setminus \tilde C_i$$ is a  product neighborhood of an end $\CE_i$ of $\tilde M_i$. This end $\CE_i$ is convex cocompact: if  not, Canary's Covering Theorem \cite{Canarycovering}   implies that either
\begin{itemize}
\item $M_i$  fibers over the circle and level surfaces of $W_i$ are fibers,	\item $M_i \setminus C_i$ is a product neighborhood of a degenerate end of $M_i$.
\end{itemize} The first option is impossible, since then any level surface of $W_i$ is a barrier that is not peripheral in $M_i$, while in the second case, $M_i$ is $(K,L)$-exceptional. So, Corollary~\ref{extending product regions} says that  the diameter of the intersection $\tilde E_i\cap \CN_1(CC(\tilde M_i)) $  is  at most some constant depending only on $K,\epsilon$.
Informally, we have  now constructed $C_i$ so that $\partial C_i$  is at bounded distance from the boundary of the `convex hull' of $C_i$ in $M_i$.

\medskip
\noindent \it Step 3, the modification. \rm  Here, we modify each carrier graph $f_i$ by `chopping off'  the part of it that lies in $C_i$ and replacing it by a short carrier graph for $C_i$ that lies near  $\partial C_i$.   We will obtain  a new sequence of carrier graphs $$f_i' : X' \longrightarrow M_i,$$ and   the goal  is to show that these $f_i'$ are still $D'$-relatively connected for some $D'$ independent of $i$, and that the edge $e_0=(v_0,w_0) \subset X$  corresponds to an edge $e_0'=(v_0',w_0') \subset X'$ such that properties (1)--(3) in Claim \ref{thatclaim} hold.

\begin {claim}[Short markings near $\partial_{e} C$]\label {short markings claim} Let $W$ be a graph with one vertex  $\star \in W$ and $6k$ edges.  Then there are
 $\pi_1$-surjective maps $$\omega_i : W\longrightarrow V_i,	\ \ \star_i := \omega_i(\star)
$$
 such that each edge of $W$ is mapped to a geodesic segment by $\omega_i$,  the images $\omega_i(W)$ have length bounded above independent of $i$, and 
$$\sup_i d(\star_i, \partial C_i) < \infty.$$\end {claim}

Recall that $genus(V_i)=genus(\hat \CE)\leq 3k.$

\begin {proof}
Pick $\epsilon$-thick simplicial ruled surfaces in $V_i$ in the homotopy classes of level surfaces that lie at bounded distance from $\partial_{e} C_i$, and apply Corollary~\ref{lipschitzcorollary} to find short graphs marking these surfaces.
\end {proof}

 Proceeding with the construction, for each $i$ homotope $\partial C_i$ slightly so that it is transverse to $f_i$, and then let $\mathcal C_i \subset X$ be the union of all components of $f^{-1}_i(C_i) \subset X$ that contain vertices of $X$. Each component of $ \mathcal C_i$ is the union of a subgraph of $X$ with some adjacent subsegments of edges. After passing to a subsequence, we can assume that the combinatorics of $\mathcal C_i$ is independent of $i$, i.e.\ the only difference when $i$ varies is that the endpoints of the subsegments change. And then after reparameterizing edges\footnote{ These reparametrizations may change the fact that the restrictions $f_i |_Y$ pointwise converge, but we will not use this again in the proof.} we can actually assume that $\mathcal C_i$ is itself independent of $i$. So, instead of $\mathcal C_i$ we now write
$$\mathcal C \subset X.$$

%Let $X'$ be the graph obtained from $X $ by collapsing all edges $e \subset X$ with $f_i(e)\subset C_i$ to a single point $\star$; here, we pass  to a subsequence so that  the same edges are collapsed for each $i$.  Abusively, if $e \subset X$ is an edge that is not collapsed, we identify it with the corresponding edge in $X'$.  (So since no vertex of $Y$ maps into $C_i$, the graph $Y$ is also embedded in $X'$.) 

Form a new graph $X / \mathcal C$ by  collapsing $\mathcal C$ to a point $\star$, and let $$q : X \longrightarrow X / \mathcal C$$  be the quotient map.  Glue $\star \in  X / \mathcal C$ to $\star \in W$ to create  the wedge $$X':=X/\mathcal C \vee_\star W.$$ 

%Let $\mathcal C_0 \subset\mathcal C$ be the union of all edges of $X$ that are completely contained in $\mathcal C$.   We then have the following one-one correspondences between certain edges and subgraphs of $X$ and those of $X'$:
%\begin{align}
%	\label {edge correspondence}\text{edges } e \subset X \setminus \mathcal C_0 \subset X &\longleftrightarrow \text{edges } e' \subset X /  \mathcal C \subset X'\\ 
%e & \longmapsto e'=q(e), \nonumber \\
%\label{subgraph correspondence}
%{Z \subset X \text{ with }  \mathcal C_0 \subset Z} & \longleftrightarrow { Z' \subset X' \text{ with }  W \subset Z' }, \\
%Z & \longmapsto Z'=h(Z) \cup W. \nonumber \end{align}
%

\begin {claim}[Pinch and replace]\label {pinch and replace} There is a homotopy $F^t_i : X  \longrightarrow M_i$, where $t\in [0,1]$, from $F_i^0 = f_i$ to a composition $$X \overset{h}{\longrightarrow} X' \overset{f_i'}\longrightarrow M_i, \ \ f_i' \circ h = F_i^1,$$  such that the following properties hold.
\begin {enumerate}
\item $f_i'$ has geodesic edges and $f_i' |_W= \omega_i$, as in Claim \ref{short markings claim}.
\item $h=q$ on $X \setminus \mathcal C$, while $h(\mathcal C) \subset W$.
\item If $x \in X$ is on an edge that does not intersect $\mathcal C$, then the path $t \mapsto F_i^t(x)$ is constant.
\item If $x\in X \setminus \mathcal C$, the path $t \mapsto F_i^t(x)$ has length at most some universal  constant independent of $i$ and $x$.
\item If $x \in \mathcal C$, the path $t \mapsto F_i^t(x)$ is contained in $C_i$.
 \end{enumerate}\end{claim}

See Figure \ref{pinch and replace fig} for an illustration.

\begin {proof}
 Since $\sup_i d(\star_i, \partial C_i) < \infty$, there is for each $i$ a homotopy $$(F_i^t), \ t\in [0,1/3], \ \ F_i^0=f_i$$   such that (3), (4), (5) all hold, and where $f^{1/3}_i(x) =\star_i$  for all $x\in \partial C_i$. (One can even make this homotopy supported in an arbitrarily small neighborhood of $\partial \mathcal C$: just pick a  bounded length path from each $f_i(x), \ x\in \partial \mathcal C$ to $\star_i$ and near each $x$, drag $f_i$ along this path.) 

Since each $\omega_i : W \longrightarrow V_i$    is $\pi_1$-surjective,  we can now find a homotopy $(F_i^t), \ t\in [1/3,2/3]$  supported on $\mathcal C$ and satisfying (5),  such that $f^{2/3}_i |_{\mathcal C}$ factors  through a map $h : (\mathcal C,\partial C) \longrightarrow (W,\star)$. Finally, extend this to a homotopy $(F_i^t), \ t\in [2/3,1]$ that is constant on $\mathcal C$ and homotopes each edge of $X \setminus int(\mathcal C)$ rel endpoints to be geodesic. (Here,  we are regarding $X \setminus int(\mathcal C)$ as a graph itself, so an `edge' can be a segment of an edge of $X$ that has an endpoint in $\partial \mathcal C$.)  Since $f_i$ had geodesic edges,  we can do this so that (3) and (4) hold. The resulting map $F_i^1$  clearly factors as indicated in the claim, and satisfies conditions (1)--(5).\end {proof}

\begin{figure}
	\centering
	\includegraphics{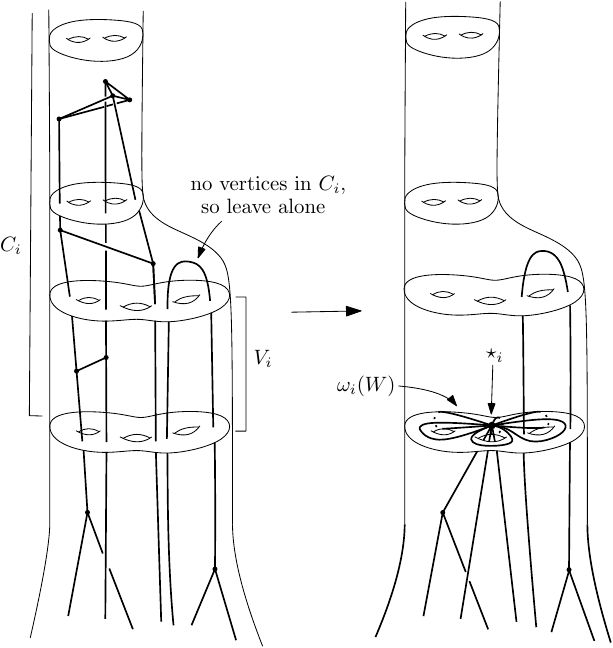}
		\caption{In Claim \ref{pinch and replace}, we create new carrier graphs $f_i'$ by pinching all intersections of $f_i$ with $\partial C_i$ to the point $\star_i$, and then replacing $f_i(X) \cap C_i$ with $\omega_i(W)$. However, if part of an edge of $f_i(X)$ enters and exits $C_i$ without going through any vertices, we leave it alone. }
	\label{pinch and replace fig}
	
\end{figure}

Focus now on $f_i' : X' \longrightarrow M_i$. This $f_i'$ is a carrier graph for $M_i$, since $f_i$ is homotopic to something that factors through it.  The number of edges of $X'$ is at most $9k=3k+6k$, where $3k$ is a bound for the number of edges of $X$ and $6k$ is the number of edges in $W$.

\begin{claim}[Convex hulls can't shrink]\label {rellength don't change}
Fix a  proper subgraph $Z \subset X$ and let $Z'=q(Z) \cup W$. Suppose $e$ is an edge of $X$ that is not completely contained in $\mathcal C$. Let $v$ be a vertex of $e$ and let $e'=q(e)$ and $v'=q(v)$  be the corresponding edge and vertex in $Z'. $ Let $e_{v,i}, e_{v',i}'$  be the segments of $e,e'$  beginning at $v,v'$  that have length
$$\length e_{v,i} = CHL_{Z}^{f_i}(e,v), \ \ \length e_{v',i}' = CHL_{Z'}^{f_i'}(e',v'), \ \ .$$ Then if $x\in e_{v,i}$,  we have that $$\sup_i d_{e'}(q(x),e_{v',i}')< \infty,$$ where here $d_{e'}(\cdot,\cdot)$  is the distance along the edge $e'$, endowed with the pullback metric induced by $f_i'$.\end{claim}

Here,  we assume that $e$ is not completely contained in $\mathcal C$ just  so that $e'=q(e)$  is actually an edge, rather than a vertex. However, suitably interpreted the claim also holds  in this degenerate case.

 Deferring the  proof for a moment, we have:

\begin{kor}
	The $f_i'$ are  uniformly relatively connected.
\end{kor}
\begin {proof}
	Let $Z' \subset X'$  be a proper subgraph.  It suffices to show that there is an edge $e'$ in $X'\setminus Z'$ with $$\sup_i \length_{Z'}^{f_i'}(e') < \infty.$$ 
 If $Z'$  does not contain $W$, then  we can set $e'$ to be any edge of $W$ since $f_i'(W)$ has bounded  (absolute) length. 

Suppose $Z' \supset W$,  and let $Z\subset X$ be the  subgraph  formed by all edges $e \subset X$ with $q(e)\subset Z'$.  Since $Z'$  is a proper subgraph of $X'$ and contains $W$,  the subgraph $Z \subset X$  is also proper.  The $f_i$  are uniformly relatively connected, so there is an edge $e=[v,w] \subset X \setminus Z$  with
$$\sup_i \length_{Z}^{f_i}(e) < \infty.$$
 Note that $e \not \subset \mathcal C$, since $\mathcal C \subset Z$, and let $e'=[v',w']$ be the $q$-image of $e=[v,w].$
Apply Claim \ref{rellength don't change} twice, once for each vertex of $e$, remembering that the two subsegments $e_{v,i},e_{w,i}$ exhaust all but a segment of $e$ with bounded $f_i$-length. In light of part (4) of Claim \ref{pinch and replace}, the result is that the two segments $e'_{v',i},e'_{w',i}$ exhaust all but a segment of $e'$  with bounded $f_i'$-length. So, $e'$ has bounded length relative to $Z'$  as desired.
 \end {proof}

 We now prove the claim.

\begin{proof}[Proof of Claim \ref{rellength don't change}]
 The claim is trivial unless $v$ lies in some component $Z_0 \subset Z$, for if it does not then $e_{v,i}=\{v\} \subset e_{v',i}'$. Let $Z_0' \subset Z'$ be the component containing $q(Z_0)$.
 
Let $\tilde M_i'$ be  the cover of $M_i$  corresponding to the subgroup $$(f_i')_*(\pi_1 Z_0') \subset \pi_1 M_i.$$
Letting $Z_0 \cup_v e$ and $Z_0' \cup_{v'}e'$ be the graphs obtained by attaching $e,e'$ to $Z_0,Z_0'$ exactly at $v,v'$, we can now lift the entire homotopy $F_i^t|_{Z_0 \cup_v e}$ to $\tilde M_i'$, along with the associated maps $f_i|_{Z_0 \cup_v e}$ and $f_i'|_{Z_0' \cup_{v'}e'}$,  giving 
\begin{align*}
\tilde F^t_i : Z_0 \cup_v e  \longrightarrow \tilde M_i', \ \ \tilde f_i : Z_0 &\cup_v e  \longrightarrow \tilde M_i', \ \  \tilde f_i' : Z_0' \cup_{v'} e' \longrightarrow \tilde M_i' \\ \tilde F_i^0 =\tilde  f_i, \ \ & \tilde F_i^1 = \tilde f_i' \circ h. \end{align*}

We claim that for some $r>0$ independent of $i$, we have
\begin{equation}\label {contains2258}
	\tilde f_i(Z_0) \subset \CN_r(CH(\tilde f_i'(Z_0'))),  
\end{equation}
where here as in \S \ref{carrier}, $CH$ stands for convex hull in $\tilde M_i'$. If $Z_0\cap \mathcal C=\emptyset$, then the homotopy from $f_i $ to $ f_i' \circ q$ is constant on $Z_0$, so lifting, we have $\tilde f_i = \tilde f_i' \circ q$ on $Z_0$ as well. As $Z_0'=q(Z_0)$, we have $\tilde f_i(Z_0)=\tilde f_i'(Z_0')$, so \eqref{contains2258} is trivial. Therefore, we may assume from now on that $Z_0$ intersects $\mathcal C$, in which case we have $Z_0' \supset W$. 

Since $W \subset Z_0'$ and the map $$\omega_i : W \longrightarrow V_i \subset C_i$$ is $\pi_1$-surjective, we can then lift $V_i \subset C_i$ homeomorphically to $$\tilde V_i \subset \tilde C_i \subset \tilde M_i'$$ in a way that is compatible with the lifted homotopy above; that is, so that $ f_i' |_W$ is the  composition of $\tilde f_i' |_W$ and the projection $\tilde V_i \longrightarrow V_i$. So, \begin{equation}\tilde f_i(x) \in\tilde C_i , \ \ \forall x \in Z_0 \cap \mathcal C.\label{inc} \end{equation}

By Lemma \ref{geosinprs} and the fact that the $M_i$ are all $\epsilon $-thick, $\tilde V_i$  is contained in an $r$-neighborhood of $CC(\tilde M_i')$  for some $r$  independent of $i$. Any components of $\tilde C_i \setminus \CN_r(CC(\tilde M_i'))$ must then be product neighborhoods of  convex cocompact ends of $\tilde M_i'$ that map down homeomorphically to components of $M_i \setminus \CN_r(CC(M_i))$, see Lemma \ref{convex cocompact cover}. By Lemma \ref{cgincc}, after increasing $r$ by some amount depending only on $D$,  we can assume that the image $f_i(X)$ is contained in $\CN_r(CC(M_i))$. So, it follows that for all $x$ such that $\tilde f_i(x) \in \tilde C_i$, e.g.\ by \eqref{inc} for all $x\in Z_0 \cap \mathcal C$, we have $$\tilde f_i(x) \in \CN_r(CC(\tilde M_i')) \subset \CN_r(CH(\tilde f_i'(Z_0'))).$$
Here as in \S \ref{carrier}, $CH$ stands for the convex hull in $\tilde M_i'$. But if $x \in Z_0 \setminus \mathcal C$, part (2) of Claim \ref{pinch and replace} says that the path $x \longmapsto \tilde F_i^t(x)$ has length bounded independently of $i$, so after increasing $r$ further, we have that equation \eqref{contains2258} holds as desired.

\medskip

Let's return now to the edges $e$ and $e'$. Note that \begin {equation}
 	e_{v',i}'=(\tilde f_i')^{-1}(\CN_1(CH(\tilde f_i'(Z_0')))) \cap e'.\label {echeq}
 \end {equation}
 A similar formula holds for $e_{v,i}$: one passes to the cover of $M_i$  corresponding to the subgroup $(f_i)_*(\pi_1 Z_0)$, and takes the part of $e$ that an appropriate lift of $f_i$ takes into the $1$-neighborhood of the convex hull of the image of $Z_0$. This cover factors through $\tilde M_i$, and the  corresponding convex hull maps into $CH(\tilde f_i(Z_0)) \subset \tilde M_i$. So by \eqref{contains2258}, for some bounded $r$ we have
\begin{equation}
	\tilde f_i(e_{v,i}) \subset \CN_r(CH(\tilde f_i'(Z_0')).\label{subs2}
\end{equation}

So, let $x \in e_{v,i}.$ If $x\in \mathcal C$, then $q(x) = v'$, so we're done. Therefore, we may assume $x \not \in \mathcal C$. By part (2) of Claim \ref{pinch and replace},  the path $x \longmapsto \tilde F_i^t(x)$  has uniformly bounded length, independent of $x,i$. So,  for some new uniform $r$, we have by \eqref{subs2} that  $\tilde f_i' \circ q(x) \in  \CN_r(CH(\tilde f_i'(Z_0'))$.  In particular, Lemma \ref{convlem1} implies that there is a point on $e'$ at a uniformly bounded distance from $q(x)$ that is in $e_{v',i}'$.
\end{proof}

  We now need to show that the carrier graphs $f_i'$ satisfy properties (1)--(3) in the statement of  Claim \ref{thatclaim}.   In other words, we need to show that there is an edge $e_0' = (v_0',w_0')$ in $X'$ such that 
\begin{enumerate} \item $\sup_i \big ( \length f_i'(e_0') - CHL^{f_i'}_{(X')^{sh}}(e_0',w_0') \big )<\infty.$
\item $\sup_i \big ( \length f_i'(e_0')- \length f_i(e_0) +CHL^{f_i}_{X^{sh}}(e_0,v_0) \big ) < \infty$, and\vspace{1mm}
\item  the number of `long' edges in $X' \setminus e_0'$, i.e.\ those $e'$ for which $$\sup_i \length f_i'(e') = \infty,$$ is at most the number of long edges in $X \setminus e_0$.
\end{enumerate}
 
If our original edge $e_0 \subset \mathcal C$,  just choose $e_0'$ to be any edge of $W$. Since $f_i'(e_0)$ has bounded length, (1) and (2) are trivially satisfied, and (3) is satisfied since no new long edges are created when we pinch and replace $f_i$ to get $f_i'$,  but the long edge $e_0 \subset \mathcal C$ is collapsed to a point.

So, suppose $e_0$  is not contained in $\mathcal C$. Let $e_0'=q(e_0)$ be the corresponding edge of $X'$, and let $v_0'=q(v_0)$ and $w_0'=q(w_0)$.  Property (3)  is immediate, since no new long edges were created in $f_i'$.
 To  prove (1) and (2),  we need to remember some of the details of the proof so far. 

Recall   that $e_0=[v_0,w_0]$  was an edge of $X$ with $$\sup_i \length_{X^{sh}}^{f_i}(e_0) < \infty,$$ and  that if $Y $ is the component of $X^{sh}$  in which $v_0$ lies, then $$l_i:=CHL_Y^{f_i} (e_0,v_0) \to \infty.$$    We  defined the sequence of covers $\pi_i : \hat M_i \longrightarrow M_i$ corresponding to the  subgroups $(f_i)_*(\pi_1 Y)\subset\pi_1 M_i$,  which  converged strongly to $\hat M_A$ when marked with the appropriate lifts $\hat f_i : Y \longrightarrow \hat M_i$.  After extending these lifts to $Y \cup_{v_0} e_0$,  we saw that they limited to a map
$$\hat f_A : Y \cup_{v_0} [0,\infty) \longrightarrow\hat M_A,$$ where $\hat f_A |_{[0,\infty)}$ is a distance minimizing ray that exits a  degenerate end $\CE $ of $\hat M_A$.  Moreover, if $M_G$ is the corresponding geometric limit of $(M_i)$ then this $\CE$ has a neighborhood that embeds  under the covering $\hat M_A \longrightarrow M_G$. We   then constructed product regions $$\hat W_i \subset \hat V_i \subset \hat M_i, \ \ W_i \subset V_i \subset M_i, \ \ W_i = \pi_i(\hat W_i), \ \ V_i=\pi_i(\hat V_i),$$   showed that $V_i$ is a neighborhood of the exterior boundary inside a compression body $C_i \subset M_i$  with missing interior boundary, and used this $C_i$  in the construction above. Let $\hat C_i \subset \hat M_i$ be the homeomorphic lift of $C_i$ such that $\hat V_i \subset \hat C_i$.

We now  break into cases, depending on whether $ \hat f_i(Y)$ is contained deep inside of the compression body $  \hat C_i$ or not. Namely, let $$D=D(\epsilon,K)$$ be a constant to be determined later. After passing to a subsequence, we can assume that either
\begin{enumerate}
\item the distance $d( \hat f_i(Y), \hat M_i \setminus  \hat C_i)$  is less than $D$  for all $i$, or
\item it is greater than or equal to $D$ for all $i$.
\end{enumerate}  
These will be the two cases below.  Intuitively, one should imagine the two cases as depending on which of the two components of $M_i \setminus int(W_i)$ contains our subgraph $f_i(Y)$. One component is a compressionbody contained in $C_i$, and the other component may have more complicated topology. If $f_i(Y)$ lies in the compressionbody component, we're in Case 2 below, while otherwise we're in Case 1.  %For the most part, we will be working in $\hat M_i$, so before breaking into cases note that the dichotomy can also be seen in $\hat M_i$, since $$d( \hat f_i(Y), \hat M_i \setminus  \hat C_i)=d( f_i(Y), M_i \setminus  C_i).$$

\medskip

\noindent \it Case 1, when $d( \hat f_i(Y), \hat M_i \setminus  \hat  C_i)\leq D$ for all $i$. \rm Recall that if $(\phi_i)$  are the almost isometric maps coming from the geometric convergence $\hat M_i \to \hat M_A$, the product regions $\hat W_i$  were constructed inside the images $\phi_i(\hat U)$,  where $\hat U$  is a fixed product region in the degenerate end $\hat \CE$ of $\hat M_A$.  So as $\phi^{-1}_i \circ \hat f_i(v_0)$ converges to a point in $\hat M_A$,
\begin {equation}
	\sup_i \ \diam_{\hat M_i}(\hat f_i(v_0) \cup \hat W_i) < \infty. \label {diambdd}
\end {equation}

As long as the $\hat W_i$ were constructed to be sufficiently wide,  we can use Lemma~\ref{level surfaces}  to pick  level  surfaces $S_i \subset \hat W_i$ such that
\begin{equation}
	d(S_i,\partial \hat C_i) >2\diam(\hat f_i(Y)) + D + \diam(\partial_{e} \hat C_i).\label{thosesurfaces}
\end{equation} 
Here,  we are using that both terms on the right are bounded independently of $i$. The first term is bounded since $Y \subset X^{sh}$ and is connected, while the second term is bounded by the Bounded Diameter Lemma, since it is a boundary component of a product region and therefore lies in the $1$-neighborhood of a NAT simplicial ruled surface.

In our construction, we chose  the  product region $\hat U \subset \hat M_A$ from a neighborhood of $\CE$  that is disjoint from  $\hat f_A(Y)$. The ray $\hat f_A |_{[0,\infty)}$ starts on $\hat f_A(Y)$ and exits $\CE$, so it passes through $\hat U$.  Excluding finitely many $i$, we may therefore assume that $\hat f_i(e_0)$  intersects $S_i$ for all $i$. Write
$$ e_0 =  b_i \cup  c_i,$$
where $ b_i $ is the segment \emph{b}eginning at $v_0$ and ending at the first point $x\in e_0$ where $\hat f_i(x) \in  S_i$. By Claim \ref{geominimizing}, $\hat f_i(e_0)$  minimizes the distance in $\hat M_i$ between its endpoints, so \eqref{diambdd} implies that $$\sup_i \length \hat f_i( b_i )< \infty.$$

Combining \eqref{thosesurfaces}, the fact that we're in case 1, and the fact that $\hat f_i(e_0)$  is distance minimizing, we  see that $\hat f_i (c_i) \cap \partial_{e} \hat C_i=\emptyset$.  So,
$$\hat f_i(c_i) \subset \hat C_i.$$
It follows that $c_i \subset \mathcal C$. So  by Claim \ref{pinch and replace}, we have that $\length f_i'(e_0')$ is  at most a bounded amount larger than $\length \hat f_i( b_i )$. In particular,
$$\sup_i \length f_i'(e_0') < \infty.$$
So, $e_0'$ is short, and properties (1) and (2) of Claim \ref{thatclaim}  follow.

\medskip

\noindent \it Case 2, when $d(\hat f_i(Y),\hat M_i\setminus \hat C_i) \geq D$ for all $i$. \rm  Here, we have $f_i(Y) \subset C_i$. Recalling that $C_i$ lifts homeomorphically to $\hat C_i \subset \hat M_i$, which is the cover corresponding to $(f_i)_*(\pi_1 Y)$, it follows that $\hat M_i$ is also the cover corresponding to the inclusion-induced images $$\mathrm{Im}(\pi_1 V_i)=\mathrm{Im}(\pi_1 C_i)\subset \pi_1 M_i.$$ We called this cover $\tilde M_i$ at the end of Step 2. So from Step 2,
\begin{itemize}
	\item  the complement $\hat E_i := \hat M_i \setminus \hat C_i$ is  a  product neighborhood of a convex cocompact end $\CE_i$ of $\hat M_i$,
 	\item  the diameter of the intersection $\hat  E_i\cap \CN_1(CC(\hat M_i))$ is at most some constant depending only on $K,\epsilon$.
\end{itemize}

As long as  the constant $D$ is large enough, we can use Lemma \ref{level surfaces} to find for each $i $ a level surface $S_i \subset\hat V_i$ such that 
\begin{itemize}
	\item $S_i$  separates $\hat f_i(Y)$  from $\hat M_i \setminus \partial_{e} \hat C_i$ in $\hat M_i$,
\item $\diam S_i < 2d(S_i,\partial_{e} \hat C_i)$,
\item $\sup_i d(S_i,\partial_{e} \hat C_i) < \infty.$
\end{itemize}
 Note that if $X_i$ is the topological product region bounded by $S_i$ and $\partial_{e} \hat C_i$,  then this last condition implies that $$\sup_i \diam X_i < \infty,$$ for if $
\gamma_i$  is a geodesic realizing the distance between $S_i$ and $\partial_{e} \hat C_i$, and we choose a point  $p\in X_i$, then any bounded diameter surface through $p$  that is in the homotopy class of a level surface must intersect $S_i \cup \partial_{e} C_i \cup \gamma_i$, a set with  diameter bounded independently of $i$.

 Recall that we have assumed that it is not the case that $e_0\subset \mathcal C$. Since $v_0 \in \mathcal C$,  this means that we can write $$e_0 = c \cup b,$$ where $c$ is the  maximal segment of $e_0$ containing $v_0$ that lies in $\mathcal C$. Now $\hat f_i(c)$ intersects $S_i$, since $S_i$  separates $\hat f_i(v_0)$ from $\partial_{e} \hat C_i$. It follows that $\hat f_i(b)$  \emph {cannot} intersect $S_i$:  if it did, the length minimizing geodesic $\hat f_i(e_0)$ would intersect $S_i$, then $\partial_{e} \hat C_i$,  and then $S_i$ again,  contradicting that $$\diam S_i < 2d(S_i,\partial_{e} \hat C_i).$$
 So, $\hat f_i(b) \cap \CN_1(CC(\hat M_i))$  is contained in the union 
 of the topological product region $X_i$  bounded by $S_i$ and $\partial_{e} \hat C_i$  and the  intersection $\tilde E_i\cap \CN_1(CC(\tilde M_i))$. Since $\hat f_i(e_0)$ is length minimizing and these  sets both have bounded diameter,  it follows that $$\sup_i \, \length(\hat f_i(b) \cap \CN_1(CC(\hat M_i))) <\infty.$$

 Adopting the notation of Claim \ref{rellength don't change}, let $(e_0)_{v_0,i}$ be the segment of $e_0$ adjacent to $v_0$ that has length $CHL_{X^{sh}}^{f_i}(e_0,v_0)$. By definition, this is just the part of $e_0$ that $\hat f_i$ maps into $\CN_1(CC(\hat M_i)),$ so  we have \begin{equation}\label{asdf}\sup_i \, \length f_i\big ((e_0)_{v_0,i} \cap b \big ) < \infty.\end{equation}
But since $\length_{X^{sh}}^{f_i}(e_0) $  is bounded,  it follows from \eqref{asdf} that 
\begin{equation}
	\sup_i \, \length f_i\big (b \setminus (e_0)_{w_0,i} \big ) < \infty.\label{asdf2}
\end{equation}

 Informally, the convex hull corresponding to $w_0$ contains all but a bounded subsegment of the part of $e_0$ that is not collapsed to a point when we create  the new graph $X'$.  Claim \ref{rellength don't change} and part (4) of Claim \ref{pinch and replace} then show that 
$$\sup_i \big ( \length f_i'(e_0') - CHL^{f_i'}_{(X')^{sh}}(e_0',w_0') \big )<\infty,$$
 which was the first of the properties that we wanted to verify for $f_i'$. Similarly, part (4) of Claim \ref{pinch and replace} implies that $$\sup_i \, \length f_i'(e_0') - \length f_i(b) < \infty,$$
and \eqref{asdf} shows that
$$\sup_i \,  \length f_i(b) - \length f_i(e_0) +CHL^{f_i}_{X^{sh}}(e_0,v_0)  < \infty,$$
implying the second property we were trying to verify for the $f_i'$:
$$\sup_i \big ( \length f_i'(e_0')- \length f_i(e_0) +CHL^{f_i}_{X^{sh}}(e_0,v_0) \big ) < \infty. $$
 This concludes the proof of Claim \ref{thatclaim}, and  therefore Lemma \ref{kernel}.

\section{When all barriers are peripheral}\label {when all barriers}

%
%Recall from \S \ref{sec: WPR}  that if $K,A$ are fixed, a collection $\mathcal U$ of disjoint product regions  is called $(K,A)$-\emph{maximal} if
%\begin{itemize}
%	\item there is no pair of product regions $U_1,U_2 \in \mathcal U$ that co-bound a trivial interval bundle in $M \setminus \cup_{U \in \mathcal U} int(U)$,
%\item there is no $(K,A)$-product region $V \subset M \setminus \cup_{U \in \mathcal U} int(U)$.
%\end{itemize}  

This section is devoted to a proof of the following theorem. 

%****say exists(!) product region with n comps below**** then extend those in 10.2 to a max collection of (K,L)-regions, which has bdd complexity. to control the c-bodies, pass to subsequence where edges of the graph are `long' or `short', disregard short edges (use that there's some universal L' bounding their width) and show complementary comps of the long subgraph have bdd diam.

\begin{sat}[No barriers gives PRs and diameter bounds]\label {nonhakenthm}
Given $k,\epsilon,L_0$, there are constants $K_0,g,n,B$, where $g,n$ depend only on $k$, as follows. 
Let $M$ be an orientable hyperbolic $3$-manifold  such that 
 \begin {itemize}
 \item 	$\rank(\pi_1(M))\le k$ and $\inj(M)\ge\epsilon,$
 \item every essential closed curve on  $\partial CC(M)$  that is nullhomotopic in $M$ has length at least $\epsilon $,
\item all barriers of $(K_0,L_0)$-product regions in $M$ are peripheral.
 \end {itemize}
 Then there is a collection $\mathcal U$ of at most $n$ $(g,L_0)$-product regions in $int(CC(M))$ such that $\diam (N) \leq B$ for every component $$N \subset \CN_1(CC(M)) \setminus \cup_{U \in \mathcal U} int(U).$$
\end{sat}

Although it is not necessary to state this above,  as mentioned in the introduction the complementary components $N$ are connected by the product regions $U$  in the pattern of a rooted tree, where all $N$  other than the root are compression bodies. See Figure~\ref{nonhakenpic}. 
 
\medskip

The proof of  Theorem \ref{nonhakenthm}  is a complicated contradiction argument, similar to that employed in Proposition \ref{no barriers implies}.   First, we define some of the constants above. Let $K_0,r_0$ be as  in Proposition \ref{no barriers implies},  and set
$$g = 2 r_0, \  \ n = 4g = 8r_0.\ \ \ $$
%The constant  $R=R(K,\epsilon)>1$  is chosen 

The constant $B$  will not be explicitly constructed. Rather, we assume that there is no $B$ as required in the statement of the theorem, and take a sequence of counterexamples $M_i$ that illustrate the failure of larger and larger constants $B$.  To derive a contradiction, we will show that after passing to a subsequence, there are collections $\mathcal U_i$ of $(g,L_i)$-product regions in $CC(M_i)$, where $|\mathcal U_i| \leq n$, where $L_i \to \infty$, and where $\sup_i \, \diam(N_i) < \infty$  for every choice of components $$N_i \subset \CN_1(CC(M_i)) \setminus \cup_{U \in \mathcal U_i} int(U).$$
%
%For convenience, we will work with a neighborhood $\CN_R(CC(M_i))$  instead of the convex core itself. Namely, fix $$R=R(g,\epsilon)$$ large enough so that any sufficiently wide genus $g$ product region in an $\epsilon $-thick  hyperbolic $3$-manifold is contained in the $(R-1)$-neighborhood of the convex core.  The existence of such a constant follows from Lemma \ref{minsrsinpr}.  With $M_i,\mathcal U_i,N_i$ as above, we'll show that $$\sup_i \diam (N_i\cap \CN_R(CC(M_i)))< \infty,$$
% and the diameter bound for the intersection with $CC(M_i)$  will obviously follow.
%Complementary components in $CC(M_i)$ are contained in complementary components in $\CN_R(CC(M_i))$, so the earlier \eqref{supdi} will follow. One advantage of working with $\CN_R(CC(M_i))$ is that  whenever $N_i$  is a component of

As a first step, let's show how to construct the \emph {central component} $N_i^c$ in Figure \ref{nonhakenpic}, corresponding to the root of the rooted tree, and the collection $\mathcal U_i^c$ of product regions adjacent to $N_i^c$.

%\begin {figure}
%\centering\includegraphics{images/nonhaken}
%\caption {An illustration of the decomposition of $M$ given by the product regions in Theorem \ref{nonhakenthm}. In the picture, `ends' are trivial compression bodies and handlebodies are also compression bodies.} 
%\label {nonhakenpic}
%\end {figure}

\begin{lem}[The central component]\label {the root component}
Passing to a subsequence, there is a  sequence $L_i \to \infty$, a collection $\mathcal U_i^c$ of $(g,L_i)$-product regions in $int(CC(M_i))$, and  a component $$N_i^c  \subset \CN_1(CC(M_i)) \setminus \cup_{U \in \mathcal U_i^c} int(U),$$
 such that the following properties hold:
\begin{enumerate}
\item $\rank \pi_1 N^c_i \leq g$,
 \item  $\sup_i \diam (N^c_i) < \infty$, 
\item each component of $M_i \setminus int(N^c_i)$ is a compression body with missing interior boundary.
\end{enumerate}
\end{lem}
\begin{proof} By Proposition \ref{no barriers implies},  after passing to a subsequence  so that the isomorphism type of graph stabilizes, there is a fixed graph $X $ with
$\rank \pi_1 X \leq r_0 $ and  a sequence of carrier graphs $$f_i : X \longrightarrow M_i, \ \ \sup_i \length f_i(X) < \infty.$$

After passing to a subsequence, we can extract algebraic and geometric limits of $( M_i)$ marked by the carrier graphs $f_i$, as in Example \ref{limit example}. Moreover, Proposition \ref{limits-fg} says that \emph {after doubling the ranks of the markings}, we may assume that the algebraic and geometric limits coincide.  So, picking basepoints $  p_i \in    f_i(X) \subset   M_i$,  we can assume 
\begin {enumerate}
	\item[(a)]  the pointed manifolds $( M_i,  p_i)$  converge geometrically to an $\epsilon$-thick hyperbolic $3$-manifold $( M_G, p_G)$,
\item[(b)]  with respect to suitable markings of the fundamental groups by a free group on $g=2r_0$ generators, this convergence is strong. 
\end {enumerate}
Note that (b)  implies that $\rank \pi_1  M_G \leq g $, so $M_G$ is tame by the Tameness Theorem \cite{Agoltameness,Calegarishrinkwrapping}, and all ends of $M_G$ have neighborhoods homeomorphic to $S \times (0,\infty)$, where $S$ has genus at most $g$. 

Let $(\phi_i)$ be a sequence of almost isometric maps given by the geometric convergence $M_i \to M_G$.
 Since the convergence $M_i\to M_G $ is strong, Theorem~\ref{strong limits} gives a   standard  compact core $C  \subset M_G$ and $L_i\to \infty$  such that  for large $i$, the following holds for every component $S \subset \partial C$.
\begin{enumerate}
	\item[(a)]  If $S$  faces a degenerate end of $M_G$,  then $\phi_i(S)$ bounds a product region $V_i^S \subset CC(M_i) \setminus int(\phi_i(C))$ with width at least $L_i$.  Moreover, the  component of $M_i \setminus int(\phi_i(C))$ containing $\phi_i(S)$ is a compression body with missing interior boundary.
 \item[(b)] If $S$  faces a convex-cocompact end  of $M_G $, then $\phi_i(S)$ bounds a component $E_i^S \subset M_i \setminus int(\phi_i(C))$ that  is a  product neighborhood of a  convex cocompact end of $M_i$, where $$E_i^S \subset M_i \setminus \CN_1(CC(M_i)).$$
\end{enumerate} 
%
%a sequence of standard compact cores $g_i \subset int( \phi_i(C))$ such that the following holds after excluding finitely many $i$.  
%\begin{enumerate}
%\item [(a)]  Every component of $M_i \setminus int(K_i)$ is a compression body with missing interior boundary,	
%\end{enumerate}
% Moreover, fix a  component $S\subset\partial C$ and let $S_i \subset \partial K_i$ be the  component facing $\phi_i(S) \subset\partial \phi_i(C)$.   Then:\begin{enumerate}
%	\item[(b)]  If $S$  faces a degenerate end of $M_G$,  then  the surface $S_i$ bounds a $(K,L_i)$-product region $V_i^S \subset M_i \setminus int(K_i)$.
%\item[(c)] If $S$  faces a convex-cocompact end  of $M_G $, then $S_i$ bounds a component of $M \setminus int(\CN_R(CC(M_i)))$  disjoint from $int(K_i)$.
%

Let $\mathcal U_i^c$  be the  collection of  product regions $V_i^S$ above, where $S$  ranges over the components of $\partial C $ facing  degenerate ends of $M_G$, and let $$N_i^c \  = \ \phi_i(C) \cap \CN_1(CC(M_i)).$$%\ \ \cup \hspace{-3mm} \bigcup_{\substack{S \subset \partial C \text{ facing } \\ \text{cc-ends of } M_G}}  \hspace{-5mm}  E_i^S.$$  
Since  all our product regions are contained in $CC(M_i)$, the manifold $N_i^c$  is a component of $\CN_1(CC(M_i)) \setminus \cup_{U \in \mathcal U_i^c} \, int(U).$ Moreover,
(a) and (b)  imply that  each component of $\phi_i(C) \setminus int(\CN_1(CC(M_i)))$ is a trivial interval bundle co-bounded by some $\phi_i(S)$, where $S \subset \partial C$ bounds  a convex cocompact end of $M_G$, and a  component of $\CN_1(CC(M_i))$. Hence, each
$N_i^c$  is homeomorphic to $\phi_i(C)$, and therefore to $C$. Hence, $$\rank \pi_1 N_i^c = \rank \pi_1 M_G \leq g,$$
 so property (1) of the lemma holds.  The diameter bound in property (2)  is immediate, since each $N_i^c  \subset \phi_i(C)$, and for large $i$ the latter is a $2$-bilipschitz image of the fixed compact subset $C \subset M_G$.  Property (3) is stated directly in (a)  above.\end{proof}

 %There is one issue remaining:  the product regions $U \in \mathcal U_i^c$  may not be contained in the convex core $CC(M_i)$.  However,  Corollary \ref{prs in the convex core} and Fact \ref{width diameter} give for each $U \in \mathcal U_i^c$ a subproduct region $U_2 \subset U$  that is contained in $CC(M_i)$, and where the two  components of $U \setminus int(U_2)$ have bounded diameter.  After replacing each $U$ with $U_2$ and adding any adjacent component of $U \setminus int(U_2)$ into $N_i^c$, the properties (1)--(3)  of the lemma  still obviously hold.

We now construct  the collections of product regions $\mathcal U_i$ in $M_i$.%  referenced in the statement of Theorem \ref{nonhakenthm}.

\begin{claim}\label {max prs}
  After passing to a subsequence, there is some $A$  independent of $i$, a sequence $L_i \to \infty$, and a collection $\mathcal U_i$ of at most $n$  $(g,L_i)$-product regions in $int(CC(M_i)) $  such that:
\begin{enumerate}
\item $N_i^c$  is a component of $\CN_1(CC(M_i)) \setminus \cup_{U \in \mathcal U_i} int(U),$
\item  every $(4g,A)$-product region  in $ M_i$  intersects some $U \in \mathcal U_i$.
\end{enumerate} 
\end{claim}

Recall that $n=4g$ were constants that were fixed at the beginning of the section. The $4g$ in the statement of the claim could be replaced by any integer, but we will apply it in Proposition \ref{diameterandtopbounds} using $4g$ and so we just state the claim as above to avoid  further proliferation of constants.

\begin{proof}
Fix some $L''>0$  larger than the widths required in Lemma \ref{intervalbundle} and Corollary \ref{unknotted2} for the genus bound $g$.
 By Lemma \ref{the root component}, for large $i $ the collection $\mathcal U_i^c$  consists of $(g,L'')$-product regions in $CC(M_i)$,  and satisfies (1). We want to extend $\mathcal U_i^c$  to a \emph{maximal} such collection.  To do this, first choose for each degenerate end $\CE$ of $M_i$ a (noncompact) product region  $U_\CE$ that is a (closed) neighborhood of $\CE$, as  given by Proposition~\ref{constructing prs}. By Corollary \ref{prs in the convex core}, we can assume that each $U_\CE \cap N_i^c = \emptyset$ and $U_\CE \subset int(CC(M_i))$.  As the difference $CC(M_i) \setminus \cup_{\CE} int(U_\CE)$ is compact, we can construct a  maximal collection $\mathcal U_i'$ of $(g,L'')$-product regions  in $int(CC(M_i))$  that satisfies (1) by adding the $U_\CE$ and some other  compact product regions to our initial collection $U_i^c$. Finally, Lemma \ref{trivial interval bundles}  says that whenever two $(g,L'')$-product regions co-bound a trivial interval bundle, the union of the two regions with the interval bundle is also a product region. So, we can assume that \emph{no two  distinct product regions $U,V$ in $\mathcal U_i'$ co-bound a trivial interval bundle in $M_i \setminus \cup_{U \in \mathcal U_i^c} int(U)$, unless that trivial interval bundle is $N_i^c$.} 

Corollary~\ref{unknotted2} gives a bound of $2g-2$  for the number of non-trivial-interval-bundle complementary components of a set of  compact product regions in a compression body $C$ with genus at most $g$.  By construction, no compact component of $M_i \setminus \cup_{U \in \mathcal U_i'} int(U)$ is a trivial interval bundle, except possibly $N_i^c$.   However, there may be peripheral interval bundles if some $U \in \mathcal U_i'$ bounds a convex cocompact end of $M_i$. In any case, though,  the number of  compact product regions in each compression body $C$ that is a component of $M_i \setminus int(N_i^c)$ is at most, say, $4g$.   Every noncompact product region in $C$ is a product neighborhood of one of the (missing)  interior boundary components of $C$, and hence  there are at most $g$ of these, for a total of at most $5g$.
Since $\rank \pi_1 N_i^c \leq g,$  the number of  boundary components of $N_i^c$, and hence the number components $C$ as above, is at most $g$, by half-lives-half-dies. So summing over all $C$, the number of product regions in $\mathcal U_i'$ is at most $n:=5g^2$, say. (This bound is not optimal.)
 
We claim that after passing to a subsequence, there is some $L' \geq L''$ independent of $i$ such that if $\mathcal U_i \subset \mathcal U_i'$  is the subset of product regions with width \emph{bigger} than $L'$, we have that
\begin{equation}\lim_{i\to \infty} \min\{\width(U) \ | \ U \in \mathcal U_i\} = \infty.\label {toinft}\end{equation}
Setting $L_i$ to be the  sequence of minima above (these may be larger than the  constants called $L_i$ in Lemma \ref{the root component}), each $U \in \mathcal U_i$  will then be a $(g,L_i)$-product region, and  we will have $L_i \to \infty$  as desired. 

To   construct  the collections $\mathcal U_i$ satisfying \eqref{toinft}, pass to a subsequence so that the sets $\mathcal U_i'$ can all be marked by some common index set $\{1,\ldots,n\}$, and then pass to a further subsequence so that for each index, the widths of the corresponding product regions either converge to $\infty$ or are bounded. Take $L'$ to be the maximum width of any product region corresponding to a `bounded' index, and $\mathcal U_i$ to be the subset corresponding to the `unbounded' indices.

Property (1) still holds for the collections $\mathcal U_i$, since by Lemma \ref{the root component} the product regions adjacent to $N_i^c$ have widths tending to infinity, so are still present in $\mathcal U_i$  after excluding finitely many $i$.

Finally, we prove (2).   By Lemma \ref{the root component} (2) and  Fact \ref{width diameter}, the diameters of $N_i^c$  and of all the product regions $U \in \mathcal U_i' \setminus \mathcal U_i$ are  uniformly bounded.  Since there are only boundedly many of these sets, Corollary \ref{cutting-product}  says there is some huge $A$  independent of $i$ such that any $(4g,A)$-product region $V \subset M_i$ contains a subproduct region $V'$ of width at least $L''$  that is disjoint from all $U \in \mathcal U_i' \setminus \mathcal U_i$ and $N_i^c$.  Moreover,  Corollary \ref{prs in the convex core}  says that  if we take $A$  slightly larger we may assume that $V' \subset CC(M_i)$.

 So, assume that $V$ is a $(4g,A)$-product region in $M_i$  that is disjoint from all $U \in \mathcal U_i$. If $V'$  is as above, it lies in some component $C \subset M_i \setminus int(N_i^c) $, and  Lemma \ref{the root component} (3)  says that $C$ is homeomorphic to a compression body with missing interior boundary.  The exterior boundary $\partial_e C$  is a boundary component of $N_i^c$, so has  genus at most $$\rank \pi_1 N_i^c \leq g.$$
Increasing $A$ and discarding finitely many $i$,  we can assume that all our product regions are wide enough so that
 Corollary \ref{unknotted2} applies,  in which case we see that in fact  the product region $V' \subset C$ has  genus at most $g$. Since $V $  does not intersect any $U\in \mathcal U_i$,  it follows from the definition  that $V'$  is a $(g,L'')$-product region  that intersects neither $N_i^c$ nor any $U \in \mathcal U_i'$. But $\mathcal U_i'$  was chosen to be a maximal collection of $(g,L'')$-product regions in $CC(M_i)$ satisfying (1), which is a contradiction.
\end{proof}
 %
%
%Let $D$ be the constant from Lemma \ref{intersectprs},  and $$B_i := N_i^c \ \cup \ \bigcup_{U \in \mathcal U_i} \{p\in U \ | \ d(p,\partial U) \leq D\} \ \cup  \bigcup_{U \in \mathcal U_i' \setminus \mathcal U_i} \hspace{-2mm}U.$$ 
%By Lemma \ref{the root component} and  Lemma \ref{bdlpr}, for instance, every component of $B_i$ has diameter bounded above independently of $i$, and  the  number of components is also  uniformly bounded. So, if $V\subset M_i$ is a $(g,L)$-product region, and $L$  is huge relative to all these constants, there will be some subproduct region $V' \subset V$ such that
%$$V'\cap B_i=\emptyset, \ \ d(V',\partial V) \geq D, \ \ \width(V') \geq L''.$$
%Since $\mathcal U_i'$ was constructed to be a maximal collection of $(g,L'')$-product regions in $N_i^c$,  it follows that $V'$ must intersect some $U\in \mathcal U_i'$.  By construction of $V'$,  this intersection must happen at a distance of at least $D$ from $\partial U$, so Lemma \ref{intersectprs}  implies that $V$ and $U$ share a level surface.  This finishes the proof of the claim. 

For each $i$,  fix now some component  $$N_i \subset \CN_1(CC(M_i)) \setminus \cup_{U \in \mathcal U_i} int(U).$$ We  claim that  after passing to a subsequence,  we have $$\rank \pi_1 N_i \leq 2g, \ \ \sup_i \diam N_i < \infty.$$ Lemma \ref{the root component}  shows that $\rank \pi_1 N_i^c \leq g$ and gives a uniform diameter bound for $N_i^c$, so from now on we assume that $N_i\neq N_i^c$. 

% The rank bound is easy. If $N_i \neq N_i^c$, it lies in  some component $C \subset M_i \setminus int(N_i^c)$, which is a compression body with missing interior boundary.  The exterior boundary $\partial_{e} C$ has genus at most $rank \pi_1 N_i^c \leq g$, and  Corollary \ref{unknotted2}  implies that $N_i \subset C$  is a compression body (possibly with some missing interior boundary components) whose exterior boundary also has genus at most $g$. Hence, $\rank \pi_1 N_i \leq 2g=r.$
%
% The rest of the proof is devoted to showing that $$\sup \hspace{-.5mm}_i \, \diam N_i < \infty.$$
%Since the widths of the products regions in $\mathcal U_i$  tend to infinity with $i$, after  discarding  finitely many terms we may assume that  Corollary~\ref{unknotted2} applies to  all the collections $\mathcal U_i$. So, as **************
The rank bound is easy.  Discarding finitely many $i$,  we can assume that all our project regions are wide enough so that Corollary \ref{unknotted2} applies.  Since $N_i \neq N_i^c$, it lies in  some component $C \subset M_i \setminus int(N_i^c)$, which is a compression body with missing interior boundary.  The exterior boundary $\partial_{e} C$ has genus at most $\rank \pi_1 N_i^c \leq g$, and Corollary~\ref{unknotted2}  implies that \emph {$N_i$  is a compression body (possibly with some missing interior boundary components)} whose exterior boundary also has genus at most $g$. Hence, $\rank \pi_1 N_i \leq 2g.$

A first step toward the diameter bound is the following, which says that the rank bound can be realized with short carrier graphs. %

\begin {claim}[Short graphs]\label {short graphs claim}
 After discarding  finitely many $i$, there are graphs $X_i$ with $\rank \pi_1 X_i \leq 2g$ and maps $$f_i: X_i \longrightarrow M_i, \ \ (f_i)_*(\pi_1 X) = [\pi_1N_i]$$ such that 
$\sup _i \length(f_i(X))<\infty$  and $\sup_i d(f_i(X),N_i) < \infty.$
\end {claim}

Here, $[\pi_1N_i]$  is the image of  the map $\pi_1 N_i \longrightarrow \pi_1 M_i$ induced by  the inclusion $N_i \hookrightarrow M_i$.

\begin {proof}  Let $U_i^{e}$ be  the  product region that contains the exterior boundary $\partial_e N_i$  as a boundary component. Discarding finitely many $i$,  we may assume that $U_i^{e}$ is wider  than the constants in Lemma \ref{minsrsinpr}.  We can then construct a NAT simplicial ruled surface $\phi_i : S_i \longrightarrow U_i^{e}$  in the homotopy class of a level surface,    whose image lies at a bounded distance from $N_i$.  Since $M_i$ is $\epsilon$-thick and $S_i$ has genus at most $g$, Corollary \ref{lipschitzcorollary} implies that if $X$ is a wedge of $2g$ circles, there is a  carrier graph $$\psi_i : X_i \longrightarrow S_i$$ where  the length of $ \phi_i \circ \psi_i(X_i)$ is bounded independently of $i$.  The maps $f_i=\phi_i \circ \psi_i$ are carrier graphs for $N_i^{e}$, so we are done.
\end{proof}

 In light of Claims \ref{max prs} and \ref{short graphs claim},  the following  proposition will finish our proof. We state it carefully  here  so that we can reuse it in \S \ref{the proof}.

\begin{prop}[Diameter and topology bounds]\label{diameterandtopbounds}
Suppose that $M_i$  is a  sequence of complete hyperbolic $3$-manifolds with $\inj(M_i) \geq \epsilon$ and $N_i \subset \CN_1(CC(M_i))$ are submanifolds with the following two properties:
\begin{itemize}
\item there are graphs $X_i$ with $\rank \pi_1 X_i \leq g$ and  a sequence of carrier graphs $f_i: X_i \longrightarrow M_i $ such that  $(f_i)_*(\pi_1 X) = \mathrm{Im}(\pi_1N_i) \subset \pi_1 M_i$,
$\sup _i \length(f_i(X))<\infty$  and $\sup_i d(f_i(X),N_i) < \infty.$\item each component $S_i \subset \partial N_i$  is either a component of $\CN_1(CC(M_i))$, or is a boundary component of  a $(g,L_i)$-product region $U_i \subset CC(M_i) \setminus int(N_i)$, where $L_i \to \infty$,
	\item  for some $A$ independent of $i$, there is no $(2g,A)$-product region  contained  in $N_i$.
\end{itemize}
Then after  passing to a subsequence, all the manifolds $N_i$ are homeomorphic, and $\sup_i \diam( N_i \cap \CN_R(CC(M_i))) < \infty$.
\end{prop}

Note that while proving Proposition \ref{diameterandtopbounds} we can assume that any product regions given in the second bullet point are compact. For if they are not, we can  replace them with appropriate truncations using Corollary \ref{truncating noncompact}.

\medskip

 The rest of the section is devoted to the proof of Proposition \ref{diameterandtopbounds}.  The idea is to use the short graphs to extract algebraic and geometric limits,  in the same way that we did when we constructed the central component $N_i^c$ in  Lemma \ref{the root component}.  In fact, the following argument is really a fancier version of the proof of Lemma \ref{the root component},  so we suggest that the reader understand that fully before proceeding further. 
Let $$N_i^{ext} \subset M_i$$ be  the union of $N_i$ with all adjacent product regions $ U \in\mathcal U_i$, and all  adjacent components of $M_i \setminus int(\CN_1(CC(M_i)))$, where in each case \emph{adjacent}  means sharing a boundary component with $N_i$. Let 
 $$\pi_i: \hat M_i \longrightarrow M_i$$   be the cover corresponding to  the  subgroup $\mathrm{Im}(\pi_1 N_i)\subset\pi_1 M_i$. Then $N_i \subset N_i^{ext}$ lift isometrically to submanifolds $$\hat N_i \subset \hat N_i^{ext} \subset \hat M_i$$  and the $f_i$ lift to  bounded length carrier graphs $$\hat f_i : X_i \longrightarrow  \hat M_i$$ that lie at bounded distance from the submanifolds $\hat N_i \subset \hat M_i$. So, passing to a subsequence we can extract algebraic and geometric limits of $(\hat M_i)$ marked by these carrier graphs, as in Example \ref{limit example}. Moreover,  Proposition \ref{limits-fg} says that \emph {after doubling the ranks of the markings}, we may assume that the algebraic and geometric limits coincide.  In other words, for any choice of basepoints $\hat p_i \in \hat M_i $ such that $$\sup_i d(\hat p_i, \hat f_i(X_i)) < \infty,$$  we can assume  after passing to a subsequence that 
\begin {enumerate}
	\item  the pointed manifolds $(\hat M_i,\hat p_i)$  converge geometrically to an $\epsilon$-thick hyperbolic $3$-manifold $( M_G, p_G)$,
\item  with respect to suitable markings of the fundamental groups by a free group on $4g$ generators, this convergence is strong. 
\end {enumerate}
Note that (2)  implies that $\rank \pi_1  M_G \leq 4g$.   In fact, \begin {enumerate}
	\item[(3)] if $p_i =\pi(\hat p_i) \in M_i,$  the pointed manifolds $(M_i,p_i)$  also converge geometrically to $(M_G,p_G)$, 
\end {enumerate}  which explains  the notation $M_G$.  To prove (3), note that $\pi_i$ is an embedding on $N_i^{ext}$ and that $d(\hat p_i ,\partial \hat N_i^{ext})\to \infty$, so we have: \begin {enumerate}
	\item[(4)] For any fixed $t>0$, the map $\pi_i : \hat M_i \longrightarrow M_i$  restricts to an embedding on the radius-$t$ ball $B_{\hat M_i}(\hat p_i,t)$.\end {enumerate}
Finally, we note that since  $\hat f_i(X_i) $ lies at a bounded distance from $\hat N_i$, we can assume above that
\begin {enumerate}
	\item[(5)] $\hat p_i \in \hat N_i$ and $p_i \in N_i.$\end {enumerate}

%
%word is in order about the base points. Since $\hat M_i$ is $\epsilon$-thick, Lemma \ref{distancetogeodesic}  implies that the injectivity radius of $\hat M_i$ outside $\CN_t(CC(\hat M_i))$ is at least some   uniform function $F(t,\epsilon)$ that converges to $\infty$ with $t$. The bounded length carrier graphs $\hat f_i(X_i) $ must\footnote{The careful reader can check that Proposition \ref{diameterandtopbounds} is trivially true when infinitely many $N_i$ are isometric to $\BH^3$.} carry some homotopically essential loop, so they lie at a bounded distance from $CC(\hat M_i)$.  Since the $\hat f_i(X_i) $ also lie at bounded distance from $\hat N_i$ the, it follows that we can take
%\begin {enumerate}
%	\item[(5)] $\hat p_i \in \CN_R(CC(\hat M_i)).$\end {enumerate}
%

 Let
$( \hat \phi_i)$ and $(\phi_i)$ be a  sequences of almost isometric maps given by the geometric convergence in (1) and (3) above, as in Definition \ref{geolimits}, and where $\pi_i \circ \hat \phi_i  = \phi_i$  for all $i$.
Since $\hat M_i \to M_G$  strongly, Theorem~\ref{strong limits} implies that  the following holds after discarding  finitely many $i$. \emph{There is a standard compact core $ C \subset M_G$ containing  the basepoint $p_G$ and a sequence $L_i \to \infty$ such for every component $S \subset \partial C$, we have:
\begin{enumerate}
	\item[(a)]  If $S$  faces a degenerate end of $M_G$,  then $\hat \phi_i(S)$ bounds a product region $\hat V_i^S \subset CC(\hat  M_i) \setminus int(\hat \phi_i(C))$ with width at least $L_i$.  %Moreover, the  component of $M_i \setminus int(\phi_i(C))$ containing $\phi_i(S)$ is a compression body with missing interior boundary.
 \item[(b)] If $S$  faces a convex-cocompact end  of $M_G $, then $\hat \phi_i(S)$ bounds a component $\hat  E_i \subset \hat M_i \setminus int(\hat \phi_i(C))$ that  is a  product neighborhood of a  convex cocompact end of $\hat M_i$, and $\hat E_i \subset \hat M_i \setminus \CN_1(CC(\hat M_i)).$
\end{enumerate}}
%
%\emph{There is a standard compact core $C \subset M_\infty $ and  further standard compact cores $$ \hat g_i \subset int( \hat \phi_i(C)) $$  such that  the following holds for some $\hat L_i \to \infty.$
% Given a component $S \subset \partial C$, let $ \hat S_i \subset \partial \hat g_i$ be the components that face $\hat \phi_i(S)$. Then:
%\begin{enumerate}
%	\item[(b)] If $S$  faces a degenerate end of $M_G$,  then $\hat S_i$ bounds a product region $\hat V_i^S \subset \hat M_i \setminus int(\hat g_i)$ with width at least $\hat L_i$.
%   \item[(c)]   
%If $S$  faces a convex-cocompact end  of $M_G $, then $\hat S_i$ bounds a component of $\hat M_i \setminus int(\CN_R(CC(\hat M_i)))$  disjoint from $int(\hat g_i)$.
%\end{enumerate}}
 We would like to  translate  statements (a) and (b) above, which are about $\hat M_i$, to  similar statements in $M_i$.   We first claim:
\begin{claim}
 Let $A>0$. If $S$  bounds a degenerate end of $M_G$, then for large $i$ the surfaces $\phi_i(S) \subset M_i$ bound  product regions $$W_i^{S,A} \subset CC(M_i)\setminus int(\phi_i(C))$$
 such that $ A \leq \width W_i^{S,A}$  for each $i $, and $\sup_i \width W_i^{S,A} <\infty $.
\end{claim}
\begin {proof}
Since the widths of the product regions $\hat V_i^S$ tend to infinity, we can  use Corollary \ref{cutting-product} to find a subproduct region $\hat W_i^{S,A} \subset \hat V_i^S$ that shares the boundary component $\hat S_i$, has width  at least $A$, and diameter bounded independently of $i$. (Corollary \ref{cutting-product} allows us to prescribe  the width of  $\hat W_i^{S,A}$  up to some bounded error, and Fact \ref{width diameter} translates a width upper bound into a diameter upper bound.) Since $\hat W_i^{S,A}$ shares a  boundary component with $\hat \phi_i(C)$, which has bounded diameter and contains $\hat p_i$,  it follows that $\hat W_i^{S,A}$ lies in a $t$-ball around $p_i$  for some $t$  independent of $i$. So by (4), the covering map $\pi_i$  restricts to an embedding on $\hat W_i^{S,A}$, and its image is  the desired  $W_i^{S,A}$.
\end{proof}

We would also like to say that part (b) can be stated in $M_i$.

\begin {claim}\label {cc-ends embed}
If $S$  faces a convex-cocompact end  of $M_G $, then for large $i$,  the surface $  \phi_i(S)$ bounds a component $  E \subset  M_i \setminus int( \phi_i(C))$ that  is a  product neighborhood of a  convex cocompact end of $ M_i$, and $$ E \subset  M_i \setminus \CN_1(CC( M_i)).$$
\end {claim}
\begin {proof}
Let $\hat E_i$ be  as in (b)  above. By Lemma \ref{convex cocompact cover}, it suffices to show that $\hat E_i$ embeds in $M_i$.  
Since $\phi_i(C)$  has bounded diameter, the base points $\hat p_i \in \hat \phi_i(C)$ lie at bounded distance from the surfaces $\hat \phi_i(S)$. Hence,
 $$\sup_i d(\partial \hat E_i,\hat p_i) < \infty.$$
Since $d(\hat p_i,\partial \hat N_i^{ext}) \to \infty$  with $i$ and $\hat p_i \in \hat N_i \subset \hat N_i^{ext}$,  it follows that $$\partial \hat E_i \subset \hat N_i^{ext}, \ \ \forall \text{ large } i.$$  

 Suppose for a moment that $\hat N_i^{ext}$ has a boundary component $\Sigma$ contained in $\hat E_i$. Since $\Sigma$ is adjacent to a wide product region, Lemma \ref{geosinprs} implies that there is a closed geodesic $\gamma$ in $\hat M_i$ of  bounded length at a bounded distance from $\Sigma$, where here `bounded' means independent of $i$. But  for large $i$, the surface $\Sigma $ lies very far from $\partial \hat E_i$, so if $\Sigma \subset \hat E_i$ then $\Sigma$ lies \emph{deep} within $\hat E_i$ for large $i$.  This means that  for large $i$, the closed geodesic $\gamma \subset int(\hat E_i)$, which is impossible since closed geodesics always lie in the convex core.

So, we may assume that for large $i$, there are no boundary components of $\hat N_i^{ext}$ contained in $\hat E_i$. It follows that $\hat E_i \subset \hat N_i^{ext}$, and  the claim follows since $\pi_i$  restricts to an embedding on $\hat N_i^{ext}$. \end {proof}

 To finish Theorem \ref{nonhakenthm}, it suffices to prove the following  claim.  

\begin{claim}\label {last claim}
 For large $i$, all the $N_i$ are homeomorphic to the standard compact core $C \subset M_G$. Moreover,   for some $A>0$  independent of $i$,
$$N_i \ \subset \ C_i^A \ := \ \phi_i(C) \ \  \cup \hspace{-7mm}\bigcup_{\substack{S \subset \partial C \text{ facing }\\  \text{ degenerate ends of } M_G }} \hspace{-7mm} W_i^{S,A}.$$
\end{claim}

 Here, the desired uniform diameter bound on $N_i$  follows since $ \phi_i(C)$ has  bounded diameter and the widths (and hence diameters by Fact \ref{width diameter}) of  the product regions $W_i^{S,A}$ are also bounded above.

\begin {proof}
A word is in order about the genus of $W^{S,A}_i$. A priori, we only know $\rank \pi_1 M_G \leq 4g$, so the usual half-lives-half-dies argument gives that the genus of $W_i^{S,A}$  is at most $4g.$ Actually, though, we claim $$ genus(W_i^{S,A}) \leq g.$$

 To see this,  choose product regions $U_i \in \mathcal U_i$ such that the  component of $M_i \setminus int(U_i)$  containing $N_i$ is a compression body with missing interior boundary, and hence exterior boundary of genus at most $g$.  For instance,  one  could  take $U_i$  to be the product region adjacent to the central component $N_i^c$ that separates $N_i^c$ from $N_i$, appealing to Lemma \ref{the root component}. 
 The product regions $W_i^{S,A}$ all lie at bounded distance from $N_i$, and the widths of the $U_i$ tend to infinity, so for large $i$ there are subproduct regions $U_i' \subset U_i$  such that 
\begin{itemize}
\item $U_i'$ has width at least $A$, \item $W_i^{S,A}$ is contained in a component of $M_i \setminus int(U_i')$  that is a compression body with missing interior boundary. \end{itemize}
As long as $A$ is bigger than the constant $L$ in Corollary \ref{unknotted2}, given the genus bound of $4g$, it follows that $$genus(W_i^{S,A})\leq genus(U_i')\leq g.$$  

Let $D$ be the constant from Lemma \ref{intersectprs}. Then the set $$X=\bigcup_{U \in \mathcal U_i} \{p\in U \ | \ d(p,\partial U) \leq D\} $$ 
 is a union of at most $2n$ sets of diameter at most $D$, where we have one  such set for each boundary component of each $U$.  So, if $A$  is very large, each $W_i^{S,A}$ has a subproduct region $Y$  such that
\begin{enumerate}
\item[(i)] $Y$ is disjoint from $X$,
\item[(ii)] $d(Y,\partial W_i^{S,A}) \geq D$,
\item[(ii)] $\width(Y)$ is at least the constant $A$ in Claim \ref{max prs}.
\end{enumerate}
By Claim \ref{max prs} (2), it follows that $Y$ intersects some  product region $U_i^S \in \mathcal U_i$.
 In light of (i) and (ii), Lemma \ref{intersectprs}  then implies that $W_i^{S,A}$ shares a level surface $\Sigma_i$ with some $U_i^S \in \mathcal U_i$. 

This $U_i^S$ must share a boundary component with $N_i$, since its  distance to $N_i$ is  bounded above,   and any product region not adjacent to $N_i$ is separated from $N_i$ by other product regions, all of whose widths tend to infinity. Moreover, $\phi_i(C) \cap N_i \neq \emptyset,$  as they both contain $p_i$, so  the side of the level surface $\Sigma \subset W_i^{S,A}$ facing $\phi_i(C)$ is the same side that faces $N_i$ when we consider $\Sigma$  as a level surface of $U_i^S$.  In particular, since $d(\Sigma,N_i)$ is bounded and $\width(U_i^S) \to \infty$, the component of $\partial W_i^{S,A}$ that lies on the \emph{other} side of $\Sigma$, i.e.\ the  one that is also a component of $\partial C_i^A$, is also a level surface of $U_i^S$.

Since $C_i^A$  intersect $N_i$ and  has bounded diameter, $C_i^A \subset N_i^{ext}$  for large $i$. Moreover, we know from  Claim \ref{cc-ends embed}  that whenever $S \subset \partial C$ bounds a convex cocompact end of $M_G$,  the boundary component $\phi_i(S) \subset \partial C_i^A$ lies outside $\CN_1(CC(M_i))$ and is  isotopic to a component of $\partial \CN_1(CC(M_i))$ that is also a component of $\partial N_i$. And  whenever $S \subset \partial C$ bounds a degenerate of $M_G$,  the component of $\partial C_i^A$  that is contained in $W_i^{S,A}$  is a level surface in some product region $U_i^S \in \mathcal U_i$ adjacent to $N_i$.
 So,  it follows that $$N_i \subset C_i^A \subset N_i^{ext},$$ and that $C_i^A$ is a standard compact core for $int(N_i^{ext})$.  In particular, $N_i$ is homeomorphic to $C_i^A$, and therefore to $C$.
\end{proof}
%

%%By Lemma \ref{the root component} and  Lemma \ref{bdlpr}, for instance, every component of $B_i$ has diameter bounded above independently of $i$, and  the  number of components is also  uniformly bounded. 
% Claim \ref{max prs} (2)  implies that each $W_i^{S,A}$ has a shares a level surface with some $U \in \mathcal U_i$. But any such $U$  separates $M_i$, e.g.\ by  Corollary \ref{unknotted2}. So as the set from \eqref{2ndg}  intersects $N_i$ (at $p_i$), and $N_i \cap int(U) =\emptyset$, the boundary component of $W_i^{S,A}$  that lies in $Z_i^{S,A}$ lies outside $N_i$. In other words, $g_i^A$ is  a closed submanifold of $M_i$ that intersect $int(N_i)$ (at $p_i$), but where all of the boundary components of $g_i^A$ lie outside of $int(N_i)$. Hence, $g_i^A$ contains $int(N_i)$,  and therefore contains $N_i$.

%
%Consider from now on the function $\kappa:\BN\to\BN$, $\kappa(x)=2^{6(x-1)}x$. If $M$ is a closed hyperbolic 3-manifold with $\rank(\pi_1(M)\le k$ and $\inj(M)\ge\epsilon$ then let 
%$$\Sigma_M\subset M$$ 
%be the surface provided by Theorem \ref{cutting along barriers} for $k,\epsilon,\kappa$. We also fix the associated constant 
%$$C=C(k,\epsilon,\kappa)$$ 
%coming from the said theorem and set
%$$C_1=\kappa(C)=2^{6(g-1)}K.$$
%
%\begin{sat}
%For all $\epsilon$ and $k$ there is $L$ such that if $M$ is a closed hyperbolic 3-manifold $\rank(\pi_1(M))\le k$ and $\inj(M)\ge\epsilon$ and if $N\in\pi_0(M\setminus\Sigma_M)$ is a component of the complement of $\Sigma_M$ then there is a carrier graph
%$$\phi:X\to N$$
%with $\rank(\pi_1(X))\le K_1$ and $\ell(X)\le L$.
%\end{sat}

\section{When there are no short compressible curves}
\label {the proof}

 In this section we prove the following result.

\begin{sat}\label {main}
 Given $k,\epsilon$,  there are a  finite collection $\mathcal M$ of  compact $3$-manifolds, and constants $n=n(k), g=g(k)$ and $B=B(k,\epsilon)$ as follows.
 Let $M$ be an orientable hyperbolic $3$-manifold such that  \begin {itemize}
 \item 	$\rank(\pi_1(M))\le k$ and $\inj(M)\ge\epsilon,$
 \item every essential closed curve on  $\partial CC(M)$  that is nullhomotopic in $M$ has length at least $\epsilon $.
\end{itemize} Then $int(CC(M))$  contains a collection $\mathcal U$ of at most $n$ product regions, each with genus at most $g$, such that every component $$N \subset CC(M) \setminus \cup_{U \in \mathcal U} int(U)$$ has diameter at most $B$ and  is homeomorphic to an element of $\mathcal M$. \end{sat}

 The  outline of the proof is as follows. We use  Theorem \ref{maxsplitting} to split $M$ into  submanifolds in which all barriers are peripheral, construct collections of product regions in these submanifolds using Theorem \ref{nonhakenthm}, and show that the union of all these collections satisfies the assumptions in Theorem \ref{main}. The fact of the complementary components $N$  can be selected from a a finite number of topological types is a consequence of the diameter bound, so we will prove that separately at the end of the section.
We now fill in the details.

\subsection{Product regions and diameter bounds}

  The $K_0$  in Proposition \ref{no barriers implies} depends only on the single input $k$, so we can consider it as a function $K_0 : \BN \longrightarrow \BN$. Let $\Sigma \subset M$ be the incompressible surface provided by Theorem \ref{maxsplitting}, for our given $k,\epsilon$ and the  function $k \mapsto K_0(k)$. So, $\Sigma$  has at most $5k$  connected components, each with genus at most some $k'=k'(k)$.

By work of Schoen--Yau \cite{Schoenexistence}, Sacks--Uhlenbeck \cite{sacks1982minimal} and Freedman--Hass--Scott \cite{Freedmanleast}, we can assume $\Sigma$ is homotopic to a minimal surface via a homotopy with tracks of length at most $\frac 12$, say. One would like to just assume that $\Sigma$ is minimal,  but it may be that a minimal surface homotopic to $\Sigma$ is not embedded.  However, the only way this can happen is if a component of the minimal surface double covers an embedded  non-orientable surface, in which case we can homotope take the corresponding component of $\Sigma$ to be an  embedded perturbation of  the minimal surface. See e.g.\ Souto \cite[\S 4]{souto2007geometry}  for details.
There is a Bounded Diameter Lemma  for connected minimal surfaces in  $\epsilon $-thick hyperbolic $3$-manifolds (again, see \cite[\S 4]{souto2007geometry}), and this gives a slightly larger bound for the diameters of the components of $\Sigma $.  In particular, \emph{after an isotopy we may assume that each component of $\Sigma$ has diameter at most some constant $d=d(k,\epsilon)$.}

 For each component $N \subset M\setminus \Sigma$, let $\pi^N : \hat M^N \longrightarrow M$ be the cover of $M$ corresponding to $\pi_1 N$, so that $N \subset M$  lifts  isometrically to the \emph{interior of} a submanifold $\hat N \subset \hat M^N.$ Note that every component $S \subset \partial\hat N$ is an isometric lift of some component of $\Sigma $, and bounds a submanifold of $\hat M^N$  homeomorphic to $S \times [0,\infty).$

Then  for each $N$,  we have that
 \begin {enumerate}
 \item 	$\rank(\pi_1(\hat M^N))\le k$ and $\inj(\hat M^N)\ge\epsilon,$
 \item every essential closed curve on  $\partial CC(\hat M^N)$  that is nullhomotopic in $\hat M^N$ has length at least $\epsilon $,
\item  every barrier of a $(K_0,L)$-product region in $N$ is peripheral, where $L$ is as given by Theorem \ref{maxsplitting}.
\end{enumerate} 
 To see that (2) holds, note that every component of $\partial CC(\hat M^N)$  either projects homeomorphically down to a component of $\partial CC(M)$, in which case (2) follows from the similar assumption on $M$, or projects to a surface isotopic to some component of $\Sigma$, and hence is incompressible.

 We would like to apply Theorem \ref{nonhakenthm}, but (3) is not quite what we need, since it applies to $N$  and not $\hat M^N$.  However:

\begin {claim}\label {l0}
	 There is some $L_0=L_0(k,\epsilon)$  such that every barrier of a $(K_0,L_0)$-product region in $\hat M^N$ is peripheral.
\end {claim}

\begin{proof}
Fix $L_0$, to be determined during the course of the proof. Let $U$ be a $(K_0,L_0)$-product region in $\hat M^N$. If $L_0$ is large enough,  we can use Corollary \ref{cutting-product}  to find a subproduct region $U' \subset U$  such that $$\width U' \geq L, \ \ d(U',\partial U) > d,$$ where $d$ is  the upper bound for the diameters of components of $\Sigma$, and $L$  is from (3)  above.

First assume $U' \cap \partial\hat N \neq \emptyset$. By our choice of $d$,  the product region $U$  must contain some component $S \subset \partial\hat N$. Any incompressible embedded surface in a project region is a level surface, so this means $U$  is a regular neighborhood of $S$, which is incompressible in $\hat M^N$. Hence, all barriers of $U$ are isotopic to $S$, and are therefore peripheral.

 If $U' \subset \hat N$, then $U'$ projects to a product region in $N$ and we are done by (3).  If $U \cap \hat N = \emptyset$,  then $U' $ is contained in  a component $$C \subset \hat M^N \setminus \hat N,$$ and $C \cong S \times [0,\infty)$  for some  surface $S$ of genus at most $k'$.  Applying Corollary \ref{unknotted2} in the cover of $\hat M^N$  corresponding to $ \pi_1 C$, it follows that $U'$ is a topological subproduct region of $C$ as long as $L_0$ is  larger than the constant $L$ from  Corollary \ref{unknotted2}, given the input $g=\max\{4K_0,k'\}$. Then as before, all barriers of $U$ are isotopic to level surfaces of $C$, and  therefore are peripheral in $\hat M^N$.
\end{proof}

 We can now apply Theorem \ref{nonhakenthm} with the inputs of $k',\epsilon$ and the $L_0$ in  Claim \ref{l0}. The result is that there are $g,n'$ depending on $k'$ and $B'=B'(k',\epsilon)$ such that $CC(\hat M^N)$  contains a collection $\mathcal U^N$ of at most $n'$ $(g,L_0)$-product regions such that every component of $$\CN_1(CC(\hat M^N)) \setminus \cup_{U \in \mathcal U^N} int(U)$$
 has diameter at most $B'$.   We  now want to alter the product regions $U \in \mathcal U^N$  so that they lie in $\hat N$.  Namely, we will show:

\begin{claim}\label {pruning prs}
	 There is a collection $\mathcal V^N$ of at most $n'$  product regions $$V \subset CC(\hat M^N) \cap int(\hat N),$$  each with  genus at most $g$, such that  every component $$C \subset \hat N \cap \CN_1(CC(\hat M^N)) \setminus \cup_{V \in \mathcal V^N} int(V)$$
 has diameter at most some $B'=B'(k',\epsilon).$
\end{claim} 
\begin{proof}
 We will construct each $ V \in \mathcal V^N$  as a subproduct region of some $U \in \mathcal U^N$.   To begin with,  recall that each component $S \subset \hat \partial N$  is incompressible in $\hat M^N$ and has  diameter at most some $d=d(k,\epsilon)$.  So if $S$  intersects some $U \in \mathcal U^N$,  then either
\begin{itemize}
\item $S \subset U$,  and is therefore a level surface of $U$, or
	\item $S \cap \partial U \neq \emptyset$, and $S$  is contained in a $d$-neighborhood of $ \partial U $.
\end{itemize} 
 Since $\hat N$  separates $\hat M^N $,  \emph{at most two} components of $\partial \hat N$  can be level surfaces of $U$. (This only happens if $\hat N $  itself happens to be a topological subproduct region of $U$.)  So, the intersections of $U$ with $\partial \hat N$  can be roughly grouped into at most four   bounded diameter sets:  at most two adjacent to  the boundary components of $U$, and at most two  enclosing level surfaces of $U$.  More precisely,  Corollary \ref{cutting-product} gives some $D=D(g,\epsilon)$ and a collection of disjoint subproduct regions \begin{equation}
 	U_i \subset U, \ i\in I,\  \text{ where } |I|\leq 4, \label {ui}
 \end{equation}
 such that $\partial \hat N \cap U \subset \cup_i int(U_i)$, and where  for each $i$, we have $$\diam U_i \leq D.$$

 By construction, each component of $U \setminus \cup_i int(U_i)$  is a product region\footnote{Here, take $int(U_i)$ to mean the interior of $U_i$ as a subset of $U$. The $U_i$  may share boundary components with $U$, so if we take the interior in $\hat M^N$ then the difference may contain components that are surfaces.} and either lies in $int(\hat N)$,  or lies outside $\hat N$. In fact, \emph{at most one} such component can lie inside $int(\hat N)$,  since $\hat N$  separates $\hat M^N$. If  such a component exists, call it $V_U$ and call $U$ \emph{admissible}. Let $$\mathcal V^N = \{ V_U \ | \ U \in \mathcal U^N \text{ is admissible} \}.$$  Then $\mathcal V^N$  is a collection of at most $n'$ product regions in $CC(\hat M^N) \cap \hat N$,  each with  genus at most $g$.  Moreover, every component $$C \subset \hat N \cap \CN_1(CC(\hat M^N)) \setminus \cup_{V \in \mathcal V^N} int(V)$$
 is a union of (boundedly many)   submanifolds of the  following types:
\begin{itemize}
\item 	components of $\CN_1(CC(\hat M^N)) \setminus \cup_{U \in \mathcal U^N} int(U)$, 
\item  subproduct regions $U_i \subset U$  as in \eqref{ui},
\item  components of $U \setminus \cup_i U_i$ as in \eqref{ui} that lie outside $\hat N$.
\end{itemize} 
All three of these types of submanifolds have bounded  diameter intersections with $\hat N$,  so the claim follows.\end{proof}

 We now construct the collection of product regions in $M$. Let
$$\mathcal U = \bigcup_{N} \big \{ \pi^N(V) \ | \ V \in \mathcal V^N \big  \},$$
 where the union is over all components $N \subset M \setminus \Sigma$.  Since for each $N$,  we have that $int(\hat N) \subset \hat M^N$  projects isometrically under $\pi^N$ to $N \subset M $, the elements of $\mathcal U$ are all product regions, and are pairwise disjoint.   The size $|\mathcal U|$ is bounded, at worst, by $$n:=10k \cdot n',$$where $10k$  is a bound for the number of components $N$ and $n'$  is a bound for the size of each $\mathcal V^N$.  Since for each $N$ we have
$$\pi^N(CC(\hat M^N)) \subset CC(M),$$  every $U\in \mathcal U$  is contained in the convex core of $M$.  

\begin {claim}\label{ccsub}
	 For each  component $N\subset M \setminus \Sigma$, we have that 
$$N \cap CC(M) \subset  \pi^N(\hat N \cap \CN_1(CC(\hat M^N))).$$ 
\end {claim}
\begin{proof}
Above, we assumed that $\Sigma$ was homotopic to a minimal surface in $M$ via a  homotopy with tracks of length at most $\frac 12$. Minimal surfaces are always contained in the convex core, so lifting the homotopy  and applying this observation in $\hat M^N$, we see that \begin{equation}
	\partial \hat N \subset \CN_{\frac 12}(CC(\hat M^N)).\label{hatbound}
\end{equation} 

So, assume $p\in N \cap CC(M)$, and lift $p$ to $\hat p \in \hat N$, $\pi^N(\hat p)=p$.  Suppose that $\hat p \not \in \CN_1(CC(\hat M^N))$, for otherwise we are done. Then  \eqref{hatbound}  implies that the  entire component $E \subset \hat M^N \setminus \CN_1(CC(\hat M^N))$  containing $\hat p$  is contained in $\hat N$. (If not,  there would have to be some component of $\partial \hat N$ that intersects $E$.) It follows that $\pi^N$  restricts to an embedding on $E$, in which case Lemma \ref{convex cocompact cover} implies that $\pi^N(E) \subset M \setminus CC(M)$, contrary to the assumption that $p\in CC(M)$.
\end{proof}

 It now follows from Claims \ref{pruning prs} and \ref{ccsub} that for each  component
\begin{equation}
	C \subset N \cap \CN_1(CC(M)) \setminus \cup_{V \in \mathcal V^N} int(\pi^N(V)),\label {thec}
\end{equation} 
 the diameter of $C \cap CC(M)$ is at most some $B'=B'(k',\epsilon)$.  But  since every component of $CC(M)$ is intrinsically an $\epsilon$-thick hyperbolic surface, the  Bounded Diameter Lemma implies that every component of $\CN_1(CC(M)) \setminus CC(M)$ has diameter bounded above by some constant depending on $k,\epsilon$. In particular, this means that the components $C$ in \eqref{thec} themselves have diameter bounded above in terms of $k,\epsilon$. Since  every component of $\CN_1(CC(M)) \setminus \cup_{U \in \mathcal U} int(U)$ is a union of (boundedly many) components $C$ as in \eqref{thec}, we are done.

\subsection{Bounding the number of topological types}
 We have shown that in each  hyperbolic $3$-manifold $M$ such that 
\begin{enumerate}
\item $\rank \pi_1 M \leq k$,
\item $\inj(M)\geq \epsilon$,
\item every essential closed curve on  $\partial CC(M)$  that is nullhomotopic in $M$ has length at least $\epsilon $,
\end{enumerate}  there is a collection $\mathcal U$  of at most $n=n(k)$ product regions in $CC(M)$, each with genus at most $g=g(k)$,  such that every component $$C \subset \CN_1(CC(M)) \setminus \cup_{U \in \mathcal U} int(U)$$
 has diameter at most some $B=B(k,\epsilon)$.  We want to show that we can choose these $\mathcal U$  so that the complementary components $C$ take on only finitely many topological types.

 The argument is by contradiction.  Choose an enumeration of the set of all homeomorphism types of compact $3$-manifolds. If the claim is false,  there is a sequence $(M_i)$ of  hyperbolic $3$-manifolds satisfying (1)--(3) above such that $M_i$ \emph{does not} admit any collection $\mathcal W_i$ of at most $n$ genus $g$  product regions in $CC(M_i)$ such that for every component $$C \subset \CN_1(CC(M_i)) \setminus \cup_{W \in \mathcal W_i} int(W),$$
 we have that both $\diam C \leq B$, and $C$ is homeomorphic to one of the first $i$ compact $3$-manifolds, with respect to our enumeration. 

For each $i$, let $\mathcal U_i$ be a collection of at most $n$ genus $g$ product regions in $CC(M_i)$  such that every component of $\CN_1(CC(M_i)) \setminus \cup_{U \in \mathcal U_i} int(U)$ has diameter at most some $B$  independent of $i$. We can assume that $$\min \{\width(U) \ | \ U \in \mathcal U_i \} \to \infty,$$
 for if it does not,  the fact that the number of product regions is bounded independently of $i$ means that there is some $L$  independent of $i$ such that the widths of all product regions that are wider than $L$ go to infinity with $i$. We can then exclude the product regions with width less than this $L$ from our collections, at the expense of replacing $B$ by some larger constant independent of $i$. See the proof of Claim \ref{max prs}, where the same argument was used, for more details.

By our assumption, for each $i$ there must be a component $$N_i \subset \CN_1(CC(M_i)) \setminus \cup_{U \in \mathcal U_i} int(U)$$
 that is not homeomorphic to one of the first $i$  compact $3$-manifolds.  We would like to apply Proposition \ref{diameterandtopbounds} to $N_i$, so we now verify the relevant assumptions.

First, the ranks of  the subgroups $[\pi_1 N_i] \subset \pi_1 M$  are bounded above, independently of $i$.   For if we take an $\epsilon$-net in $N_i$, the number of points in the net is bounded above independent $i$, by the uniform diameter bound, and the incidence graph $X_i$ of the collection of $\epsilon$-balls in $M$ centered at points in our net $\pi_1$-surjects onto $[\pi_1 N_i]$. Note that since the boundary components of $N_i$ are framed by wide topological product regions, the image of $\pi_1 X_i$  cannot strictly contain $[\pi_1 N_i]$. Also, it is perhaps worth mentioning that one could get much better rank bounds by tracing through the construction of the $\mathcal U_i$, but this is irrelevant for the  current argument.

 By construction, each component of $\partial N_i$  is either a component of $\CN_1(CC(M_i))$, or is adjacent to a product region whose widths goes to infinity with $i$. Since $\diam N_i \leq B$, no $N_i$  can contain a product region with width bigger than $2B$, by Fact \ref{width diameter}. So,  Proposition \ref{diameterandtopbounds} applies, where we set the referenced $g$ to be the maximum of the genus bound for the product regions in our $\mathcal U_i$ and the rank bound for the $N_i$, say.  It follows that there is some subsequence $N_{i_j}$ in which all the manifolds are homeomorphic, say to some fixed compact $3$-manifold $N$.  For some $j$, we must have that $i_j$ is larger then the index of $N$ in our enumeration. So,  this contradicts our assumption that no $N_i$  is  homeomorphic to one of the first $i$ compact $3$-manifolds.

\section{The general case}
\label {section general}
Let  $M$  be a complete, $\epsilon_0$-thick hyperbolic $3$-manifold with finitely generated fundamental group. The  convex core boundary $\partial CC(M)$ is a compact pleated surface \cite{Thurstongeometry}, so  in particular it is intrinsically hyperbolic. Assume for convenience that $\epsilon_0$  is less than the $2$-dimensional Margulis constant. Then if $\epsilon<\epsilon_0$, any component $$A \subset (\partial CC(M))_{< \epsilon }$$ of the $ \epsilon $-thin part of  the convex core boundary  is a (compact) annulus that is \emph {compressible}, i.e.\ nullhomotopic in $M$. 

 Somewhat following Bowditch \cite{bowditch2010upper}, we define:

\begin{defi}\label {epsilon handle}
	An \emph{$\epsilon$-thin $1$-handle} is  a topological ball $H \subset CC(M)$ whose  boundary  decomposes as a union $\partial H = D_1 \cup A \cup D_2,$ where  $A$  is a component of $( \partial CC(M))_{< \epsilon}$ and $D_1,D_2$ are  properly embedded discs in $CC(M)$, each with  intrinsic diameter less than $2$, say, such that $$\partial A  = \partial D_1 \cup \partial D_2.$$
\end{defi}

The terminology  reflects the fact that  the ball $H$ is most naturally thought of as a $1$-handle.  For convenience, we say $H$ is \emph{associated to} $A$ and we define the \emph{horizontal} and \emph{vertical} boundaries of $H$ as $$\partial_h H := A, \ \ \partial_v H := D_1 \cup D_2,$$  respectively. Of course, the precise constant $2$ in Definition \ref{epsilon handle} is not important, but it will be useful to have \emph{some} diameter bound.   

At this point, the reader could skip ahead and read the statement of Theorem \ref{general case}, if desired. However, we will find it convenient to first establish a few properties about $\epsilon$-thin $1$-handles.

\begin{claim}
For every $\epsilon<\epsilon_0$ and every component $A \subset 	( \partial CC(M))_{< \epsilon}$, there is an $\epsilon$-thin $1$-handle $H$ with $\partial_h H =A$. 
\end{claim}
\begin{proof}
As $\epsilon$ is less than the $2$-dimensional Margulis constant, which is approximately $.2629$ by work of Yamada \cite{yamada1981marden}, see \cite[pg 440]{gehring1996margulis}, one can check that each component of $\partial A$ has length less than $\frac 12$. Coning each component to one of its points, we obtain properly embedded discs $D_1,D_2 \subset CC(M)$, each with intrinsic diameter strictly less than $1$, such that $\partial A  = \partial D_1 \cup \partial D_2$. Perturbing to general position and  employing the usual surgery arguments, these disks can be made disjoint, at the expense of doubling (after perturbing) the diameter bounds. So, $2$ is a diameter bound for the new $D_1,D_2$. The union $D_1 \cup \partial A \cup D_2$ is  then an embedded sphere in $CC(M)$, which bounds a ball $H$.\end{proof}

\begin{claim}\label{nestedhandles}
 For fixed $\epsilon'<\epsilon_0$,  the following holds for sufficiently small $\epsilon < \epsilon'$. If  $A\subset A'$ are  components of the $\epsilon$-thin and $\epsilon'$-thin parts  of $\partial CC(M)$, and $H, H'$ are associated $\epsilon$-  and $\epsilon'$-thin $1$-handles, we have  $$H \subset H', \ \ d_M(H, \partial_v H') \geq 2.$$
\end{claim}
\begin {proof}
It suffices to show that the \emph{intrinsic} distance $d_{H'}(\partial H,\partial_v {H'}) \geq 2$, since any geodesic in $M$ from $H$ to $\partial_v {H'}$ has a segment that  runs from $\partial H$ to $\partial_v {H'}$ within ${H'}$. Note that  if we prove this, we will  obviously have $\partial_v H \cap \partial_v {H'}=\emptyset$, so it will follow automatically that $H \subset {H'}$.

Since ${A'}$  is compressible,  there must be some spanning arc $\alpha$ of ${A'}$ that is part of the  bending lamination of $\partial CC(M)$. Each component of ${A'} \setminus int(A)$ is spanned by a segment of $\alpha$, and  if we take $\epsilon'<<\epsilon$  we can assume that the length of any spanning arc in ${A'} \setminus int(A)$ is at least $7$. Since $\alpha$ is a geodesic in ${H'}$, which is simply connected, it is the shortest path in ${H'}$ between its endpoints, so we then have
$$d_{H'}(\partial A,\partial {A'}) \geq 7 - \epsilon' - \epsilon > 6,$$
where the $\epsilon',\epsilon$ are  bounds for the diameters of the boundary components of $A,{A'}$,  respectively. Since  the components of $ \partial_v H$ and $\partial_v {H'}$ all have diameter at most $2$, it   follows that \[d_{H'}(\partial H,\partial_v {H'}) > 6-2-2 =2.\qedhere\]
\end {proof}
%
%\begin{claim}\label{1handles}
% Fixing any sufficiently small $\epsilon<\epsilon_0$, there is a collection of \emph{pairwise disjoint} $\epsilon$-thin $1$-handles, one  for each component of $(\partial CC(M))_{<\epsilon}$.
%\end{claim}  
%\begin{proof}
%Fix some $\epsilon'<\epsilon_0$. Applying the previous claims, if $\epsilon$ is small, we can choose a collection $\mathcal H$ of $\epsilon$-thin $1$-handles, one for each $(\partial CC(M))_{<\epsilon}$, and a collection $\mathcal H'$ of $\epsilon'$-handles, such that for each $H \in \mathcal H$, there is some $H'\in \mathcal H'$ with 
%$H \subset H'$ and $ d_M(H, \partial_v H') \geq 1.$ For any two $H_1,H_2 \in \mathcal H$, we have that $\partial_h H_2 \cap H_1'=\emptyset$, so as each  component of $\partial_v H_2$ has  diameter at most $1$, we have $H_2 \cap H_1=\emptyset$.\end{proof}
%
We can now state the most general version of our result.  It is slightly more convenient to state the following using $CC(M)$ rather than its $1$-neighborhood, which we used in  Theorem \ref{main}, so we include below the assumption that $CC(M)$ is a $3$-dimensional submanifold of $M$.  There is no loss of generality in assuming this, since if Theorem \ref{main} does not already apply, then $M$  is the interior of a manifold with compressible boundary,  in which case $CC(M)$ is automatically $3$-dimensional.

\begin{sat}[Decomposition by $\epsilon$-thin $1$-handles and PRs]\label{general case}
Given $k,\epsilon_0$, there are a  finite collection $\mathcal M$ of  compact $3$-manifolds, and constants $n=n(k),$ $ g=g(k)$, $\epsilon=\epsilon(\epsilon_0)<\epsilon_0$ and $B=B(k,\epsilon_0)$ as follows.

 Suppose $M$ is a  complete,  orientable hyperbolic $3$-manifold with $$\rank(\pi_1(M))\le k, \ \ \inj(M)\ge\epsilon_0,$$ and assume that $CC(M)$ is a $3$-dimensional submanifold of $M$. Then $CC(M)$ contains a  collection $\mathcal H$ of  pairwise disjoint $\epsilon$-thin $1$-handles and a  collection $\mathcal U$ of at most $n$ product regions, each with genus at most $g$, such that  for each $U\in \mathcal U$,  we have $$U \subset int(CC(M)) \setminus \cup_{H \in \mathcal H} H,$$ and  where the closure of every component of $$N \subset CC(M) \setminus ((\cup_{U \in \mathcal U} U ) \cup (\cup_{H \in \mathcal H} H))$$ has diameter at most $B$ and  is homeomorphic to an element of $\mathcal M$. 
\end{sat}

\begin {proof}
Pick $\epsilon'<\epsilon_0$ and let $\epsilon$ be given by Claim \ref{nestedhandles}. Let  $\mathcal H$  be a collection of $\epsilon$-thin $1$-handles, one for each component of $(\partial CC(M))_{<\epsilon}$, choose a collection $\mathcal H'$ of $\epsilon'$-thin $1$-handles, such that  there is a bijection 
\begin{equation}
\label {HH'}	\mathcal H \xrightarrow{H \mapsto H'} \mathcal H',  \text{ where } H \subset H', \ d_M(H, \partial_v H') \geq 2.
\end{equation}

Note that for any two $H_1,H_2 \in \mathcal H$, we have that $\partial_h H_2 \cap H_1'=\emptyset$, so as each  component of $\partial_v H_2$ has  diameter less than's $2$,  $H_2 \cap H_1=\emptyset$. That is, the $\epsilon$-thin $1$-handles in $\mathcal H$  are pairwise disjoint, as required.

Fix a component
$N \subset CC(M) \setminus \cup_{H \in \mathcal H} H,$
and define the \emph {horizontal} and \emph{vertical} boundaries of $N$ as
\begin {equation}\label {boundarieseq}\partial_h N = \partial N \cap \partial CC(M), \ \ \partial_v N = \partial N \cap int(CC(M)).\end{equation}
So, $\partial_v N$ is a union of vertical boundary components of $\epsilon$-thin $1$-handles.
Since $\epsilon$-thin $1$-handles are balls and $CC(M) \hookrightarrow M$  is a homotopy equivalence, $\pi_1 N$ injects into $\pi_1 M$  as a free factor, and hence has rank at most $k$ by Grushko's theorem. Call the corresponding cover $$\pi^N : M^N \longrightarrow M.$$

%Explicitly, $M^N$ can be constructed as follows. For each component $D \subset \partial_v N$, extend $D$ to a properly embedded open disc $$D^{ext} \subset M$$ by adding to $D$ some annulus in $M \setminus CC(M)$. Take the $D^{ext}$ to be disjoint as $D$ varies, and let $N^{ext}$ be the  component of $M \setminus \cup_D D^{ext}$ containing $N$. So, $N^{ext} \setminus N$ is homeomorphic to $\partial_h N \times (0,\infty)$. Take a copy of $\BH^3$, regarded as the universal cover of $M$, and lift $D^{ext}$ to a properly embedded open disc $\tilde D^{ext} \subset \BH^3$. Delete the component of $\BH^3 \setminus \tilde D^{ext}$ that faces the side of $\tilde D^{ext}$ whose projection to $M$ faces $N$, and glue the resulting half-space $\BH_D \subset \BH^3$ to $N^{ext}$ along $D^{ext}$. Then
%$$M^N \cong N^{ext}\ \ \cup \  \bigcup_{\substack{ \text{components} \\  D \subset \partial_v N}} \BH_D,$$
%where each $\BH_D$ is glued to $N^{ext}$ as  indicated above.

 This $M^N$  has rank at most $k$ and injectivity radius at least $\epsilon$, and $N$ lifts homeomorphically to $\hat N \subset M^N$. %Moreover, it  follows from the explicit construction above that every  component of $M^N \setminus int(\hat N)$ is  homeomorphic to $S \times [0,\infty)$, where $S$ is a  component of $\partial \hat N$. 

\begin{claim} \label{verify}
There are constants $\delta,B$ depending only on $\epsilon,k$ such that
\begin{enumerate}
	\item $CC(M^N) \subset \hat N, $
\item  each component of $\hat N \setminus CC(M^N)$  is a trivial interval bundle that has diameter at most $B,$
\item every  essential closed curve on $\partial CC(M^N)$ has length at least $\delta$.
\end{enumerate}
\end{claim}

	%%Every essential closed curve on $\partial CC(M^N)$  that is nullhomotopic in $M^N$ has length at least some $\epsilon'=\epsilon'(\epsilon,k)$.

\begin{proof}
The manifold $N$  contains a unique component $$N' \subset CC(M) \setminus \cup_{H' \in \mathcal H'} int(H').$$
 Note that $N \setminus N'$ is a union of balls,  each of which is a component of some $H \setminus H'$, where $H\in \mathcal H$. Define the horizontal and vertical boundaries of $N'$ just as in \eqref{boundarieseq}, e.g.\ $\partial_h N' = \partial N' \cap \partial CC(M).$

%Let $\hat N' \subset \hat N$  be the homeomorphic lift of $N'$ to $M^N$, and lift the horizontal/vertical decompositions of $\partial \hat N' $ and $\partial \hat N$ to decompositions
%$$\partial \hat N = \partial_h \hat N \cup \partial_v \hat N, \ \ \partial \hat N' = \partial_h \hat N' \cup \partial_v \hat N'.$$
% Since $\pi^N(int(CC(M^N))) \subset int(CC(M)),$
% we have that 
%\begin{equation*}
%\partial_h \hat N' \cap int(CC(M^N)) = \emptyset.
%\end{equation*} 
% So in particular, if $D'$ is a  component of $\partial_v \hat N'$, we have $$\partial D' \cap int(CC(M^N)) = \emptyset.$$
%
%Each component of $M^N \setminus int(\hat N')$ is  homeomorphic to $S \times [0,\infty)$, where $S$ is a  component of $\partial \hat N'$, since the analogous statement was true for $N$ and $N \setminus N'$  is just a union of balls. Clearly, every  component of $M^N \setminus int(CC(M^N))$ is  also a product. Since $\partial D'$ contracts to a point in $\partial \hat N'$, it contracts to a point in $M^N \setminus int(CC(M^N)) $. For each such $D$, let $$D' \looparrowright M^N \setminus int(CC(M^N))$$ be a map from a disc whose  boundary maps to $c$.

 Pick a point $\hat  p \in \hat N'$, and let $\hat \gamma$  be a geodesic segment in $M^N$  that starts and ends at $ \hat p$. To show (1), it suffices to prove that $\hat \gamma \subset \hat N$.  We will work in $M$ rather than $M^N$, so project $\hat \gamma$ down to a geodesic $\gamma$ that  starts and ends at some $p\in N'$. We want  to show that $\gamma \subset N$. 

Now $\gamma$ is contained in $CC(M)$ and  is homotopic within $CC(M)$ into $N'$, so we can pick a smooth map $$f: A \longrightarrow CC(M)$$
 from an annulus $A$ such that $\partial A = a \sqcup b$ and $f|_a $  parameterizes $\gamma$ while $f(b) \subset N'$.   After a perturbation, we can assume that $\partial_v N'$ is smooth and transverse to $f$.  Consider the $1$-manifold $$I_f := f^{-1}(\partial_v N')\subset A.$$

Since $\gamma$  is homotopically nontrivial and each  component of $\partial_v N'$  is a disc, no component of $I_f$ is an essential simple closed curve on $A$. And as $CC(M)$ is irreducible, we can assume after a homotopy of $f$ that no component of $I_f$  is an  inessential  simple closed curve. So as $f(b) \subset N'$,  every component of $I_f$  is an arc in $A$  with both endpoints on $a$. Let $a_{sep} \subset a$ be the set of all points that are separated from $b$ by $I_f$. Then $$f(a\setminus a_{sep}) \subset N',$$
since $f(b) \subset N'$.  Every component  $a'\subset a_{sep}$ maps under $f$ to a segment $\gamma' \subset \gamma$ that is homotopic rel endpoints to an arc in some component of $\partial_v N'$. Since every component of $\partial_v N'$ has intrinsic diameter less than $2$, the geodesic segment  $\gamma'$ has length less than $2$. By \eqref{HH'}, we have $d(\partial_v N', H) \geq 2$ for all $H \in \mathcal H$, so it follows that $\gamma \subset N$ as desired.

We now have proven (1), i.e.\ that $CC(M^N) \subset \hat N$. Let $\CE$ be a convex-cocompact  end of $M^N$, let $\partial_\CE CC(M^N)$  be the  component of the  convex core boundary that faces $\CE$, and let $$E \subset M^N \setminus int(CC(M^N))$$ be the  component that is a  neighborhood of $\CE$.  We cannot have $E \subset \hat N$, since then $\pi^N(E) \subset N$ would be  a convex cocompact end of $M$, but $N \subset CC(M)$. So, some  component  $S \subset \partial \hat N$ is contained in $E$.  By the uniqueness of compact cores\footnote{Both $CC( M^N)$ and $\hat N$ may be  noncompact, but  one can apply the uniqueness theorem to compact cores of each.} up to homeomorphism \cite{mccullough1985uniqueness}, $S$ is  homeomorphic to $\partial_\CE CC(M^N)$. Since $S$ separates $\partial_\CE CC(M^N)$ from infinity in $E$,  it follows that $S$ is a level surface of $E$.

Now $\partial N$ is  a union of components of the $\epsilon$-thick part of $\partial CC(M)$ and   a collection of discs, each with  intrinsic diameter at most $2$. So by the bounded diameter lemma, each component of $\partial N$ has intrinsic  diameter bounded by  some constant depending on $\epsilon$ and its genus. As the genus of any component of $\partial CC(M)$ is at most $\rank \pi_1 M \leq k$ and our $S$ above is isometric to the component $\pi^N(S) \subset \partial N$, we have that $S$ has  intrinsic diameter at most some constant depending only on $k,\epsilon$.

 We are now in a position to prove (2) and (3). Let 
$$p : M^N \longrightarrow CC(M^N)$$
 be the nearest point retraction. Then $p(S) = \partial_\CE CC(M^N)$ and $p$ is $1$-lipschitz, so $\partial_\CE CC(M^N)$ has  intrinsic diameter bounded above by some constant depending only on $k,\epsilon$. As $\partial_\CE CC(M^N)$ is a hyperbolic surface, its injectivity radius is then  at least some $\delta$  depending only on $k,\epsilon$, which proves (3).  For (2), recall that if $x \in E $, then the lipschitz constant of $p$ near $x$ is  an exponential in $-d(x,p)$, which follows from Fact \ref{expdecrease}. So, the fact that $p(S) =\partial_\CE CC(M^N)$, which has some diameter bounded above by some uniform constant, implies that $S$ intersects a radius $r$  neighborhood of $\partial_\CE CC(M^N)$, for some $r=r(k,\epsilon)$. Since $S$ has bounded diameter, it is then \emph{contained} in $\CN_r(\partial_\CE CC(M^N))$ for some larger $r=r(k,\epsilon)$.   Finally, as  $\partial_\CE CC(M^N))$  has bounded diameter and $S$ is a level surface of $E$, it  follows that the component of $\hat N \setminus CC(M^N)$  contained in $E$ has  diameter at most some  constant $B=B(k,\epsilon)$, as desired.
\end{proof}

 We now finish the proof of Theorem \ref{general case}. For each  component $N \subset CC(M) \setminus \cup_{H \in \mathcal H}H$,  the manifold $M^N$ has  rank at most $k$ and injectivity radius at least $\epsilon$, so Claim \ref{verify} (3) implies that $M^N$  satisfies the assumptions of Theorem \ref{main}, at least after replacing $\epsilon$ with $\delta$, which only depends on $k,\epsilon$. Hence, for each $N$ there is a collection $\hat {\mathcal U}^N$ of at most $n=n(k)$ product regions in $int(CC(M^N))$, each with genus at most $g=g(k)$, such that every component of
$$CC(M^N) \setminus \cup_{\hat U \in \hat {\mathcal U}^N} int(\hat U)$$
has diameter at most $B=B(k,\epsilon)$ and is homeomorphic to some element of a  subset $\mathcal M$ of  compact $3$-manifolds depending on $k,\epsilon$. Define $$\mathcal U^N = \{ \pi^N(\hat U) \ | \ \hat U \in \hat {\mathcal U}^N\}, \ \ \mathcal U = \cup_N \mathcal U^N.$$ 
By Claim \ref{verify} (1), each $U^N \in \mathcal U$ lies in $int(N) \subset int(CC(M))$. By Claim \ref{verify} (2), for each $N$ every component
$$C \subset N \setminus \cup_{U \in \mathcal U^N} int(U)$$
is  isometric to the union of a component of $$CC(M^N) \setminus \cup_{\hat U \in \hat {\mathcal U}^N} int(\hat U)$$
with some  bounded diameter trivial interval bundles glued onto its boundary components. Hence, $C$ has diameter at most some constant  depending on $k,\epsilon$, and  is homeomorphic to an element of $\mathcal M$ above.  As  the closure of  every component of $$CC(M) \setminus ((\cup_{U \in \mathcal U} U ) \cup (\cup_{H \in \mathcal H} H))$$ is equal to such a $C$, we are done.\end {proof}

\subsection{Proofs of the three corollaries}
\label {corollaries}

The reader might guess that we prove here the three corollaries from the introduction.

Before doing so, we first show that  both $\epsilon$-thin $1$-handles and  product regions  are quasi-isometric to intervals.  The proofs are similar.

\begin {lem}\label {QI1}
Fix $0<\epsilon<1$ smaller than the Margulis constant.  Suppose that $M$  is a hyperbolic $3$-manifold, $H \subset CC(M)$ is an $\epsilon$-thin $1$-handle. Then  for some  compact interval $I\subset \BR$, there is a $(1,12)$-quasi-isometry $\rho : (H,d_H) \longrightarrow I$  such that the two components of $\partial_v H$ map to  the the two endpoints of $I$,  respectively. 
\end {lem}

Here, $d_H$  is the  path metric in $H$.

\begin {proof}
Since  the annulus $\partial_h H$ is compressible,  it must have a spanning arc $\gamma : I \longrightarrow \partial_h H $ that is part of the  bending lamination of $\partial CC(M)$, where here $I \subset \BR$  is a compact interval. So, this $\gamma$  is a geodesic segment in $M$, which is distance minimizing  with respect to $d_H$.
 We claim that for each point $x \in H$  we have
$$d_H(x,\gamma(I)) \leq 5.$$

Since $H$  is a ball,  it lifts  isometrically to $\tilde H \subset\BH^3.$ Lift $\gamma$ as well to $\tilde \gamma : I \longrightarrow \partial_h \tilde H.$
  Foliate $\partial_h \tilde H$ by  circles $C_t$ with length at most $2\epsilon$, where each $C_t$ intersects $\tilde \gamma$ transversely at $\gamma(t)$. In $\BH^3$, cone each circle $C_t$ to $\tilde \gamma(t)$,  producing a  continuously varying family of disks $D_t$,  each with diameter at most $2\epsilon$. (For instance, the circles $C_t$  can be the boundary components of regular neighborhoods of the core geodesic of the hyperbolic annulus $\partial_h \tilde H$.) The discs $D_a$ and $D_b$  have the same boundaries as the components of $D_1,D_2 \subset \partial_v \tilde H$, each of which has diameter at most $2$.  So, $D_a \cup D_1$ and $D_b \cup D_2$ are two (singular) spheres, each with diameter at most $3$, say. Fill in these spheres  with (singular) balls of  diameter at most $3$, e.g.\ by  coning again. Regard the balls as  homotopies rel $\partial$ from  $D_a $ to $ D_1$, and from $D_b $ to $ D_2$, respectively.   Concatenating with our original interpolation,  this gives an interpolation via disks of diameter at most $3$ between $D_1,D_2$.  Moreover, if we regard this  interpolation as a $3$-chain in $\BH^3, $ its boundary is $\partial \tilde H$. So, the interpolation must pass through every point of $\tilde H$.  This implies that every point $x \in \tilde H$  is within $3$ of $\tilde \gamma(I)$ in $\BH^3$.  Since every point of $\partial \tilde H$  can be joined to $\tilde \gamma(I)$  by a path on $\partial \tilde H$  of length at most $2$, the \emph {intrinsic}  distance from any $x \in \tilde H$ to $\gamma(I)$  is at most $5$.

So, given $x \in H$, let $\rho(x)$ be any $t\in I$  for which the distance $d_H(x,\gamma(t))$ is minimal. Since $\gamma$ is $d_H$-minimizing, we have  \begin {equation}
 	d_H(x,y)-10 \leq |\rho(x)-\rho(y)| \leq d_H(x,y)+10.\label {qi}
 \end {equation}
 Moreover, $\rho$ is surjective, since $\rho(\gamma(t))=t$. Hence, $\rho$ is a $(1,6)$-quasi-isometry.  Since the endpoints of $\gamma$  lie on components of $\partial_v H$,  it follows that these  components map within the $2$-neighborhoods of the endpoints of $I$. So, we can just change $\rho$  so that it maps the components of $\partial_v H$  to the endpoints, at the expense of changing $10$ to $12$.
 \end {proof}

\begin {lem}\label {QI2}
 Given $g,\epsilon>0$,  there is some $C$  as follows. Suppose that $M$  is $\epsilon $-thick, $U\subset M $ is a product region and consider $U$  with its path metric $d_U$. Then for some properly embedded interval $I \subset \BR$, there is a $(1,C)$-quasi-isometry $\rho : (U,d_U) \longrightarrow I$ such that each component $S \subset \partial U$  maps to an endpoint of $I$,  in a way that induces  a bijection between the components of $\partial U$ and the  endpoints of $I$. \end {lem}
\begin {proof}
Pick a $d_U$-minimizing, unit speed geodesic $\gamma : I \longrightarrow U$ that realizes the width of $U$. Here, we can take $I=\BR, \ [0,\infty)$ or $I= [0,L] $ for some $ L.$ By Lemma \ref{level surfaces},  there is some uniform  constant $D=D(\epsilon,g)$  such that for each $x\in U$,  there is a level surface $S_x \subset U $ that is contained in the $D$-neighborhood of $x$,  say with respect to  the metric $d_U$. This $S_x$ must  intersect $\gamma$, so $$d_U(x,\gamma(I))\leq D.$$

Given $x \in U$, let $\rho(x)$ be any $t\in I$  for which the distance $d_U(x,\gamma(t))$ is minimal. Since $\gamma$ is $d_U$  minimizing, it follows that
 \begin {equation}
 	d_U(x,y)-2D \leq |\rho(x)-\rho(y)| \leq d_U(x,y)+2D.\label {qi2}
 \end {equation}
 Moreover, $\rho$ is surjective, since $\rho(\gamma(t))=t$. Hence, $\rho$ is a $(1,C)$-quasi-isometry.   As in the previous lemma, we can ensure that the components of $\partial U$  map to endpoints of $I$  by changing $\rho$ and increasing $C$.\end {proof}

\begin{kor}\label{qicor}
  There is some $C=C(k,\epsilon)$  such that the convex core of any  hyperbolic $3$-manifold $M$ with $$\rank(\pi_1(M))\le k, \ \  \inj(M)\ge\epsilon$$
is $(1,C)$-quasi-isometric to a metric graph  with at most $n=n(k)$ edges, some of which some of which may be half-infinite or biinfinite.\end{kor}
\begin {proof}
 The corollary is obvious if $CC(M)$ is not $3$-dimensional. For if it is $2$-dimensional, it is a totally geodesic surface with bounded genus, and hence has bounded diameter, by the Bounded  Diameter Lemma. If it is $1$-dimensional, it is a  closed geodesic, which is itself  a metric graph. If $M=\BH^3$, then $CC(M)=\emptyset$ is an empty graph. We can  therefore assume that Theorem \ref{general case} applies.

Take $C$  to be the maximum of the constants $12$ and $C$ from  Lemmas \ref{QI1} and \ref{QI2}, and the constant $B$  from Theorem \ref{general case}. Let $\mathcal H, \mathcal U$ be the $\epsilon$-thin $1$-handles and product regions in Theorem \ref{general case}. For each  $H \in \mathcal H$, pick an interval $I_H \subset \BR$ and a  $(1,C)$-quasi-isometry $$\rho_H : H \longrightarrow I_H $$ as in Lemma  \ref{QI1}, and define  for each $U \in \mathcal U$ a  $(1,C)$-quasi-isometry $\rho_U : H \longrightarrow I_U $  similarly using Lemma \ref{QI2}. Finally, for each  component $$N \subset CC(M) \setminus ((\cup_{U \in \mathcal U} U ) \cup (\cup_{H \in \mathcal H} H))$$
define a  point called $v_N$.

  Build a  graph $G$ with vertices the $v_N$ and edges the $I_H$ and $I_U$ as follows.   For every surface $S$  that is a component of some $\partial N$ and also  of some $\partial U$,  we identify the endpoint $\rho_U(S) \in I_U$ with $v_N$.  Similarly, if $S\subset\partial N$ and $S \subset \partial_v H$,  we identify $\rho_H(S) \in I_H$ with $v_N$. Then the maps $\rho_H,\rho_U$  combine  to give a  surjective map
$$\rho : CC(M) \longrightarrow G,$$
 which we can  define to be constant on each $N$.  Let $n=n(k)$  be an upper bound for the total number of thin $1$-handles $H$  and  product regions $U$,  as given by Theorem \ref{general case}. 

We claim  that the map $\rho$  is a $(1,C')$-quasi-isometry, for some uniform $C'=C'(C,n)$. So, let $x,y\in CC(M)$.   For simplicity, let us assume that $x,y$ do not lie in the union of all the thin $1$-handles and  product regions. We will comment on the modifications necessary to deal with the remaining case during the proof.

For the upper bound, take any minimizing geodesic $\gamma$  connecting points $x,y\in CC(M)$, and  orient $\gamma$ from $x$ to $y$. Then  there are  disjoint segments $\gamma_1,\ldots,\gamma_k$ of $\gamma$, indexed in the order they appear along $\gamma$, with
\begin {enumerate}
\item  the initial endpoint of $\gamma_1$  lies in the same complementary component $N$ as $x$, and the  terminal endpoint of $\gamma_k$  lies in the same complementary component as $y$,
\item each $\gamma_i$ is contained in  some $H \in \mathcal H$, or in some $U\in \mathcal U$, and  the initial and terminal endpoints of $\gamma_i$  lie on  distinct components of $\partial_v H$ or $\partial U$,  respectively,
\item the terminal  endpoint of $\gamma_i$  is  lies on the boundary of the same complementary component $N$ as the initial endpoint of $\gamma_{i+1}$,
\item if $i\neq j$, then $\gamma_i,\gamma_j$ do not lie in the same $1$-handle $H$ or product region $U$.
\end {enumerate}
Here, (4) implies that the number of segments $k$  is at most $n$. Then each $\gamma_i$ maps into some edge $e_i$ of $G$ under $\rho$, and the lengths $\gamma_i$ is within $C$ of  the length of $e_i$. By (3), $e_1,\ldots,e_k$ is  an edge path in $G$, so $$d_G(\rho(x),\rho(y)) \leq \sum_i \length \gamma_i + nC \leq d(x,y)+nC.$$
 Note that if $x$ or $y$ are allowed to lie in  thin $1$-handles $H$ or product regions $U$,  the same argument essentially works, but the constant will be $(n+1)C$ instead of $nC$ since  the path $\gamma$ may start and end inside of a single $H \in \mathcal H$, say, but may  exit this $H$ in between. (So, (4) will have an exception for $\{i,j\}=\{1,k\}.$)

 For the lower bound, take $x,y\in CC(M)$ and a length minimizing edge path $\gamma$ in $G$  connecting $\rho(x)$ and $\rho(y)$.  This path is simple, and hence traverses at most $n$ edges. For each  such edge, find an intrinsically minimizing geodesic in the associated thin $1$-handle or product region that  joins the corresponding (vertical,  if a $1$-handle) boundary components.
Concatenate these  geodesics with (at most $n+1$) minimizing geodesics in  complementary component $N$, to create a path from $x$ to $y$. Since $\rho$  restricts to a $(1,C)$-quasi-isometry on each segment of our concatenation,  the resulting path has length at most $\length(\gamma) - (2n+1)C.$ So, we have
$$d(x,y)-(2n+1)C \leq d_G(\rho(x),\rho(y)).$$
(The proof if $x,y$ lie in $1$-handles product regions is basically the same, but the constant will be $(2n+2)C$, for the same reason as above.)
\end {proof}

\begin{kor}\label{embedded balls cor}
	If $M$  is a hyperbolic $3$-manifold with rank  at most $k$ and injectivity radius at least $\epsilon$,  the radius $r_M$ of the largest embedded ball in $CC(M)$ is at most some  function of $k,\epsilon$.
\end{kor}
\begin{proof}
 This follows from the previous corollary. For  suppose that $$\rho : CC(M) \longrightarrow G$$ is a $(1,C)$-quasi-isometry, where $C=C(k,\epsilon)$ and $G$  is a graph with at most $n=n(k)$  edges. If $CC(M)$ contains an  embedded ball with  huge radius, it also contains an embedded ball  with (not quite as) huge radius whose image under $\rho$  is contained in a  single edge of $G$. But  for every $C$,  it is easy to see that there is some $r$  such that an $r$-ball in $\BH^3$  is not $(1,C)$-quasi-isometric to an interval. \end{proof}

\begin{kor}\label{genus cor}
The Heegaard genus $g(M)$ of a  hyperbolic $3$-manifold $M $ is bounded above by some function of $\rank(\pi_1 M)$ and $\inj(M)$.	
\end{kor}
\begin {proof}
 With the notation of Theorem \ref{general case}, each component
$$N \subset CC(M) \setminus ((\cup_{U \in \mathcal U} U ) \cup (\cup_{H \in \mathcal H} H))$$  is homeomorphic to some element of $\mathcal M$,  which is finite and depends only on $k,\epsilon$. So, each such $N$ has Heegaard genus at most some $g=g(k,\epsilon)$. But then a  standard compact core $C \subset M$  can be created topologically by gluing these components $N$  along subsurfaces  of their boundaries. One can then combine the Heegaard splittings for the components $N$  to give a Heegaard splitting for $C$ via  a process called \emph {amalgamation}, see Schultens \cite{schultens1993classification}.   Since  the number of components $N$  is bounded above  by some function of $k$,   the genus of the resulting splitting is bounded above by some function of $k,\epsilon$.
\end {proof}

%
%
%\section{Proof of Theorem \ref{main2}}
%
%\begin{sat}
%Given $k,\epsilon >0 $, there are compact hyperbolic $3$-manifolds with boundary $N_1,\ldots, N_d $ as follows.  Every closed hyperbolic $3$-manifold $M $ with $\rank (M)\leq k $ and $\text{inj}(M)\geq \epsilon $ contains a collection of at most $ C (k) $ disjoint, compact $3 $-submanifolds $M_i\subset M $  such that\label{bigconjecture}
%\begin{enumerate}
%\item  each $M_i $ is $2 $-bilipschitz to one of the manifolds $N_1,\ldots, N_d $.
%\item every component of $M \setminus \bigcup_i M_i $ is homeomorphic to a product $S \times (0,1) $ in such a way that the level surfaces $S \times \{ t\} $ have diameter and area bounded above by some $C (k,\epsilon) $.
%\end{enumerate}
%\end{sat}

\bibliographystyle{amsplain}
\bibliography{total}

\end{document}